\documentclass[11pt]{article}

\usepackage{fullpage}
\usepackage{authblk}
\usepackage{natbib}
\usepackage{algorithm}
\usepackage{algpseudocode}
\usepackage{amsmath,amsthm,amsfonts}
\usepackage{amsmath}
\usepackage[subnum]{cases}
\usepackage[most]{tcolorbox}
\usepackage{mathrsfs}
\usepackage{amssymb}
\usepackage{color}
\usepackage{mathrsfs}
\usepackage{empheq}
\usepackage{enumitem}
\usepackage{bm}
\usepackage{multirow}
\usepackage{booktabs}
\usepackage{makecell}
\usepackage{graphicx}
\usepackage{subcaption}
\usepackage{comment}
\usepackage{cases}
\usepackage{appendix}
\usepackage{tikz}
\usetikzlibrary{arrows,shapes}
\usepackage[colorlinks= true, linkcolor=brown, citecolor=blue, urlcolor=black]{hyperref}
\usepackage{cleveref}
\usepackage{marginnote}
\usepackage{diagbox} 
\usepackage{bbm}

\numberwithin{equation}{section}

\theoremstyle{definition}

\newtheorem{theorem}{Theorem}[section]
\newtheorem{lemma}[theorem]{Lemma}

\newtheorem{remark}[theorem]{Remark}

\newtheorem{condition}[theorem]{Condition}
\newtheorem{prop}[theorem]{Proposition}
\newtheorem{defn}[theorem]{Definition}

\newcommand{\defi}{\overset{\triangle}{=}}

\usepackage{booktabs}
\usepackage{pgfplots}
\usetikzlibrary{calc,intersections}
\usepackage{tikz}
\usetikzlibrary{patterns}

\usetikzlibrary{shapes.geometric, arrows, positioning}

\tikzset{
	mybox/.style  = {draw, rectangle, minimum width=2.0cm, minimum height=0.8cm, text centered, text width=2.0cm,   
		font=\normalsize,  align=left},
	box/.style  = {draw, rectangle, minimum width=2.0cm, minimum height=0.6cm, text centered, text width=7.4cm,   
		font=\normalsize,  align=left},		
	box1/.style  = {draw, rectangle, minimum width=2.0cm, minimum height=0.6cm, text centered, text width=5.6cm,   
		font=\normalsize},	
	box2/.style  = {draw, rectangle, minimum width=2.0cm, minimum height=0.6cm, text centered, text width=7.4cm,   
		font=\normalsize,  align=left},	
	myarrow/.style = {line width=0.2pt, draw=black, -triangle 60, postaction={draw, line width=0.2pt, shorten >=10pt,-}}	
}

\tikzstyle{arrow} = [->, >=stealth, -triangle 60]

\allowdisplaybreaks

\makeatletter
\newcommand{\leqnomode}{\tagsleft@true}
\newcommand{\reqnomode}{\tagsleft@false}
\makeatother

\begin{document}

\title{Numerical Construction of Quasi-Periodic Solutions for Nonlinear PDEs:  I. Bounded Perturbation}
\author[1,2]{Mingwei Fu}
\author[1,2]{Bin Shi\thanks{Corresponding author: \url{binshi@fudan.edu.cn} } }
\affil[1]{Center for Mathematics and Interdisciplinary Sciences, Fudan University, Shanghai 200433, China}
\affil[2]{Shanghai Institute for Mathematics and Interdisciplinary Sciences, Shanghai 200433, China}

\date\today

\maketitle

\begin{abstract}
In this paper, we propose an alternating numerical scheme for constructing 
quasi-periodic solutions to a class of nonlinear partial differential 
equations (PDEs) with bounded perturbations, including the one-dimensional 
nonlinear Schr\"odinger (NLS) equations and nonlinear wave (NLW) equations 
under periodic and Dirichlet boundary conditions.  A key observation concerning the resonant $Q$-equations allows us to establish directly a local diffeomorphism between the amplitudes and the drifted frequencies. This renders the proof considerably cleaner, since it avoids the Birkhoff normal form and the associated coordinate transformations required in classical KAM theory for PDEs.
Unlike in the finite-dimensional setting, the lattice boxes must also include 
the spatial index, so that the multiscale induction has to be carried out in 
one additional dimension. 
This introduces two difficulties. 
First, the measure estimate for the excluded set of ``bad'' frequencies must 
account for the spatial index. 
For both the NLS and NLW equations, the linear frequencies and their 
differences are integers or close to integers; this allows the spatial 
dependence to be controlled through the modulus of the temporal index, so 
that the excluded measure remains under control. 
Second, a singular set may contain sites with the same temporal index but 
widely separated spatial indices, which obstructs the 
``inversion implies localization'' process. 
For the one-dimensional NLS and NLW equations, however, such sites occur 
only in pairs with opposite spatial indices. 
They can therefore be separated by small boxes, after which the global 
inverse is obtained from the local inverses via the resolvent identity. 
Numerical experiments illustrate the scheme for nonlinear traveling and 
standing waves, exhibiting both the computed quasi-periodic solutions and the 
convergence of the iteration.
\end{abstract}

%

\section{Introduction}
\label{sec: intro}

Nearly integrable Hamilton systems constitute a central class in modern 
dynamical systems. 
They arise naturally in a broad range of physical models; a prominent example 
is the dynamics of the Solar system, often investigated through simplified 
models such as the three-body problem. 
Although small perturbations may profoundly affect the long-term behavior of 
an integrable system, much of its regular structure typically persists. 
Under suitable nondegeneracy conditions, a large-measure set of invariant tori 
survives the perturbation and continues to support quasi-periodic motions, 
thereby providing a fundamental mechanism for long-term dynamical stability. 
This persistence phenomenon lies at the heart of Kolmogorov-Arnold-Moser (KAM) 
theory, which originated in the seminal works 
of~\citet{kolmogorov1954preservation, arnold1963proof} 
and~\citet{moser1962invariant}. 
The classical theory was subsequently extended to lower-dimensional invariant 
tori; in particular,~\citet{eliasson1988perturbations} 
and~\citet{poschel1989elliptic} established fundamental persistence results 
for elliptic lower-dimensional tori.

The extension of KAM theory to Hamiltonian partial differential equations (PDEs) 
was initiated by~\citet{kuksin1993nearly}. 
Unlike their finite-dimensional counterparts, Hamiltonian PDEs give rise to 
infinite-dimensional dynamical systems whose spectral properties are governed 
by dispersion relations linking spatial and temporal scales. 
This infinite-dimensional setting introduces substantial analytical 
challenges, most notably the requirement to control infinitely many small 
divisors. 
A major breakthrough was achieved by~\citet{kuksin1996invariant}, who 
established the existence of small-amplitude quasi-periodic solutions to the 
one-dimensional nonlinear Schr\"odinger (NLS) equation subject to 
Dirichlet boundary conditions.~\citet{poschel1996quasi} subsequently extended 
this approach to the one-dimensional nonlinear wave (NLW) equation, likewise 
under Dirichlet boundary conditions. 
Complementing these existence results,~\citet{cong2016stability} established 
stability over polynomially long time scales for the one-dimensional NLS 
equation under sufficiently small perturbations of the initial data.

In parallel,~\citet{craig1993newton,craig1994periodic} developed an alternative 
approach based on space-time Fourier expansions, which transform a nonlinear 
PDE into an infinite system of nonlinear lattice equations for its Fourier 
coefficients. 
The resulting system is solved by coupling a Lyapunov-Schmidt decomposition 
with a Newton-type iteration on an expanding sequence of 
finite-dimensional subspaces. 
Using this framework, they constructed time-periodic solutions, characterized 
by a single temporal frequency, to the one-dimensional NLW and NLS equations 
under both periodic and Dirichlet boundary conditions.~\citet{bourgain1994construction} 
subsequently incorporated multiscale analysis into this scheme and established 
the existence of quasi-periodic solutions associated with low spatial 
frequencies. 
This methodology is now widely known as the Craig-Wayne-Bourgain (CWB) scheme. 
A major advantage of the CWB scheme is its ability to handle the growing 
multiplicities of spatial eigenvalues and the resulting clusters of small 
divisors that arise in higher spatial dimensions. 
Although nearly resonant sites may form clusters, the underlying geometry of 
the resonant sets ensures that these clusters remain of controlled size and are 
well separated. 
This separation allows local inverse estimates on individual clusters to be 
combined within the multiscale analysis, thereby enabling control of the large 
matrices generated by multiple spatial eigenvalues. 
By exploiting this geometric structure,~\citet{bourgain1995construction} 
constructed time-periodic solutions to higher-dimensional NLW equations, and 
later established the existence of quasi-periodic solutions to two-dimensional NLS 
equations under periodic boundary conditions~\citep{bourgain1998quasi}.

Motivated by the constructive nature of the CWB scheme, we recently developed 
an executable numerical framework for finite-dimensional nonlinear dynamical 
systems. 
The method combines a dimension-enlarged Newton iteration with alternating 
updates of the Fourier coefficients and frequencies. 
It was first applied to the construction of full-dimensional quasi-periodic 
solutions~\citep{Fu2026numerical-a}, and subsequently extended to elliptic 
lower-dimensional quasi-periodic solutions~\citep{Fu2026numerical-b}. 
In this paper, we further extend this framework to nonlinear PDEs, addressing 
the core computational and analytical challenges posed by their 
infinite-dimensional spatial dynamics.

\subsection{Fourier expansion for the NLS equation}
\label{subsec: fourier-nls}

We consider the one-dimensional NLS equation on the interval $[0, 2\pi]$ with 
periodic boundary conditions,
\begin{subequations}
    \label{eqn: nls-periodic}
    \begin{empheq}[left=\empheqlbrace]{align} 
       &iu_t - u_{xx} + \mu u + \varepsilon |u|^2 u = 0,  \label{eqn: nls}  \\
       &u(0, t) = u(2\pi, t).                             \label{eqn: per-bdy} 
    \end{empheq}        
\end{subequations} 
which is a Hamiltonian PDE. 
The associated Hamiltonian is $H = H_0 + \varepsilon H_1$, where
\[
    H_0 = \frac12 \int_{0}^{2\pi} |u_x|^2 + \mu |u|^2dx \quad \mathrm{and} \quad H_1 = \frac14 \int_{0}^{2\pi} |u|^4dx.
\]
Throughout the paper, we assume that $\mu \in \mathbb{R}$ and that the 
perturbation parameter satisfies $0 < \varepsilon \ll 1$. 
Since $u = u(x, t)$ is $2\pi$-periodic in $x$, we expand it in a spatial 
Fourier series:
\begin{equation}
    \label{eqn: fourier-expansion}
    u(x,t) = \sum_{n\in\mathbb Z} \hat{u}_n(t)e^{inx}, 
\end{equation}
where the Fourier coefficients are given by
\begin{equation}
    \label{eqn: fourier-coefficient}
    \hat{u}_n(t) = \frac{1}{2\pi} \int_{0}^{2\pi} u(x, t) e^{-inx}dx. 
\end{equation}
Hereafter, we suppress the explicit $t$-dependence of the coefficients for 
notational simplicity. 
It follows directly from~\eqref{eqn: fourier-coefficient} that 
$\hat{\overline{u}}_n = \overline{\hat{u}_{-n}}$. 
Substituting~\eqref{eqn: fourier-expansion} into~\eqref{eqn: nls}, we obtain 
an infinite system of coupled ordinary differential equations (ODEs) for each 
Fourier coefficient $\hat{u}_n$, $n \in \mathbb{Z}$: 
\begin{equation}
    \label{eqn: nls-fourier-ode}
    i \dot{\hat u}_n + (n^2 + \mu) \hat{u}_n + \varepsilon \sum_{\substack{n_1 + n_2 + n_3 = n  \\  n_1, n_2, n_3 \in \mathbb{Z}}} \hat{u}_{n_1} \hat{\overline{u}}_{n_2} \hat{u}_{n_3} = 0,  
\end{equation}
where the middle factor $\hat{\overline{u}}_{n_2}$ is the $n_2$-th Fourier 
coefficient of $\overline{u}$.

The associated linear equation, obtained by setting $\varepsilon = 0$, is 
given by:
\begin{equation}
    \label{eqn: nls_linear}   
    iu_t - u_{xx} + \mu u = 0,
\end{equation}
whose dispersion relation is therefore
\begin{equation}
    \label{eqn: nls-linear-frequency}
    \varpi_n = n^2 + \mu, 
\end{equation}
for any $n \in \mathbb{Z}$. 
We select a finite set of $m$ distinct wavenumbers in the low-frequency range, 
referred to as the tangential modes,
\begin{equation}
    \label{eqn: tangent-spatial-mode}
    \mathcal{N} = \left\{ n_1 < n_2 < \cdots < n_m \right\} \subset \mathbb{Z}, 
\end{equation}
with $\max\{ |n_1|, |n_{m}| \} \leq M_{0}^{\frac13}/2 - 1$, where 
$M_0=M_0(\varepsilon)$ is a truncation parameter depending on the perturbation 
parameter $\varepsilon$, to be specified in~\eqref{eqn: M0}. 
By~\eqref{eqn: nls-linear-frequency}, the corresponding frequencies are 
$\omega_{j} := \varpi_{n_j}$ for $j=1, \ldots, m$, and we collect them in the 
frequency vector $\omega = (\omega_{1}, \ldots, \omega_{m}) \in \mathbb{R}^{m}$. 
A solution of~\eqref{eqn: nls_linear} supported on these modes is the linear 
superposition
\begin{equation}
    \label{eqn: lin-soln}
    u(x, t) = \sum_{j=1}^{m} \sqrt{I_{0,j}} e^{i \left(\omega_{j} t + \theta_{0,j} \right)} e^{in_jx},
\end{equation}
where $I_0 = (I_{0,1}, \ldots, I_{0,m}) \in \mathbb{R}_{+}^{m}$ and 
$\theta_0 = \left(\theta_{0, 1}, \ldots, \theta_{0, m} \right) \in \mathbb{T}^m$ 
are the initial action and angle variables, respectively. 
For convenience, we set $a_j = \sqrt{I_{0,j}} $ for $j=1, \ldots, m$, and assume 
that the amplitudes of the selected tangential modes are bounded away from zero, 
with $\min_{1 \leq j \leq m} a_j \geq \delta_0 > 0$. 
Denoting by $e_j$ the $j$-th standard basis vector of $\mathbb{R}^m$ for 
$j=1, \ldots, m$, we may then rewrite~\eqref{eqn: lin-soln} as
\begin{equation}
    \label{eqn: lin-soln-new}
    u(x, t) = \sum_{j=1}^{m} a_j e^{i \left\langle e_j, \; \omega t + \theta_0 \right\rangle} e^{in_jx}. 
\end{equation}
Accordingly, we define the resonant set
\begin{equation}
    \label{eqn: resonant-index}
    \mathcal{S} := \left\{ (n_j, e_j) \in \mathbb{Z}^{1+m} ~\big|~ j = 1, \ldots, m \right\}. 
\end{equation}
Expanding $u=u(x, t)$ over the space-time Fourier lattice indexed by the 
multi-index $(n,k) \in \mathbb{Z}^{1+m}$, the linear 
solution~\eqref{eqn: lin-soln-new} is represented by the coefficients 
$\hat{u}(n_j, e_j) = a_j$ for $j=1, \ldots, m$ and $\hat{u}(n, k) = 0$ for 
each non-resonant index $(n, k) \notin \mathcal{S}$. 
For the nonlinear equation~\eqref{eqn: nls-periodic}, we seek a quasi-periodic 
solution of the form
\begin{equation}
    \label{eqn: nonlin-soln}
    u(x, t) = \sum_{(n, k) \in \mathbb{Z}^{1+m}} \hat{u}(n, k) e^{i\left(\langle k,\; \omega' t + \theta_0 \rangle + n x\right)},
\end{equation}
where the coefficients of the resonant modes are fixed by 
$\hat{u}(n_j, e_j) = a_j$ for $j = 1, \ldots, m$, while the frequency vector 
$\omega' \in \mathbb{R}^m$, a small drift of the linear frequency vector 
$\omega$, is to be determined together with the non-resonant coefficients 
$\hat{u}(n, k) $ for any $(n, k) \notin \mathcal{S}$.

\subsection{Numerical scheme}
\label{subsec: numerical-scheme}

In this section, we propose a numerical scheme that iteratively deforms the 
linear solution~\eqref{eqn: lin-soln-new} into the nonlinear quasi-periodic 
solution~\eqref{eqn: nonlin-soln}. 
Substituting the ansatz~\eqref{eqn: nonlin-soln} into the one-dimensional NLS 
equation~\eqref{eqn: nls-periodic} and matching the coefficients of 
$e^{i \left( \langle k, \; \omega' t + \theta_0 \rangle + nx \right)}$, we 
obtain the lattice equation
\begin{equation}
    \label{eqn: nls-lattice}
    \left( - \langle k, \omega' \rangle + \varpi_{n} \right) \hat{u}(n, k) + \varepsilon \widehat{|u|^2u}(n,k) = 0,
\end{equation}
for each multi-index $(n, k) \in \mathbb{Z}^{1+m}$, where the nonlinear term 
is given by the convolution 
\begin{equation}
    \label{eqn: nonlinear-convolution}
    \widehat{|u|^2u}(n, k) = \sum_{n_1 + n_2 + n_3 = n} \sum_{k_1 + k_2 + k_3 = k} \hat{u}(n_1, k_1) \hat{\overline{u}}(n_2, k_2) \hat{u}(n_3, k_3).  
\end{equation}
Collecting the coefficients $\hat{u}(n, k)$ into a vector $\hat{u}$, we may 
write~\eqref{eqn: nls-lattice} in vector form as 
\begin{equation}
    \label{eqn: nls-lattice-vector}
    -\langle k, \omega' \rangle \hat{u} + \varpi \odot \hat{u} + \varepsilon \hat{X}(\hat{u}) = 0,
\end{equation}
where $\hat{X}(\hat{u})$ denotes the nonlinear vector field and the symbol 
$\odot$ denotes componentwise multiplication with respect to the spatial index 
$n$. 
According to the resonant set $\mathcal{S}$ defined 
in~\eqref{eqn: resonant-index}, we decompose the Fourier coefficients into 
resonant and non-resonant components, $\hat{u} = (\hat{u}_q, \hat{u}_p)^\top$ 
and $\hat{X} = (\hat{X}_q, \hat{X}_p)^\top$, where $\hat{u}_q$ and $\hat{X}_q$ 
denote the components indexed by $\mathcal{S}$, whereas $\hat{u}_p$ and 
$\hat{X}_p$ collect those indexed by $(\mathbb{Z}^{1+m}) \setminus \mathcal{S}$. 
Recall from~\Cref{subsec: fourier-nls} that the resonant coefficients are 
prescribed by $\hat{u}(n_j, e_j) = a_j$ for $j=1, \ldots, m$, so that 
$\hat{u}_q = a := (a_1, \ldots, a_m)^{\top}$ is known. 
Restricting~\eqref{eqn: nls-lattice-vector} to the resonant modes, we obtain 
the $Q$-equations as 
\begin{equation}
    \label{eqn: q-eqn}
    (- \omega' + \omega) \odot a + \varepsilon \hat{X}_q = 0,
\end{equation}
which serve to determine the frequency vector $\omega'$. 
The remaining equations, referred to as the $P$-equations, are
\begin{equation}
    \label{eqn: p-eqn}
    - \langle k, \omega' \rangle \hat{u}_p + \varpi \odot \hat{u}_p + \varepsilon \hat{X}_p = 0,
\end{equation}
which determine the unknown non-resonant coefficients $\hat{u}_p$. 
Since $\hat{X}_q$ and $\hat{X}_p$ both depend on the full vector 
$\hat{u} = (a, \hat{u}_p)^{\top}$, the two systems are coupled and must be 
solved simultaneously for the unknowns $(\omega', \hat{u}_p)$. 
This $Q/P$ splitting forms the basis of the alternating numerical scheme 
described below.

Since the $P$-equations~\eqref{eqn: p-eqn} constitute an infinite-dimensional 
nonlinear algebraic system, we solve them numerically by a dimension-enlarged 
Newton scheme. 
To this end, we define the nonlinear operator $F$, acting on the non-resonant 
Fourier coefficients $\hat{u}_p$, as
\begin{equation}
    \label{eqn: f-operator}
    F(\omega', \hat{u}_p; a) = - \langle k, \omega' \rangle \hat{u}_p + \varpi \odot \hat{u}_p + \varepsilon \hat{X}_p(\hat{u}_p, a).   
\end{equation}
The associated linearized, or tangent, operator is $T = D + \varepsilon S$, 
where $D$ is diagonal over the non-resonant set 
$\mathbb{Z}^{1+m} \setminus \mathcal{S}$, with entries
\begin{equation}
    \label{eqn: D-linear}
    D_{n, k} = - \langle k, \omega' \rangle + n^2 + \mu, 
\end{equation}
and $S$ is the Jacobian of the nonlinear perturbation with respect to 
$\hat{u}_p$: 
\begin{equation}
    \label{eqn: S-non-linear}
    S = \frac{\partial \hat{X}_p(\hat{u}_p, a)}{\partial \hat{u}_p}.
\end{equation}
Let $M =M(\varepsilon)$, to be specified in~\eqref{eqn: M}, and set the 
truncation scale $N_r := M^{r+1}$. 
Based on the $Q$-equations~\eqref{eqn: q-eqn} and the 
$P$-equations~\eqref{eqn: p-eqn}, we propose the following alternating 
numerical scheme: at each iteration, we first update the frequency vector and 
then the non-resonant coefficients,
\begin{subequations}
    \label{eqn: alternating-numerical}
    \begin{empheq}[left=\empheqlbrace]{align} 
       &\omega^{(r+1)} = \omega + \varepsilon \hat{X}_q(\hat{u}_p^{(r)}; a) \odot a^{-1},                                                                                                                     \label{eqn: alternating-frequency}  \\
       &\hat{u}_p^{(r+1)} = \hat{u}_p^{(r)} - \left[ (T + \varepsilon B)_{N_{r+1}}^{-1}\left( \hat{u}_p^{(r)}, \omega^{(r+1)}; a \right) \right] F\left( \hat{u}_p^{(r)}, \omega^{(r+1)}; a \right).  \label{eqn: alternating-coefficient} 
    \end{empheq}        
\end{subequations} 
where $a^{-1} = (a_1^{-1}, \ldots, a_{m}^{-1})^\top$ denotes the componentwise 
reciprocal of $a$. 
The correction operator $B$ is introduced to ensure the convergence of the 
iteration and is given by
\begin{equation}
    \label{eqn: operator-b}
    B = - \frac{\partial \langle k, \hat{X}_q \odot a^{-1} \rangle}{\partial \hat{u}_p} \hat{u}_p^{\top}.
\end{equation}
The restricted operator is defined as 
$(T + \varepsilon B)_N = P_N (T + \varepsilon B) P_N$, where $P_N$ is the 
projection onto the lattice box $\Lambda_N^{1+m}$, given by
\[
(P_N \hat{u})(n, k) := 
\left\{ 
    \begin{aligned}
        & \hat{u}(n, k), && \mathrm{if}\; (n, k) \in \Lambda^{1+m}_{N}, \\
        & 0,             && \mathrm{if}\; (n, k) \notin \Lambda^{1+m}_{N},   
    \end{aligned} 
\right.
\]
with $\Lambda_N := \left\{ n \in \mathbb{Z} ~\big|~ |n| \leq N \right\}$. 
Throughout this paper, unless otherwise specified, $|\cdot|$ and 
$|\cdot|_{\infty}$ denote the $\ell_1$-norm and the $\ell_{\infty}$-norm on 
$\mathbb{Z}^m$ or $\mathbb{R}^m$, respectively.

\subsection{Numerical demonstration}
\label{subsec: numerical-perform}

In this section, we demonstrate the performance of the alternating numerical 
scheme~\eqref{eqn: alternating-numerical} for the NLS equation with periodic 
boundary conditions~\eqref{eqn: nls-periodic}.

\paragraph{Periodic solutions ($m = 1$)}
As a benchmark case, we first consider periodic solutions, corresponding to 
$m=1$. 
The linear solution~\eqref{eqn: lin-soln} then reduces to 
\begin{equation}
    \label{eqn: lin-1}
    u(x, t) = a_1 e^{i \left( n_1 x + \omega_{1}t + \theta_{0, 1} \right)}.
\end{equation}
The associated resonant set consists of a single mode, 
$\mathcal{S} = \left\{ (n_1, 1) \right\}$, with Fourier coefficients 
$\hat{u}(n_1, 1) = a_1$ and $\hat{u}(n, k) = 0$ for all 
$(n, k) \notin \mathcal{S}$. 
Substituting~\eqref{eqn: lin-1} into the nonlinear term yields
\[
    |u|^2u = a_1^3 e^{i \left( n_1 x + \omega_{1}t + \theta_{0, 1} \right)}, 
\]
which has the same Fourier support as the linear solution. 
Thus, the nonlinearity generates no additional modes, and the non-resonant 
equation $F=0$ is satisfied with $\hat{u}_p = 0$. 
The frequency update in the alternating numerical 
scheme~\eqref{eqn: alternating-numerical} gives
\[
    \omega' = \omega_{1} + \varepsilon a_1^2 = n_1^2 + \mu + \varepsilon a_1^2. 
\]
Consequently, the nonlinear periodic solution is given by
\begin{equation}
    \label{eqn: non-lin-1}
    u(x, t) = a_1 e^{i \left[ n_1 x + \left(\omega_{1} + \varepsilon a_1^2 \right) t + \theta_{0, 1} \right]}.
\end{equation}
Comparing~\eqref{eqn: lin-1} and~\eqref{eqn: non-lin-1}, we see that the nonlinearity 
introduces only a frequency shift of $\varepsilon a_1^2$, leaving the amplitude 
and spatial profile unchanged.

\paragraph{Quasi-periodic solutions ($m = 2$)}
We now consider $m=2$, for which the linear quasi-periodic 
solution~\eqref{eqn: lin-soln} takes this form 
\begin{equation}
    \label{eqn: lin-2}
    u(x, t) = a_1 e^{i \left( n_1 x + \omega_{1}t + \theta_{0, 1} \right)} + a_2 e^{i \left( n_2 x + \omega_{2}t + \theta_{0, 2} \right)}.
\end{equation}
For simplicity, we set the tangential mode indices $n_1=1$ and $n_2=2$, the 
amplitude $(a_1, a_2)=(1, 1)$, the initial phase 
$\theta_0 = (\theta_{0, 1}, \theta_{0, 2}) = (0, 0)$, and the parameter 
$\mu=\sqrt{2}$. 
The corresponding frequency vector is 
$\omega = (\omega_1, \omega_2) = \left( 1 + \sqrt{2}, \; 4 + \sqrt{2} \right)$, 
so~\eqref{eqn: lin-2} reduces to 
\begin{equation}
    \label{eqn: lin-2-simple}
    u(x, t) = e^{i \omega_1 t} e^{ix} + e^{i \omega_2 t} e^{2ix} = \hat{u}_1(t) e^{ix} + \hat{u}_2(t) e^{2ix}, 
\end{equation}
where $\hat{u}_j(t) = e^{i \omega_j t }$ for $j=1, 2$. 
The nonlinear quasi-periodic solution~\eqref{eqn: nonlin-soln} can be expressed as
\begin{equation}
    \label{eqn: non-lin-2}
    u(x, t) = \sum_{n \in \mathbb{Z}} \hat{u}_n(t) e^{inx} = \sum_{n \in \mathbb{Z}} \left( \sum_{k \in \mathbb{Z}^2} \hat{u}(n, k) e^{i \langle k, \; \omega'\rangle t} \right) e^{inx},
\end{equation}
where $\omega' = (\omega_1', \omega_2')$ is the updated frequency vector. 
Since the resonant Fourier coefficients are fixed at 
$\hat{u}(1, (1, 0)) = \hat{u}(2, (0, 1)) = 1$, the spatial Fourier coefficients 
satisfy
\begin{subequations}
    \label{eqn: nls-quasi-periodic-2}
    \begin{empheq}[left=\empheqlbrace]{align} 
        & \hat{u}_1(t) = e^{i \omega_1' t} + \sum_{k \in \mathbb{Z}^2 \setminus \{(1, 0)\}} \hat{u}(1, k) e^{i \langle k, \; \omega'\rangle t},   \label{eqn: nls-quasi-periodic-2-1}  \\
        & \hat{u}_2(t) = e^{i \omega_2' t} + \sum_{k \in \mathbb{Z}^2 \setminus \{(0, 1)\}} \hat{u}(2, k) e^{i \langle k, \; \omega'\rangle t},   \label{eqn: nls-quasi-periodic-2-2}  \\
        & \hat{u}_n(t) = \sum_{k \in \mathbb{Z}^2} \hat{u}(n, k) e^{i \langle k, \; \omega'\rangle t }, \qquad \qquad \qquad \qquad n \in \mathbb{Z} \setminus \{ 1, 2 \}.                                                                                                                                 \label{eqn: nls-quasi-periodic-2-3} 
    \end{empheq}        
\end{subequations} 
where both $\hat{u}_1(t)$ and $\hat{u}_2(t)$ correspond to the tangential modes, 
while $\hat{u}_n(t)$, $n \in \mathbb{Z} \setminus \{1, 2\}$, correspond to the 
normal modes. 
In the numerical computation, we set $\varepsilon = 0.2$. 
As shown in~\Cref{fig: nls-p-2tori-traj}, the tangential modes maintain amplitudes 
close to unity but exhibit noticeable frequency shifts relative to the linear 
solution. 
In contrast, the representative normal mode $\hat{u}_3(t)$, which is absent 
from the linear solution, develops small-amplitude oscillations due to 
nonlinear interactions.
\begin{figure}[htb!]
    \centering
    \begin{subfigure}[t]{0.3\linewidth}
    \centering
    \includegraphics[scale=0.18]{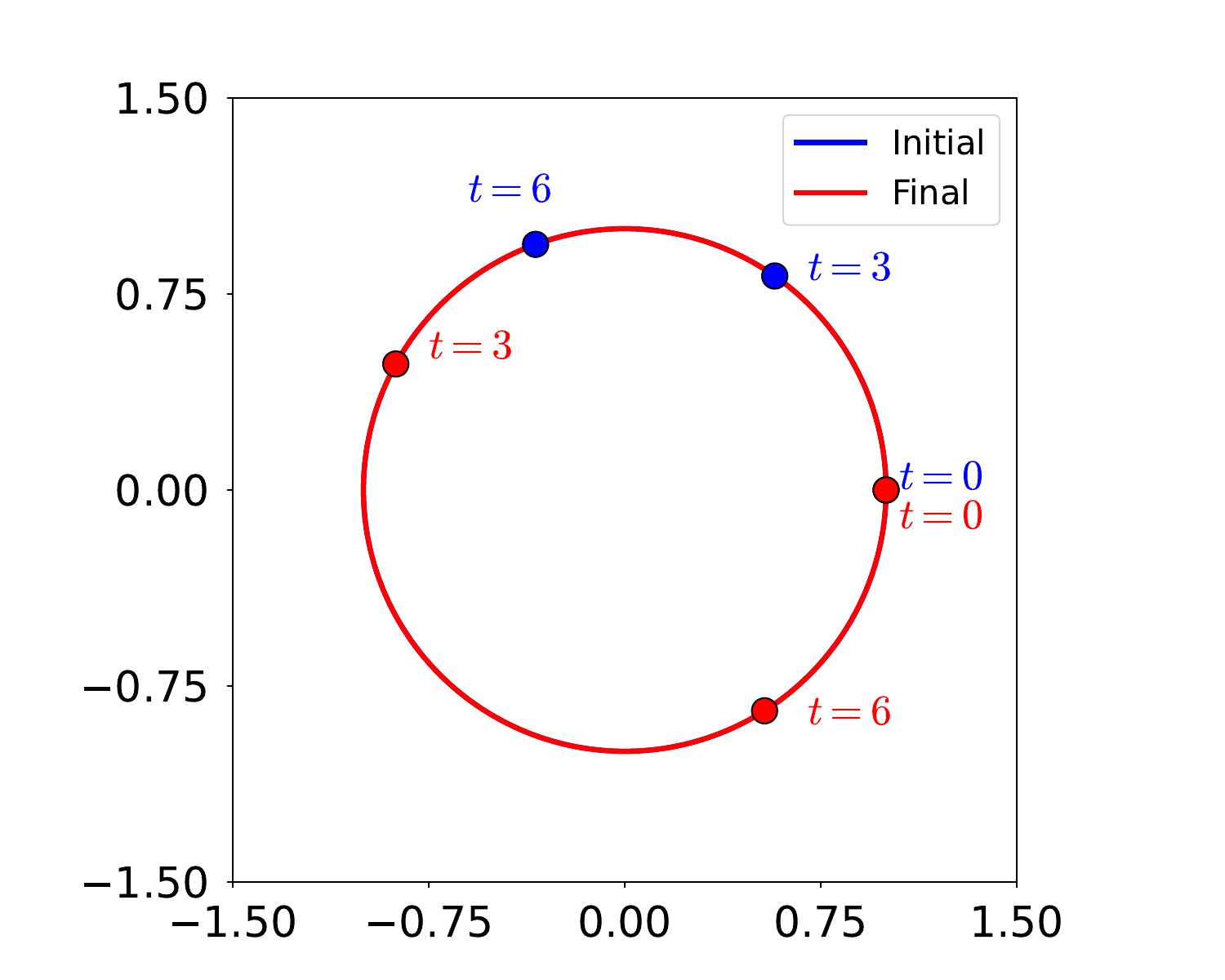}
    \caption{Tangential mode: $\hat{u}_1(t)$}
    \end{subfigure}
    \begin{subfigure}[t]{0.3\linewidth}
    \centering
    \includegraphics[scale=0.18]{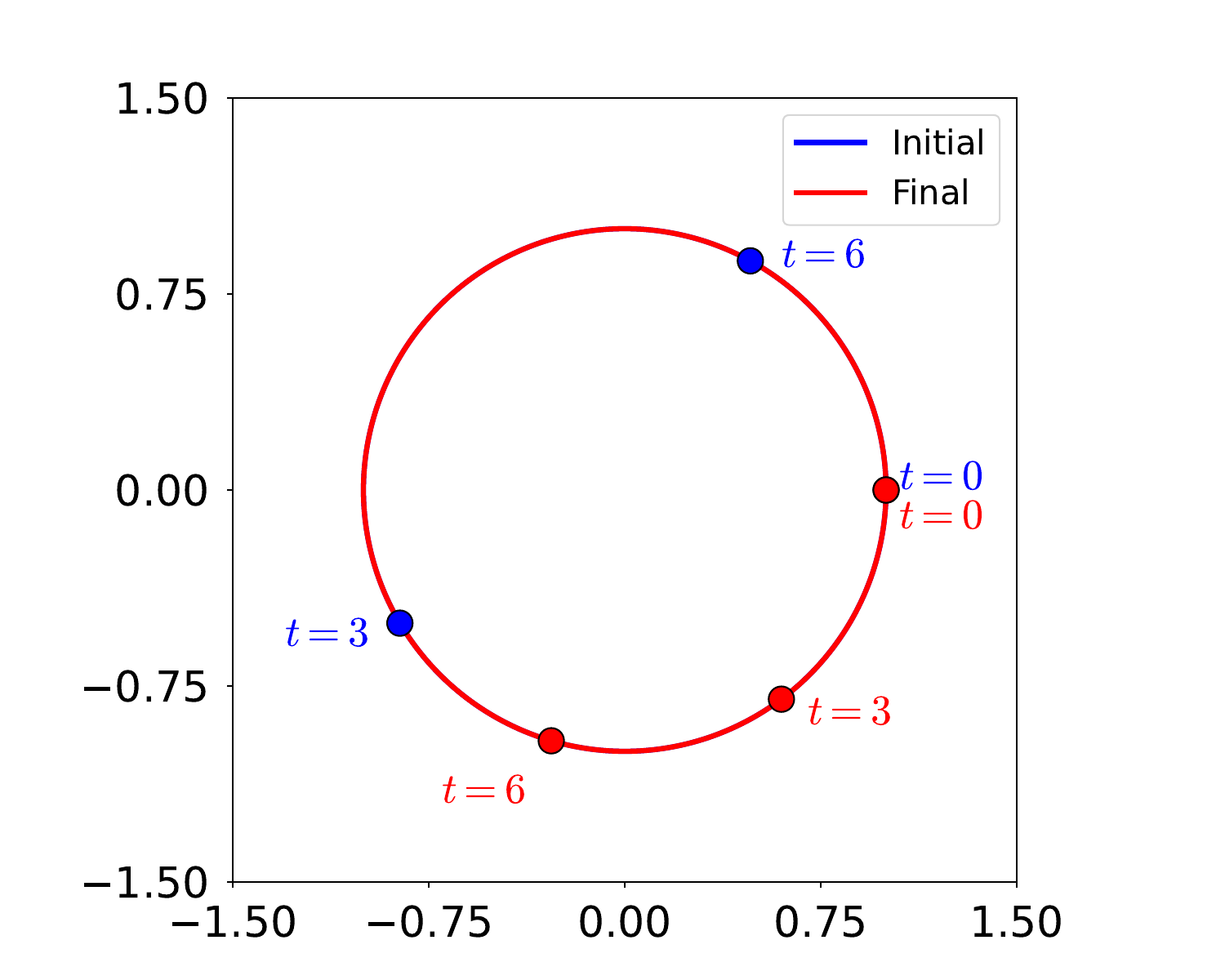}
    \caption{Tangential mode: $\hat{u}_2(t)$}
    \end{subfigure}
    \begin{subfigure}[t]{0.3\linewidth}
    \centering
    \includegraphics[scale=0.18]{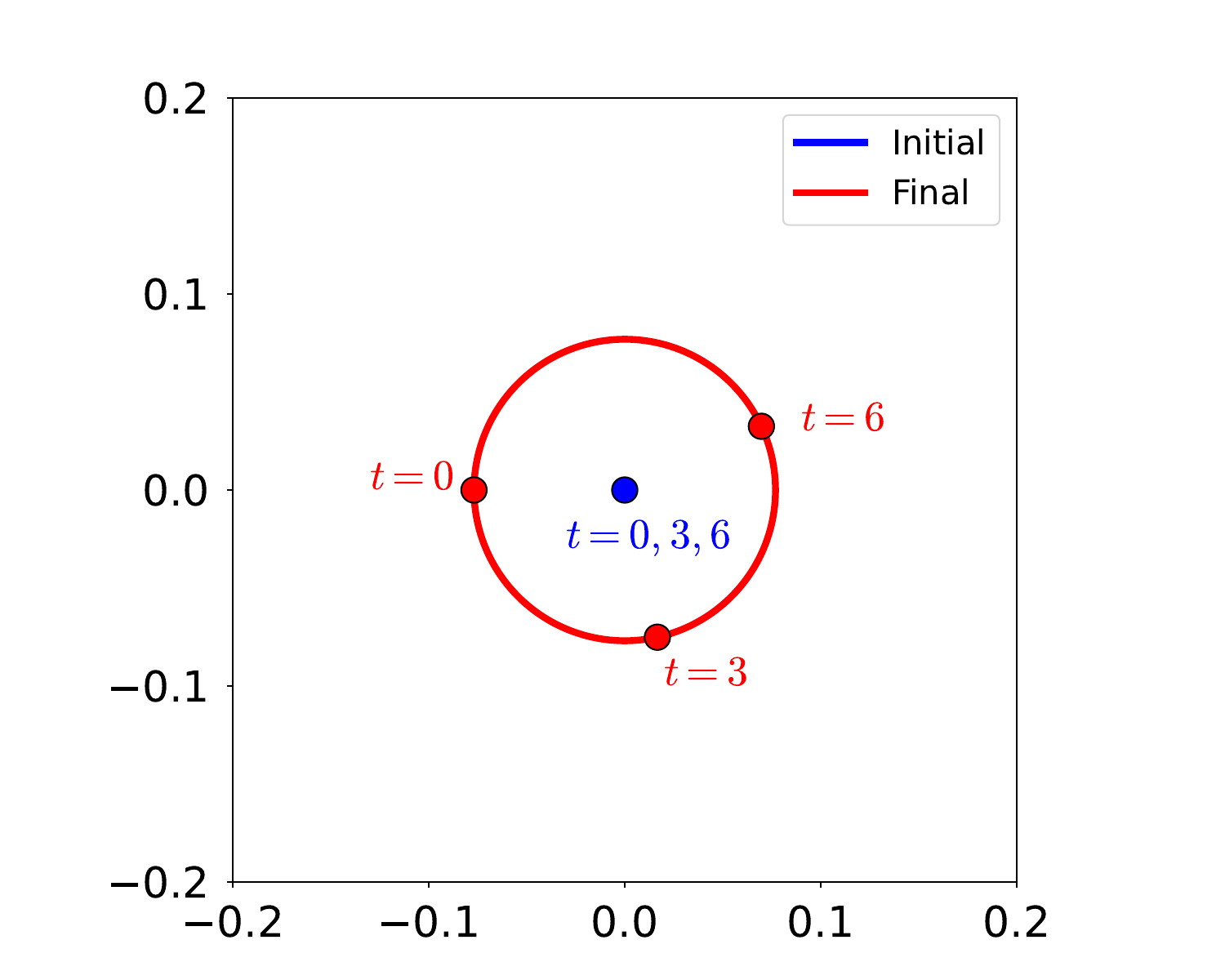}
    \caption{Normal mode: $\hat{u}_3(t)$}
    \end{subfigure}
    \caption{Phase-space trajectories of the NLS equation with periodic 
             boundary conditions~\eqref{eqn: nls-periodic}. 
             ``Initial'' and ``Final'' denote the initial linear solution and 
             the computed nonlinear solution, respectively. 
             Markers indicate specific time points at $t = 0$, $3$, and $6$.} 
    \label{fig: nls-p-2tori-traj}
\end{figure}

We next examine the spatiotemporal distributions of $|u(x,t)|^2$. 
Taking the squared modulus of the linear solution~\eqref{eqn: lin-2-simple} 
yields
\begin{equation}
    \label{eqn: lin-square}
    |u|^2 = 2 + 2 \cos\left( x + 3t \right),
\end{equation}
which depends only on the combined phase $x+3t$ and thus exhibits a 
traveling-wave structure with propagation velocity, 
as shown in~\Cref{subfig: nls-p-2tori-heat-slice-initial}.~\Cref{subfig: nls-p-2tori-heat-slice-final} 
displays the corresponding distribution for the nonlinear solution obtained 
using the alternating numerical scheme~\eqref{eqn: alternating-numerical}. 
Compared with the linear case, the nonlinear wave pattern has lower, flatter 
crests and narrower troughs. 
The propagation bands are also more nearly horizontal in the $(x, t)$-plane, 
indicating a slightly higher propagation speed in the negative $x$-direction.
\begin{figure}[htb!]
    \centering
    \begin{subfigure}[t]{0.42\linewidth}
    \centering
    \includegraphics[scale=0.24]{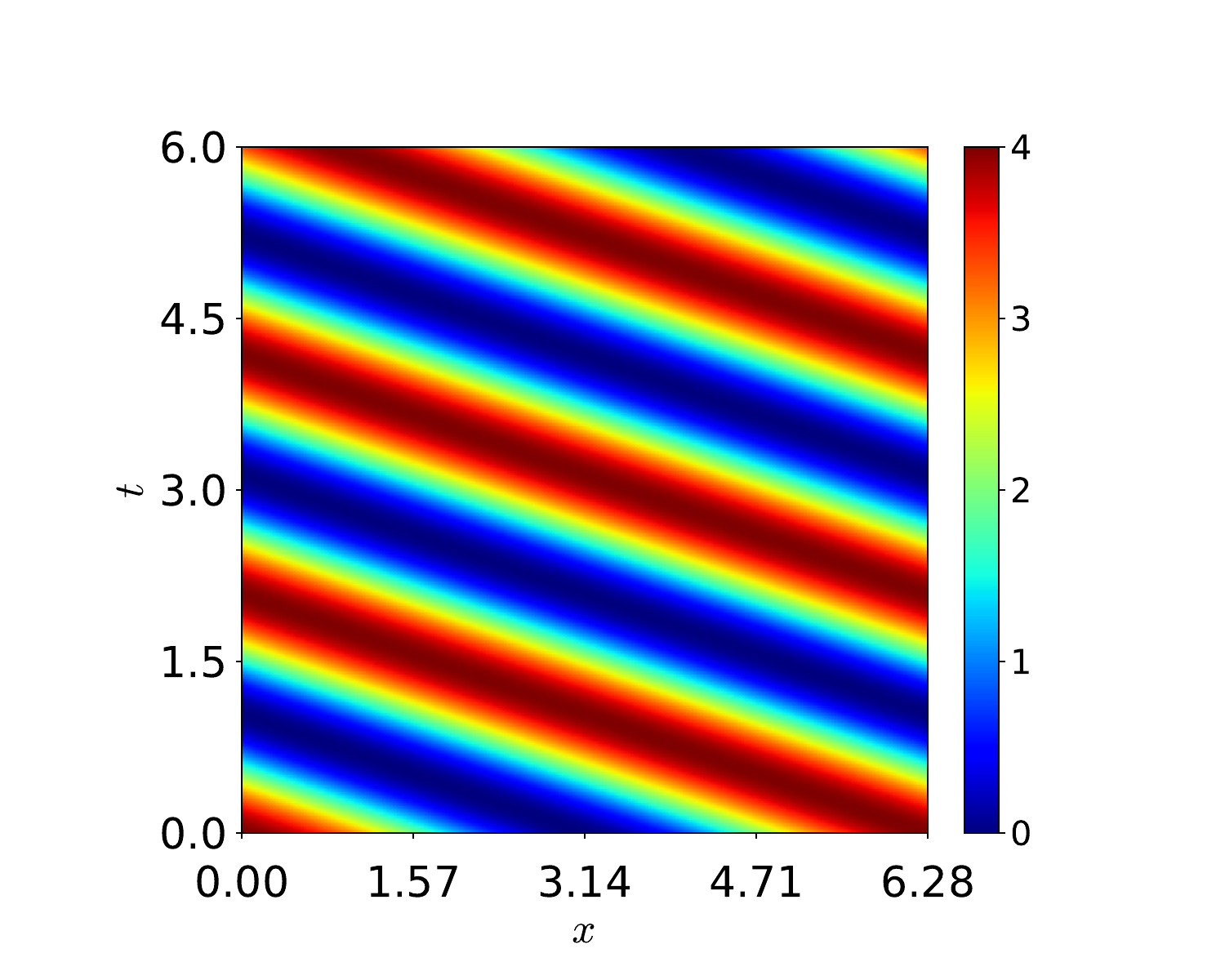}
    \caption{Linear solution}
    \label{subfig: nls-p-2tori-heat-slice-initial}
    \end{subfigure}
    \begin{subfigure}[t]{0.42\linewidth}
    \centering
    \includegraphics[scale=0.24]{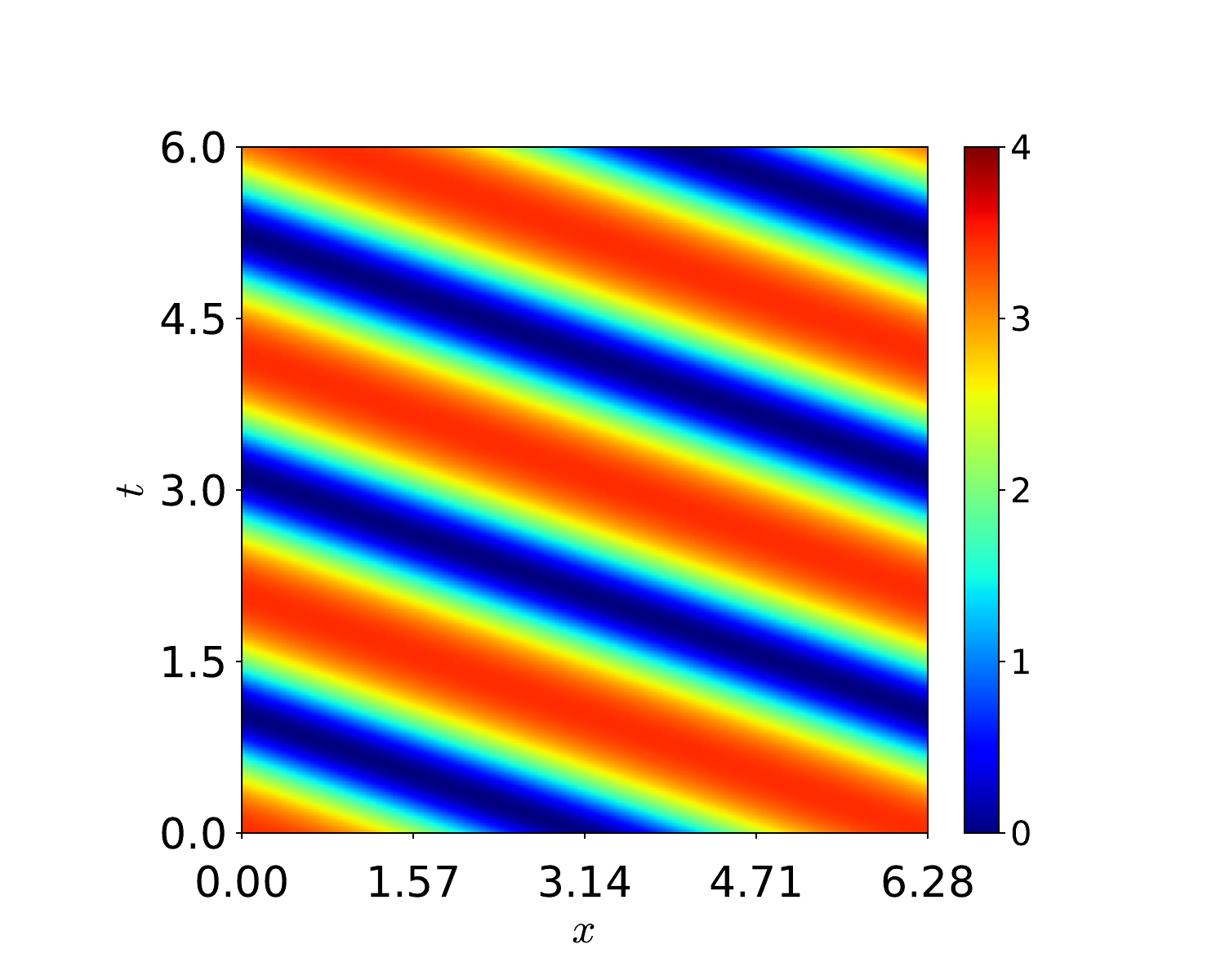}
    \caption{Nonlinear solution}
    \label{subfig: nls-p-2tori-heat-slice-final}
    \end{subfigure} 
    \caption{Spatiotemporal distributions of $|u(x,t)|^2$ for the NLS equation 
             with periodic boundary conditions~\eqref{eqn: nls-periodic}. 
             The linear solution is~\eqref{eqn: lin-square}, and the nonlinear 
             solution is obtained via the alternating numerical 
             scheme~\eqref{eqn: alternating-numerical}.}
    \label{fig: nls-p-2tori-heat-slice}
\end{figure}

\subsection{The main theorem: a priori bound}
\label{subsec: priori-estimate}

In this section, we establish convergence and a priori bounds for the 
alternating numerical procedure~\eqref{eqn: alternating-numerical}. 
For a Fourier vector $\hat{u} = \left\{ \hat{u}(n,k) \right\}_{(n,k) \in \mathbb{Z}^{1+m}}$, 
the standard $\ell_2$-norm is defined by  
\begin{equation}
    \label{eqn: l2-norm}
    \|\hat{u}\|_2 = \left( \sum_{(n,k) \in \mathbb{Z}^{1+m}} |\hat{u}(n,k) |^2 \right)^{\frac12}. 
\end{equation}
Since the Fourier coefficients are continuously differentiable with 
respect to the drifted frequency vector $\omega'$, we introduce a norm that incorporates 
both the coefficients and their frequency derivatives:
\begin{equation}
    \label{eqn: new-norm}
    \|\hat{u}\| = \left( \|\hat{u}\|_2^2 + \sum_{\ell=1}^{m} \| \partial_{\omega'_{\ell}} \hat{u} \|_2^2  \right)^{\frac12}. 
\end{equation}
Equivalently, this norm can be expressed componentwise as  
\[
    \|\hat{u}\|= \left[ \sum_{n \in \mathbb{Z}} \sum_{k \in \mathbb{Z}^m} \left( |\hat{u}(n, k)|^2 + \sum_{\ell = 1}^{m} |\partial_{\omega'_{\ell}}\hat{u}(n, k)|^{2} \right) \right]^{1/2}. 
\]
For each $(n, k) \in \mathbb{Z}^{1+m}$, we also define the coefficient-wise norm  
\begin{equation}
    \label{eqn: new-norm-entries}
    \| \hat{u}(n, k) \| = \left( |\hat{u}(n, k)|^2 + \sum_{\ell = 1}^{m} |\partial_{\omega'_{\ell}}\hat{u}(n, k)|^{2} \right)^{\frac12},  
\end{equation}
which, unlike the absolute value $|\hat{u}(n, k)|$, also accounts for the 
frequency derivatives. 
For such frequency-dependent Fourier vectors, we introduce the normed vector 
space
\begin{equation}
    \label{eqn: l2-space}
    \mathscr{H}(\mathbb{Z}^{1+m}) =
    \left\{
        \hat{u} = 
        \left\{ \hat{u}(n,k) \right\}_{(n,k) \in \mathbb{Z} \times \mathbb{Z}^m} ~\bigg|~
        \hat{u}(n,k) \in \mathbb{R}, \; \|\hat{u}\| < \infty
    \right\}. 
\end{equation}
To characterize the regularity of the solutions, we define the Gevrey decay set
\begin{equation}
    \label{eqn: gevrey-decay-set}
    \mathcal{K}(s) =
    \left\{ \hat{u} \in \mathscr{H}(\mathbb{Z}^{1+m}) ~\bigg|~ \sup_{(n,k) \in \mathbb{Z} \times \mathbb{Z}^m} \left( \| \hat{u}(n,k) \| \exp \left\{ (|n|+|k|)^s \right\} \right) \leq 1 \right\},
\end{equation}
where $0<s<1$. 
This condition imposes subexponential decay on both the Fourier coefficients and 
their frequency derivatives, providing the regularity control needed throughout 
the iteration. 
We now state the main theorem, which establishes convergence of the alternating 
numerical procedure~\eqref{eqn: alternating-numerical} and provides a priori 
error bounds for its iterates.

\begin{theorem}[Convergence and a priori bound]
    \label[theorem]{thm: main}
    Fix $\delta_0 \in (0, 1)$ and $\tau > m$, and let 
    $\mathscr{A} = [\delta_0, 1]^{m}$. 
    There exists a threshold $\varepsilon_0 > 0$ such that, for any 
    $0 < \varepsilon \leq \varepsilon_0$ the following statements hold: 
    \begin{itemize}
        \item[1.] \textbf{Measure of resonance}: 
                  There exist constants $\kappa = \kappa(m, \mu, \tau) >0$ and 
                  $\delta = \delta(\varepsilon) > 0$, with $\delta \rightarrow 0$ 
                  as $\varepsilon \rightarrow 0$, such that the excluded set 
                  $\mathscr{A}^{\star} \subset \mathscr{A}$ of initial amplitudes 
                  satisfies $\mathrm{mes}(\mathscr{A}^{\star}) \leq \kappa + \delta$. 

        \item[2.] \textbf{Iterative sequence}: 
                  For any $ a \in \mathscr{A} \setminus \mathscr{A}^{\star}$, the 
                  alternating numerical scheme~\eqref{eqn: alternating-numerical}, 
                  initialized at $(\omega^{(0)}, \hat{u}_p^{(0)}) = (\omega, 0)$, 
                  generates a well-defined sequence of iterates 
                  $\{ (\omega^{(r)}, \hat{u}_p^{(r)}) \}_{r=0}^{\infty}$. 
                  Moreover, for every $r\geq0$, the updated frequency map
                  $a\mapsto\omega^{(r+1)}(a)$ is a local diffeomorphism.

        \item[3.] \textbf{Convergence}: 
                  There exists a Gevrey exponent $s = s(\varepsilon) \in (0, 1)$ 
                  such that this sequence converges to an exact solution pair 
                  $(\omega^\star, \hat{u}_p^\star)$. 
                  The limiting frequency map $a\mapsto\omega^\star(a)$ is also 
                  a local diffeomorphism, and $\hat{u}_p^\star\in\mathcal{K}(s)$.

        \item[4.] \textbf{A priori error bounds}: 
                  For any $r \geq 0$, the iterative sequence 
                  $\{ \omega^{(r)}, \hat{u}_p^{(r)} \}_{r=0}^{\infty}$ satisfies 
                  the following error bounds:
                  \begin{subequations}
                      \label{eqn: convergence-vector-frequency}
                      \begin{empheq}[left=\empheqlbrace]{align} 
                          \|\hat{u}^{(r)}-\hat{u}^{\star}\| &\leq \exp\left\{ -\frac{3}{2}(M^s)^{r+1} \right\},  \label{eqn: u-hat-converge}  \\
                          |\omega^{(r)}-\omega^{\star}| &\leq \exp\left\{ -\frac{3}{2}(M^s)^{r+1} \right\}.      \label{eqn: frequency-converge}
                      \end{empheq}
                  \end{subequations}
    \end{itemize}
    Furthermore, define the $r$-th approximate solution by 
    \begin{equation}
        \label{eqn: approximate-pde-solution}
        u^{(r)}(x,t) = \sum_{ (n, k) \in \mathbb{Z}^{1+m} } \hat{u}^{(r)}(n, k) e^{i \langle k, \; \omega^{(r)} t + \theta_0 \rangle } e^{inx}. 
    \end{equation}
    Then, for any  $t \in [0, M^r]$, the following time-domain error bound holds:
    \begin{equation}
        \label{eqn: pde-solution-error}
        \|u^{(r)}(x,t)-u^{\star}(x,t)\|_{L^2} \leq \exp \left\{ -(M^s)^{r+1} \right\},
    \end{equation}
    where $u^{\star}(x, t)$ is the exact quasi-periodic solution of the NLS 
    equation with periodic boundary conditions~\eqref{eqn: nls-periodic} 
    corresponding to $(\omega^{\star}, \hat{u}_p^{\star})$. 
    Consequently, for any $t\geq0$, $u^{(r)}(x,t)$ converges in $L^2$ to 
    $u^\star(x,t)$ as $r \rightarrow \infty$.
\end{theorem}

The proof of~\Cref{thm: main} is completed in~\Cref{sec: iterative-scheme} by 
combining the implementation framework developed in~\Cref{sec: implement} with 
the multiscale analysis presented in~\Cref{sec: multi-scale}.

\subsection{Organization of the paper}
\label{subsec: organization}

The remainder of this paper is organized as follows.~\Cref{sec: basic-setting} 
establishes uniform properties of the associated vectors and operators within 
the Gevrey decay class, which serve as the foundation for the subsequent 
analysis.~\Cref{sec: implement} formulates the implementation conditions 
corresponding to these uniform properties and verifies them by inductive 
arguments.~\Cref{sec: multi-scale} derives the inversion and localization 
conditions for the tangent operator restricted to outer boxes via a multiscale 
analysis, built upon inductive bases for two types of small-scale 
cases.~\Cref{sec: iterative-scheme} establishes the iterative lemma, proves the 
main convergence results, and presents numerical experiments for the NLS 
equation~\eqref{eqn: nls} under Dirichlet boundary conditions.~\Cref{sec: nlw} 
generalizes the above results to the NLW equation~\eqref{eqn: nlw} under both 
periodic and Dirichlet boundary conditions, together with the corresponding 
numerical demonstrations. Finally,~\Cref{sec: conclu} concludes the paper and 
outlines directions for future research.

\section{Gevrey decay properties}
\label{sec: basic-setting}

In this section, we establish uniform Gevrey decay estimates and norm bounds 
for the nonlinear vector field and its associated operators on $\mathcal{K}(s)$. 
Let $\mathbb{T}^{1+m} = [0, 2\pi]^{1 + m}$ denote the extended 
$(1+m)$-dimensional space-time torus, with coordinates $(x, \theta)$. 
The alternating numerical scheme~\eqref{eqn: alternating-numerical} preserves 
the reality of the frequency vectors and Fourier coefficients. 
Since the initial data $(\omega^{(0)}, \hat{u}_p^{(0)}) = (\omega, 0)$ and the 
amplitude $a$ are real, all iterates $(\omega^{(r)}, \widehat{u}_p^{(r)})$ 
remain real. 
Accordingly, we restrict our analysis to Fourier coefficient vector satisfying 
$\hat{u}(n, k) \in \mathbb{R}$ for all $(n, k) \in \mathbb{Z}^{1+m}$. 
The components of the nonlinear vector field $\hat{X}$ are given by 
\begin{equation}
    \label{eqn: vector-field-def}
    \hat{X}(n, k) = \frac{1}{(2\pi)^{1+m}} \iint_{\mathbb{T}^{1+m}} |u|^2 u e^{ - i \left( \langle k, \theta \rangle + nx \right) } d\theta dx. 
\end{equation}
Since the Fourier coefficients of $|u|^2 u$ are convolution sums of products of 
the real coefficients $\hat{u}(n, k)$, $\hat{X}(n, k)$ is real as well. 
Taking the complex conjugate of~\eqref{eqn: vector-field-def}, we therefore 
obtain 
\begin{equation}
    \label{eqn: symmetric-relation}
    \frac{1}{(2\pi)^{1+m}} \iint_{\mathbb{T}^{1+m}} |u|^2 u e^{ - i \left( \langle k, \theta \rangle + nx \right) } d\theta dx = \frac{1}{(2\pi)^{1+m}} \iint_{\mathbb{T}^{1+m}} |u|^2 \overline{u} e^{ i \left( \langle k, \theta \rangle + nx \right) } d\theta dx. 
\end{equation}
The following proposition, whose proof is given in~\Cref{subsec: vector-field}, 
provides uniform decay estimates and norm bounds for the vector field 
$\hat{X} = (\hat{X}_q, \hat{X}_p)^{\top}$.

\begin{prop}
    \label[proposition]{prop: vector-field}
    Suppose that $\hat{u} \in \mathcal{K}(s)$. 
    Then there exist constants $\gamma_1, \gamma_2 > 0$, depending on the 
    number $m$ of resonant modes, the Gevrey exponent $s$, the initial 
    amplitude vector $a$ and the non-resonant vector $\hat{u}_p$, such that 
    \begin{subequations}
        \label{eqn: vector-field}
        \begin{empheq}[left=\empheqlbrace]{align} 
            & \sup_{(n,k) \in \mathbb{Z}^{1+m}} \left( \|\hat{X}(n,k)\| e^{ (|n| + |k|)^s } \right) \leq \gamma_1,  \label{eqn: vector-field-comp}  \\
            & \max \big \{ \|\hat{X}_q\|, \|\hat{X}_p\| \big \} \leq \|\hat{X}\| \leq \gamma_2.                     \label{eqn: vector-field-bound}
        \end{empheq}  
    \end{subequations} 
\end{prop}

We next consider the derivative of $\hat{X}$ with respect to the non-resonant 
Fourier coefficients, 
$\partial \hat{X} / \partial \hat{u}_p = ( \partial \hat{X}_q / \partial \hat{u}_p, \partial \hat{X}_p / \partial \hat{u}_p )^{\top}$. 
Recall from~\eqref{eqn: S-non-linear} that the tangent operator is 
$S = \partial \hat{X}_p / \partial \hat{u}_p$. 
For non-resonant indices $(n,k), (n',k') \notin \mathcal{S}$, we 
utilize~\eqref{eqn: symmetric-relation} to express its entries as
\begin{align*}
    S\left( (n, k), (n',k') \right) & = S\left( n - n', k - k' \right) \\
                                    & = \frac{4}{(2\pi)^{1+m}} \iint_{\mathbb{T}^{1+m}} |u|^{2} \cos\left[ \langle k - k', \theta \rangle + (n- n')x \right] d\theta dx.
\end{align*}
The following proposition establishes the symmetry, off-diagonal decay, and 
boundedness of $S$, with its formal proof given in~\Cref{subsec: convolution}.

\begin{prop}
    \label[proposition]{prop: convolution}
    Suppose that $\hat{u} \in \mathcal{K}(s)$. 
    Then there exist constants $\gamma_3, \gamma_4 > 0$, depending on the 
    number $m$ of resonant modes, the Gevrey exponent $s$, the initial 
    amplitude vector $a$ and the non-resonant vector $\hat{u}_p$, such that $S$ 
    is symmetric and satisfies:
    \begin{subequations}
        \label{eqn: tangent}
        \begin{empheq}[left=\empheqlbrace]{align} 
            & \sup_{ (n, k), (n', k') \notin \mathcal{S}} \left(\|S\left( (n, k), (n', k') \right)\| e^{ (|n - n'| + |k - k'|)^s } \right) \leq \gamma_3,                                   \label{eqn: tangent-comp} \\
            & \max \left\{ \left\| \frac{\partial \hat{X}_q}{\partial \hat{u}_p} \right\|, \|S\| \right\} \leq \left\| \frac{\partial \hat{X}}{\partial \hat{u}_p} \right\| \leq \gamma_4.  \label{eqn: tangent-bound}
        \end{empheq}        
    \end{subequations}
\end{prop}

The iteration lemma (\Cref{lem: iteration}) also requires a bound on the second 
derivative of $\hat{X}$. 
We state this estimate below and provide its proof in~\Cref{subsec: tensor}.

\begin{prop}
    \label[proposition]{prop: tensor}
    Suppose $\hat{u} \in \mathcal{K}(s)$ and, in addition, that 
    $\mathrm{supp}~\hat{u} \subseteq \Lambda_N^{1+m}$. 
    Then there exists a constant $\gamma_5 = \gamma_5(m, s, a, \hat{u}_p)$ such 
    that 
    \begin{equation}
        \label{eqn: tensor}
        \max \left\{ \left\| \frac{\partial^2 \hat{X}_q}{\partial \hat{u}_p^2} \right\|, \left\| \frac{\partial S}{\partial \hat{u}_p } \right\| \right\} \leq \left\| \frac{\partial^2 \hat{X}}{\partial \hat{u}_p^2} \right\| \leq \gamma_5 (2N+1)^\frac{1+m}{2}. 
    \end{equation}
\end{prop}

To control the error associated with the frequency 
update~\eqref{eqn: alternating-frequency}, we also require estimates for the 
auxiliary operator $B$. 
These estimates are stated in the following proposition, with the proof given 
in~\Cref{subsec: b-operator}.

\begin{prop}
    \label[proposition]{prop: B-operator}
    Suppose that $\hat{u} \in \mathcal{K}(s)$. 
    Then there exist constants $\gamma_6, \gamma_7 > 0$, depending on the 
    number $m$ of resonant modes, the Gevrey exponent $s$, the initial amplitude 
    vector $a$ and the non-resonant vector $\hat{u}_p$, such that the 
    operator $B$ satisfies: 
    \begin{subequations}
        \label{eqn: b-operator-bound}
        \begin{empheq}[left=\empheqlbrace]{align} 
            & \sup_{(n, k), (n', k') \notin \mathcal{S}} \left( \left\| B\left( (n, k), (n', k') \right) \right\| \exp\left\{ (|n|+|k|)^s + (|n'|+|k'|)^s \right\} \right) \leq \gamma_6,  \label{eqn: b-comp} \\
            & \| B \| \leq \gamma_7.                                                                                                                                                       \label{eqn: b-bound}
        \end{empheq}        
    \end{subequations}
\end{prop}

Since the number $m$ and the exponent $s$ are fixed throughout, we set 
$\gamma = \max_{1 \leq j \leq 7} \gamma_j$, which depends only on the resonant 
amplitude $a$ and the non-resonant vector $\hat{u}_p$. 
Since $a \in \mathscr{A}$ and $\hat{u}_p \in \mathcal{K}(s)$, we may therefore 
set $\varepsilon_0 > 0$ such that 
$\sqrt{\varepsilon_0}\gamma \leq (1/2) \delta_0 $ without loss of generality. 
The bounds on $\hat{X}_q$ and its derivative in~\Cref{prop: vector-field}, 
together with the $Q$-equation~\eqref{eqn: q-eqn}, yield the following 
estimates for the frequency drift.

\begin{prop}
    \label[proposition]{prop: frequency-drift}
    Suppose that $\hat{u} \in \mathcal{K}(s)$. 
    The frequency drift determined by the $Q$-equation~\eqref{eqn: q-eqn} 
    satisfies
    \begin{equation}
        \label{eqn: frequency-drift-bound}
        \left| \omega' - \omega \right| \leq \sqrt{\varepsilon}.
    \end{equation}
    Moreover, the frequency map $a \mapsto \omega'(a)$ is a local
    diffeomorphism.
\end{prop}

Indeed, the $Q$-equation~\eqref{eqn: q-eqn} defines an implicit relation
between the amplitude $a$ and the frequency $\omega'$:
\begin{equation}
    \label{eqn: frequency-map-nls}
    Q(\omega', a) = (-\omega' + \omega)\odot a + \varepsilon \hat{X}_q(\omega', a) = 0.
\end{equation}
Regarding amplitude $a$ locally as a function of $\omega'$ along the solution and 
differentiating~\eqref{eqn: frequency-map-nls} with respect to $\omega'$
gives
\begin{equation}
    \label{eqn: frequency-map-nls-derivative}
    \frac{\partial Q}{\partial \omega'} + \frac{\partial Q}{\partial a}~\frac{\partial a}{\partial \omega'} = 0.
\end{equation}
Moreover, we have
\begin{equation}
    \label{eqn: q-freq-derivative}
    \frac{\partial Q}{\partial \omega'} = -\mathrm{diag}(a) + \varepsilon \frac{\partial \hat{X}_q}{\partial \omega'}.
\end{equation}
Since each component of the amplitude vector $a$ is bounded below by 
$\delta_0$,~\Cref{prop: vector-field} implies that 
$\partial Q/\partial \omega'$ is invertible for sufficiently small $\varepsilon$. 
By~\eqref{eqn: frequency-map-nls-derivative}, the product of the two square
matrices $\partial Q/\partial a$ and $\partial a/\partial \omega'$ is therefore
invertible. Hence both matrices are invertible, and in particular
$\partial a/\partial \omega'$ is invertible. The inverse function theorem then
implies that the frequency map $a \mapsto \omega'(a)$ is a local diffeomorphism.

\section{Implementation conditions}
\label{sec: implement}

The heart of the alternating numerical scheme~\eqref{eqn: alternating-numerical} 
is the update of the non-resonant vector via the dimension-enlarged Newton 
iteration~\eqref{eqn: alternating-coefficient}. 
To overcome the small-divisor problem, we must rigorously control the inverse 
of the truncated operator $(T + \varepsilon B)_N$ as the truncation scale $N$ 
increases. 
To this end, we introduce the scale-dependent threshold:  
\begin{equation}
    \label{eqn: scale-dependent}
    \varepsilon_N := e^{ - (\log N)^{15} }. 
\end{equation}
We impose the following implementation conditions to ensure convergence of the 
alternating numerical scheme~\eqref{eqn: alternating-numerical}. 
For clarity, the analysis in this section and~\Cref{sec: multi-scale} is carried 
out in $\ell_2(\mathbb{Z}^{1+m})$, without incorporating derivatives with 
respect to the frequency vector. 
The estimates are extended to include these derivatives 
in~\Cref{subsec: induction}. 
For convenience, throughout the remainder of this paper, we write 
$\Gamma_{n_0, k_0, N} := (n_0, k_0) + \Lambda_N^{1+m}$, and denote the 
restriction of the operator $T$ to $\Gamma_{n_0, k_0, N}$ by $T_{n_0, k_0, N}$. 
In particular, when $k_0 = 0$, we abbreviate $\Gamma_{n_0, 0, N}$ as 
$\Gamma_{n_0, N}$; if, in addition, $n_0 = 0$, we simply write $\Gamma_N$. 
The same abbreviations will be used for the corresponding restrictions of 
the operators considered below.

\begin{tcolorbox}[breakable]
    \begin{condition}[Implementation Conditions]
        \label[condition]{cond: implementation}
        At each scale $N$, the restricted operator $(T + \varepsilon B)_N$ is 
        required to satisfy the following two conditions:
        \begin{itemize}
            \item[(1)] \textbf{Inversion condition}: The operator is invertible, 
                       and its inverse satisfies 
                       \begin{equation}
                           \label{eqn: inverse-grow}
                           \| (T + \varepsilon B)_{N}^{-1} \|_2 \leq \frac{1}{\varepsilon_{N}}.
                       \end{equation}

            \item[(2)] \textbf{Localization condition}: For any $(n,k), (n', k') \in \Lambda_N^{1+m}$ 
                       with $|n - n'| + |k - k'| \geq N^{\frac12}$, the entries of the 
                       inverse matrix obey the Gevrey decay estimate
                       \begin{equation}
                           \label{eqn: off-diagonal}
                           | (T + \varepsilon B)_{N}^{-1}\left( (n, k), (n', k') \right) | \leq e^{ - \frac{\left(|n - n'| + |k - k'|\right)^{s}}{2} }.
                       \end{equation}
        \end{itemize}
    \end{condition}
\end{tcolorbox}

In contrast to the finite-dimensional setting considered 
in~\citet{Fu2026numerical-b}, we seek to establish these bounds uniformly over 
all central boxes $\Gamma_{n_0, N}$ with $n_0 \in \mathbb{Z}$. 
Crucially, the dispersion relation~\eqref{eqn: nls-linear-frequency} between 
wavenumber and frequency allows the measure of the union of the nearly-resonant 
sets associated with all spatial indices $n_0 \in \mathbb{Z}$ to be made 
sufficiently small, as shown in~\Cref{subsec: nearly-resonant}. 
This need for uniform control over all spatial centers is a key distinction 
from the finite-dimensional setting. 
We denote the restriction of $T + \varepsilon B$ to such a box by 
$(T + \varepsilon B)_{n_0, N}$ and formulate this stronger requirement as 
follows.

\begin{tcolorbox}[breakable]
    \begin{condition}[Uniform Conditions]
        \label[condition]{cond: implementation-proved}
        At each scale $N$, the following two conditions are required to hold 
        uniformly for all $n_0 \in \mathbb{Z}$:
            \begin{itemize}
            \item[(1)] \textbf{Inversion condition}: The restricted operator 
                       $(T + \varepsilon B)_{n_0, N}$ is invertible, and its 
                       inverse satisfies 
                       \begin{equation}
                           \label{eqn: inverse-grow-proved}
                           \| (T + \varepsilon B)_{n_0, N}^{-1} \|_2 \leq \frac{1}{\varepsilon_{N}}.
                       \end{equation}

            \item[(2)] \textbf{Localization condition}: For any $(n,k), (n', k') \in \Gamma_{n_0, N}$ 
                       with $|n - n'| + |k-k'| \geq N^{\frac12}$, the entries of the 
                       inverse matrix obey the Gevrey decay estimate
                       \begin{equation}
                           \label{eqn: off-diagonal-proved}
                           | (T + \varepsilon B)_{n_0, N}^{-1}\left( (n, k), (n', k') \right) | \leq e^{ - \frac{ \left(|n - n'| + |k - k'|\right)^{s}}{2} }.
                       \end{equation}
        \end{itemize}
    \end{condition}
\end{tcolorbox}

\Cref{cond: implementation-proved} implies~\Cref{cond: implementation} upon 
taking $n_0 = 0$. 
In what follows, our main objective is to 
establish~\Cref{cond: implementation-proved} through the systematic exclusion 
of ``bad'' frequencies. 
At a fixed iteration, suppressing the iteration index, we write 
$\Omega := \omega'(\mathscr{A})$ for the corresponding 
frequency set, where $\omega' : a \mapsto \omega'(a)$ denotes the local 
diffeomorphism established in~\Cref{prop: frequency-drift}.
We divide the analysis into a small-scale regime and a large-scale regime, 
separated by the threshold
\begin{equation}
    \label{eqn: M}
    M = e^{ \left( \log \frac{1}{\varepsilon} \right)^{\frac{1}{20}} }, 
\end{equation}
which is determined by the perturbation parameter $\varepsilon$. 
The small-scale regime corresponds to $M_0 \leq N \leq M$, whereas the 
large-scale regime corresponds to $N > M$, where the initial scale is defined by 
\begin{equation}
    \label{eqn: M0}
   M_0 := e^{ \left( \log M \right)^{\frac{1}{20}} } = e^{ \left( \log \frac{1}{\varepsilon} \right)^{\frac{1}{400}} }.
\end{equation}    
For notational simplicity, we write $L := T + \varepsilon B$ for the linearized 
operator throughout the remainder of the paper.

\subsection{Nearly-resonant sets}
\label{subsec: nearly-resonant}

We first consider the small-scale regime $M_0 \leq N \leq M$. 
To establish uniform bounds on the inverse of the restricted operators, we 
exclude frequencies associated with near resonances. 
The nearly-resonant sets defined below are analogous to those introduced 
by~\citet{Fu2026numerical-b} in the low-dimensional setting. 
In the present setting, however, there are infinitely many normal frequencies, 
explicitly given by $n^2 + \mu$ for $n \in \mathbb{Z} \setminus \mathcal{N}$. 
We now introduce the following nearly-resonant sets:

\begin{itemize}
    \item \textbf{Tangential nearly-resonant set}: 
    \begin{equation}
        \label{eqn: tangential-nearly-resonant}
        \Omega_{0,M}^{\tau} := \bigcup_{k \in \Lambda_{2M}^m \setminus \{0\}} \left\{ \omega' \in \Omega ~\bigg|~ |\langle k,\omega' \rangle| < \frac{1}{|k|^\tau} \right\}.
    \end{equation}

    \item \textbf{Single-mode nearly-resonant set}: 
    \begin{equation}
        \label{eqn: single-nearly-resonant}
        \Omega_{1,M}^{\tau} :=
        \bigcup_{k \in \Lambda_{2M}^m} \bigcup_{n \in \mathbb{Z} \setminus \mathcal{N}}
        \left\{ \omega' \in \Omega ~\bigg|~ \left| \langle k,\omega' \rangle - n^2 - \mu \right| < \frac{1}{(1+|k|)^\tau} \right\}.
    \end{equation}

    \item \textbf{Difference nearly-resonant set:}
    \begin{equation}
        \label{eqn: difference-nearly-resonant}
        \Omega_{2,M}^{\tau} :=
        \bigcup_{k\in \Lambda_{2M}^m} \bigcup_{\substack{n_1, n_2 \in \mathbb{Z} \setminus \mathcal{N}  \\  |n_1| \neq |n_2|}}
        \left\{ \omega' \in \Omega ~\bigg|~ \left| \langle k,\omega' \rangle - n_1^2 + n_2^{2} \right| < \frac{1}{(2+|k|)^\tau} \right\}.
    \end{equation}
\end{itemize}

The total nearly-resonant set is the union of the sets defined 
in~\eqref{eqn: tangential-nearly-resonant}--\eqref{eqn: difference-nearly-resonant} as
\begin{equation}
    \label{eqn: nls-total-nearly-resonant}
    \Omega_M^{\tau} := \Omega_{0,M}^{\tau} \cup \Omega_{1,M}^{\tau} \cup \Omega_{2,M}^{\tau}.
\end{equation}
The following lemma provides a measure bound for this excluded set.

\begin{lemma}
    \label[lemma]{lem: near-measure-initial}
    Let $\tau > m$ be fixed. There exists a constant 
    $\kappa = \kappa(m, \tau, \Omega) > 0$, depending on the dimension $m$, the 
    parameter $\tau$ and the domain $\Omega$, such that the nearly-resonant set 
    satisfies
    \begin{equation}
        \label{eqn: near-measure-initial}
        \mathrm{mes}
        \left( \Omega_M^{\tau} \right) \leq \kappa.
    \end{equation}
\end{lemma}

The proof relies on elementary measure estimates and is deferred 
to~\Cref{sec: near-measure-initial}. 
It is important to note that the present setting requires $\tau > m$, in 
contrast to the condition $\tau > m-1$ in the finite-dimensional setting 
considered by~\citet{Fu2026numerical-b}.

\subsection{Restricted operators on small boxes}
\label{subsec: restricted-small}

In the small-scale regime, we establish invertibility and Gevrey decay estimates 
for the restrictions of $L$ to central boxes, with frequencies 
$\omega' \in \Omega \setminus \Omega_{M}^{\tau}$. 
The following theorem provides bounds that are uniform in the spatial center 
$n_0 \in \mathbb{Z}$.

\begin{theorem}
    \label{thm: small-size}
    Suppose that $\hat{u} \in \mathcal{K}(s)$ and 
    $\omega' \in \Omega \setminus \Omega_{M}^{\tau}$. 
    Then, for any $n_0 \in \mathbb{Z}$ and any scale $M_0 \leq N \leq M$, the 
    restricted operator $L_{n_0, N}$ is invertible and satisfies the operator 
    norm bound
    \begin{equation}
        \label{eqn: small-size-inverse}
        \| L_{n_0, N}^{-1} \|_2 \leq \frac{1}{\varepsilon_N}.
    \end{equation}
    Moreover, for any $(n, k) \neq (n', k')$, its entries satisfy the Gevrey 
    decay estimate
    \begin{equation}
        \label{eqn: small-size-localization}
        \left| L_{n_0, N}^{-1} \left( (n, k), (n',k') \right) \right| \leq e^{ - (|n-n'| + |k-k'|)^{s} }. 
    \end{equation} 
\end{theorem}

By~\Cref{prop: frequency-drift} and the exclusion of $\Omega_{M}^{\tau}$, the 
diagonal entries of $D_{n_0, N}$ satisfy
\begin{equation}
    \label{eqn: nls-uniform-lower-original-small}
    | D_{n_0, N}(n, k) | = \left| - \langle k, \omega' \rangle + n^2 + \mu \right| \geq \frac{1}{(1 + mN)^{\tau}},   
\end{equation}
for any non-resonant index $(n, k) \in \Gamma_{n_0, N} \setminus \mathcal{S}$. 
Thus, at each scale $N$, the diagonal entries of $D_{n_0, N}$ are bounded away 
from zero uniformly in $n_0 \in \mathbb{Z}$. 
This lower bound allows us to estimate $L_{n_0, N}^{-1}$ via a Neumann series. 
Since the off-diagonal terms satisfy uniform Gevrey decay estimates, the 
remaining estimates follow from standard arguments. 
The detailed proof is deferred to~\Cref{sec: small-size-original}.

\subsection{Verification of the uniform conditions}
\label{subsec: verification}

In this section, we verify that the restrictions of $L$ to the central boxes 
satisfy the uniform conditions (\Cref{cond: implementation-proved}) throughout 
the large-scale regime $N > M$. 
The proof proceeds by induction, with~\Cref{thm: small-size} providing the base 
case in the small-scale regime $M_0 \leq N \leq M$. 
Assuming that~\Cref{cond: implementation-proved} holds for all scales 
$M_0 \leq N \leq M^r$, we extend these bounds to the next range of scales, 
$M^r < N \leq M^{r+1}$. 
The inductive step relies on estimates for the restrictions of $L$ to the outer 
lattice boxes $\Gamma_{n_0, k_0, N}$, where $n_0 \in \mathbb{Z}$ and 
$k_0 \notin \Lambda_{2N}^m$. 
We denote these restricted operators by $L_{n_0, k_0, N}$. 
The following theorem establishes the invertibility and off-diagonal decay 
estimates required for the induction. 
Its proof, based on multi-scale analysis, is presented in~\Cref{sec: multi-scale}.

\begin{theorem}
    \label{thm: off-side-box}
    Suppose $\hat{u} \in \mathcal{K}(s)$ and 
    $\omega' \in \Omega \setminus (\Omega_{M}^{\tau} \cup G)$, where the 
    exceptional set $G$ is defined in~\Cref{sec: multi-scale}. 
    Then for any scale $N > M_0$, for any $n_0 \in \mathbb{Z}$ and any $k_0 \notin \Lambda_{2N}^m$, the 
    restricted operator $L_{n_0, k_0, N}$ is invertible and satisfies: 
    \begin{equation}
        \label{eqn: off-side-inverse}
        \| L_{n_0, k_0, N}^{-1} \|_2 \leq \frac{1}{\varepsilon_N}.
    \end{equation}
    Moreover, for any $(n,k), (n', k') \in \Gamma_{n_0, k_0, N}$ such that 
    $|n - n'| + |k - k'| \geq N^{1/2}$, its entries satisfy:
    \begin{equation}
        \label{eqn: off-side-localization}
        \left| L_{n_0, k_0, N}^{-1} \left( (n, k), (n', k') \right) \right| \leq e^{ - \frac{ \left(|n-n'| + |k-k'| \right)^{s}}{2} }.
    \end{equation} 
\end{theorem}

As in the finite-dimensional setting, we set $K = N^{\frac{1}{10}}$ to 
establish~\Cref{cond: implementation} at scale $N$. 
In the temporal index $k$, we retain the covering by the central subbox 
$\Lambda^m_{10K}$ and the translated subboxes $k + \Lambda_{K}^m$ with 
$k \in \Lambda_N^m \setminus \Lambda_{9K}^m$. 
To obtain the spatially uniform bounds required 
by~\Cref{cond: implementation-proved}, however, we also take the union over all 
integer spatial centers in $n_0 + \Lambda_N$, which leads to the following 
covering  
\begin{equation}
    \label{eqn: center-gluing-9k-implementation}
    \Gamma_{n_0, N} = \left( \bigcup_{n \in n_0 + \Lambda_N } \Gamma_{n, 10K} \right) \bigcup \left( \bigcup_{k \in \Lambda_N^m \setminus \Lambda_{9K}^m } \bigcup_{n \in n_0 + \Lambda_N } \Gamma_{n, k, K} \right). 
\end{equation}
Thus, the additional spatial index replaces each subbox of the temporal 
covering by a union of $2N+1$ spatial translates. 
The resulting covering is illustrated in~\Cref{fig: glue}. 
\begin{figure}[htb!]
    \centering
    \begin{tikzpicture}[scale=0.95, >=stealth]
        \draw[line width=1.0pt] (-5.0, -1.5) rectangle (5.0, 1.5);

        \draw[line width=0.8pt] (-1.5, -1.5) -- (-1.5, 1.5);
        \draw[line width=0.8pt] (1.5, -1.5) -- (1.5, 1.5);
        
        \draw[dashed, line width=0.8pt] (-1.2, -1.5) -- (-1.2, 1.5);
        \draw[dashed, line width=0.8pt] (1.2, -1.5) -- (1.2, 1.5);

        \fill[pattern=north east lines, pattern color=black!70] (-1.5, 0.4) rectangle (1.5, 1.0);
        \draw[line width=0.8pt] (-1.5, 0.4) rectangle (1.5, 1.0);

        \fill[pattern=north east lines, pattern color=black!70] (-1.5, -0.8) rectangle (1.5, -0.2);
        \draw[line width=0.8pt] (-1.5, -0.8) rectangle (1.5, -0.2);

        \draw[line width=0.8pt] (-3.2, 0.2) rectangle (-2.4, 0.5);
        \fill[pattern=north east lines, pattern color=gray!60] (-3.2, 0.2) rectangle (-2.4, 0.5);
        
        \draw[line width=0.8pt] (-4.2, -0.5) rectangle (-3.4, -0.2);
        \fill[pattern=north east lines, pattern color=gray!60] (-4.2, -0.5) rectangle (-3.4, -0.2);

        \draw[decorate, decoration={brace, mirror, amplitude=4pt}, line width=0.8pt] (-1.5, -1.6) -- (1.5, -1.6);
        \node at (0, -2.15) {\small $k \in \Lambda_{10K}^{m}$};
        
       \draw[decorate, decoration={brace, amplitude=4pt}, line width=0.8pt] (-1.2, 1.6) -- (1.2, 1.6);
        \node at (0, 2.15) {\small $k\in \Lambda_{9K}^{m}$};
        
       \draw[decorate, decoration={brace, mirror, amplitude=4pt}, line width=0.8pt] (5.1, -1.5) -- (5.1, 1.5);
        \node at (6.8, 0) {\small $n \in n_0 + \Lambda_N$};

        \node (sing) at (3.8, 2.2) [align=center] {\small Central boxes  $\Gamma_{n, 10K}$};
        \draw[->, line width=0.4pt] (-0.2, 1.0) -- (sing.south);
        \draw[->, line width=0.4pt] (-0.2, -0.2) -- (sing.south);

        \node (reg) at (-4.2, 2.2) [align=center] {\small Outer boxes $\Gamma_{n, k, K}$, where \\ \small  $n \in n_0 + \Lambda_N$ and $k \notin \Lambda_{9K}^{m}$};
        \draw[->, line width=0.4pt] (-2.8, 0.5) -- (reg.south);
        \draw[->, line width=0.4pt] (-3.8, -0.2) -- (reg.south);
    \end{tikzpicture}
    \caption{Schematic illustration of the lattice-box covering in~\eqref{eqn: center-gluing-9k-implementation}.}
    \label{fig: glue}
\end{figure}
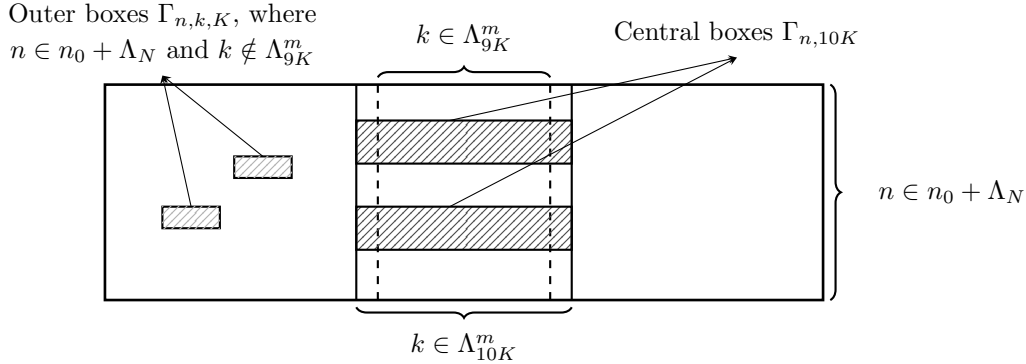
Despite this increase in the number of subboxes, the dispersion relation ensures 
that the measure of the union of the corresponding nearly-resonant sets remains 
sufficiently small, as demonstrated by the small-scale estimate 
in~\Cref{lem: near-measure-initial}. 
The gluing procedure therefore remains essentially unchanged, as detailed below.

\paragraph{Central region ($n \in n_0 + \Lambda_N$ and $k \in \Lambda_{9K}^m$)} 
For each central subbox, we consider the decomposition 
$\Gamma_{n_0, N} = \Gamma_{n, 10K} \cup \Gamma_{n, 10K}^{c} $, where 
$\Gamma_{n, 10K}^{c} = \Gamma_{n_0, N} \setminus \Gamma_{n, 10K}$.
Applying the resolvent identity, we obtain the block representation:
\begin{align}
    \begin{pmatrix} 
        L^{-1}_{n, 10K} & 0  \\ 
        0 & L^{-1}_{n, 10K,-} 
    \end{pmatrix} 
    = & \left[ \mathrm{Id} + \varepsilon 
    \begin{pmatrix} 
        0 & L_{n, 10K}^{-1}P^{*}  \\  
        L^{-1}_{n, 10K,-}P & 0 
    \end{pmatrix}  
    \right]  \nonumber \\
    & \cdot
    \begin{pmatrix} 
        (L^{-1}_{n_0, N})_{ \Gamma_{n, 10K} \times \Gamma_{n, 10K} }     & (L^{-1}_{n_0, N})_{\Gamma_{n, 10K} \times \Gamma_{n, 10K}^c}  \\ 
        (L^{-1}_{n_0, N})_{ \Gamma_{n, 10K}^{c} \times \Gamma_{n, 10K} } & (L^{-1}_{n_0, N})_{\Gamma_{n, 10K}^c \times \Gamma_{n, 10K}^c} 
    \end{pmatrix},      \label{eqn: resolvent-center}
\end{align}
where $P$ and $P^*$ are the coupling operators inherited from $S+B$, and 
$L^{-1}_{n, 10K,-}$ denotes the inverse of the operator $L$ restricted to 
$\Gamma_{n, 10K}^{c}$. 
Comparing the blocks in the first row gives
\begin{subequations}
    \label{eqn: resolvent-center-relation}
    \begin{empheq}[left=\empheqlbrace]{align} 
        & (L^{-1}_{n_0, N})_{\Gamma_{n, 10K} \times \Gamma_{n, 10K}} = L^{-1}_{n, 10K} - \varepsilon L_{n, 10K}^{-1}P^* (L^{-1}_{n_0, N})_{ \Gamma_{n, 10K}^c \times \Gamma_{n, 10K} },           \label{eqn: resolvent-center-relation-1} \\
        & (L^{-1}_{n_0, N})_{\Gamma_{n, 10K} \times \Gamma_{n, 10K}^c} = - \varepsilon L_{n, 10K}^{-1}P^*(L^{-1}_{n_0, N})_{\Gamma_{n, 10K}^c \times \Gamma_{n, 10K}^c}.                     \label{eqn: resolvent-center-relation-2}
    \end{empheq}        
\end{subequations}
Combining~\eqref{eqn: resolvent-center-relation-1} 
and~\eqref{eqn: resolvent-center-relation-2} and extending the row operators by 
zero to the full column index set $ \Gamma_{n_0, N}$, we obtain
\begin{equation}
    \label{eqn: resolvent-center-final}
    (L^{-1}_{n_0, N})_{ \Gamma_{n, 10K} \times \Gamma_{n_0, N} } = \begin{pmatrix} L_{n, 10K}^{-1} & 0 \end{pmatrix} - \varepsilon \begin{pmatrix} 0 & L_{n, 10K}^{-1}P^* \end{pmatrix} L_{n_0, N}^{-1}.
\end{equation}
The inductive hypothesis, together with~\Cref{prop: convolution} 
and~\Cref{prop: B-operator}, yields entrywise estimates for 
$L^{-1}_{n, 10K}P^{*}$. 
For $(n, k) \in \Gamma_{n, 9K} $ and 
$(n', k') \in (n_0 + \Lambda_N) \times \left( \Lambda_N^{m} \setminus \Lambda^m_{10K} \right)$, 
we have $|n-n'| + |k-k'| \geq |k-k'| \geq K$. 
Consequently, the corresponding entries of 
$(L^{-1}_{n, 10K}P^{*}) \left( (n,k), (n',k') \right)$ exhibit Gevrey decay. 
More precisely, splitting the defining sum into short-range and long-range 
contributions gives
\begin{equation}
    \label{eqn: inversion-localization-center}
    \left| (L^{-1}_{n, 10K}P^*)\left( (n,k), (n',k') \right) \right| \leq \left\{ 
    \begin{aligned}
        & 0,                                                           &&\mathrm{if}\; (n, k) = (n', k') ; \\
        & e^{ - \frac{ \left(|n - n'| + |k - k'| \right) ^{s}}{2} },   &&\mathrm{if}\; 0 <|n-n'| + |k-k'| < 100K^{\frac32},  \\
        & e^{ - \frac{ 3\left(|n - n'| + |k - k'| \right)^{s}}{4} },   &&\mathrm{if}\; |n-n'| + |k-k'| \geq 100K^{\frac32}.
    \end{aligned} 
    \right.
\end{equation}
Finally, restricting~\eqref{eqn: resolvent-center-final} to the rows indexed by 
$\{n\} \times \Lambda_{9K}^m $, we obtain
\begin{equation}
    \label{eqn: resolvent-center-final-9k}
    (L^{-1}_{n_0, N})_{ (\{n\} \times \Lambda_{9K}^m) \times \Gamma_{n_0, N} } = \begin{pmatrix} L_{n, 10K}^{-1} & 0 \end{pmatrix}_{ (\{n\} \times \Lambda_{9K}^m) \times \Gamma_{n_0, N} } - \varepsilon \begin{pmatrix} 0 & L_{n, 10K}^{-1}P^* \end{pmatrix}_{ (\{n\} \times \Lambda_{9K}^m) \times \Gamma_{n_0, N} } L_{n_0, N}^{-1},   
\end{equation}
where the entries of the zero-extended coupling operator 
$(0, \;L_{n, 10K}^{-1}P^*)$, restricted to 
$(\{n\} \times \Lambda_{9K}^m) \times \Gamma_{n_0, N}$, satisfy the Gevrey decay 
bounds~\eqref{eqn: inversion-localization-center}.

\paragraph{Outer region ($n \in n_0 + \Lambda_N$ and $k \in \Lambda_{N}^{m} \setminus \Lambda_{9K}^m$)} 
For each lattice point outside the central region, we consider the decomposition 
$\Gamma_{n_0, N} = \Gamma_{n, k, K} \cup \Gamma_{n, k, K}^c$, where 
$\Gamma_{n, k, K}^c = \Gamma_{n_0, N} \setminus \Gamma_{n, k, K}$. 
Repeating the block decomposition and resolvent identity argument 
in~\eqref{eqn: resolvent-center-relation}--\eqref{eqn: resolvent-center-final-9k}, 
and restricting to the row indexed by $(n, k)$, we obtain
\begin{equation}
\label{eqn: resolvent-outside-final-9k}
    (L^{-1}_{n_0, N})_{\{(n, k)\} \times \Gamma_{n_0, N}} = \begin{pmatrix} (L_{n, k, K}^{-1})_{\{(n,k)\} \times\Gamma_{n, k, K} } & 0 \end{pmatrix} - \varepsilon \begin{pmatrix} 0 & (L_{n, k, K}^{-1}P^*)_{\{ (n, k)\} \times \Gamma_{n, k, K}^c} \end{pmatrix} L_{n_0, N}^{-1},     
\end{equation}
where $P^*$ denotes the coupling operator associated with this decomposition, 
and both row operators are extended by zero to the full column index set 
$\Gamma_{n_0, N}$. 
For any $(n', k') \in \Gamma_{n, k, K}^c$, the coupling term satisfies the 
entrywise Gevrey decay estimate
\begin{equation}
    \label{eqn: inversion-localization-outside}
    \left| (L^{-1}_{n, k, K}P^*)\left( (n,k), (n',k') \right) \right| \leq \left\{ 
    \begin{aligned}
        & 0,                                                        &&\mathrm{if}\;  |n-n'| + |k-k'|  < K,                              \\
        & e^{ - \frac{ \left(|n-n'| + |k-k'| \right)^{s}}{2} },     &&\mathrm{if}\; K \leq |n-n'| + |k-k'| < K^{\frac32},     \\
        & e^{ - \frac{ 3\left(|n-n'| + |k-k'| \right)^{s}}{4} },    &&\mathrm{if}\; |n-n'| + |k-k'| \geq K^{\frac32}.
    \end{aligned} 
    \right.
\end{equation}

Assembling the central-row identities~\eqref{eqn: resolvent-center-final-9k} 
over all spatial centers $n \in n_0 + \Lambda_N$ and the outer-row 
identities~\eqref{eqn: resolvent-outside-final-9k} over all 
$(n,k) \in (n_0 + \Lambda_N) \times \left( \Lambda_{N}^{m} \setminus \Lambda_{9K}^m \right)$, 
we express the global inverse as
\begin{equation}
    \label{eqn: resolvent-relation}
    L_{n_0, N}^{-1} = \Psi_{n_0,N} - \varepsilon \Phi_{n_0,N} L_{n_0, N}^{-1},
\end{equation}
where $\Psi_{n_0,N}$ is assembled row by row from the lower-scale inverses 
$L_{n, 10K}^{-1}$ and $L_{n, k, K}^{-1}$. 
Specifically, for any $(n,k), (n',k') \in \Gamma_{n_0, N}$, its entries are 
given by 
\begin{multline}
    \Psi_{n_0, N}\left( (n,k), (n',k') \right) \\
    = \left\{ 
    \begin{aligned}
        & L_{n, 10K}^{-1}\left( (n,k), (n',k') \right),   &&\mathrm{if}\; (n,k) \in \left( n_0 + \Lambda_N \right) \times \Lambda_{9K}^m ,\; (n', k') \in \Gamma_{n, 10K},        \\
        & L_{n, k, K}^{-1}\left( (n,k), (n',k') \right),  &&\mathrm{if}\; (n,k) \notin \left( n_0 + \Lambda_N \right) \times \Lambda_{9K}^m ,\; (n', k') \in \Gamma_{n, k, K},    \\
        & 0,                                              &&\mathrm{otherwise}.   
    \end{aligned} 
    \right.  \label{eqn: entries-psi}
\end{multline}
The estimates~\eqref{eqn: inversion-localization-center} 
and~\eqref{eqn: inversion-localization-outside} imply the following entrywise 
bounds for $\Phi_{n_0,N}$:
\begin{equation}
    \label{eqn: entries-phi}
    \left| \Phi_{n_0,N}\left( (n,k), (n',k') \right) \right| \leq \left\{ 
    \begin{aligned} 
        & 0,                                                                              &&\mathrm{if}\; (n, k) = (n', k') ; \\
        & e^{ - \frac{ \left(|n-n'| + |k-k'| \right)^{s}}{2} },        &&\mathrm{if}\; 0 < |n - n'| + |k - k'|  < 100K^{\frac32};  \\ 
        & e^{ - \frac{ 3\left(|n-n'| + |k-k'| \right)^{s}}{4} },       &&\mathrm{if}\; |n - n'| + |k - k'| \geq 100K^{\frac32}.
    \end{aligned} 
    \right.
\end{equation}                                    
Thus, the operator $\mathrm{Id} + \varepsilon \Phi_{n_0, N}$ is invertible and 
its inverse is represented by a norm-convergent Neumann series. 
For any $|n-n'| + |k-k'| \geq N^{1/4}$, this expansion yields
\begin{equation}
    \label{eqn: inverse-decay-long}
    | (\mathrm{Id} + \varepsilon \Phi_{n_0,N})^{-1}\left( (n,k), (n',k') \right) | \leq e^{ - \frac{ \left(|n-n'| + |k-k'| \right)^{s} }{2} - \frac{1}{4} \left(100K^{\frac32}\right)^s }. 
\end{equation}
To derive this bound, we distinguish between paths containing at least one 
long-range jump and paths consisting entirely of short-range jumps. 
Paths of the first type benefit from the stronger decay 
in~\eqref{eqn: entries-phi}. 
For paths of the second type, the strict concavity of the power function for 
$0<s<1$ yields an additional gain in the cumulative decay exponent. 
More precisely, the sum of the s-th powers of the jump lengths along a purely 
short-range path admits the lower bound
\begin{align}
    \frac{|n-n'| + |k-k'|}{100K^{\frac32}} \left(100K^{\frac32}\right)^{s} 
    & \geq \left( \frac{N^{\frac14}}{ 100K^{\frac32} } \right)^{1-s} \cdot \left(|n - n'| + |k - k'| \right)^{s} \nonumber \\ 
    & \geq \left(|n-n'| + |k-k'| \right)^{s} + \left(100K^{\frac32}\right)^s,     \label{eqn: inverse-decay-short-inequality}
\end{align}
where the last inequality holds for sufficiently large $K$, since 
$K = N^{\frac{1}{10}}$. 
Thus, purely short-range paths also provide the additional decay required 
in~\eqref{eqn: inverse-decay-long}. 
The preceding argument applies to every fixed Gevrey exponent $0<s<1$. 
In the analytic case $s=1$, however, the strict concavity 
underlying~\eqref{eqn: inverse-decay-long} is absent. 
A stronger smallness condition on the perturbation parameter $\varepsilon$ is 
therefore required to compensate for the loss of this additional decay along 
purely short-range paths.

For sufficiently small $s>0$, the transition threshold 
in~\eqref{eqn: inversion-localization-center} can be lowered from 
$100K^{\frac32}$ to $100K$, making it linear in $K$. More generally, a refined 
Gevrey analysis involving~\Cref{lem: app-1} yields a two-regime decay estimate 
for the off-diagonal entries of 
$(\mathrm{Id} + \varepsilon \Phi_{n_0,N})^{-1}\left( (n,k), (n',k') \right)$: 
the long-range decay rate in~\eqref{eqn: inverse-decay-long} is retained, 
while the corresponding short-range estimate holds with the transition threshold 
increased from $100K$ to $K^2$. 
We can now rewrite~\eqref{eqn: resolvent-relation} as
\begin{equation}
    \label{eqn: l-n-inverse}
    L_{n_0, N}^{-1} = (\mathrm{Id} + \varepsilon \Phi_{n_0, N})^{-1} \Psi_{n_0,N}. 
\end{equation}
Finally,~\eqref{eqn: entries-psi} and the inductive bounds on the lower-scale 
inverses yield the $\ell_2$-operator norm estimate 
$\|\Psi_{n_0,N}\|_2 \leq (20K)^{1+m} \varepsilon_{10K}^{-1}$, while a bound for 
$\|\Phi_{n_0,N}\|_2$ follows from the Gevrey entry 
bound~\eqref{eqn: entries-phi}. 
Combining these bounds with the Neumann-series norm bound and the entrywise 
decay estimate~\eqref{eqn: inverse-decay-long}, we conclude 
from~\eqref{eqn: l-n-inverse} that $L_{n_0, N}^{-1}$ satisfies the required norm and 
localization bounds at scale $N$, uniformly in the spatial center $n_0$. 
This completes the verification.

\section{Multi-scale analysis}
\label{sec: multi-scale}

In this section, we employ multi-scale analysis to 
establish~\Cref{thm: off-side-box} for the ``outer'' boxes 
$\Gamma_{n_0, k_0, N}$, whose centers satisfy $n_0 \in \mathbb{Z}$ and 
$k_0 \notin \Lambda^m_{2N}$. 
For any $(n,k), (n',k') \in \Gamma_{n_0, k_0, N}$, the condition 
$k_0 \notin \Lambda^m_{2N}$ implies that $|k| + |k'| \geq 2N$. 
Consequently, the restricted operator satisfies
\[
    \|B_{n_0, k_0, N}\|_2 \leq \sqrt{ \|B_{n_0, k_0, N}\|_1 \cdot \|B_{n_0, k_0, N}\|_\infty } \leq (2N+1)^{1+m} e^{ -N^{s} }.
\]
We therefore regard $L_{n_0, k_0, N} = T_{n_0, k_0, N} + \varepsilon B_{n_0, k_0, N}$ 
as a perturbation of $T_{n_0, k_0, N}$, which is symmetric under interchange of 
the indices $(n, k)$ and $(n', k')$. 
By symmetry, $\| T_{n_0, k_0, N}^{-1} \|_2$ is the reciprocal of the smallest 
absolute eigenvalue of $T_{n_0, k_0, N}$; in particular, if this eigenvalue is 
bounded below by $2\varepsilon_N$, then 
$\|T_{n_0, k_0, N}^{-1} \|_2 \leq \varepsilon_N^{-1}/2$. 
Together with the rapid decay of $B_{n_0, k_0, N}$, we therefore obtain
\[
    \varepsilon \|T_{n_0, k_0, N}^{-1}\|_2 \|B_{n_0, k_0, N}\|_2 \leq \varepsilon (2N+1)^{1+m} e^{ (\log N)^{15} - N^s} \leq \frac{1}{2}.
\]
Applying the Neumann series to the factorization 
\[
    L_{n_0, k_0, N} = T_{n_0, k_0, N} \left( \mathrm{Id} + \varepsilon T^{-1}_{n_0, k_0, N} B_{n_0, k_0, N} \right),
\]
we conclude that $L_{n_0, k_0, N}$ is invertible with 
$\|L_{n_0, k_0, N}^{-1}\|_2 \leq \varepsilon_{N}^{-1}$. 
Thus, to obtain the desired inverse norm bound~\eqref{eqn: off-side-inverse}, 
it suffices to show that $\|T_{n_0, k_0, N}^{-1} \|_2 \leq \varepsilon_N^{-1}/2$, 
and the remainder of this section is devoted to establishing this bound together
with the corresponding localization estimate.

\subsection{Inductive base for small boxes}
\label{subsec: small-outer}

To establish the base case for the induction 
in~\Cref{sec: inversion-localization}, we require inverse estimates for the 
small boxes in the regime $M_0 \leq N \leq M$. 
The two preliminary results of this subsection, the classification of 
singular sites (\Cref{prop: small-eig}) and the inverse estimate on the 
nonsingular set (\Cref{lem: small-singular}), are in fact valid in the wider 
range $1 \leq N \leq M$; only the ``inversion implies localization'' argument 
for the full box (\Cref{lem: small-inverse-localization}) requires 
$M_0 \leq N \leq M$.

\paragraph{Classification of singular sites} 
By the frequency-drift estimate in~\Cref{prop: frequency-drift}, the difference 
between any two diagonal entries $D_{n_0, k_0, N}$ satisfies 
\begin{equation}
    \label{eqn: d-eigenvalue-gap}
    | D_{n_0, k_0, N}(n_1, k_1) - D_{n_0, k_0, N}(n_2, k_2)| = | - \langle k_1 - k_2, \omega' \rangle + \varpi_{n_1} - \varpi_{n_2} |.
\end{equation} 
For any $\omega' \in \Omega \setminus \Omega_M^{\tau}$, the exclusion of the 
tangential nearly-resonant set $\Omega_{0,M}^{\tau}$ 
in~\eqref{eqn: tangential-nearly-resonant}, the single-mode nearly-resonant set 
$\Omega_{1,M}^{\tau}$ in~\eqref{eqn: single-nearly-resonant}, and the difference 
nearly-resonant set $\Omega_{2,M}^{\tau}$ 
in~\eqref{eqn: difference-nearly-resonant} yields the following dichotomy: 
either 
\[
    | - \langle k_1 - k_2, \omega' \rangle + \varpi_{n_1} - \varpi_{n_2} | \geq \frac{1}{(4mN)^{\tau}}
\]
or the two unperturbed diagonal entries coincide exactly,
\begin{equation}
    \label{eqn: d-eigenvalue-equal}
    - \langle k_1 - k_2, \omega' \rangle + \varpi_{n_1} - \varpi_{n_2} = 0. 
\end{equation}
In the latter case, the exclusion of~\eqref{eqn: single-nearly-resonant} rules 
out mixed tangential-normal pairs: the spatial indices $n_1$ and $n_2$ must 
either both belong to $\pm \mathcal{N}$ or both lie outside $\pm \mathcal{N}$. 
These two cases have the following structure. 
\begin{itemize}
    \item \textbf{Tangential case}. 
    Suppose $n_1, n_2 \in \pm \mathcal{N}$. 
    Then the equality~\eqref{eqn: d-eigenvalue-equal} can be rewritten as 
    \begin{align}
        - \langle k_1 - k_2, \omega' \rangle + \varpi_{n_1} - \varpi_{n_2} 
        &= - \langle k_1 - k_2, \omega' \rangle + \omega_{i} - \omega_{j} = - \langle k_1 - k_2 - e_{i} + e_{j}, \omega' \rangle  \nonumber  \\
        &= -(\omega_i-\omega_i') + (\omega_j-\omega_j'),  \label{eqn: d-eigenvalue-gap-1}
    \end{align}
    where $i,j \in \{ 1,\ldots, m \}$. 
    The right-hand side of~\eqref{eqn: d-eigenvalue-gap-1} is bounded by 
    $2\varepsilon$.
    By the exclusion of~\eqref{eqn: tangential-nearly-resonant}, the shifted 
    temporal difference must vanish, yielding $k_1 - e_{i} = k_2 - e_{j}$. 
    Thus, all sites in the same tangential equality class share a common 
    shifted temporal index $k_{\star}$, with $k_{\star} := k_j - e_j$ for each 
    site in the class associated with the $j$-th tangential mode.

    \item \textbf{Normal case}. 
    Suppose $n_1, n_2 \notin \pm \mathcal{N}$. 
    The exclusion of~\eqref{eqn: difference-nearly-resonant} forces 
    $\varpi_{n_1} = \varpi_{n_2}$, and the 
    equality~\eqref{eqn: d-eigenvalue-equal} therefore reduces to
    \begin{equation}
        \label{eqn: d-eigenvalue-gap-2}
        - \langle k_1 - k_2, \omega \rangle = 0.
    \end{equation}
    The exclusion of~\eqref{eqn: tangential-nearly-resonant} then gives 
    $k_1 = k_2:=k_{\star}$. 
    Moreover, $\varpi_{n_1} = \varpi_{n_2}$ implies $n_1 = \pm n_2$. 
    Consequently, all sites in the same normal equality class are contained in 
    a pair of the form $\{ (n_{\star}, k_{\star}), (-n_{\star}, k_{\star}) \}$, 
    where $n_{\star} \notin \pm \mathcal{N}$. 
\end{itemize}
For any box $\Gamma_{n_0, k_0, N}$ with $n_0 \in \mathbb{Z}$ and 
$1 \leq N \leq M$, we define the set of \textit{singular sites} as: 
\begin{equation}
\label{eqn: singular-sites}
\Pi := \left\{ (n, k) \in \Gamma_{n_0, k_0, N} ~\big|~ |D_{n_0, k_0, N}(n, k)| \leq 4 \varepsilon_N \right\}. 
\end{equation}
The preceding classification therefore gives the following proposition.

\begin{prop}
    \label[proposition]{prop: small-eig}
    Suppose $\hat{u} \in \mathcal{K}(s)$ and 
    $\omega' \in \Omega \setminus \Omega_M^\tau$. 
    If $\Pi$ is nonempty, then it is contained in a resonant configuration of 
    exactly one of the following two types:
    \begin{itemize}
        \item[\textbf{(I)}] \textbf{Tangential configuration}: There exists  
        a common shifted temporal index $k_{\star} \in k_0 + \Lambda_N^m$ 
        such that 
        \begin{equation}
            \label{eqn: singular-sites-tangent}
            \Pi \subseteq \left\{ (\pm n_j, k_{\star} + e_j ) ~\big|~ j=1, \ldots, m \right\},
        \end{equation}
        where $n_j \in \mathcal{N}$ for $j=1, \ldots, m$ are the tangential 
        indices. 

        \item[\textbf{(II)}] \textbf{Normal configuration}: There exist 
        a temporal index $k_{\star} \in k_0 + \Lambda_N^m$ and a spatial 
        index $n_{\star} \notin \pm \mathcal{N}$ such that
        \begin{equation}
            \label{eqn: singular-sites-normal}
            \Pi \subseteq \left\{ ( n_{\star}, k_{\star} ), \; ( - n_{\star}, k_{\star} ) \right\}.
        \end{equation}
    \end{itemize}
    In particular, $\Pi$ cannot contain both tangential and normal sites. 
\end{prop}

\paragraph{Inverse on the nonsingular set} We next estimate the inverse of the 
restriction of $T$ to the nonsingular set $\Gamma_{n_0, k_0, N} \setminus \Pi$. 
The estimates below hold throughout the range $1 \leq N \leq M$. 
Applying~\Cref{prop: small-eig} and arguing as in the proof 
of~\Cref{thm: small-size}, we obtain the following lemma.

\begin{lemma}
    \label[lemma]{lem: small-singular}
    Suppose that $\hat{u} \in \mathcal{K}(s)$ and 
    $\omega' \in \Omega \setminus \Omega_{M}^{\tau}$. 
    Then the restricted operator $T_{\Gamma_{n_0, k_0, N} \setminus \Pi}$ is 
    invertible, and its inverse satisfies 
    \begin{equation}
        \label{eqn: small-size-inverse-prepare}
        \| T_{\Gamma_{n_0, k_0, N} \setminus \Pi}^{-1} \|_2 \leq \frac{1}{2\varepsilon_N}.
    \end{equation}
    Moreover, its off-diagonal entries satisfy 
    \begin{equation}
        \label{eqn: small-size-off-diagonal-prepare}
        \left| T_{\Gamma_{n_0, k_0, N} \setminus \Pi}^{-1} \left( (n,k), (n',k') \right) \right| \leq e^{ - (|n-n'| + |k-k'|)^{s} } < e^{ - \frac{ (|n-n'| + |k-k'|)^{s}}{2} },
    \end{equation}
    for any $(n,k) \neq (n', k') \in \Gamma_{n_0, k_0, N} \setminus \Pi$.
\end{lemma}

\paragraph{Inversion implies localization} The preceding analysis applies 
throughout the range $1 \leq N \leq M$. 
We now restrict attention to the small-scale regime $M_0 \leq N \leq M$ and 
establish the estimates needed for the base case of the induction 
in~\Cref{sec: inversion-localization}. 
For this purpose, we only need the Gevrey decay estimates 
from~\Cref{lem: small-singular} for the inverse of the nonsingular restriction 
at separations $|n-n'| + |k-k'| \geq N^{\frac12}$. 
We apply the ``inversion implies localization'' argument to transfer these 
decay estimates to the full inverse. 
With respect to the decomposition 
$\Gamma_{n_0, k_0, N} = (\Gamma_{n_0, k_0, N} \setminus \Pi) \cup \Pi$, the 
resolvent identity gives
\begin{align}
    \begin{pmatrix} 
        T_{\Gamma_{n_0, k_0, N} \setminus \Pi}^{-1} & 0 \\ 
        0 & T_{\Pi}^{-1} 
    \end{pmatrix} 
    = & \left[ \mathrm{Id} + \varepsilon 
    \begin{pmatrix} 
        0 & T_{\Gamma_{n_0, k_0, N} \setminus \Pi}^{-1} P^* \\ 
        T_{\Pi}^{-1} P & 0 
    \end{pmatrix}  
    \right]    \nonumber \\
    & \cdot  \left( \begin{array}{ll} 
        (T^{-1}_{n_0, k_0, N})_{(\Gamma_{n_0,k_0, N} \setminus \Pi) \times (\Gamma_{n_0, k_0, N} \setminus \Pi)} & (T^{-1}_{n_0, k_0, N})_{(\Gamma_{n_0,k_0, N} \setminus \Pi) \times \Pi}  \\ 
        (T^{-1}_{n_0, k_0, N})_{\Pi \times (\Gamma_{n_0, k_0, N} \setminus \Pi)} \mathrel{\phantom{=l}} & (T^{-1}_{n_0, k_0, N})_{\Pi \times \Pi} \mathrel{\phantom{=l}}
    \end{array} \right).     \label{eqn: resolvent-small-inversion-local}
\end{align}
Comparing the corresponding blocks 
in~\eqref{eqn: resolvent-small-inversion-local}, we obtain:
\begin{subequations}
    \label{eqn: resolvent-small-inversion-local-relation}
    \begin{empheq}[left=\empheqlbrace]{align} 
        & (T_{n_0,  k_0, N}^{-1})_{(\Gamma_{n_0, k_0, N} \setminus \Pi) \times (\Gamma_{n_0, k_0,N} \setminus \Pi)} \nonumber \\
        & \mathrel{\phantom{=============}}= T_{\Gamma_{n_0, k_0, N} \setminus \Pi}^{-1} - \varepsilon T_{\Gamma_{n_0, k_0, N} \setminus \Pi}^{-1} P^* (T_{n_0, k_0, N}^{-1})_{\Pi \times (\Gamma_{n_0, k_0, N} \setminus \Pi)},                                                                                                                                                                                \label{eqn: resolvent-small-inversion-local-relation-1} \\
        & (T_{n_0, k_0, N}^{-1})_{(\Gamma_{n_0, k_0, N} \setminus \Pi) \times \Pi} = - \varepsilon T_{\Gamma_{n_0, k_0, N} \setminus \Pi}^{-1} P^* (T_{n_0, k_0, N}^{-1})_{\Pi \times \Pi}.            \label{eqn: resolvent-small-inversion-local-relation-2}
    \end{empheq}        
\end{subequations}
The decay estimates~\eqref{eqn: small-size-off-diagonal-prepare} for 
$T_{\Gamma_{n_0, k_0,N} \setminus \Pi}^{-1} $ and~\eqref{eqn: tangent-comp} for
$P^*$ imply that, whenever $|n-n'| + |k-k'| \geq N^{\frac12}/4$, the coupling 
operator satisfies the following estimate 
\begin{equation}
    \label{eqn: coupling-operator-gevrey-small}
    \left| \left( T_{\Gamma_{n_0, k_0,N} \setminus \Pi}^{-1} P^* \right) \left( (n, k), (n', k') \right) \right| \leq e^{ - \frac{2 (|n-n'| + |k-k'|)^s}{3}}.
\end{equation}
In the finite-dimensional setting considered 
by~\citet{Fu2026numerical-a, Fu2026numerical-b}, suppose that the inverse bound 
$\| T_{n_0, k_0, N}^{-1} \|_2 \leq \varepsilon_{N}^{-1}/2$ holds. 
Since the operator norm of each block is bounded by that of the full inverse 
itself, we substitute the estimate~\eqref{eqn: coupling-operator-gevrey-small} 
into~\eqref{eqn: resolvent-small-inversion-local-relation-1}--\eqref{eqn: resolvent-small-inversion-local-relation-2} 
and use the symmetry of $T_{n_0, k_0, N}^{-1}$ and the clustering property of 
the singular sites to obtain
\begin{equation}
    \label{eqn: full-gevrey-small}
    \left| \left( T_{n_0, k_0,N}^{-1} \right) \left( (n, k), (n', k') \right) \right| \leq e^{ - \frac{(|n-n'| + |k-k'|)^s}{2} },
\end{equation}
for any $|n-n'| + |k-k'| \geq N^{\frac12}$. 
The key feature of the finite-dimensional setting is that truncation applies 
only to the temporal index $k$, while the spatial index ranges over a fixed 
finite set. 
Moreover, the singular sites are concentrated near a single temporal index 
$k_{\star}$ and form a cluster whose diameter is bounded independently of $N$. 
By taking the perturbation parameter $\varepsilon>0$ sufficiently small and 
then $N$ sufficiently large, we ensure that this diameter is negligible 
compared with $N^{\frac12}$. 
The block identities above therefore allow us to transfer decay from the 
inverse of the nonsingular restriction to the full inverse.

In the infinite-dimensional case, we first consider the tangential 
configuration~\eqref{eqn: singular-sites-tangent}. 
By the definition of the set of tangential modes $\mathcal{N}$, we have
\[
    M_0 \geq \left(2 \max_{n_1, n_2 \in \mathcal{N}} |n_{1} \pm n_{2}| + 2 \right)^{3}. 
\]
Consequently, any two singular sites $(n', k'), (n'', k'') \in \Pi$ satisfy 
$|n'-n''| + |k'-k''| \leq M_0^{\frac13}$. 
Let $(n, k) \in \Gamma_{n_0, k_0,N} \setminus \Pi$ be a nonsingular site and 
$(n', k') \in \Pi$ a singular site such that $|n-n'| + |k-k'| \geq N^{\frac12}/2$. 
In the block 
identities~\eqref{eqn: resolvent-small-inversion-local-relation-1}--\eqref{eqn: resolvent-small-inversion-local-relation-2}, 
the potentially problematic contributions involve intermediate singular sites 
$(n'', k'') \in \Pi$. 
The diameter bound above ensures, however, that each such site remains 
sufficiently far from $(n, k)$. 
Indeed, provided that $N^{\frac12} \geq 6M_0^{\frac13}$, the triangle 
inequality yields
\begin{equation}
    \label{eqn: jump-singular-small-tangential}
    |n-n''| + |k-k''| \geq |n-n'| + |k-k'| - M_0^{\frac13} \geq \frac{2 \left( |n-n'| + |k-k'|\right)}{3}. 
\end{equation}
Applying the estimate~\eqref{eqn: jump-singular-small-tangential} in the block 
identities~\eqref{eqn: resolvent-small-inversion-local-relation-1}--\eqref{eqn: resolvent-small-inversion-local-relation-2}, 
together with the symmetry of $T_{n_0, k_0, N}^{-1}$, we extend the decay 
estimates to the full inverse. 
Thus, under the inverse bound 
$\|T_{n_0, k_0, N}^{-1} \|_2 \leq \varepsilon_{N}^{-1}/2$, we obtain the 
localization estimate~\eqref{eqn: full-gevrey-small} whenever 
$|n-n'| + |k-k'| \geq N^{\frac12}$.

We next consider the normal configuration~\eqref{eqn: singular-sites-normal}, 
distinguishing between the cases $|n_{\star}| \leq N^{\frac13}/2$ and 
$|n_{\star}| > N^{\frac13}/2$. 
In the first case, $\Pi$ consists of two singular sites whose distance is at 
most $2|n_{\star}| \leq N^{\frac13}$. 
Let $(n, k) \in \Gamma_{n_0, k_0, N} \setminus \Pi$ and $(n', k') \in \Pi$ 
satisfying $|n-n'| + |k-k'| \geq N^{\frac12}/2$, and denote the other singular 
site in $\Pi$ by $(n'', k'')$. 
Provided that $N^{\frac12} \geq 6N^{\frac13}$, the triangle inequality yields   
\begin{equation}
    \label{eqn: jump-singular-small-normal}
    |n-n''| + |k-k''| \geq |n-n'| + |k-k'| - N^{\frac13} \geq \frac{2 \left( |n-n'| + |k-k'| \right) }{3}. 
\end{equation}
The argument for the tangential configuration therefore applies and, under the 
same inverse bound, yields the localization 
estimates~\eqref{eqn: full-gevrey-small} whenever 
$|n-n'| + |k-k'| \geq N^{\frac12}$. 
In the second case, the two possible singular sites are separated by a distance 
$2|n_{\star}| > N^{\frac13}$. 
If both lie in the lattice box $\Gamma_{n_0, k_0,N}$, the diameter of $\Pi$ 
need not be small relative to $N^{\frac{1}{2}}$, so a triangle-inequality 
estimate of the form~\eqref{eqn: jump-singular-small-tangential} 
or~\eqref{eqn: jump-singular-small-normal} is no longer guaranteed. 
To overcome this difficulty, we use the fact that~\Cref{lem: small-singular} 
holds at any scale between $1$ and $M$, and consider subboxes 
$\Gamma_{n,k,K} \subset \Gamma_{n_0, k_0,N}$ of scale $K = N^{\frac{1}{10}}$. 
Since their diameters are smaller than $N^{\frac{1}{3}}$, each subbox 
$\Gamma_{n,k,K}$ contains at most one singular site, as illustrated 
in~\Cref{fig: singular-site-divide}. 
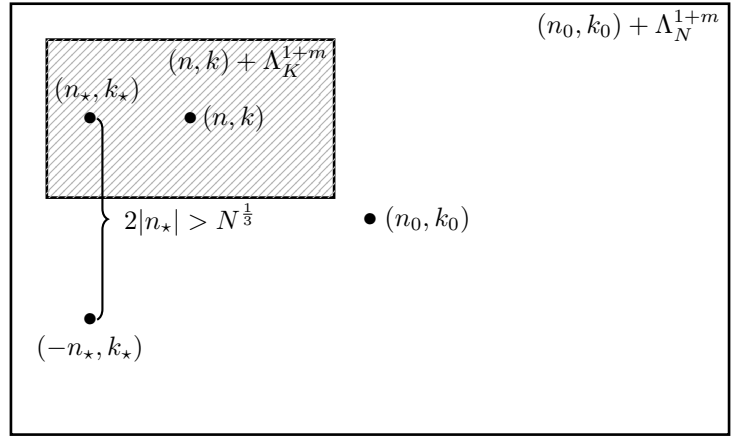
\begin{figure}[htb!]
    \centering
    \begin{tikzpicture}[scale=0.95, >=stealth]
        \draw[line width=1.0pt] (-5.0, -3) rectangle (5.0, 3);
        \draw[line width=1.0pt] (-4.5, 0.3) rectangle (-0.5, 2.5);
            \fill[pattern=north east lines, pattern color=gray!60] (-4.5, 0.3) rectangle (-0.5, 2.5);
        
        \draw[line width=0.2pt] (-2.5, 1.4) node {$\bullet$};
        \draw[line width=0.5pt] (-3.9, 1.4) node {$\bullet$};
        \draw[line width=0.5pt] (-3.9, -1.4) node {$\bullet$};
        \draw[line width=0.5pt] (0, 0) node {$\bullet$};
        \draw[decorate, decoration={brace, mirror, amplitude=4pt}, line width=0.8pt] (-3.8, -1.4) -- (-3.8, 1.4);
        \node at (-2.5, 0) {\small $2|n_{\star}| > N^{\frac13}$};  
          
         \node at (3.6, 2.7) {\small $(n_0,k_0) + \Lambda_N^{1+m}$};  
         \node at (-1.7, 2.2) {\small $(n,k) + \Lambda_K^{1+m}$};  
         \node at (-3.8, 1.8) {\small $(n_{\star}, k_{\star})$};  
         \node at (-3.9, -1.8) {\small $(-n_{\star}, k_{\star})$};  
         \node at (-1.9, 1.4) {\small $(n,k)$};  
         \node at (0.8, 0) {\small $(n_0,k_0)$}; 
    \end{tikzpicture}
    \caption{Scale reduction from $\Gamma_{n_0, k_0,N}$ to $\Gamma_{n, k,K}$ in the case $|n_{\star}| > N^{\frac13}/2$.}
    \label{fig: singular-site-divide}
\end{figure}
Provided that $\|T_{n_0, k_0, N}^{-1}\|_2 \leq \varepsilon_N^{-1}/2$, the 
Schur complement formula yields $\|T_{n, k, K}^{-1}\|_2 \leq \varepsilon_N^{-1}/2$. 
The preceding block-identity argument therefore applies within each such subbox. 
Under the local inverse bound $\|T_{n, k, K}^{-1} \|_2 \leq \varepsilon_{N}^{-1}/2$, 
we obtain the corresponding localization 
estimates~\eqref{eqn: full-gevrey-small}, with $N$ replaced by $K$ whenever 
$|n-n'| + |k-k'| \geq K^{\frac12}$. 
Finally, we invoke the gluing argument of~\Cref{subsec: verification} to 
propagate these local inverse estimates to the full box $\Gamma_{n_0, k_0,N}$. 
The details are given as follows.

With respect to the decomposition 
$\Gamma_{n_0, k_0, N} = \Gamma_{n, k, K} \cup \Gamma_{n, k, K}^{c}$, where 
$\Gamma_{n, k, K}^{c} = \Gamma_{n_0, k_0, N} \setminus \Gamma_{n, k, K}$, the 
resolvent identity gives the block representation  
\begin{align}
    \begin{pmatrix} 
        T^{-1}_{n,k, K} & 0  \\ 
        0 & T^{-1}_{n, k, K,-} 
    \end{pmatrix} 
    = & \left[ \mathrm{Id} + \varepsilon 
    \begin{pmatrix} 
        0 & T_{n,k, K}^{-1}P^{*}  \\  
        T^{-1}_{n,k, K,-}P & 0 
    \end{pmatrix}  
    \right]  \nonumber \\
      &  \cdot
    \begin{pmatrix} 
        (T^{-1}_{n_0, k_0, N})_{\Gamma_{n, k, K} \times \Gamma_{n, k, K} } & (T^{-1}_{n_0, k_0, N})_{\Gamma_{n, k, K} \times \Gamma_{n, k,K}^{c}}  \\ 
        (T^{-1}_{n_0, k_0, N})_{\Gamma_{n, k, K}^{c} \times \Gamma_{n, k, K} } & (T^{-1}_{n_0,k_0, N})_{\Gamma_{n, k, K}^c \times \Gamma_{n, k, K}^c} 
    \end{pmatrix},      \label{eqn: resolvent-small-offside}
\end{align}
where $P$ and $P^*$ are the coupling operators inherited from $S$ and 
$T^{-1}_{n, k, K, -}$ denotes the inverse of the operator $T$ restricted to 
$\Gamma_{n, k, K}^{c}$. 
Combining the two identities in the first row 
of~\eqref{eqn: resolvent-small-offside}, extending the terms to the full 
index set $\Gamma_{n_0, k_0, N}$, and extracting the row indexed by $(n,k)$, 
we obtain
\begin{align}
    \left( T_{n_0, k_0,N}^{-1} \right)_{ \{ (n, k) \} \times \Gamma_{n_0, k_0, N} } = 
    & \begin{pmatrix} (T_{n,k, K}^{-1})_{ \{ (n, k) \} \times \Gamma_{n, k, K}} & 0 \end{pmatrix} \nonumber \\ 
    & - \varepsilon \begin{pmatrix} 0 & (T_{n,k,K}^{-1}P^*)_{\{ (n, k) \} \times \Gamma_{n, k, K}^{c}} \end{pmatrix} T_{n_0, k_0, N}^{-1}.  \label{eqn: block-representation-small-offside}
\end{align}
Assembling~\eqref{eqn: block-representation-small-offside} over all indices 
$(n, k) \in \Gamma_{n_0, k_0, N}$ yields the operator identity
\begin{equation}
    \label{eqn: resolvent-relation-small-offside}
    T_{n_0, k_0, N}^{-1} = \Psi_{n_0, k_0, N} - \varepsilon \Phi_{n_0, k_0, N} T_{n_0, k_0, N}^{-1}, 
\end{equation}
where the linear operators $\Psi_{n_0, k_0, N}$ and $\Phi_{n_0, k_0, N}$ satisfy
\begin{equation}
    \Psi_{n_0, k_0, N}\left( (n,k), (n',k') \right) 
    = \left\{ 
    \begin{aligned}       
        & T_{n,k, K}^{-1}\left( (n,k), (n',k') \right),  &&\mathrm{if}\;  (n', k')  \in \Gamma_{n,k, K},    \\
        & 0,                                             &&\mathrm{otherwise},   
    \end{aligned} 
    \right.   \label{eqn: entries-psi-small-offside}
\end{equation}
and
\begin{equation}
    \label{eqn: entries-phi-small-offside}
    \left| \Phi_{n_0, k_0, N}\left( (n,k), (n',k') \right) \right| \leq \left\{ 
    \begin{aligned} 
        & 0,                                                                              &&\mathrm{if}\; |n - n'| + |k - k'| < K, \\
        & e^{- \frac{ \left(|n - n'| + |k - k'| \right)^{s}}{2} },         &&\mathrm{if}\;  K \leq |n - n'| + |k - k'|  < K^{\frac32},  \\ 
        & e^{ - \frac{ 3\left(|n - n'| + |k - k'| \right)^{s}}{4} },       &&\mathrm{if}\; |n - n'| + |k - k'| \geq K^{\frac32}.
    \end{aligned} 
    \right.
\end{equation} 
Following the procedure in~\Cref{subsec: verification}, we divide the paths 
into two classes and show that $(\mathrm{Id} + \varepsilon \Phi_{n_0, k_0, N})^{-1}$ satisfies an estimate of 
the form~\eqref{eqn: inverse-decay-long} whenever 
$|n-n'| + |k - k'| \geq N^{\frac14}$. 
From~\eqref{eqn: entries-psi-small-offside}, we further obtain 
$\|\Psi_{n_0, k_0, N}\|_{2} \leq N^{m+1} \varepsilon_N^{-1}$. 
We thus obtain the required localization bounds for $T_{n_0, k_0, N}^{-1}$ at 
any scale $N$ in the small scale regime $M_0 \leq N \leq M$, uniformly in the 
spatial center $n_0$. 
We summarize this conclusion in the following ``inversion implies localization'' 
lemma.

\begin{lemma}
    \label[lemma]{lem: small-inverse-localization}
    Suppose that $\hat{u} \in \mathcal{K}(s)$, 
    $\omega' \in \Omega \setminus \Omega_M^\tau$, and $M_0 \leq N \leq M$. 
    If the restricted operator $T_{n_0, k_0, N}$ is invertible and satisfies 
    the norm bound:
    \begin{equation}
        \label{eqn: trans-small-size-original-inverse}
        \| T_{n_0,  k_0, N}^{-1} \|_2 \le \frac{1}{2\varepsilon_N},
    \end{equation}
    then its entries satisfy the Gevrey decay estimate
    \begin{equation}
        \label{eqn: trans-small-size-off-diagonal}
        \left| T_{n_0, k_0, N}^{-1}\left( (n, k), (n', k') \right) \right| \leq e^{-\frac{\left( |n - n'| + |k - k'| \right)^s}{2}}, 
    \end{equation}
    for any $|n-n'| + |k-k'| \geq N^{\frac12}$. 
\end{lemma}

\subsection{Scale reduction: temporal-index clustering of singular sites}
\label{sec: scale-reduce-cluster}

We now turn to the large-scale regime $N > M$. 
When the singular set $\Pi$ has either the tangential 
configuration~\eqref{eqn: singular-sites-tangent} or the normal 
configuration~\eqref{eqn: singular-sites-normal} with 
$|n_{\star}| \leq N^{\frac13}/2$, we work with lattice boxes of size $N$. 
For the normal configuration~\eqref{eqn: singular-sites-normal} of 
$|n_{\star}| > N^{\frac13}/2$, we instead work with boxes of size 
$N^{\frac{1}{10}}$. 
The analysis below applies to both cases; for notational convenience, we 
denote the corresponding box size by $N$ throughout this subsection and the 
next, and return to the distinction between the two scales at the end of the 
next subsection.

Before starting the scale reduction, we first record an elementary observation. 
The eigenvalues of the restricted operator $T_{n_0, k_0, N}$ are given by
\begin{equation}
    \label{eqn: t-n0-N-eigen}
    T_{n_0, k_0, N}(n,k) = - \langle k, \omega' \rangle + \varpi_n + \nu(n, k), 
\end{equation}
where the perturbation $\nu(n, k)$ depends on $\omega'$. 
By~\Cref{prop: convolution}, we have the derivative estimate 
\begin{equation}
    \label{eqn: perturb-omega-t}
    \max_{(n,k) \in \mathbb{Z}^{1+m}} \left\| \frac{\partial \nu(n,k)}{\partial \omega'} \right\|_2 \leq \varepsilon  \left\| \frac{\partial S}{\partial \omega'} \right\|_2 \leq \frac{\sqrt{\varepsilon}}{2}. 
\end{equation}
The difference between any two eigenvalues of $T_{n_0, k_0, N}$ can therefore 
be written as 
\begin{equation}
    \label{eqn: eigenvalue-difference}
    T_{n_0, k_0, N}(n_1, k_1) - T_{n_0, k_0, N}(n_2, k_2) = - \langle k, \omega' \rangle + \varpi_{n_1} - \varpi_{n_2} + \nu(n_1, k_1) - \nu(n_2, k_2),
\end{equation}
where $k = k_1 - k_2$. To control the corresponding small divisor, we exclude 
the resonant set along the line parallel to $k / \| k \|_2$. 
By~\eqref{eqn: perturb-omega-t}, the absolute value of the directional 
derivative of the eigenvalue difference is bounded below by 
$\|k\|_2 - \sqrt{\varepsilon}$. 
Consequently, along any such line segment within the parameter domain, the 
one-dimensional measure of the resonant set is bounded above by 
$2 (\| k\|_2 - \sqrt{\varepsilon})^{-1}$ times the resonance threshold.

We first introduce a reduced scale $N'$, derived from the primary scale $N$ via 
the relation:
\begin{equation}
    \label{eqn: 10-log-size}
    N' = e^{ (\log N)^{\frac{1}{10}} }. 
\end{equation}
The difference nearly-resonant set for the large-scale regime is then defined 
as the union of its components across all scales $N>M$. 
\begin{defn}
    \label[definition]{defn: second-melnikov}
    The difference nearly-resonant set in the large-scale regime is defined as 
    \[
        G_2 = \bigcup_{N > M} G_{2, N},
    \] 
    where each component set $G_{2, N}$ is given by
    \begin{equation}
        \label{eqn: difference-nearly-resonant-multiscale}
        G_{2, N} = \bigcup_{k \in \Lambda^m_{2N} \setminus \Lambda^{m}_{2N'}} \bigcup_{n_1, n_2 \in \mathbb{Z}} \left\{ \omega' \in \Omega ~\bigg|~ \left| -\langle \omega', k \rangle + n_1^2 - n_2^2 + \nu(n_1, k_1) - \nu(n_2, k_2) \right| < 4\varepsilon_{N'} \right\},
    \end{equation}
\end{defn}

As detailed in~\Cref{sec: near-measure-initial}, it suffices to count the 
distinct admissible values of the difference $n_1^2 - n_2^2$ rather than the 
pairs $(n_1, n_2)$ themselves. 
Thus, based on the scale relation~\eqref{eqn: 10-log-size}, the component set 
satisfies 
\[
    \mathrm{mes}(G_{2}) \leq \sum_{N>M} \mathrm{mes}(G_{2, N}) \leq \sum_{N>M} [\text{diam}(\Omega)]^{m-1} 16\sqrt{m}N(2N+1)^m e^{-(\log N)^{3/2}} < \infty.
\]
In particular, since $M \rightarrow \infty$ as $\varepsilon \rightarrow 0$, the 
convergence of the above series implies $\mathrm{mes}(G_2) \rightarrow 0$.

We now consider two subboxes $\Gamma_{n_{0,1}, k_{0,1}, N'}$ and 
$\Gamma_{n_{0,2}, k_{0,2}, N'}$ with centers $(n_{0,1}, k_{0,1})$, 
$(n_{0,2}, k_{0,2}) \in \Gamma_{n_0, k_0, N}$. 
If the two centers satisfy the separation condition 
$|k_{0,1} - k_{0,2}|_{\infty} > 4N'$, then at least one of the restricted 
operators $T_{n_{0,1}, k_{0,1}, N'}$ and $T_{n_{0,2}, k_{0,2}, N'}$ has its 
smallest eigenvalue (in absolute value) bounded below by $2\varepsilon_{N'}$. 
Equivalently, at most one of the resolvent norms, 
$\|T_{n_{0,1}, k_{0,1}, N'}^{-1}\|_2$ and 
$\|T_{n_{0,2}, k_{0,2}, N'}^{-1}\|_2$, can exceed $\varepsilon_{N'}^{-1}/2$. 
In other words, the centers of the subboxes on which the restricted operator is 
nearly singular cannot be far apart in the temporal direction. 
To describe the resulting clustering in the temporal indices, we define the set 
of singular centers within $\Gamma_{n_0, k_0, N}$ as
\begin{equation}
    \label{eqn: eqn-violated-center-pde}
    \Sigma(N, N') = \left\{ (n_{0}', k_{0}') \in \Gamma_{n_0, k_0, N} ~\bigg|~ \| T_{n_{0}', k_{0}', N'}^{-1} \|_2 > \frac{1}{2\varepsilon_{N'}} \right\}.
\end{equation}
The following lemma characterizes this clustering.
\begin{lemma}
    \label[lemma]{lem: second-melnikov}
    Suppose that $\hat{u} \in \mathcal{K}(s)$ and $\omega' \in \Omega \setminus (\Omega_M^\tau \cup G_2)$. If $(n_{0,1}, k_{0, 1}), (n_{0,2}, k_{0, 2}) \in \Sigma(N, N')$, then 
    they must be clustered such that:
    \begin{equation}
        \label{eqn: center-clustering}
        k_{0, 1} - k_{0, 2} \in \Lambda_{4N'}^{m}.
    \end{equation}
\end{lemma}
It is worthwhile to note that the clustering in~\Cref{lem: second-melnikov} 
concerns only the temporal index, which is confined to a set of the form
$k_{0}' + \Lambda_{4N'}^{m}$; it imposes no additional restriction on 
the spatial index $n_0'$, which may range over all of $n_0 + \Lambda_N$, as 
shown in~\Cref{fig: k-box-reduction}. 
\begin{figure}[htb!]
    \centering
    \begin{tikzpicture}[scale=0.95, >=stealth]
        \draw[line width=1.0pt] (-4.0, -2) rectangle (4.0, 2);
        \draw[dashed, line width=0.8pt] (0.5, -2) -- (0.5, 2);
        \draw[dashed, line width=0.8pt] (2.5, -2) -- (2.5, 2);
        \fill[pattern=north east lines, pattern color=black!70] (0.5, -2) rectangle (2.5, 2);
        
        \draw[line width=0.2pt] (1.0, 1) node {$\bullet$};
        \draw[line width=0.2pt] (2.0, -1) node {$\bullet$};
        \draw[line width=0.2pt] (0.0, 0) node {$\bullet$};
       
        \node at (2.0, 1.0) {\small $(n_{0,1}', k_{0,1}')$};  
        \node at (3.0, -1.0) {\small $(n_{0,2}', k_{0,2}')$}; 
        \node at (0.8, 0.0) {\small $(n_{0}, k_{0})$}; 
        
        \draw[decorate, decoration={brace, mirror, amplitude=4pt}, line width=0.8pt] (-4.0, -2) -- (4.0, -2);
        \node at (0.2, -2.5) {\small $k_0 + \Lambda_{N}^{m}$};
        
        \draw[decorate, decoration={brace, amplitude=4pt}, line width=0.8pt] (0.5, 2) -- (2.5, 2);
        \node at (1.6, 2.5) {\small $k_0' + \Lambda_{4N'}^{m}$};
        
        \draw[decorate, decoration={brace, mirror, amplitude=4pt}, line width=0.8pt] (4.0, -2) -- (4.0, 2);
        \node at (5.0, 0.0) {\small $n_0 + \Lambda_{N}$};
    \end{tikzpicture}
    \caption{Clustering of the temporal indices.}
    \label{fig: k-box-reduction}
\end{figure}
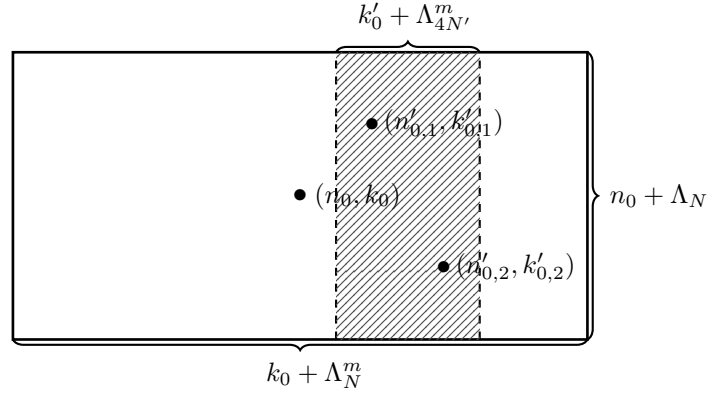

\subsection{Multi-scale induction: inversion implies localization}
\label{sec: inversion-localization}

Based on~\Cref{lem: second-melnikov}, the singular centers are contained in a 
union of subboxes $\Gamma_{n_0', k_0',4N'}$ that share a common temporal index 
$k_0' \in k_0 + \Lambda_N^{m}$, while $n_0'$ ranges over $n_0 + \Lambda_N$. 
More precisely, there exists $k_0' \in k_0 + \Lambda_N^{m}$ such that
\[
    \Sigma(N, N') \subseteq \bigcup_{n_0' \in n_0 + \Lambda_N} \Gamma_{n_0', k_0',4N'} \subseteq \bigcup_{n_0' \in n_0 + \Lambda_N} \Gamma_{n_0', k_0',10N'},
\]  
where the second inclusion is immediate from 
$\Gamma_{n_0', k_0',4N'} \subseteq \Gamma_{n_0', k_0',10N'}$. 
We refer to the enlarged boxes $\Gamma_{n_0', k_0', 10N'}$ as singular boxes, 
and we call any subbox whose center does not belong to $\Sigma(N, N')$ a 
regular box. 
With this classification, the box $\Gamma_{n_0, k_0, N} $ decomposes into a 
union of singular boxes, all sharing the common temporal index $k_0'$ while 
$n_0'$ ranges over $n_0 + \Lambda_N$, together with a collection of regular 
boxes:
\begin{equation}
    \label{eqn: multi-scale-decomposition}
    \Gamma_{n_0, k_0, N} = \left( \bigcup_{n_0' \in n_0 + \Lambda_N} \Gamma_{n_0', k_0',10N'} \right) \bigcup \left( \bigcup_{k_0'' \notin  k_0' + \Lambda_{9N'}^{m}} \bigcup_{n_0'' \in n_0 + \Lambda_N} \Gamma_{n_0'', k_0'',N'} \right).
\end{equation}
Note that a box $\Gamma_{n_0'', k_0'', N'}$ with 
$k_0'' \in (k_0 + \Lambda_{N}^{m}) \setminus ( k_0' + \Lambda^{m}_{9N'})$ may 
overlap some singular boxes $\Gamma_{n_0', k_0', 10N'}$, but it is 
nevertheless regular: since $|k''_0 - k_0'| \geq 9N'$, its temporal range lies 
at distance no less than $8N'$ from $k_0'$, so it is disjoint from any smaller 
box $\Gamma_{n_0', k_0', 4N'}$, and it is these smaller boxes that contain 
$\Sigma(N, N') $. 
This deliberate enlargement facilitates the subsequent gluing procedure, 
consistent with the strategy detailed in~\Cref{subsec: verification}.

We repeat the scale reduction process for each box in the 
decomposition~\eqref{eqn: multi-scale-decomposition}. 
This iteration proceeds through a sequence of scales satisfying the hierarchy 
in~\eqref{eqn: 10-log-size} the scales enter the small-scale regime $[M_0, M]$. 
For each terminal singular box in the small-scale regime,
\Cref{prop: small-eig} gives one of the singular configurations
classified in~\eqref{eqn: singular-sites-tangent}--\eqref{eqn: singular-sites-normal},
and the corresponding arguments of~\Cref{subsec: small-outer} apply.
Within the small-scale regime, we have established the following results.
\begin{itemize}
    \item[(i)] For the singular box with the singular set $\Pi$ 
    removed,~\Cref{lem: small-singular} guarantees that the inversion 
    condition~\eqref{eqn: small-size-inverse-prepare} and the localization 
    condition~\eqref{eqn: small-size-off-diagonal-prepare} hold 
    simultaneously. 
    
    \item[(ii)] For the regular boxes,~\Cref{lem: small-inverse-localization} 
    establishes that the inversion 
    condition~\eqref{eqn: trans-small-size-original-inverse} directly implies 
    the localization condition~\eqref{eqn: trans-small-size-off-diagonal}. 
\end{itemize}

With these small-scale results serving as the inductive bases, we now show 
that inversion implies localization throughout the large-scale regime $N > M$, 
rigorously stated as the following lemma.

\begin{lemma}
    \label[Lemma]{lem: multi-scale-singular}
    Suppose that $\hat{u} \in \mathcal{K}(s)$ and 
    $\omega' \in \Omega \setminus (\Omega_{M}^{\tau} \cup G_{2})$. 
    If the restricted operator $T_{n_0, k_0, N}$ satisfies the inversion 
    condition~\eqref{eqn: trans-small-size-original-inverse} for any scale 
    $M^r < N \leq M^{r+1}$, then the entries of its inverse satisfy the 
    localization condition
    \begin{equation}
        \label{eqn: trans-large-size-off-diagonal}
        \left| T_{n_0, k_0, N}^{-1}\left( (n, k), (n', k') \right) \right| \leq e^{ - \frac{ (|n-n'| + |k-k'|)^s}{2} },
    \end{equation}
    for any $|n-n'| + |k-k'| \geq N^{\frac12}$. 
\end{lemma}

The proof consists of two successive steps:  a ``gluing'' procedure, followed 
by a demonstration that inversion implies localization.

\paragraph{The ``gluing'' procedure} 
By the inductive hypothesis for $M_0 \leq N \leq M^{r}$, the following 
statements hold:
\begin{itemize}
    \item For each singular box $\Gamma_{n_0', k_0' ,10N'}$ with 
    $M_0 \leq 10N' \leq M^r$, the operator restricted to 
    $\Gamma_{n_0', k_0' ,10N'} \setminus \Pi$, that is, to the box with the 
    singular sites removed, satisfies the inversion 
    condition~\eqref{eqn: small-size-inverse-prepare} and the localization condition~\eqref{eqn: small-size-off-diagonal-prepare}. 
    
    \item For each regular box $\Gamma_{n_0'', k_0'' ,N'}$ with 
    $k_0'' \in (k_0 + \Lambda_{N}^{m}) \setminus ( k_0' + \Lambda^{m}_{9N'})$ 
    and $M_0 \leq N' \leq M^r$, the inversion 
    condition~\eqref{eqn: trans-small-size-original-inverse} for the restricted 
    operator implies the localization 
    condition~\eqref{eqn: trans-small-size-off-diagonal}. 
\end{itemize}
Now let $M^r < N \leq M^{r+1}$. 
By the decomposition~\eqref{eqn: multi-scale-decomposition}, the set 
$\Gamma_{n_0, k_0, N} \setminus \Pi$ is covered by the two types of boxes 
described above: singular boxes with the singular sites removed and regular 
boxes. 
Gluing the corresponding local inverses into a global inverse, exactly as in 
the procedure detailed in~\Cref{subsec: verification}, we conclude 
$T_{\Gamma_{n_0, k_0, N} \setminus \Pi}$ for $M^{r} < N \leq M^{r+1}$ satisfies 
the inversion condition~\eqref{eqn: trans-small-size-original-inverse} and the 
localization condition~\eqref{eqn: trans-small-size-off-diagonal}. 
We therefore omit the routine details.

\paragraph{Inversion implies localization} This step proceeds exactly as in the 
argument for the small-scale regime $M_0 \leq N \leq M$ 
in~\Cref{subsec: small-outer} 
(see the paragraph ``Inversion implies localization'' there). 
For the normal configuration~\eqref{eqn: singular-sites-normal} of 
$|n_{\star}| > N^{\frac13}/2$, the estimates obtained there concern boxes of 
size $K = N^{\frac{1}{10}}$; repeating the 
procedure~\eqref{eqn: resolvent-small-offside}--\eqref{eqn: entries-phi-small-offside}, 
we complete the proof.

\subsection{Measure estimate for ``bad'' frequencies}
\label{subsec: measure-estimate}

The final stage of the multiscale analysis consists in quantifying the set of 
``bad''  frequencies that must be excluded in order for the inversion 
condition~\eqref{eqn: trans-small-size-original-inverse} to remain valid. 
In this section, we establish the rigorous measure estimates for these 
resonant sets. 
Recall the restricted operator $T_{n_0, k_0, N}$ admits the block 
decomposition 
\[
    T_{n_0, k_0, N} = 
    \begin{pmatrix} 
        T_{\Gamma_{n_0, k_0, N} \setminus \Pi} & \varepsilon P^*  \\ 
        \varepsilon P & T_{\Pi} 
    \end{pmatrix}.
\]
In terms of this block structure, the Schur complement of the block 
$T_{\Gamma_{n_0, k_0, N} \setminus \Pi} $ is
\begin{equation}
    \label{eqn: schur-complement}
    U = T_{\Pi} - \varepsilon^2 P T_{\Gamma_{n_0, k_0, N} \setminus \Pi}^{-1} P^*.
\end{equation}
From~\eqref{eqn: schur-complement}, we establish an upper bound for the 
$\ell_2$-norm of the inverse operator:
\[
    \| T_{n_0, k_0, N}^{-1} \|_2 \leq 4 \| T_{\Gamma_{n_0, k_0, N} \setminus \Pi}^{-1} \|_2^2 \cdot \| U^{-1} \|_2 + \| T_{\Gamma_{n_0, k_0, N} \setminus \Pi}^{-1} \|_2. 
\]
Since the bound $\| T_{\Gamma_{n_0, k_0, N} \setminus \Pi}^{-1} \|_2 \leq (20N'+1)^{1+m} e^{( \log 10N')^{15}}$ 
has already been established, we deduce the inclusion of resonant sets
\begin{equation}
    \label{eqn: inverse-relation-schur-complement}
    \left\{ \omega' \in \Omega ~\bigg|~ \| T_{n_0, k_0, N}^{-1} \|_2 > \frac{1}{2\varepsilon_N} \right\} \subseteq \left\{ \omega' \in \Omega ~\bigg|~ \| U^{-1} \|_2 > \frac{1}{\sqrt{\varepsilon_N}} \right\},
\end{equation}
so that it suffices to analyze the Schur complement $U$ for the purposes of the 
measure estimates. 
We therefore write~\eqref{eqn: schur-complement} more explicitly; two cases 
arise. 
For the tangential configuration~\eqref{eqn: singular-sites-tangent}
\begin{equation}
    \label{eqn: u-tangential}
    U = \mathrm{diag}\left( - \langle k_{\star}, \omega' \rangle - \omega'_j + \omega_j\right) + \varepsilon S_{\Pi} - \varepsilon^2 P T_{\Gamma_{n_0, k_0, N} \setminus \Pi}^{-1} P^*. 
\end{equation}
and for the normal configuration~\eqref{eqn: singular-sites-normal},
\begin{equation}
    \label{eqn: u-normal}
    U = \mathrm{diag}\left( - \langle k_{\star}, \omega' \rangle + n_{\star}^2 + \mu \right) + \varepsilon S_{\Pi} - \varepsilon^2 P T_{\Gamma_{n_0, k_0, N} \setminus \Pi}^{-1} P^*. 
\end{equation} 
Accordingly, for the tangential configuration we set
$\sigma_{\star}=-\langle k_{\star},\omega'\rangle$, whereas for the normal
configuration we set
$\sigma_{\star}=-\langle k_{\star},\omega'\rangle+n_{\star}^2+\mu$.
Then, in both cases, we write
\begin{equation}
    \label{eqn: u-bar-form}
    U = \sigma_{\star} \mathrm{Id} + \Upsilon.
\end{equation}
Since $\| T_{\Gamma_{n_0, k_0, N} \setminus \Pi}^{-1} \|_2 \leq (20N'+1)^{1+m} e^{( \log 10N')^{15}}$, 
both~\eqref{eqn: u-tangential} and~\eqref{eqn: u-normal} implies that 
$\| \Upsilon \|_2 \leq (20N'+1)^{2+m} e^{( \log 10N')^{15}}$. 
Combining~\eqref{eqn: inverse-relation-schur-complement} 
and~\eqref{eqn: u-bar-form}, we define the single-mode nearly-resonant set at 
the scale $N$ as:
\begin{equation}
    \label{eqn: single-mode-nearly-resonant-multiscale}
    G_{1,N} = \bigcup_{(n, k) \in \Lambda_{N^{10}}^{1+m}} \left\{ \omega' \in \Omega ~\bigg|~ \big| |\sigma_{\star}| - \|\Upsilon\|_2 \big| < \sqrt{\varepsilon_N} \right\},
\end{equation}
which ensures that the inversion 
condition~\eqref{eqn: trans-small-size-original-inverse} holds for all frequencies 
$\omega' \notin G_{1,N}$. 
To estimate the measure of this set, we utilize a simplified version of the 
Cartan estimate~\citep{levin1996lectures}.

\begin{lemma}[Theorem 4 of Lecture 11.3 in~\citet{levin1996lectures}]
    \label{lem: cartan}
    Let $f(z)$ be an analytic function in the complex disk 
    $\{ z \in \mathbb{C} \mid |z| \leq 2eR \}$ with $|f(0)| = 1$. 
    Let $M_f(r) = \max_{|z|=r} |f(z)|$, and let $\eta > 0$ be an arbitrarily 
    small positive number. 
    Then the estimate
    \[
        \log |f(z)| > - H(\eta) \log M_f(2eR), \qquad H(\eta) = \log \frac{15e^3}{\eta},
    \]
    holds everywhere in the disk $\{z \in \mathbb{C} \mid |z| \leq R\}$ except 
    for a family of disks $(C_j)$ with sum of radii $\sum r_j \leq \eta R$.
\end{lemma}

To complete the measure estimates, we analyze the magnitude of the residual 
term $\Upsilon$ in two distinct regimes. 
For small residuals, namely $ \| \Upsilon \|_2 < \varepsilon_N^{\frac{1}{4}}$, 
the single-mode nearly-resonant set is contained within
\begin{equation}
    \label{eqn: single-msa-small-residual}
    G_{1,N} \subseteq \bigcup_{(n, k) \in \Lambda_{N^{10}}^{1+m}} \left\{ \omega' \in \Omega ~\bigg|~ |\sigma_{\star}| < 2\varepsilon_N^{\frac{1}{4}} \right\},
\end{equation}
and we note that this case encompasses the small-scale regime, where the tighter 
condition $|\sigma_\star| < \varepsilon_N$ is in fact sufficient; the measure 
of this set is bounded by
\begin{equation}
    \label{eqn: mes-g-1-N-1}
    \mathrm{mes}(G_{1,N}) \leq 4\left[\text{diam}(\Omega)\right]^{m-1} (2N^{10}+1)^{1+m} \varepsilon_N^{\frac{1}{4}}. 
\end{equation}
When the residual term is large, i.e., 
$\| \Upsilon \|_2 \geq \varepsilon_N^\frac{1}{4}$, 
we define an analytic function:
\[
    f(\sigma_{\star}) = \frac{|\sigma_{\star}|}{\| \Upsilon \|_2 } - 1. 
\]
with the disk radius $R$ chosen so that $\| \Upsilon \|_2 = 2e(1+e)^{-1}R$, 
whence $|f(0)| = 1$ and $M_f(2eR) = e$; accordingly, the single-mode 
nearly-resonant set satisfies:
\begin{equation}
    \label{eqn: single-msa-large-residual}
    G_{1,N} \subseteq \bigcup_{(n, k) \in \Lambda_{N^{10}}^{1+m}} \left\{ \omega' \in \Omega \bigg | |f(\sigma_{\star})| < \varepsilon_N^{1/4} \right\},
\end{equation}
and taking $\eta = 15e^3 \varepsilon_N^{1/4}$ in~\Cref{lem: cartan} yields the measure estimate
\begin{align}
    \mathrm{mes}(G_{1,N}) 
    & \leq 8(1+e)e^2 \left[\text{diam}(\Omega)\right]^{m-1} (2N^{10}+1)^{1+m} (20N'+1)^{1+m} e^{(\log 10N')^{15}} \varepsilon_{N}^{\frac{1}{4}}  \nonumber  \\
    & \leq 8(1+e)e^2 \left[\text{diam}(\Omega)\right]^{m-1} (2N^{10}+1)^{2+2m} \varepsilon_{N}^{\frac{1}{5}}.   \label{eqn: mes-g-1-N-2}
\end{align}
In view of~\eqref{eqn: single-mode-nearly-resonant-multiscale}, the total 
single-mode resonant set is the union over all scales $N \geq M_0$ as 
\begin{equation}
    \label{eqn: final-msa-single-mode-nearly-resonant}
    G_1 = \bigcup_{N \geq M_{0}} G_{1,N}. 
\end{equation}
and combining~\eqref{eqn: mes-g-1-N-1} and~\eqref{eqn: mes-g-1-N-2},  we obtain
\[
    \mathrm{mes}(G_1) \leq \sum_{N \geq M_0} 8(1+e)e^2 \left[\text{diam}(\Omega)\right]^{m-1} (2N^{10}+1)^{2+2m} \varepsilon_{N}^{\frac{1}{5}} < \infty.
\]
Since $M_0 \rightarrow \infty$ as $\varepsilon \rightarrow 0$, the convergence 
of this series implies $\mathrm{mes}(G_1) \rightarrow 0$ as 
$\varepsilon \rightarrow 0$. 
The set $G = G_1 \cup G_2$ corresponds to the set defined 
in~\Cref{thm: off-side-box}, and consequently, $\text{mes}(G) \leq \delta$ with 
$\delta \to 0$ as $\varepsilon \to 0$, in agreement with the statement 
of~\Cref{thm: main}.

\section{Convergence of the numerical scheme}
\label{sec: iterative-scheme}

In this section, we establish the theoretical foundation for the convergence of 
the proposed alternating numerical 
scheme~\eqref{eqn: alternating-numerical}.~\Cref{subsec: induction} derives the 
iteration lemma, which serves as the technical cornerstone of our analysis. 
Subsequently, in~\Cref{subsec: proof-main}, we utilize this result to complete 
the formal proof of the main statement.

\subsection{The iteration lemma}
\label{subsec: induction}

Before proceeding to the inductive estimates, we need to characterize the 
derivative of the inverse operator $\partial L_N^{-1}$, specifically its norm 
bounds and off-diagonal decay properties. 
Throughout, the derivative is taken with respect to the frequency $\omega'$, 
which we suppress from the notation for brevity. 
Differentiating the identity $L_N^{-1} L_N = I$ yields:
\begin{equation}
    \label{eqn: t-n-derivative-inverse}
    \partial L_{N}^{-1} = - L_{N}^{-1}(\partial L_{N} ) L_{N}^{-1}.
\end{equation}
This identity allows us to establish the following lemma on the norm bound and 
localization.

\begin{lemma}
    \label{lem: inverse-derivative}
    Suppose that the restricted operator $L_{N}$ 
    satisfies~\Cref{cond: implementation}. 
    Then its inverse satisfies the norm bound:
    \begin{equation}
        \label{eqn: inverse-restricted-norm}
        \| \partial L_{N}^{-1} \|_2 \leq \frac{4\sqrt{m}N}{\varepsilon_N^2}. 
    \end{equation}
    Furthermore, for sufficiently separated lattice sites in $\Lambda_N^{1+m}$, 
    specifically when $|n-n'| + |k-k'| \geq N^{\frac34}$, the entries of the 
    derivative of the inverse exhibit Gevrey-type decay:
    \begin{equation}
        \label{eqn: inverse-restricted-entries}
        \left| \partial L_{N}^{-1}\left( (n, k), (n', k') \right) \right| \leq e^{ - \frac{ \left( |n-n'| + | k - k' |\right)^{s}}{4}} .
    \end{equation}
\end{lemma}

The norm bound~\eqref{eqn: inverse-restricted-norm} follows directly from the 
inversion condition~\eqref{eqn: inverse-grow} and the derivative 
identity~\eqref{eqn: t-n-derivative-inverse}. 
The localization estimate~\eqref{eqn: inverse-restricted-entries} relies on 
elementary inequalities associated with Gevrey regularity; the detailed proof 
is provided in~\Cref{subsec: derivative-app}. 
Given that~\Cref{cond: implementation} has been verified in the preceding 
sections,~\Cref{lem: inverse-derivative} implies the iteration lemma.

\begin{theorem}[The Iteration Lemma]
    \label{lem: iteration}
    Let $\omega' \in \Omega \setminus (\Omega_{M}^{\tau} \cup G)$, and let 
    $\varepsilon_0 > 0$ be the threshold provided in~\Cref{thm: main}. 
    For any $0 < \varepsilon \leq \varepsilon_0$, there exists a Gevrey exponent 
    $s(\varepsilon) > 0$ with the following property. 
    Suppose that at the $r$-th iteration, the approximate solution 
    $\hat{u}_p^{(r)}$ is supported in: 
    \begin{equation}
        \label{eqn: vector-support-r}
        \mathrm{supp} \; \hat{u}_p^{(r)} \subseteq  \Lambda^{1+m}_{N_r}, 
    \end{equation}
    that for any index $i \in \{ 0, 1, \ldots, r\}$ and any 
    $(n,k) \in \Lambda^{1+m}_{N_i} \setminus \Lambda^{1+m}_{N_{i-1}}$, 
    it satisfies the Gevrey decay estimate 
    \begin{equation}
        \label{eqn: vector-decay-r}   
        \| \hat{u}^{(r)}(n, k) \| \leq \sum_{\ell=i}^{r}e^{ - \frac32 N_{\ell}^{s} },
    \end{equation}
    and the non-resonant vector field $F$ satisfies 
    \begin{equation}
        \label{eqn: vector-field-decay-r}  
        \big\| F\big( \hat{u}^{(r)}_p, \;\omega^{(r+1)}; a \big) \big\| \leq  e^{-2N_{r}^{s} }.                                                                                                                                                                                                                                                                   
    \end{equation}
    Then the iterate $\hat{u}_p^{(r+1)}$ produced by the alternating numerical 
    scheme~\eqref{eqn: alternating-numerical} satisfies: 
    \begin{equation}
        \label{eqn: vector-support-r+1}
        \mathrm{supp} \; \hat{u}_p^{(r+1)} \subseteq \Lambda^{1+m}_{N_{r+1}}.
    \end{equation}
    For any $i \in \{ 0, 1, \ldots, r+1\}$ and any 
    $(n, k) \in \Lambda^{1+m}_{N_i} \setminus \Lambda^{1+m}_{N_{i-1}}$, the 
    following decay property holds: 
    \begin{equation}
        \label{eqn: vector-decay-r+1}   
        \| \hat{u}^{(r+1)}(n, k) \| \leq \sum_{\ell=i}^{r+1}e^{ - \frac32 N_{\ell}^s },
    \end{equation}
    and 
    \begin{equation}
        \label{eqn: vector-field-decay-r+1}  
        \big\| F \big( \hat{u}^{(r+1)},\omega^{(r+2)}; a \big) \big\| \leq e^{-2N_{r+1}^{s} }.                                                                                                                                
    \end{equation}
    In other words, properties~\eqref{eqn: vector-support-r}--\eqref{eqn: vector-field-decay-r} 
    are preserved under the alternating update~\eqref{eqn: alternating-numerical}.
\end{theorem}

\begin{proof}[Proof of~\Cref{lem: iteration}]
    For conciseness, we suppress the auxiliary parameters from the 
    dimension-enlarged Newton scheme~\eqref{eqn: alternating-coefficient} 
    and write the increment as:
    \begin{equation}
        \label{eqn: nash-moser-single-simplification}
        \Delta^{(r)} = \hat{u}^{(r+1)}_p - \hat{u}^{(r)}_p = - L_{N_{r+1}}^{-1} \big(  \hat{u}^{(r)}_p \big) F\big(  \hat{u}^{(r)}_p \big).
    \end{equation}
    The proof proceeds in three steps.

    \paragraph{Support enlargement} 
    The first step verifies the support inclusion~\eqref{eqn: vector-support-r+1}. 
    Since the perturbation $|u|^2u$ is cubic, the support 
    assumption~\eqref{eqn: vector-support-r} implies
    \begin{equation}
        \label{eqn: support-d-enlarge}
        \mathrm{supp} \; F\big( \hat{u}^{(r)}_p \big) \subseteq [-3N_{r}, 3N_{r}]^{1+m} \subseteq \left[ -\frac{N_{r+1}}{4}, \frac{N_{r+1}}{4} \right]^{1+m},
    \end{equation}
    which, together with the update rule~\eqref{eqn: alternating-numerical}, 
    establishes~\eqref{eqn: vector-support-r+1}

    \paragraph{Lattice vector decay} This step establishes the 
    bound~\eqref{eqn: vector-decay-r+1}. 
    By the inversion condition~\eqref{eqn: inverse-grow} and the residual 
    estimate~\eqref{eqn: vector-field-decay-r}, the iterative difference 
    satisfies
    \[
        \| \Delta^{(r)} \|_2 \leq \big \| L_{N_{r+1}}^{-1} \big( \hat{u}^{(r)}_p \big) \big\|_2 \big \| F\big( \hat{u}^{(r)}_p\big) \big\|_2 \leq \frac{1}{\varepsilon_{N_{r+1}}} e^{-2N_{r}^{s}}.
    \]
    Furthermore, applying~\Cref{lem: inverse-derivative} yields the derivative 
    estimate
    \begin{align*}
        \| \partial \Delta^{(r)} \|_2 
        & \leq \big\| \partial L_{N_{r+1}}^{-1} \big( \hat{u}^{(r)}_p \big) \big\|_2 \big\| F\big( \hat{u}^{(r)}_p\big) \big\|_2 + \big\| L_{N_{r+1}}^{-1} \big( \hat{u}^{(r)}_p \big) \big\| _2\big\| \partial F\big( \hat{u}^{(r)}_p\big) \big\|_2                                                            \nonumber \\
        & \leq \left( \frac{4\sqrt{m}N_{r+1} + \varepsilon_{N_{r+1}}}{\varepsilon_{N_{r+1}}^2}\right) e^{-2N_{r}^{s}}.       
    \end{align*} 
    Combining the two bounds above, there exists a threshold 
    $\varepsilon_0 > 0$ such that:
    \begin{equation}
        \label{eqn: difference-derivative-estimate}
        \| \Delta^{(r)} \| \leq \exp\left\{-\frac{7N_{r}^{s}}{4}\right\},
    \end{equation}
    for any $0 < \varepsilon \leq \varepsilon_0$. 
    We now set $s = s(\varepsilon)>0$ such that $6M^{s} \leq 7$. 
    Consequently, at the $(r+1)$-th iteration, for any index 
    $i \in \{0, 1, \ldots, r+1\}$ and any 
    $(n, k) \in \Lambda^{1+m}_{N_i} \setminus \Lambda^{1+m}_{N_{i-1}}$, the 
    following holds:
    \begin{align*}
        \| \hat{u}^{(r+1)}_p \| \leq \| \hat{u}^{(r)}_p \| + \|\Delta^{(r)}\| \leq \sum_{\ell=i}^{r+1} e^{ - \frac32 N_{\ell}^{s} },
    \end{align*}
    which confirms that the decay property~\eqref{eqn: vector-decay-r+1} is 
    preserved across successive iterations.

    \paragraph{Super-exponential convergence} To establish the super-convergence 
    property~\eqref{eqn: vector-field-decay-r+1}, we analyze the residual of the 
    vector field after the numerical update~\eqref{eqn: alternating-numerical}. 
    Applying a Taylor expansion to the vector field and the frequency, and 
    invoking~\Cref{prop: tensor}, the post-update residual is bounded by the 
    truncation error together with the quadratic terms of the iteration: 
    \begin{equation}
        \label{eqn: vector-estimate-taylor-r+1}
        \| F(\hat{u}_p^{(r+1)}, \omega^{(r+2)}) \|_2 \leq \| [ L - L_{N_{r+1}}] \Delta^{(r)} \|_2 + (2N_{r+1}+1)^\frac{1+m}{2} e^{- 3N_{r+1}^{s} }. 
    \end{equation}
    Similarly, the derivative of the residual satisfies:
    \begin{align}
        \| \partial F(\hat{u}_p^{(r+1)}, \omega^{(r+2)}) \|_2 
        \leq & \| \partial [ L - L_{N_{r+1}}] \Delta^{(r)} \|_2 + \| [ L - L_{N_{r+1}} ]\partial \Delta^{(r)} \|_2 \nonumber \\ 
        & + 3(2N_{r+1}+1)^\frac{1+m}{2}e^{- 3N_{r+1}^{s} }.     \label{eqn: vector-estimate-derivative-taylor-r+1}
    \end{align}
    To estimate the truncation terms in~\eqref{eqn: vector-estimate-taylor-r+1} 
    and~\eqref{eqn: vector-estimate-derivative-taylor-r+1}, we exploit the 
    off-diagonal decay of the operators $L $ and $L_{N_{r+1}}^{-1}$ established 
    in~\Cref{prop: convolution} and~\Cref{lem: inverse-derivative}, respectively. 
    Since the estimates for these terms are entirely analogous, we focus on 
    $[L- L_{N_{r+1}}] \Delta^{(r)}$ as a representative example. 
    We utilize the following decomposition:
    \begin{align}
        [L- L_{N_{r+1}}] \Delta^{(r)} 
        =  & \left[ (\mathrm{Id} - P_{N_{r+1}}) LP_{N_{r+1}} \right] \left[ (\mathrm{Id} - P_{N_{r+1}/3}) \Delta^{(r)} \right]  \nonumber \\
           & + \left[ (\mathrm{Id} - P_{N_{r+1}}) L P_{N_{r+1}/3} \right] \Delta^{(r)}.  \label{eqn: decomposition-t-tr}
    \end{align}
    Given the confirmed off-diagonal decay, there exists a threshold 
    $\varepsilon_0 > 0$ such that for any $0 < \varepsilon \leq \varepsilon_0$, 
    the following estimates hold:
    \begin{equation}
        \label{eqn: final-super-convergence-bound-1}
        \|(\mathrm{Id} - P_{N_{r+1}}) L P_{N_{r+1}/3} \| \leq \frac13 e^{ - \frac{7}{12} N_{r+1}^{s} },
    \end{equation}
    and
    \begin{equation}
        \label{eqn: final-super-convergence-bound-2}
        \|(\mathrm{Id} - P_{N_{r+1}/3}) L_{N_{r+1}}^{-1} P_{N_{r+1}/4} \| \leq \frac13 e^{ - \frac{1}{48} N_{r+1}^{s} }.
    \end{equation}
    Noting that 
    $\|(\mathrm{Id} - P_{N_{r+1}}) LP_{N_{r+1}} \| \leq \varepsilon \|S+B\| \leq 1$, 
    we combine the bounds~\eqref{eqn: final-super-convergence-bound-1} 
    and~\eqref{eqn: final-super-convergence-bound-2} with the 
    decomposition~\eqref{eqn: decomposition-t-tr} to control the truncation 
    errors in~\eqref{eqn: vector-estimate-taylor-r+1} 
    and~\eqref{eqn: vector-estimate-derivative-taylor-r+1}. 
    Hence, for a sufficiently small threshold $\varepsilon_0 > 0$, we obtain
    \begin{equation}
        \| F(\hat{u}_p^{(r+1)}, \omega^{(r+2)}) \| \leq \frac{2}{3} e^{ -\left( \frac{2}{M^{s}} + \frac{1}{48}\right) N_{r+1}^{s} } + 4(2N_{r+1}+1)^\frac{1+m}{2}e^{ - 3N_{r+1}^{s} }.   \label{eqn: final-super-convergence-1}
    \end{equation}
    By choosing $s = s(\varepsilon)>0$ such that $95 M^{s} \leq 96$, we arrive at 
    the desired super-convergence bound:
    \begin{equation}
        \label{eqn: final-super-convergence-2}
        \| F(\hat{u}_p^{(r+1)}, \omega^{(r+2)}) \| \leq e^{ - 2  N_{r+1}^{s} }, 
    \end{equation}
    which completes the proof of the Iteration Lemma~(\Cref{lem: iteration}). 
\end{proof}

The decay estimates in~\Cref{lem: iteration} show that, although individual 
terms may fluctuate from one step to the next, the cumulative decay remains 
strictly controlled. Indeed, for any multi-index 
$(n, k) \in \Lambda^{1+m}_{N_r} \setminus \Lambda^{1+m}_{N_{r-1}}$, the tail of 
the series is bounded by
\[
    \sum_{\ell = r}^{\infty} e^{-\frac{3N_{\ell}^{s}}{2} } \leq e^{-\frac{5}{4}N_{r}^{s} } \leq e^{ - \left(|n|+|k|\right)^{s} },
\]
where the last inequality follows from $|n|+|k| \leq N_r$. 
This guarantees that the sequence $\hat{u}^{(r)}$ remains uniformly bounded in 
the space $\mathcal{K}(s)$. 
Consequently, the approximate solutions retain their Gevrey regularity, or 
equivalently, the localization of their Fourier coefficients, throughout the 
iteration, and no loss of regularity occurs as the lattice size $N_r$ grows.

\subsection{Proof of the main theorem}
\label{subsec: proof-main}

Since the resonant component $\hat{u}_q^{(r)} = a$ is held fixed throughout the 
iteration, the estimate~\eqref{eqn: difference-derivative-estimate} remains 
valid with $\hat{u}_p$ replaced by $\hat{u}$. 
In particular, the displacement between successive iterates is bounded as:
\[
    \| \hat{u}^{(r+1)} - \hat{u}^{(r)} \| = \| \Delta^{(r)} \| \leq e^{ - \frac{7}{4} N_{r}^s } = e^{ - \frac{7}{4} (M^s)^{r+1} },
\]
so $\{ \hat{u}^{(r)} \}$ is a Cauchy sequence in $\mathcal{K}(s)$ and therefore 
converges to a limit $\hat{u}^{\star} \in \mathcal{K}(s)$. 
Summing the tail of the sequence yields the rate of convergence:
\[
    \| \hat{u}^{(r)} - \hat{u}^\star \| \leq \sum_{\ell = 0}^{\infty} \| \hat{u}^{(r+\ell +1)} - \hat{u}^{(r+\ell )} \| \leq \sum_{\ell = 0}^{\infty} e^{ - \frac{7}{4} (M^s)^{r+\ell +1} } \leq e^{ - \frac{3}{2} (M^s)^{r+1} },
\]
which establishes \eqref{eqn: u-hat-converge}. 
We now turn to the frequencies. 
By the frequency update~\eqref{eqn: alternating-frequency} and the uniform 
bound~\eqref{eqn: tangent-bound}, we have:
\begin{equation*}
    | \omega^{(r+2)} - \omega^{(r+1)} | \leq \varepsilon \gamma \| \hat{u}^{(r+1)} - \hat{u}^{(r)} \| \leq \varepsilon \gamma e^{ - \frac{7}{4} (M^s)^{r+1} }.
\end{equation*}
Thus, $\{ \omega^{(r)} \}$ converges to some 
$\omega^{\star} \in \Omega \setminus (\Omega_{M}^{\tau} \cup G)$, and, arguing 
as above,
\[
    | \omega^{(r)} - \omega^\star | \leq \sum_{\ell = 0}^{\infty} | \omega^{(r+\ell +1)} - \omega^{(r+\ell )} | \leq \sum_{\ell = 0}^{\infty} \varepsilon \gamma e^{ - \frac{7}{4} (M^s)^{r+\ell} } \leq e^{ - \frac{3}{2} (M^s)^{r+1} },
\]
which establishes~\eqref{eqn: frequency-converge}. 
It remains to prove the time-domain error bound~\eqref{eqn: pde-solution-error}. 
The limit $\hat{u}^{\star}$ is represented by its Fourier expansion as
\[
    u^\star(x, t) = \sum_{(n,k) \in \mathbb{Z}^{1+m}} \hat{u}^\star(n, k) e^{i \langle k, \omega^\star t + \theta_0 \rangle} e^{inx},
\]
while the $r$-th approximate solution $u^{(r)}(x,t)$ is given 
by~\eqref{eqn: approximate-pde-solution}. 
Based on the support property \eqref{eqn: vector-support-r}, the error can be 
split into two terms:
\begin{align*}
    \| u^{(r)}(x, t) - u^\star(x, t) \|_{L^2} 
    \leq & \left\| \sum_{(n,k) \in \mathbb{Z}^{1+m}} \left( \hat{u}^\star(n, k) - \hat{u}^{(r)}(n, k) \right) e^{i \left( \langle k, \omega^\star t + \theta_0 \rangle + nx \right)} \right\|_{L^2} \\
    & + \left\| \sum_{(n,k) \in \Lambda_{N_{r}}^{1+m}} \hat{u}^{(r)}(n, k) e^{inx} \left( e^{ i \langle k, \omega^\star t + \theta_0 \rangle } - e^{ i \langle k, \omega^{(r)} t + \theta_0 \rangle } \right) \right\|_{L^2}  
    \defi I_1 + I_2.
 \end{align*}
For the term $I_1$, combining the support property~\eqref{eqn: vector-support-r}, 
the Gevrey decay~\eqref{eqn: vector-decay-r}, and the Cauchy-Schwarz inequality, 
we obtain
\begin{align*}
    I_1 &\leq \sum_{(n, k) \in \Lambda_{N_{r}}^{1+m} } | \hat{u}^\star(n,k) - \hat{u}^{(r)}(n,k) | + \sum_{(n,k) \notin \Lambda_{N_{r}}^{1+m} } | \hat{u}^\star(n,k) |  \\
        &\leq \| \hat{u}^{(r)} - \hat{u}^\star \|_2 \cdot (2N_r + 1)^{(1+m)/2} + \sum_{(n, k) \notin \Lambda_{N_{r}}^{1+m} } e^{-\frac{5}{4} (|n|+|k|)^s} \leq \frac{1}{2} e^{ - N_r^s } = \frac{1}{2} e^{ - (M^s)^{r+1} }.
\end{align*}
where the last inequality follows from the bound~\eqref{eqn: u-hat-converge}. 
For the term $I_2$, by the uniform boundedness of $\| \hat{u}^{(r)} \|_2$ and the  Cauchy-Schwarz inequality once more, we have 
\begin{align*}
    I_2 &\leq \| \hat{u}^{(r)} \|_2  \left( \sum_{(n,k) \in \Lambda_{N_{r}}^{1+m}} \left| e^{i \langle k, \omega^\star t + \theta_0 \rangle } - e^{i \langle k, \omega^{(r)} t + \theta_0 \rangle } \right|^2 \right)^\frac{1}{2}   \\
        &\leq \left( \| \hat{u}^\star \|_2 + e^{ - \frac{3}{2} (M^s)^{r+1} } \right) \cdot | \omega^{(r)} - \omega^\star | \cdot \left( \sum_{(n,k) \in \Lambda_{N_{r}}^{1+m}} \|k t \|^2 \right)^\frac{1}{2} \leq \frac12 e^{-N_r^s} = \frac{1}{2} e^{- (M^s)^{r+1} },
\end{align*}
which holds uniformly for any $t \in [0, N_r] = [0, M^{r+1}]$. 
Combining the estimates for $I_1$ and $I_2$, we arrive at the final 
time-domain error bound \eqref{eqn: pde-solution-error}. 
This completes the proof of~\Cref{thm: main}.

\subsection{Generalization to the Dirichlet boundary condition}
\label{subsec: dirichlet-nls}

In this subsection, we extend the preceding results for the NLS 
equation~\eqref{eqn: nls} from the periodic setting~\eqref{eqn: per-bdy} to 
Dirichlet boundary conditions:
\begin{equation}
    \label{eqn: dirichlet-bdy}
    u(0, t) = u(\pi, t) = 0. 
\end{equation}
To make the Fourier expansion available, we extend $u$ oddly to 
$x \in [-\pi, 0]$ by setting
\[
    u(x, t) = - u(-x, t),
\]
so that the extended function is $2\pi$-periodic and odd in $x$. 
Consequently the spatial Fourier coefficients satisfy the symmetry relation
\[
    \hat{u}_{-n}(t) = -\hat{u}_n(t),
\]
for any $n \in \mathbb{Z}$, and in particular, $\hat{u}_0(t)=0$. 
The solution can therefore be written as
\begin{equation}
    \label{eqn: nls-dirichlet-sine}
    u(x,t) = 2i \sum_{n = 1}^{\infty} \hat{u}_n(t) \sin(nx).
\end{equation}
That is, under the Dirichlet condition the index set of spatial modes reduces 
from $\mathbb{Z}$ to $\mathbb{N}^+$, and the number of unknown coefficients is 
halved. 
Accordingly, the set of tangential modes~\eqref{eqn: tangent-spatial-mode} is 
now taken as
\[
    \mathcal{N} = \left\{ n_1 < n_2 < \cdots < n_m \right\} \subset \mathbb{N}^{+}.
\]
Since the nonlinearity $|u|^2 u$ maps odd functions to odd functions, the odd 
subspace is invariant under both the equation and the alternating 
scheme~\eqref{eqn: alternating-numerical}. 
Hence all the convergence results, and in particular the main 
statement~(\Cref{thm: main}) remain valid, provided that the Fourier 
coefficients are restricted to those satisfying the symmetry relation
\begin{equation}
    \label{eqn: nls-dirichlet-nk}
    \hat{u}(-n, k) = - \hat{u}(n, k)
\end{equation}

We next demonstrate the performance of the alternating numerical 
scheme~\eqref{eqn: alternating-numerical}, where we set the initial phase 
$\theta_0 = 0$ and the perturbation quantity $\varepsilon = 0.2$.

\paragraph{Periodic solutions ($m=1$)}
Different from~\eqref{eqn: lin-1} in the periodic setting, the linear periodic 
solution here reads 
\begin{equation}
    \label{eqn: nls-dirichlet-linear}
    u(x, t) = 2i a_1 e^{i \omega_1t} \sin (n_1x),
\end{equation}
which is in fact the superposition of two waves with wavenumbers $n_1$ and 
$-n_1$. Consequently, the support of the nonlinear term is no longer confined 
to a single mode, and the nonlinear solution takes the form
\begin{subequations}
    \label{eqn: nls-dirichlet-periodic-1}
    \begin{empheq}[left=\empheqlbrace]{align} 
        & \hat{u}_{n_1}(t) = a_1e^{i \omega_1' t} + \sum_{k \in \mathbb{Z} \setminus \{1\}} \hat{u}(n_1, k) e^{ik\omega_1' t},             \label{eqn: nls-dirichlet-periodic-1-1}                \\
        & \hat{u}_n(t) = \sum_{k \in \mathbb{Z}} \hat{u}(n, k) e^{ik\omega_1' t}, \quad n \in \mathbb{N}^{+} \setminus \{n_1\},          \label{eqn: nls-dirichlet-periodic-1-2}                
    \end{empheq}        
\end{subequations} 
For the numerical test we take $n_1=1$, $\mu=\sqrt{2}$, and $a_1 =1$, so that 
$\omega_1 = 1 + \sqrt{2}$; the results are shown in~\Cref{fig: nls-d-1torus-traj}. 
\begin{figure}[htb!]
    \centering
    \begin{subfigure}[t]{0.3\linewidth}
    \centering
    \includegraphics[scale=0.18]{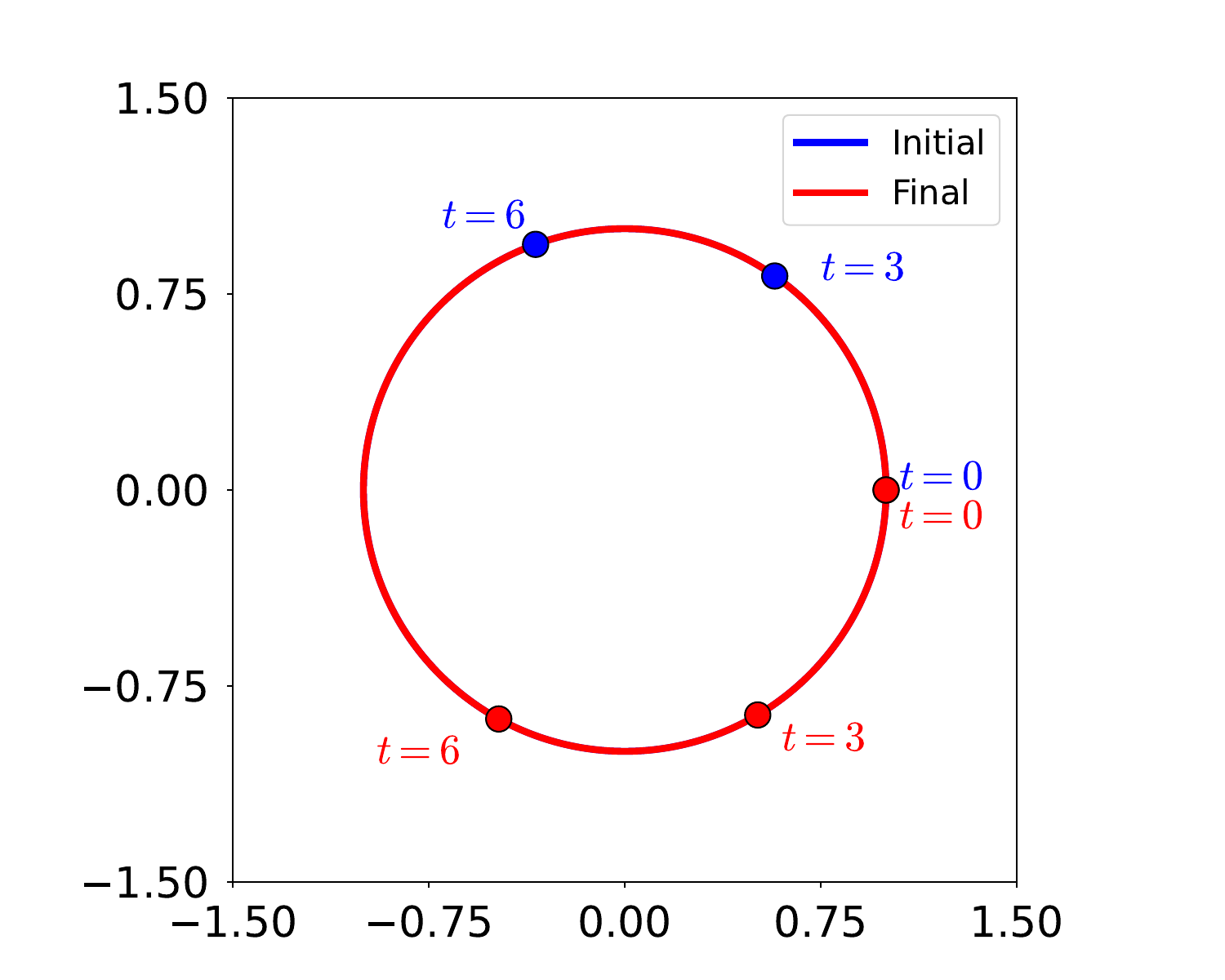}
    \caption{Tangential mode: $\hat{u}_1(t)$}
    \end{subfigure}
    \begin{subfigure}[t]{0.3\linewidth}
    \centering
    \includegraphics[scale=0.18]{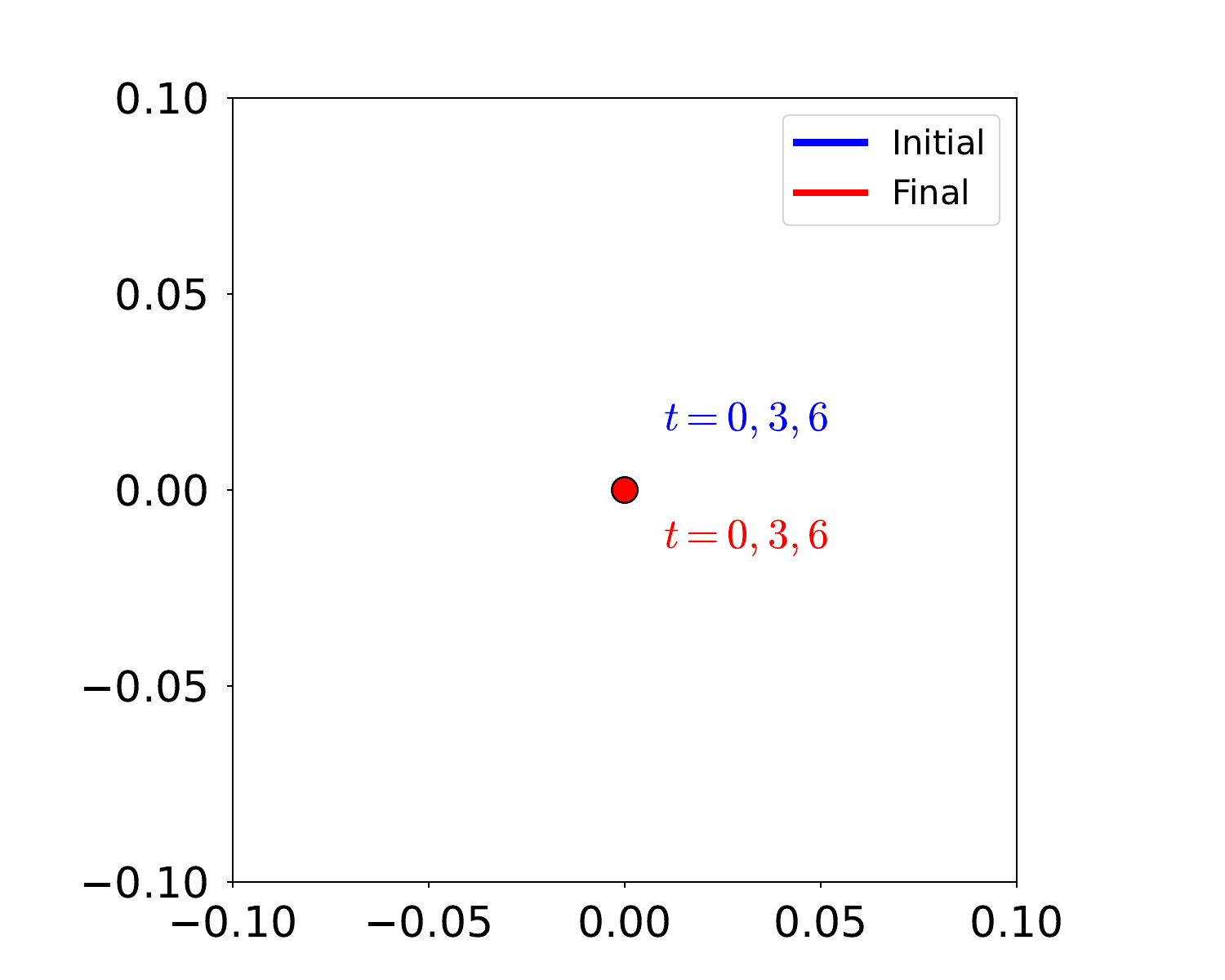}
    \caption{Normal mode: $\hat{u}_2(t)$}
    \end{subfigure}
    \begin{subfigure}[t]{0.3\linewidth}
    \centering
    \includegraphics[scale=0.18]{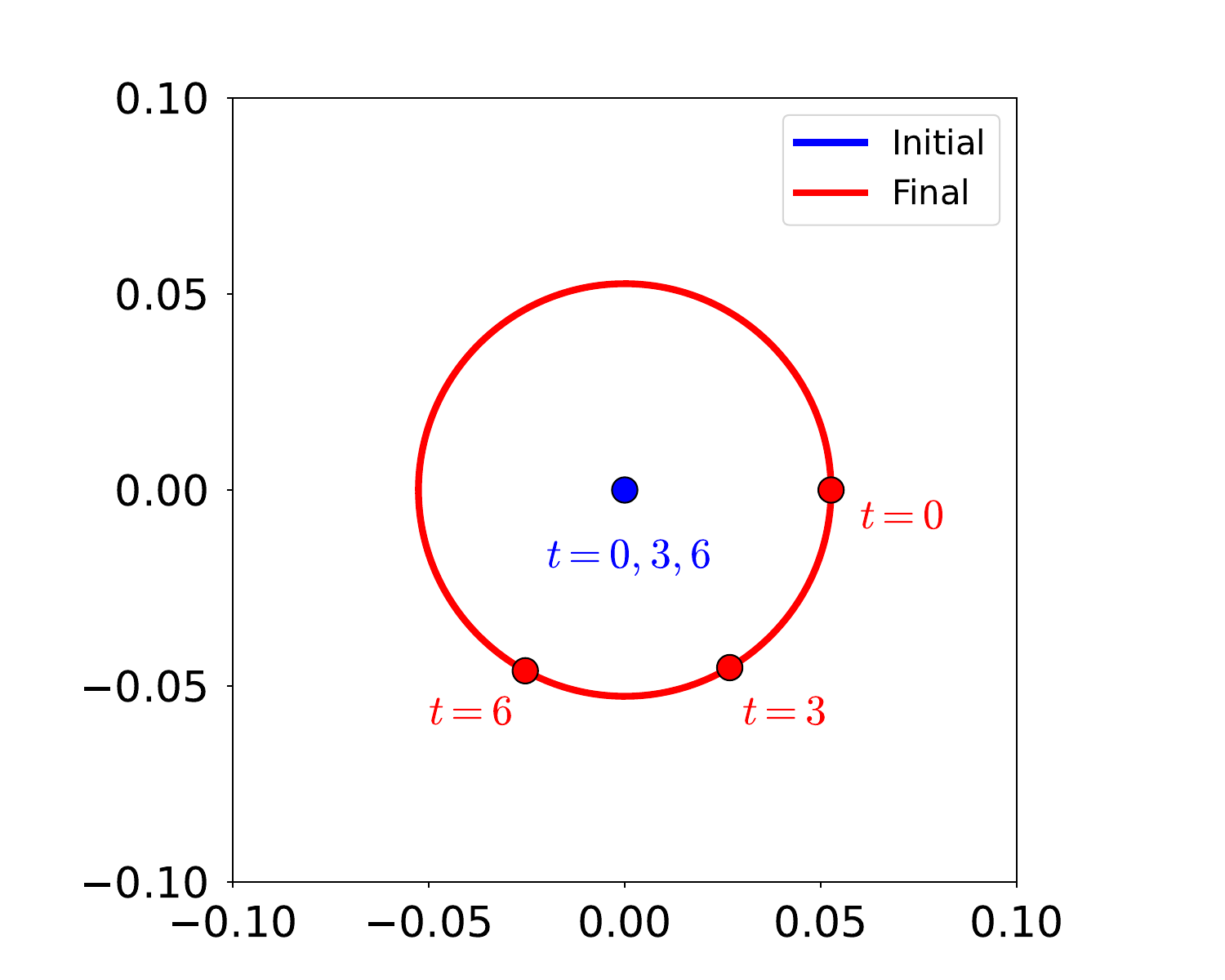}
    \caption{Normal mode: $\hat{u}_3(t)$}
    \end{subfigure}
    \caption{Phase space trajectories of the NLS equation with Dirichlet 
             boundary conditions~\eqref{eqn: dirichlet-bdy}. 
             ``Initial'' and ``Final'' denote the initial linear solution and 
             the computed nonlinear solution, respectively. 
             Markers indicate specific time points at $t = 0$, $3$, and $6$.} 
    \label{fig: nls-d-1torus-traj}
\end{figure}
The tangential mode $\hat{u}_1(t)$ retains an amplitude close to unity but 
exhibits noticeable frequency shifts relative to the linear solution. 
The normal modes, on the other hand, display two distinct types of behavior: 
the representative mode $\hat{u}_2(t)$ remains fixed at the origin, whereas 
$\hat{u}_3(t)$,  
\begin{figure}[htb!]
    \centering
    \begin{subfigure}[t]{0.3\linewidth}
    \centering
    \includegraphics[scale=0.18]{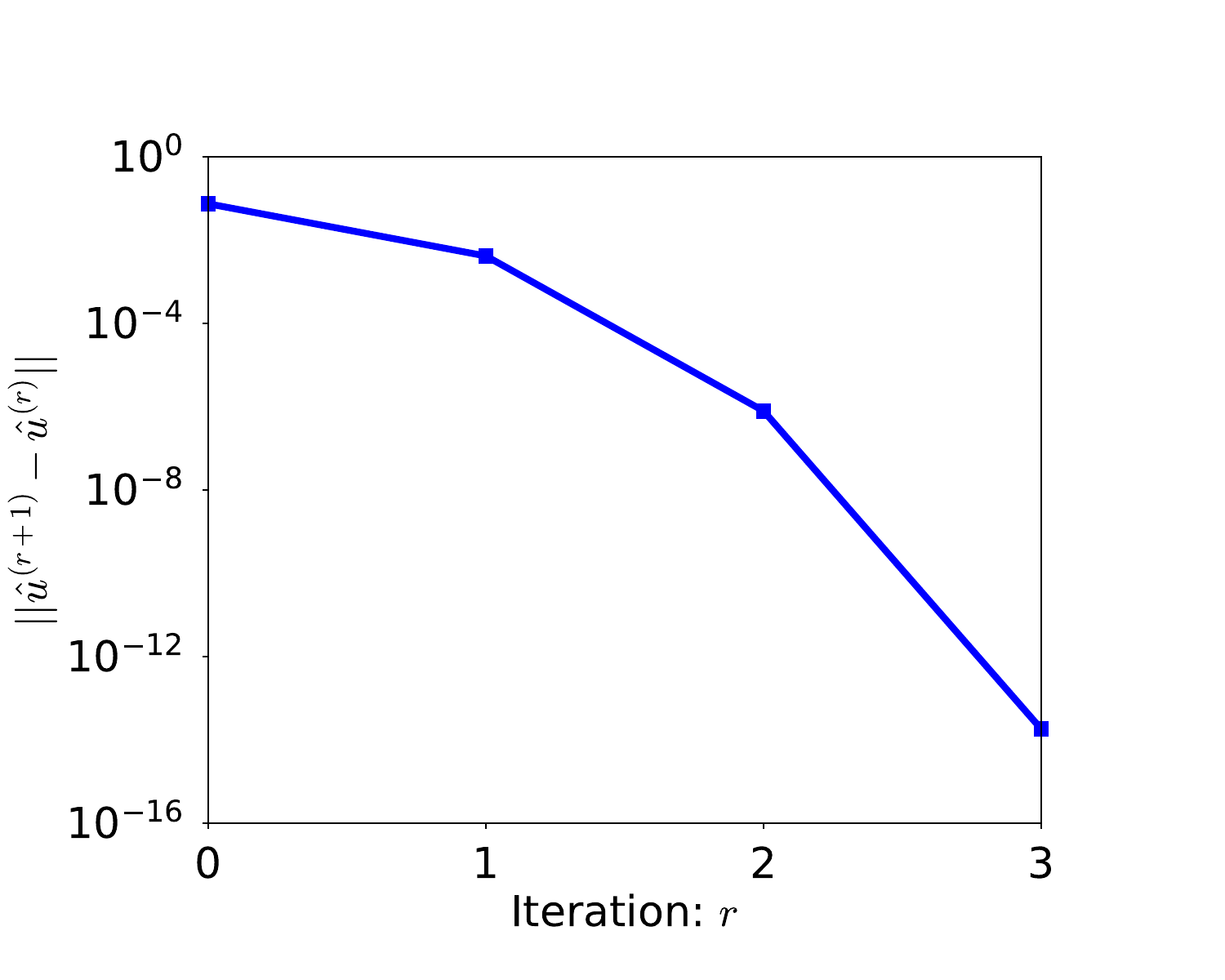}
    \caption{$\| \hat{u}_p^{(r+1)} - \hat{u}_p^{(r)} \|$}
    \end{subfigure}
    \begin{subfigure}[t]{0.3\linewidth}
    \centering
    \includegraphics[scale=0.18]{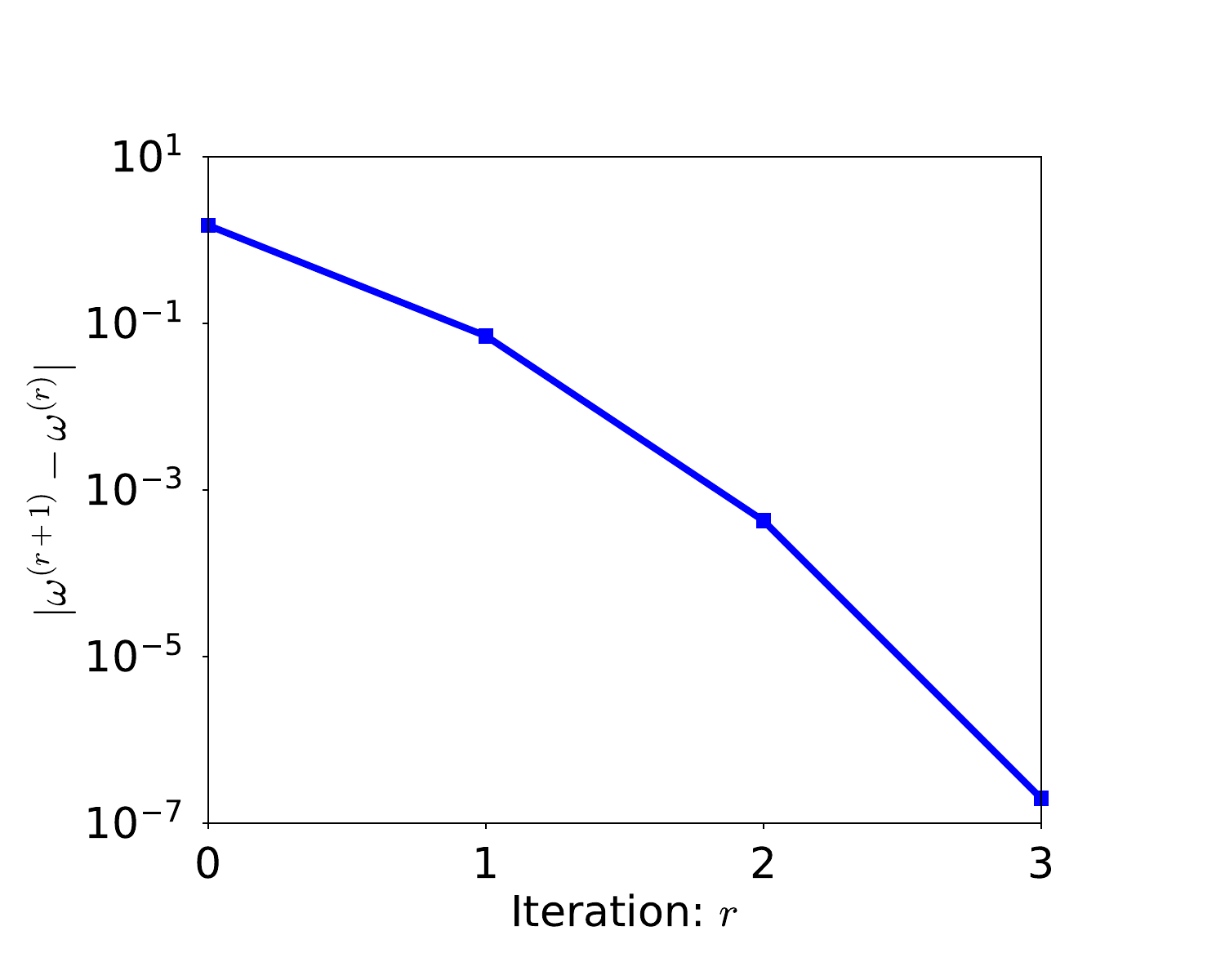}
    \caption{$| \omega^{(r+1)} - \omega^{(r)} |$}
    \end{subfigure}
    \begin{subfigure}[t]{0.3\linewidth}
    \centering
    \includegraphics[scale=0.18]{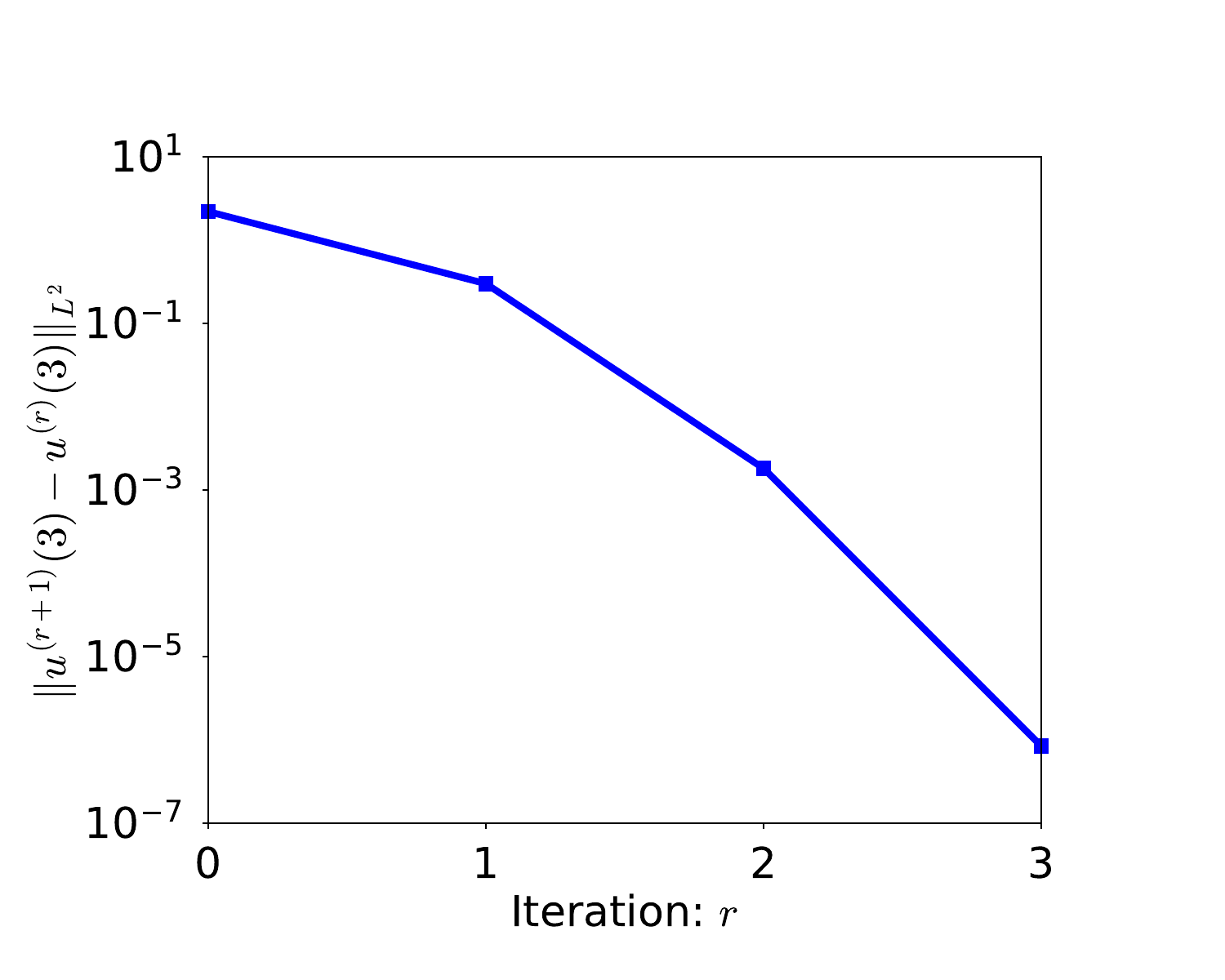}
    \caption{$\| u^{(r+1)}(x, 3) - u^{(r)}(x, 3) \|_{L^2}$}
    \end{subfigure}
    \caption{Convergence profiles across iterations for the NLS 
             equation~\eqref{eqn: nls} with Dirichlet boundary 
             conditions~\eqref{eqn: dirichlet-bdy}.} 
    \label{fig: nls-d-1torus-con}
\end{figure}
which is absent from the linear solution, develops small-amplitude oscillations 
driven by the nonlinear interaction. 
The convergence profile is shown in~\Cref{fig: nls-d-1torus-con}.

Turning to the spatiotemporal distributions, the squared modulus of the linear 
solution~\eqref{eqn: lin-square-nls-dirichlet} with the above parameters is
\begin{equation}
    \label{eqn: lin-square-nls-dirichlet}
    |u|^2 = 4\sin^2x, 
\end{equation}
which depends on $x$ alone and is independent of $t$. 
In other words, the linear solution is a stationary state: the time dependence 
enters only through the global phase factor $e^{i\omega_1 t}$, so the intensity 
profile is frozen in time. 
This is in sharp contrast to the traveling-wave pattern observed in the 
periodic setting. 
The nonlinear solution inherits this stationary character; the only visible 
change is a slight reduction in peak amplitude, accompanied by a broadening of 
the spatial profile. 
Both distributions are shown in~\Cref{fig: nls-d-1torus-heat-slice}.
\begin{figure}[htb!]
    \centering
    \begin{subfigure}[t]{0.42\linewidth}
    \centering
    \includegraphics[scale=0.24]{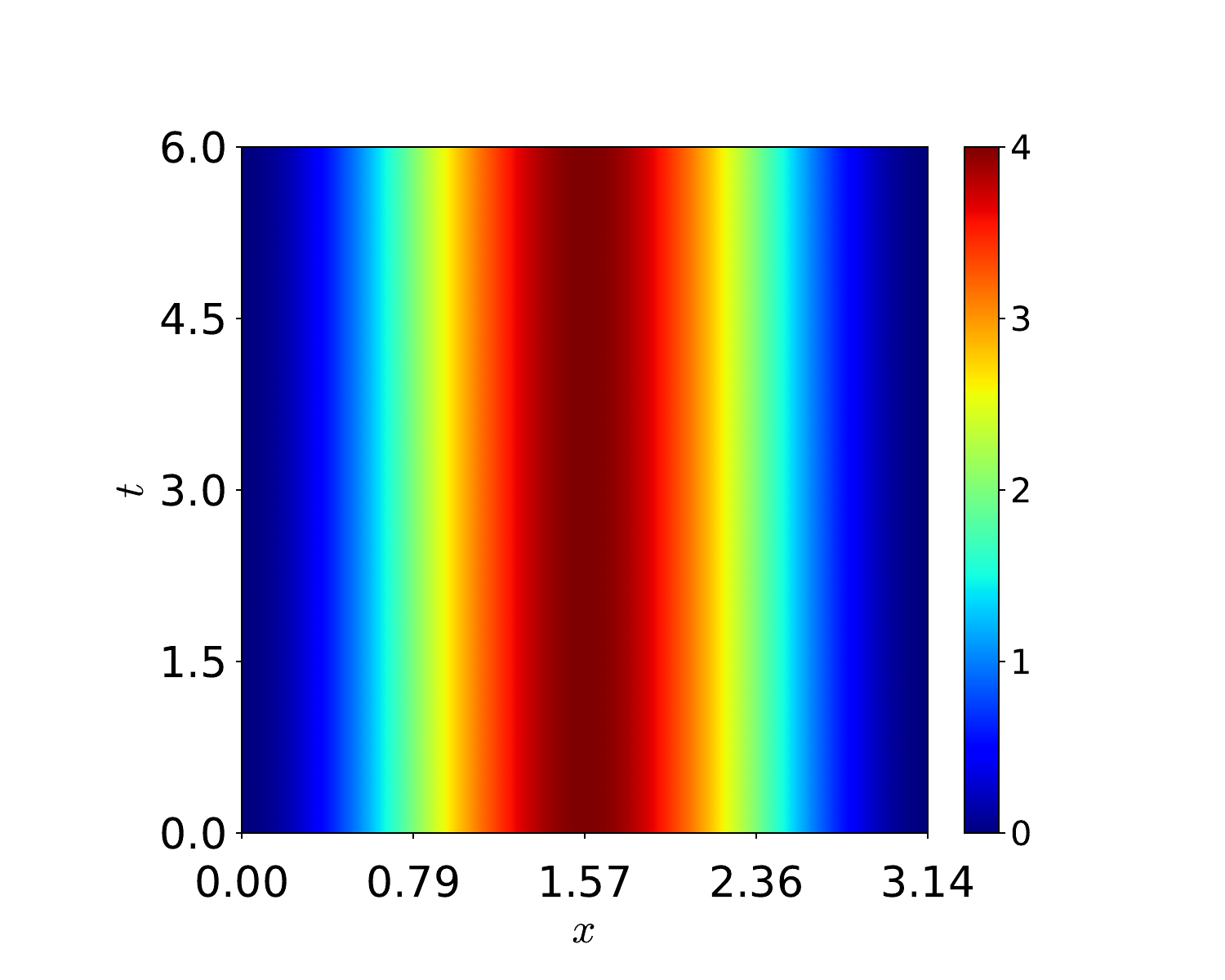}
    \caption{Linear solution}
    \end{subfigure}
    \begin{subfigure}[t]{0.42\linewidth}
    \centering
    \includegraphics[scale=0.24]{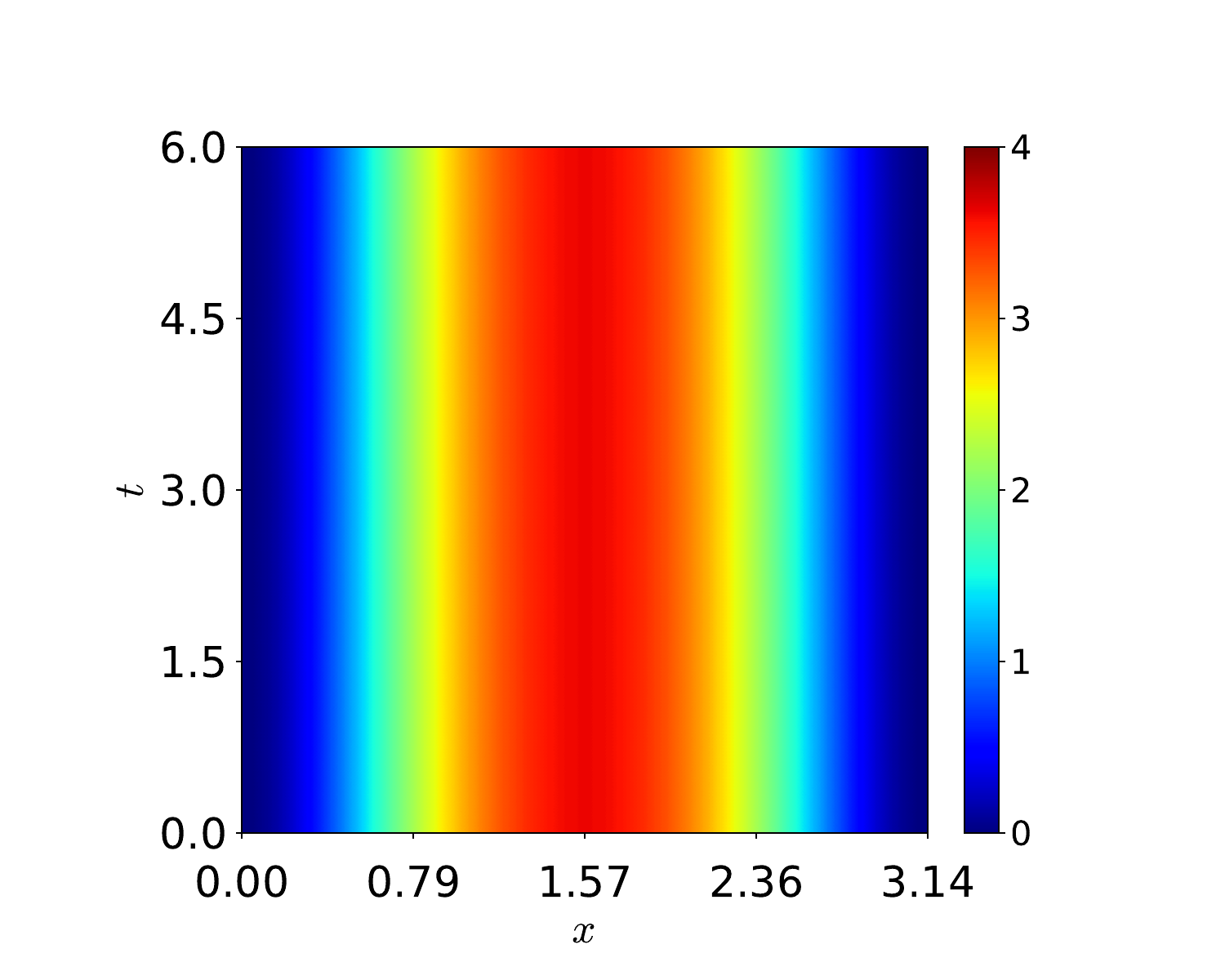}
    \caption{Nonlinear solution}
    \end{subfigure} 
    \caption{Spatiotemporal distributions of $|u(x,t)|^2$ for the NLS 
             equation~\eqref{eqn: nls} with Dirichlet boundary 
             conditions~\eqref{eqn: dirichlet-bdy}. 
             The linear solution is~\eqref{eqn: lin-square-nls-dirichlet}, 
             and the nonlinear solution is obtained via the alternating 
             numerical scheme~\eqref{eqn: alternating-numerical}.}
    \label{fig: nls-d-1torus-heat-slice}
\end{figure}

\paragraph{Quasi-periodic solutions ($m=2$)}
Similarly, under the Dirichlet boundary condition~\eqref{eqn: dirichlet-bdy}, 
the quasi-periodic solution reads 
\begin{equation}
    \label{eqn: nls-dirichlet-linear-quasi-periodic}
    u(x, t) = 2ia_1 e^{i\omega_1t}\sin(n_1 x) + 2ia_2 e^{i\omega_2t}\sin(n_2 x)
\end{equation}
which is in fact a superposition of four plane waves with wavenumbers 
$\pm n_1$ and $\pm n_2$. 
Accordingly, the nonlinear solution takes the form
\begin{subequations}
    \label{eqn: nls-dirichlet-quasi-periodic-2}
    \begin{empheq}[left=\empheqlbrace]{align} 
       & \hat{u}_{n_1}(t) = a_1e^{i \omega_1' t} + \sum_{k \in \mathbb{Z}^{2} \setminus\{(1, 0)\}} \hat{u}(n_1, k) e^{i \langle k, \omega'\rangle t},      \label{eqn: nls-dirichlet-quasi-periodic-2-1}          \\
       & \hat{u}_{n_2}(t) = a_2e^{i \omega_2' t} + \sum_{k \in \mathbb{Z}^2 \setminus\{(0, 1)\}} \hat{u}(n_2, k) e^{i \langle k, \omega' \rangle t},       \label{eqn: nls-dirichlet-quasi-periodic-2-2}         \\       
       & \hat{u}_n(t) = \sum_{k \in \mathbb{Z}^2 } \hat{u}(n, k) e^{i \langle k, \omega' \rangle t}, \quad n \in \mathbb{N}^{+}\setminus \{n_1, n_2\},                    \label{eqn: nls-dirichlet-quasi-periodic-2-3}                
    \end{empheq}        
\end{subequations} 
where $\omega' = (\omega_1', \omega_2')$ denotes the shifted frequency vector. 
For the numerical test we take $n_1=1$, $n_2=2$, $\mu=\sqrt{2}$, and 
$a_1 = a_2 =1$, so that $\omega_1 = 1 + \sqrt{2}$ and $\omega_2=4+\sqrt{2}$. 
The resulting phase-space trajectories are shown in~\Cref{fig: nls-d-2tori-traj}. 
Compared with~\Cref{fig: nls-d-1torus-traj} for the periodic solution, the 
frequency shifts are again clearly visible; more notably, the quasi-periodic 
character is now much more pronounced: the trajectories repeatedly return to 
the vicinity of earlier positions without ever closing on themselves, 
exhibiting recurrence rather than exact repetition. 
\begin{figure}[htb!]
    \centering
    \begin{subfigure}[t]{0.3\linewidth}
    \centering
    \includegraphics[scale=0.18]{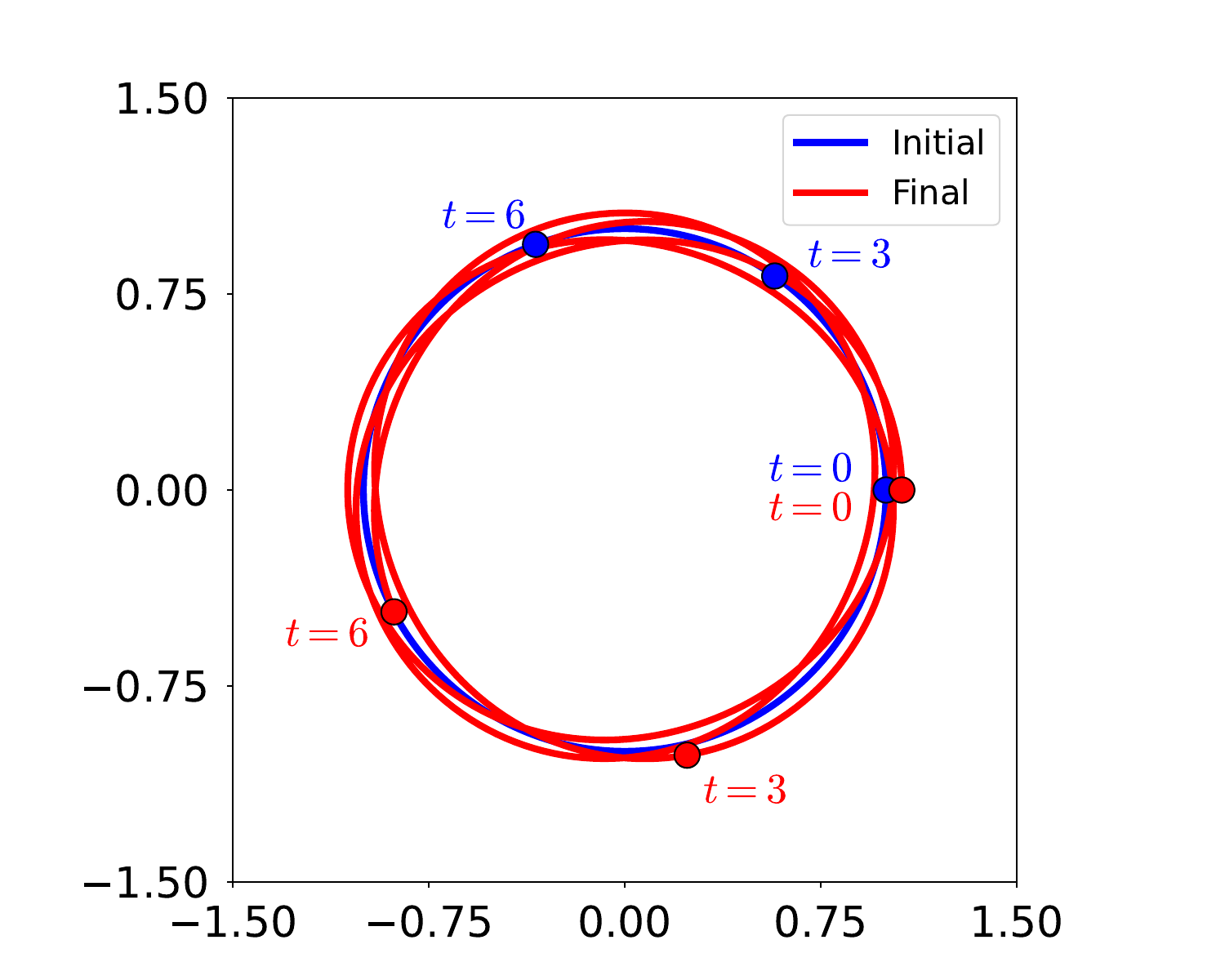}
    \caption{Tangential mode: $\hat{u}_1(t)$}
    \end{subfigure}
    \begin{subfigure}[t]{0.3\linewidth}
    \centering
    \includegraphics[scale=0.18]{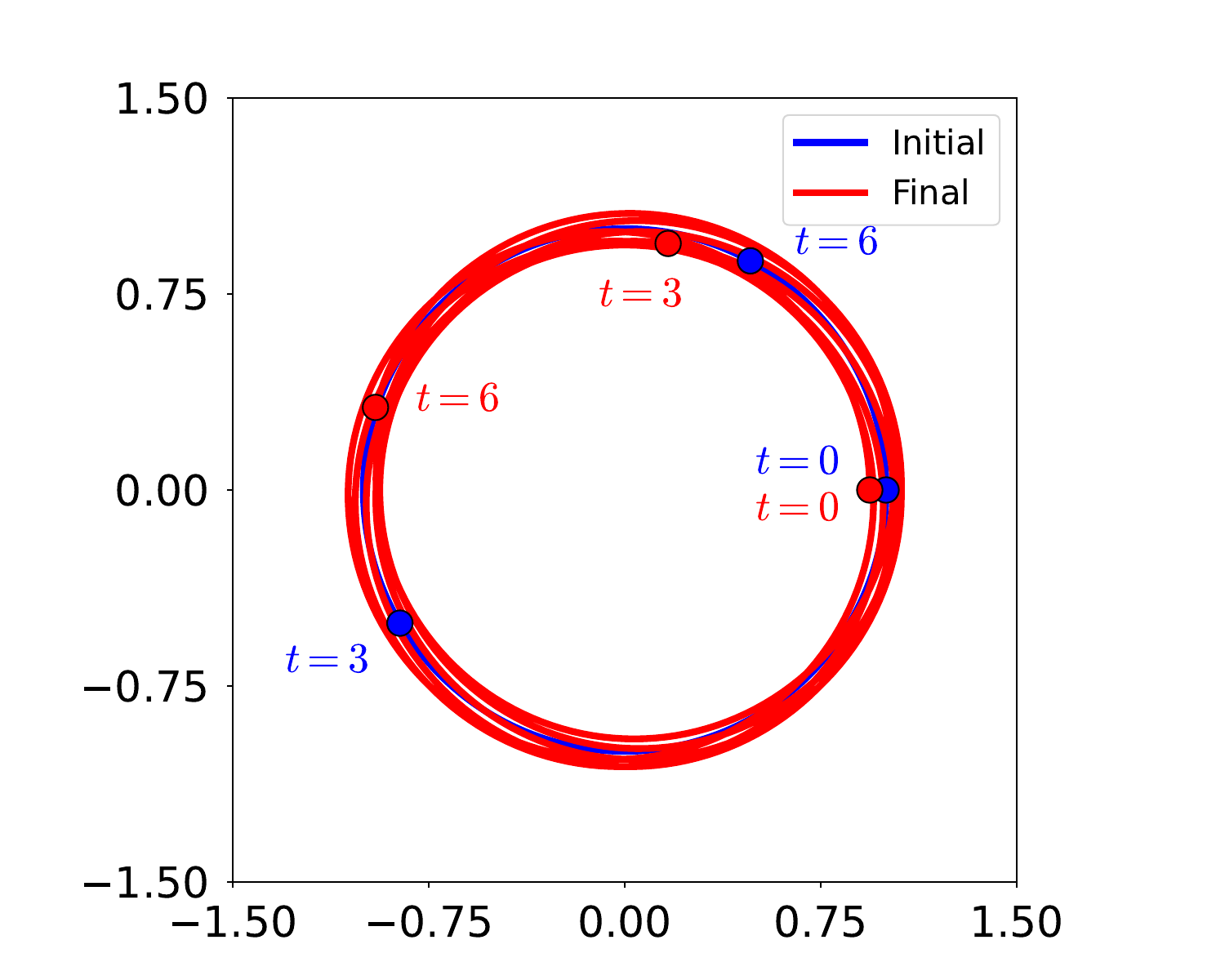}
    \caption{Tangential mode: $\hat{u}_2(t)$}
    \end{subfigure}
    \begin{subfigure}[t]{0.3\linewidth}
    \centering
    \includegraphics[scale=0.18]{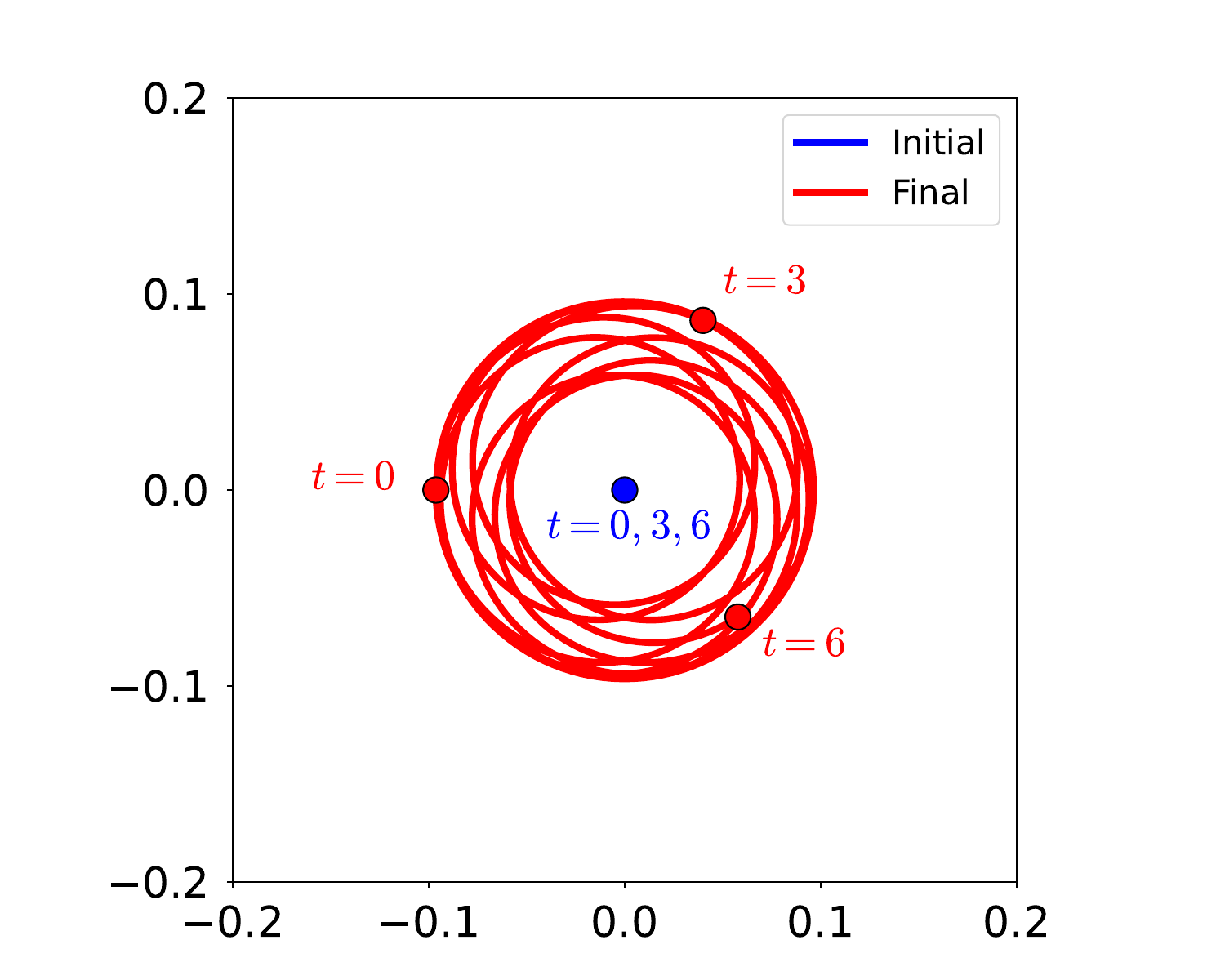}
    \caption{Normal mode: $\hat{u}_3(t)$}
    \end{subfigure}
    \caption{Phase space trajectories of the NLS equation with Dirichlet 
             boundary conditions~\eqref{eqn: dirichlet-bdy}.  
             ``Initial'' and ``Final'' denote the initial linear solution and 
             the computed nonlinear solution, respectively. 
             Markers indicate specific time points at $t = 0$, $3$, and $6$. } 
    \label{fig: nls-d-2tori-traj}
\end{figure}

Turning to the spatiotemporal distributions, the squared modulus of the linear 
solution~\eqref{eqn: nls-dirichlet-linear-quasi-periodic} with the above 
parameters is
\begin{equation}
    \label{eqn: lin-square-nls-dirichlet-quasi-periodic}
    |u|^2 = 4\sin^2(x) + 4\sin^2(2x) + 8 \cos(3t) \sin (x) \sin (2x)
\end{equation}
where $\omega_2 - \omega_1 =3$ is the beat frequency of the two modes. 
Although the intensity now depends on both $x$ and $t$, it is not a traveling 
wave but a superposition of two standing waves, $\sin(x)$ and $\sin(2x)$, 
oscillating in place through their interference term. 
\begin{figure}[htb!]
    \centering
    \begin{subfigure}[t]{0.42\linewidth}
    \centering
    \includegraphics[scale=0.24]{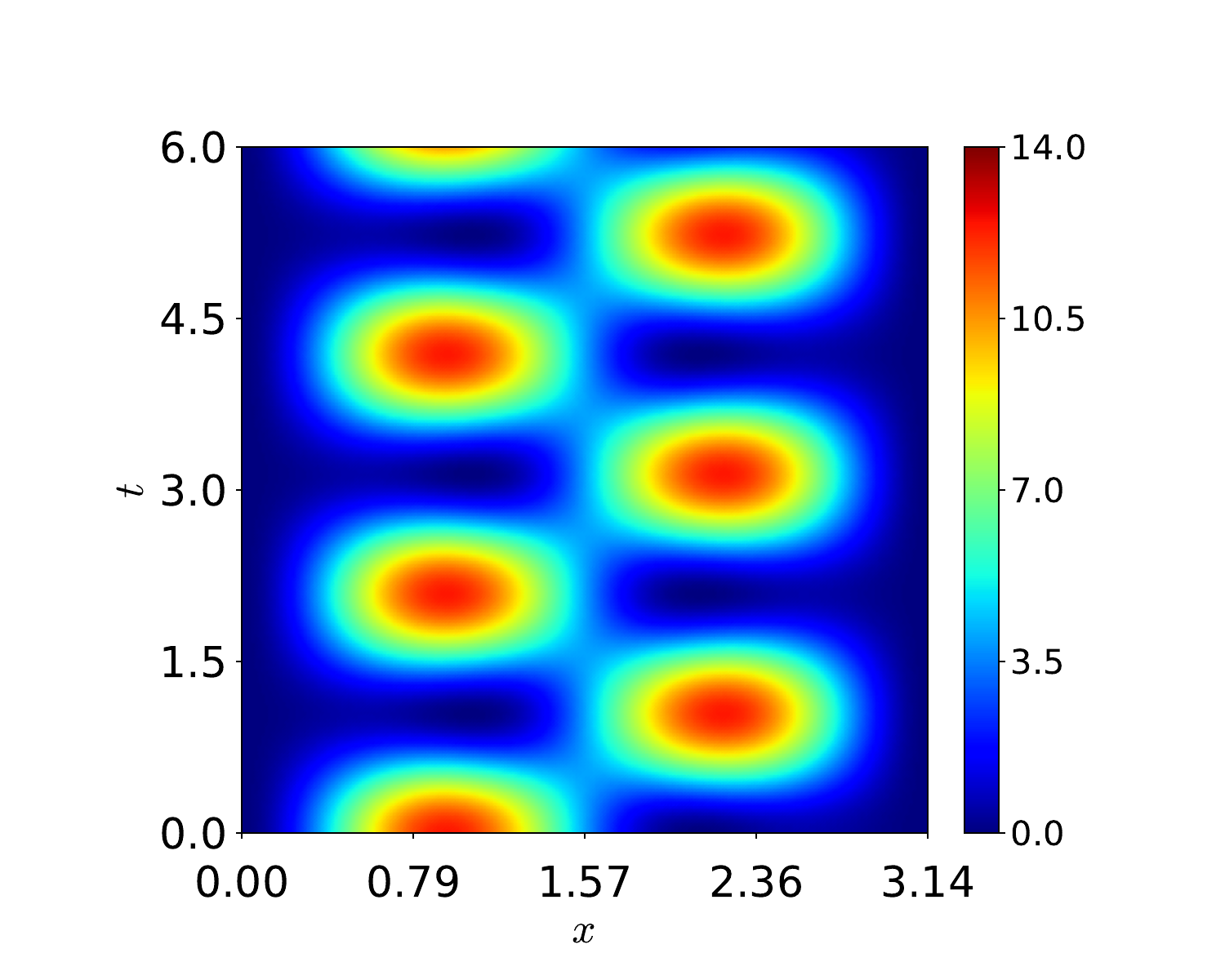}
    \caption{Linear solution}
    \end{subfigure}
    \begin{subfigure}[t]{0.42\linewidth}
    \centering
    \includegraphics[scale=0.24]{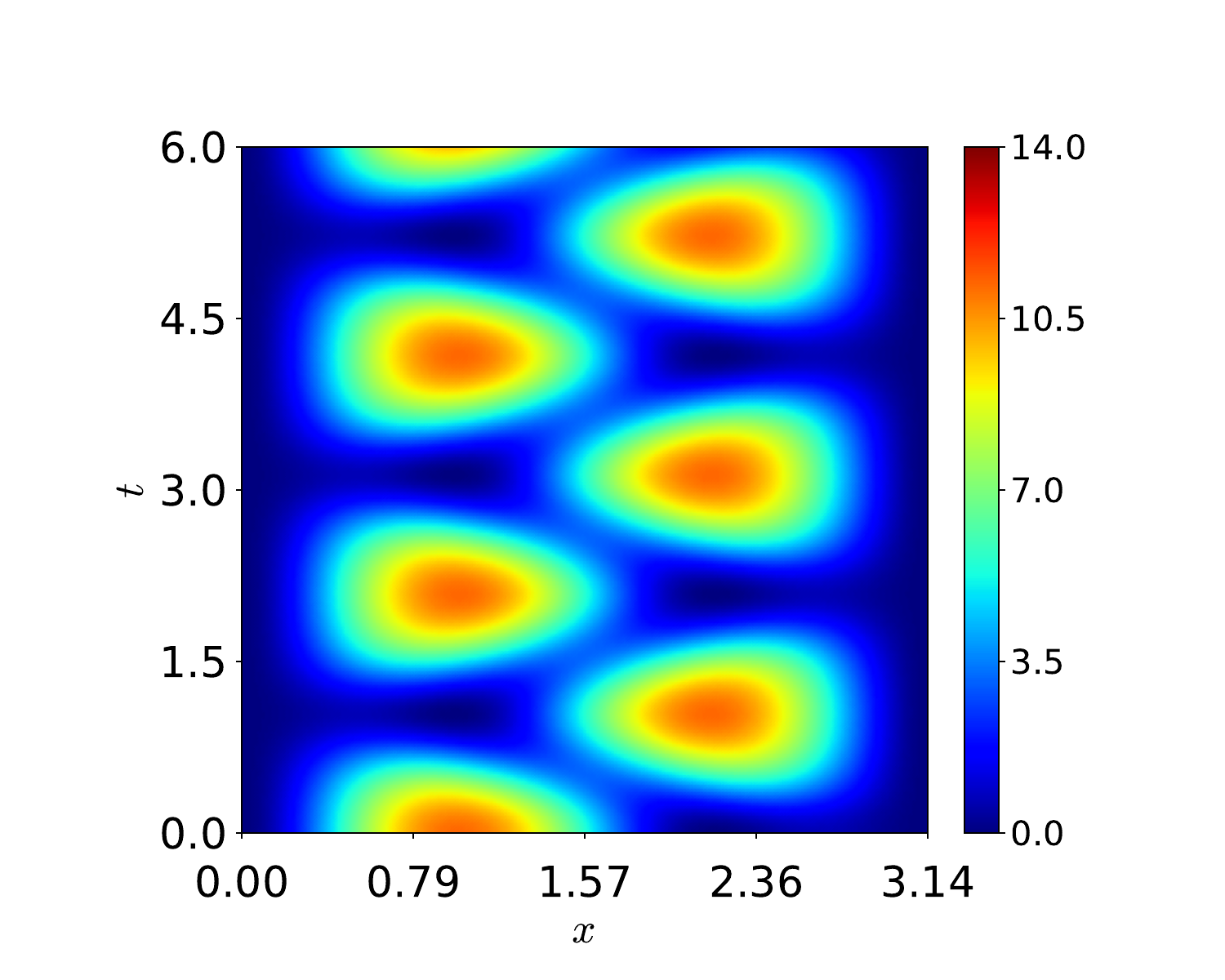}
    \caption{Nonlinear solution}    
    \end{subfigure} 
    \caption{Spatiotemporal distributions of $|u(x,t)|^2$ for the NLS equation~\eqref{eqn: nls} with Dirichlet boundary conditions~\eqref{eqn: dirichlet-bdy}. The linear solution is~\eqref{eqn: lin-square-nls-dirichlet-quasi-periodic}, and the nonlinear solution is obtained via the alternating numerical scheme~\eqref{eqn: alternating-numerical}.}
    \label{fig: nls-d-2tori-heat-slice}
\end{figure}
Once the nonlinear perturbation is introduced, the two standing waves are 
deformed and mixed together: the sharp interference pattern is smeared out, and 
the peaks are slightly lowered and broadened, as energy is redistributed to 
higher modes. 
Both distributions are shown in~\Cref{fig: nls-d-2tori-heat-slice}.

\section{Nonlinear wave equation}
\label{sec: nlw}

In this section, we extend to the one-dimensional nonlinear wave (NLW) equation 
the numerical construction of quasi-periodic solutions.
\begin{equation}
    \label{eqn: nlw}
    u_{tt} - u_{xx} + \mu u + \varepsilon u^3 = 0,
\end{equation}
subject to either the periodic boundary condition~\eqref{eqn: per-bdy} or the 
Dirichlet boundary condition~\eqref{eqn: dirichlet-bdy}. 
We first consider the case $\mu>0$, in which all linear frequencies 
$\varpi_n = \sqrt{n^2 + \mu}$ are positive, and focus on the periodic boundary 
condition~\eqref{eqn: per-bdy}, deferring the case $\mu \leq 0$ to the end 
of~\Cref{subsec: main-nlw}. 
For the NLW equation~\eqref{eqn: nlw}, the solution $u$ is real, so its spatial 
Fourier coefficients satisfy $\hat{u}_n(t) = \overline{\hat{u}_{-n}(t)}$ for 
any $n \in \mathbb{Z}$. 
Hence, the solution can be written as 
\begin{equation}
    \label{eqn: nlw-fourier-1}
    u(x, t) = \hat{u}_0(t) + 2\sum_{n=1}^{\infty} \mathrm{Re} \left( \hat{u}_n(t) e^{inx} \right). 
\end{equation}
For the space-time Fourier coefficients, the reality of $u$ translates 
into $\hat{u}(-n, -k) = \overline{\hat{u}(n,k)}$ for any 
$(n, k) \in \mathbb{Z}^{1+m}$. 
In what follows we look for solutions whose Fourier coefficients 
$\hat{u}(n, k)$ are all real. 
This restriction is consistent with~\eqref{eqn: nlw}, since if $u(x,t)$ is a 
solution, then $u(-x,-t)$ is also a solution. 
We therefore restrict the construction to the reversible class invariant under 
the reflection $(x,t)\mapsto(-x,-t)$.
For such solutions, we have $\hat{u}(n, k) = \hat{u}(-n, -k)$, and the 
expansion~\eqref{eqn: nlw-fourier-1} reduces to 
\begin{equation}
    \label{eqn: nlw-fourier-2}
    u(x, t) = \sum_{k \in \mathbb{Z}^m} \hat{u}(0, k) \cos(\langle k, \omega' t + \theta_0 \rangle) + 2 \sum_{n=1}^{\infty} \sum_{k \in \mathbb{Z}^m} \hat{u}(n, k) \cos\left( \langle k, \omega' t + \theta_0 \rangle + nx \right), 
\end{equation}
which is a superposition of real waves, in contrast to the complex-valued 
expansion~\eqref{eqn: nonlin-soln} for the NLS equation. 
Hence, for the NLW equation~\eqref{eqn: nlw}, it suffices to select the 
tangential modes among non-negative integers rather than the whole of 
$\mathbb{Z}$ as 
\begin{equation}
    \label{eqn: tangent-spatial-mode-nlw}
    \mathcal{N} = \left\{ 0 \leq n_1 < n_2 < \cdots < n_m \right\} \subset \mathbb{N}, 
\end{equation}
with $ n_{m} \leq M_{0}^{\frac13}/2 - 1$, where $M_0=M_0(\varepsilon)$ is a 
truncation parameter defined in~\eqref{eqn: M0}, which depends on the 
perturbation parameter $\varepsilon$. 
Accordingly, the linear solution takes the form
\begin{equation}
    \label{eqn: lin-soln-nlw}
    u(x, t) =  2 \sum_{j=1}^{m} a_j \cos\left( \omega_j t + \theta_{0,j} + n_j x \right),
\end{equation}
where $a_j \geq \delta_0$ and $\omega_j = \varpi_{n_j}$ for $j=1, \ldots, m$. 
For computational convenience, however, we still regard $\hat{u}(n, k)$ as 
defined on the whole lattice $\mathbb{Z}^{1+m}$, impose the symmetry relation 
$\hat{u}(n, k) = \hat{u}(-n, -k)$, and work with the full expansion
\begin{equation}
    \label{eqn: nlw-fourier-3}
    u(x, t) = \sum_{n \in \mathbb{Z}} \sum_{k \in \mathbb{Z}^m} \hat{u}(n, k) e^{ i \left( \langle k, \omega' t + \theta_0 \rangle+ nx \right) }. 
\end{equation}
In this formulation, the resonant set consists of the $2m$ lattice points and 
is given by
\begin{equation}
    \label{eqn: resonant-index-nlw}
    \mathcal{S} := \left\{ (n_j, e_j), (-n_j, -e_j) \in \mathbb{Z}^{1+m} ~\big|~ j = 1, \ldots,m \right\}. 
\end{equation}

\subsection{Numerical scheme}
\label{subsec: numerical-scheme-nlw}

We now propose a numerical scheme that iteratively deforms the linear 
solution~\eqref{eqn: lin-soln-nlw} into the nonlinear quasi-periodic solution 
of the form~\eqref{eqn: nlw-fourier-3}. 
Substituting the ansatz~\eqref{eqn: nlw-fourier-3} into the one-dimensional NLW 
equation~\eqref{eqn: nlw} and matching the coefficients 
$e^{ i \left( \langle k, \omega' t + \theta_0 \rangle+ nx \right) }$, we obtain 
the lattice equation 
\begin{equation}
    \label{eqn: nlw-lattice}
    \left( - \langle k, \omega' \rangle^2 + \varpi_{n}^2 \right) \hat{u}(n, k) + \varepsilon \widehat{u^3}(n,k) = 0,
\end{equation}
for each multi-index $(n, k) \in \mathbb{Z}^{1+m}$, where the nonlinear term is 
given by the convolution 
\begin{equation}
    \label{eqn: nonlinear-convolution-nlw}
    \widehat{u^3}(n, k) = \sum_{n_1-n_2+n_3 = n} \sum_{k_1 - k_2 + k_3 =k} \hat{u}(n_1, k_1) \hat{u}(-n_2, -k_2) \hat{u}(n_3, k_3).  
\end{equation}
Collecting the coefficients $\hat{u}(n, k)$ into a vector $\hat{u}$, we may 
write~\eqref{eqn: nlw-lattice} in vector form as
\begin{equation}
    \label{eqn: nlw-lattice-vector}
    -\langle k, \omega' \rangle^2 \hat{u} + \varpi^2 \odot \hat{u} + \varepsilon \hat{X}(\hat{u}) = 0,
\end{equation}
where $\varpi^2 = \left(\varpi_n^2 \right)_{n \in \mathbb{Z}}$ and 
$\hat{X}(\hat{u})$ denotes the nonlinear vector field. 
As in the NLS case, we decompose the Fourier vector and the vector field as 
$\hat{u} = (\hat{u}_q, \hat{u}_p)$ and $\hat{X} = (\hat{X}_q, \hat{X}_p)$ 
according to the resonant set $\mathcal{S}$ defined 
in~\eqref{eqn: resonant-index-nlw}. 
The resulting $Q$-equations and $P$-equations, however, take a different form 
from their NLS counterparts~\eqref{eqn: q-eqn}--\eqref{eqn: p-eqn}, since the 
linear multiplier in~\eqref{eqn: nlw-lattice-vector} is quadratic rather than 
linear in the frequency. 
Due to the symmetry $\hat{u}(n,k)=\hat{u}(-n,-k)$, the two $Q$-equations 
associated with $(n_j,e_j)$ and $(-n_j,-e_j)$ are identical. 
We therefore retain only one equation from each symmetric pair and regard 
$\hat{X}_q$ as the corresponding $m$-dimensional vector.
Thus, the $Q$-equations for the resonant indices read
\begin{equation}
    \label{eqn: q-eqn-nlw}
    (- \omega'^2 + \omega^2) \odot a + \varepsilon \hat{X}_q = 0, 
\end{equation}
and the $P$-equations for the non-resonant indices read
\begin{equation}
    \label{eqn: p-eqn-nlw}
    - \langle k, \omega'\rangle^2 \hat{u}_p + \varpi^2 \odot \hat{u}_p + \varepsilon \hat{X}_p = 0. 
\end{equation}

Before proposing the alternating numerical scheme, we first define the nonlinear 
operator $F$, acting on the non-resonant Fourier coefficients $\hat{u}_p$, as 
\begin{equation}
    \label{eqn: f-operator-nlw}
    F(\omega', \hat{u}_p; a) = - \langle k, \omega' \rangle^{2} \hat{u}_p + \varpi^2 \odot \hat{u}_p + \varepsilon \hat{X}_p(\hat{u}_p, a).   
\end{equation}
The associated tangent operator is $T = D + \varepsilon S$, where $D$ is 
diagonal over the non-resonant set $\mathbb{Z}^{1+m}\setminus \mathcal{S}$ with 
entries 
\begin{equation}
    \label{eqn: D-linear-nlw}
    D_{n, k} = - \langle k, \omega' \rangle^2 + n^2 + \mu, 
\end{equation}
and $S$ is the Jacobian of the nonlinear perturbation with respect to 
$\hat{u}_p$. 
Setting the truncation scale $N_r:=M^{r+1}$, we propose the following 
alternating numerical scheme, built from the $Q$-equations~\eqref{eqn: q-eqn-nlw} 
and the $P$-equations~\eqref{eqn: p-eqn-nlw}:
\begin{subequations}
    \label{eqn: alternating-numerical-nlw}
    \begin{empheq}[left=\empheqlbrace]{align} 
       &\omega^{(r+1)} = \sqrt{ \omega^2 + \varepsilon \hat{X}_q(\hat{u}_p^{(r)}; a) \odot a^{-1} },                                                                                                                                                                       \label{eqn: alternating-nlw-frequency}                \\
       &\hat{u}_p^{(r+1)} = \hat{u}_p^{(r)} - \left[ (T + \varepsilon B)_{N_{r+1}}^{-1}\left( \hat{u}_p^{(r)}, \omega^{(r+1)}; a \right) \right] F\left( \hat{u}_p^{(r)}, \omega^{(r+1)}; a \right).         \label{eqn: alternating-nlw-coefficient} 
    \end{empheq}        
\end{subequations} 
where $\sqrt{ \cdot }$ is taken componentwise. 
The correction operator $B$, which accounts for the frequency 
update~\eqref{eqn: alternating-nlw-frequency}, is introduced to ensure the convergence and is given by
\begin{equation}
    \label{eqn: operator-b-nlw}
    B = - \langle k, \omega' \rangle \cdot \frac{\partial \langle k, \hat{X}_q \odot a^{-1} \odot \omega'^{-1} \rangle}{\partial \hat{u}_p} \hat{u}_p^{\top}.
\end{equation}
Moreover, we fix the threshold $\varepsilon_0 >0$ such that 
$n_1^2 + \mu - \sqrt{\varepsilon} \geq \mu/2$ for any 
$0 < \varepsilon \leq \varepsilon_0$, which ensures that the square root 
in~\eqref{eqn: alternating-nlw-frequency} is real and well defined. 
In other words, in the present setting the frequency domain $\Omega$ is not 
only bounded from above but also bounded away from the origin: it excludes a 
neighbourhood of the origin.

\subsection{The main theorem: a priori bound}
\label{subsec: main-nlw}

In this section, we establish convergence and a priori bounds for the 
alternating numerical procedure~\eqref{eqn: alternating-numerical-nlw}. 
In this section, we establish convergence and a priori bounds for the 
alternating numerical procedure~\eqref{eqn: alternating-numerical-nlw}. 
As in the NLS case, the linear frequencies
$\omega_j=\sqrt{n_j^2+\mu}$ remain fixed, while the updated frequency
$\omega'$ is determined at each iteration by the $Q$-equation
\[
    Q(\omega', a) = (-\omega'^2+\omega^2)\odot a +\varepsilon\hat{X}_q(\omega',a)=0.
\]
For a fixed iteration, suppressing the iteration index, this equation defines
an implicit relation between the amplitude $a$ and the updated frequency
$\omega'$. For sufficiently small $\varepsilon$ and
$\min_{1\leq j\leq m}a_j\geq\delta_0>0$, the frequency map
$a\mapsto\omega'(a)$ is a local diffeomorphism.

Before stating the main convergence theorem, we first point out the key 
difference between the NLW equation~\eqref{eqn: nlw} and the NLS 
equation~\eqref{eqn: nls}, which lies in the exclusion of the ``bad'' 
frequencies. 
For the small scale $M_0 \leq N \leq M$, the tangential nearly-resonant set 
$\Omega_{0,M}^{\tau}$ to be excluded is the same 
as~\eqref{eqn: tangential-nearly-resonant}. 
Since the single nearly-resonant set is associated with the central box, we 
have the inclusion
\begin{align*}
    & \left\{ \omega' \in \Omega \bigg| \left| \langle k, \omega' \rangle^2 - \varpi_n^2 \right| < \frac{1}{(1+|k|)^{2\tau}} \right\}  \\
    &\subseteq \left\{ \omega' \in \Omega \bigg| \left| \langle k, \omega' \rangle  + \varpi_n \right| < \frac{1}{(1+|k|)^{\tau}} \right\} \bigcup \left\{ \omega' \in \Omega \bigg| \left| \langle k, \omega' \rangle  - \varpi_n \right| < \frac{1}{(1+|k|)^{\tau}} \right\},
\end{align*}
and hence, by the symmetry $k \mapsto -k$ of $\Lambda_{2M}^m$, the single-mode 
nearly-resonant set for the NLW~\eqref{eqn: nlw} equation is again the same 
as~\eqref{eqn: single-nearly-resonant}, namely
\begin{equation}
    \label{eqn: single-nearly-resonant-nlw}
    \Omega_{1,M}^{\tau} :=
    \bigcup_{k \in \Lambda_{2M}^m} \bigcup_{n \in \mathbb{Z} \setminus \pm \mathcal{N}}
    \left\{ \omega' \in \Omega \bigg| \left| \langle k,\omega' \rangle - \varpi_{n} \right| < \frac{1}{(1+|k|)^\tau} \right\}.
\end{equation}
The difference nearly-resonant set, on the other hand, is associated with the 
outer box. 
For $n_1, n_2 \in \mathbb{Z} \setminus \mathcal{N}$ with $|n_1| \neq |n_2|$, we 
have the inclusion 
\begin{align*}
    & \left\{ \omega' \in \Omega \bigg| \left| \langle k_1,\omega'\rangle^2 - \varpi_{n_1}^2 \right| < \frac{1}{4(2+|k_1|)^{2\tau}}~\mathrm{and}~\left| \langle k_2,\omega'\rangle^2 - \varpi_{n_2}^2 \right| < \frac{1}{4(2+|k_2|)^{2\tau}} \right\} \\
    & \subseteq \left( \left\{ \omega' \in \Omega \bigg|  \left| \langle k_1,\omega' \rangle + \varpi_{n_1} \right| < \frac{1}{2(2+|k_1|)^{\tau}}\right\} \bigcup \left\{ \omega' \in \Omega \bigg| \left| \langle k_1,\omega' \rangle - \varpi_{n_1} \right| < \frac{1}{2(2+|k_1|)^{\tau}}\right\} \right)\\
    & \mathrel{\phantom{sul}} \bigcap \left( \left\{ \omega' \in \Omega \bigg| \left| \langle k_2,\omega' \rangle + \varpi_{n_2} \right| < \frac{1}{2(2+|k_2|)^{\tau}}\right\} \bigcup \left\{ \omega' \in \Omega \bigg| \left| \langle k_2,\omega' \rangle - \varpi_{n_2} \right| < \frac{1}{2(2+|k_2|)^{\tau}}\right\} \right) \\
    & \subseteq \left\{ \omega' \in \Omega \bigg| \left| \langle k_1 - k_2,\omega'\rangle + \varpi_{n_1} - \varpi_{n_2} \right| < \frac{1}{(2+|k_1 - k_2|)^{\tau}}\right\} \\
    & \mathrel{\phantom{sul}} \bigcup \left\{ \omega' \in \Omega \bigg| \left| \langle k_1 - k_2,\omega'\rangle + \varpi_{n_1} + \varpi_{n_2} \right| < \frac{1}{(2+|k_1 - k_2|)^{\tau}}\right\} 
\end{align*}
Writing $k := k_1 - k_2 \in \Lambda_{2M}^{m} $ and using again the symmetry 
$k \mapsto -k$ of $\Lambda_{2M}^m$, we conclude that the difference 
nearly-resonant set must contain not only the 
set~\eqref{eqn: difference-nearly-resonant}, namely
\begin{equation}
    \label{eqn: nlw-difference-nearly-resonant}
    \Omega_{2,M}^{\tau,+} :=
    \bigcup_{k \in \Lambda_{2M}^m} \bigcup_{\substack{n_1, n_2 \in \mathbb{Z} \setminus \pm \mathcal{N}  \\  |n_1| \neq |n_2|}}
    \left\{ \omega' \in \Omega \bigg| \left| \langle k, \omega' \rangle - \varpi_{n_1}+ \varpi_{n_2} \right| < \frac{1}{(2+|k|)^\tau} \right\},
\end{equation}
but also the additional set
\begin{equation}
    \label{eqn: nlw-sum-nearly-resonant}
    \Omega_{2,M}^{\tau,-} :=
    \bigcup_{k \in \Lambda_{2M}^m} \bigcup_{n_1, n_2 \in \mathbb{Z} \setminus \pm \mathcal{N}}
    \left\{ \omega' \in \Omega \bigg| \left| \langle k, \omega' \rangle - \varpi_{n_1}- \varpi_{n_2} \right| < \frac{1}{(2+|k|)^\tau} \right\}.
\end{equation}

For the nearly-resonant sets, $\Omega_{0,M}^{\tau}$, $\Omega_{1,M}^{\tau}$, and 
$\Omega_{2,M}^{\tau,-}$, the heart of the measure estimate is that, for each 
fixed $k$, the number of spatial modes $n$ that can contribute is at most of 
order $\|k\|_2$. 
This is exactly the mechanism used in the proof 
of~\Cref{lem: near-measure-initial} and the details are shown 
in~\Cref{sec: near-measure-initial}. 
The essential new difficulty arises from the set $\Omega_{2,M}^{\tau,+}$ 
defined in~\eqref{eqn: nlw-difference-nearly-resonant}: the frequency 
difference $\varpi_{n_1} - \varpi_{n_2}$ can remain bounded even when $n_1$ and 
$n_2$ both tend to infinity, so that infinitely many pairs $(n_1, n_2)$ may 
compete for the same $k$. 
By the symmetry $\varpi_{-n} = \varpi_{n}$ and the symmetry $k \mapsto -k$ of 
$\Lambda_{2M}^m$, we may assume $0 \leq n_1 < n_2$. The key observation is the 
elementary identity $\varpi_n - n = \mu / (\varpi_n + n)$ for $n \geq 0$, which 
shows that each $\omega'$ in the set 
in~\eqref{eqn: nlw-difference-nearly-resonant} satisfies
\begin{equation}
    \label{eqn: difference-resonant-nlw}
    \left| \langle k, \omega' \rangle - n_1 + n_2\right| < \frac{1}{(2+|k|)^\tau} + \frac{\mu}{\sqrt{n_1^2 + \mu} + n_1} + \frac{\mu}{\sqrt{n_2^2 + \mu} + n_2}. 
\end{equation}
Guided by~\eqref{eqn: difference-resonant-nlw}, we count the admissible pairs 
$(n_1, n_2)$ according to whether $n_1 \geq (2+|k|)^\frac{\tau}{2}$ or 
$n_1 < (2+|k|)^\frac{\tau}{2}$. 
In the first case, both correction terms 
in~\eqref{eqn: difference-resonant-nlw} are of order 
$(2+|k|)^{-\frac{\tau}{2}}$, so that only $O(\|k\|_2)$ values of the 
integer difference $n_2 - n_1$ are admissible, each contributing a set of 
measure of order $(2+|k|)^{-\frac{\tau}{2}} \|k\|_2^{-1}$. 
In the second case, there are at most $(2+|k|)^{\frac{\tau}{2}}$ choices of 
$n_1$, and for each fixed $n_1$, at most of order $\|k\|_2$ values of $n_2$ are 
admissible, each contributing a set of measure of order 
$(2+|k|)^{-\tau} \|k\|_2^{-1}$. 
In both cases, the total measure excluded for a fixed $k \in \Lambda_{2M}^{m}$ 
is of order $(2+|k|)^{-\frac{\tau}{2}}$. 
Summing over $k \in \Lambda_{2M}^{m}$, the resulting series converges provided 
$\tau/2 > m$, that is, $\tau > 2m$.

In the large-scale regime $N > M$, the difference nearly-resonant set $G_2$ 
defined in~\eqref{defn: second-melnikov} is excluded by the same mechanism as 
above, while the single-mode nearly-resonant set $G_1$ defined 
in~\eqref{eqn: final-msa-single-mode-nearly-resonant} is treated exactly as 
before, since replacing $\sigma_{\star}$ by $\sigma_{\star}^2$ does not affect 
the analytic properties involved. 
With these key differences in place, we now state the main theorem, which 
establishes the convergence of the alternating numerical 
procedure~\eqref{eqn: alternating-numerical-nlw} and provides a priori error 
bounds for its iterates.

\begin{theorem}[Convergence and a priori bound]
    \label{thm: main-nlw}
    Fix $\delta_0 \in (0, 1)$ and $\tau > 2m$, and let 
    $\mathscr{A} = [\delta_0, 1]^{m}$. 
    There exists a threshold $\varepsilon_0 > 0$ such that, for any 
    $0 < \varepsilon \leq \varepsilon_0$ the following statements hold: 
    \begin{itemize}
        \item[1.] \textbf{Measure of resonance}: 
                  There exist constants $\kappa = \kappa(m, \mu, \tau) >0$ and 
                  $\delta = \delta(\varepsilon) > 0$, with $\delta \rightarrow 0$ 
                  as $\varepsilon \rightarrow 0$, such that the excluded set 
                  $\mathscr{A}^{\star} \subset \mathscr{A}$ of initial amplitudes 
                  satisfies $\mathrm{mes}(\mathscr{A}^{\star}) \leq \kappa + \delta$. 

        \item[2.] \textbf{Iterative sequence}: 
                  For any $ a \in \mathscr{A} \setminus \mathscr{A}^{\star}$, 
                  the alternating numerical 
                  scheme~\eqref{eqn: alternating-numerical-nlw}, initialized at 
                  $(\omega^{(0)}, \hat{u}_p^{(0)} ) = (\omega, 0)$, generates a 
                  well-defined sequence of iterates 
                  $\{ (\omega^{(r)}, \hat{u}_p^{(r)}) \}_{r=0}^{\infty}$. 
                  Moreover, for every $r\geq0$, the updated frequency map
                  $a\mapsto\omega^{(r+1)}(a)$ is a local diffeomorphism.

        \item[3.] \textbf{Convergence}: There exists a Gevrey exponent 
                  $s = s(\varepsilon) \in (0, 1)$ such that this sequence 
                  converges to an exact solution pair 
                  $(\omega^\star, \hat{u}_p^\star)$. 
                  The limiting frequency map $a\mapsto\omega^\star(a)$ is also 
                  a local diffeomorphism, and $\hat{u}_p^\star\in\mathcal{K}(s)$.

        \item[4.] \textbf{A priori error bounds}: For any $r \geq 0$, the 
                  iterative sequence 
                  $\{ \omega^{(r)}, \hat{u}_p^{(r)} \}_{r=0}^{\infty}$ satisfies 
                  the following error bounds:
                  \begin{subequations}
                      \label{eqn: convergence-vector-frequency-nlw}
                      \begin{empheq}[left=\empheqlbrace]{align} 
                          \|\hat{u}^{(r)}-\hat{u}^{\star}\| &\leq \exp\left\{ -\frac{3}{2}(M^s)^{r+1} \right\},                    \label{eqn: u-hat-converge-nlw}               \\
                          |\omega^{(r)}-\omega^{\star}| &\leq \exp\left\{ -\frac{3}{2}(M^s)^{r+1} \right\}.                   \label{eqn: frequency-converge-nlw}
                      \end{empheq}
                  \end{subequations}
    \end{itemize}
    Furthermore, define the $r$-th approximate solution by 
    \begin{equation}
        \label{eqn: approximate-pde-solution-nlw}
        u^{(r)}(x,t) =  \sum_{ (n, k) \in \mathbb{Z}^{1+m} } \hat{u}^{(r)}(n, k) e^{i \langle k, \; \omega^{(r)} t + \theta_0 \rangle } e^{inx}. 
    \end{equation}
    Then, for any  $t \in [0, M^r]$, the following time-domain error bound holds:
    \begin{equation}
        \label{eqn: pde-solution-error-nlw}
        \|u^{(r)}(x,t)-u^{\star}(x,t)\|_{L^2} \leq \exp \left\{ -(M^s)^{r+1} \right\},
    \end{equation}
    where $u=u^{\star}(x, t)$ is the exact quasi-periodic solution of the NLW 
    equation~\eqref{eqn: nlw} with periodic boundary 
    conditions~\eqref{eqn: per-bdy} corresponding to 
    $(\omega^{\star}, \hat{u}_p^{\star})$. 
    Consequently, for any $t\geq0$, $u^{(r)}(x,t)$ converges in $L^2$ to 
    $u^\star(x,t)$ as $r \rightarrow \infty$.
\end{theorem}

\begin{remark}
\label{rem: nlw}
The assumption $\mu > 0$ can be relaxed, subject to suitable nondegeneracy 
conditions on the tangential and non-tangential modes. 
In particular, when $\mu \leq 0$ and $|\mu| \leq M^2$, the alternating numerical 
procedure~\eqref{eqn: alternating-numerical-nlw} can be adapted to construct 
quasi-periodic solutions. 
The measure estimates for nearly-resonant frequencies concern the elliptic 
modes, for which $n^2 + \mu \geq 0$. 
For a hyperbolic mode with $n^2+\mu<0$, the corresponding diagonal divisor is 
$-\langle k,\omega'\rangle^2+n^2+\mu$, and therefore cannot vanish,  
so these modes require no frequency-exclusion estimate. 
The borderline case $n^2+\mu=0$ is already ruled out by the tangential 
nearly-resonant condition.
\end{remark}

\subsection{Numerical demonstration}
\label{subsec: numerical-periodic-nlw}

We demonstrate the performance of the alternating numerical 
scheme~\eqref{eqn: alternating-numerical-nlw} for the NLW 
equation~\eqref{eqn: nlw} with periodic boundary 
conditions~\eqref{eqn: per-bdy}. 
In our numerical tests, we set the initial phase $\theta_0 = 0$, and exclude 
the spatial mode $n = 0$ from the set of the tangential modes.

\paragraph{Periodic solutions ($m=1$)}
Unlike the linear solution~\eqref{eqn: lin-1} for the NLS equation, the linear 
solution here is a real wave and given by
\begin{equation}
    \label{eqn: lin-1-nlw}
    u(x, t) = 2a_1\cos(\omega_1 t + n_1 x)
\end{equation}
which is because the two nonzero space-time Fourier coefficients satisfy $\hat{u}(n_1,1) = \hat{u}(-n_1,-1)$. Consequently, the support of the nonlinear term is no longer confined to a single mode, and the nonlinear solution takes the form
\begin{subequations}
    \label{eqn: nlw-periodic-1}
    \begin{empheq}[left=\empheqlbrace]{align} 
       & \hat{u}_{n_1}(t) = \overline{\hat{u}_{-n_1}(t)} = a_1 e^{i\omega_1't} + \sum_{k \in \mathbb{Z}\setminus \{1\}} \hat{u}(n_1, k) e^{ik\omega_1' t},                  \label{eqn: nlw-periodic-1-1}                \\
       & \hat{u}_{n}(t)   = \sum_{k \in \mathbb{Z}} \hat{u}(n, k) e^{ik\omega_1' t}, \mathrel{\phantom{===============}} n \in \mathbb{Z} \setminus \{-n_1, n_1\},          \label{eqn: nlw-periodic-1-2}                
    \end{empheq}        
\end{subequations} 
We remark that the canonical variables used in the phase-space plots differ 
from those of the NLS case. 
For the NLS equation~\eqref{eqn: nls}, the canonical pair associated with each 
spatial mode is $(\mathrm{Re} (\hat{u}(t)), \mathrm{Im} (\hat{u}(t))) $, 
whereas for the NLW equation~\eqref{eqn: nlw}, which is second order in time, 
we use 
\[
    \left(\mathrm{Re} \left(\hat{u}_n(t)\right),~ \frac{\mathrm{Re} \left(\dot{\hat{u}}_{n}(t) \right)}{\sqrt{n^2 + \mu}} \right), 
\]
so that the linear solution traces a circle in each excited mode. 
For the numerical test, we take $n_1 = 1$, $\mu = \sqrt{2}$, and $a_1 = 1$, so 
that $\omega_1 = \sqrt{1 + \sqrt{2}}$; we set the perturbation quantity 
$\varepsilon = 0.5$, and the results are shown in~\Cref{fig: nlw-p-1torus-traj}. 
\begin{figure}[htb!]
    \centering
    \begin{subfigure}[t]{0.3\linewidth}
    \centering
    \includegraphics[scale=0.18]{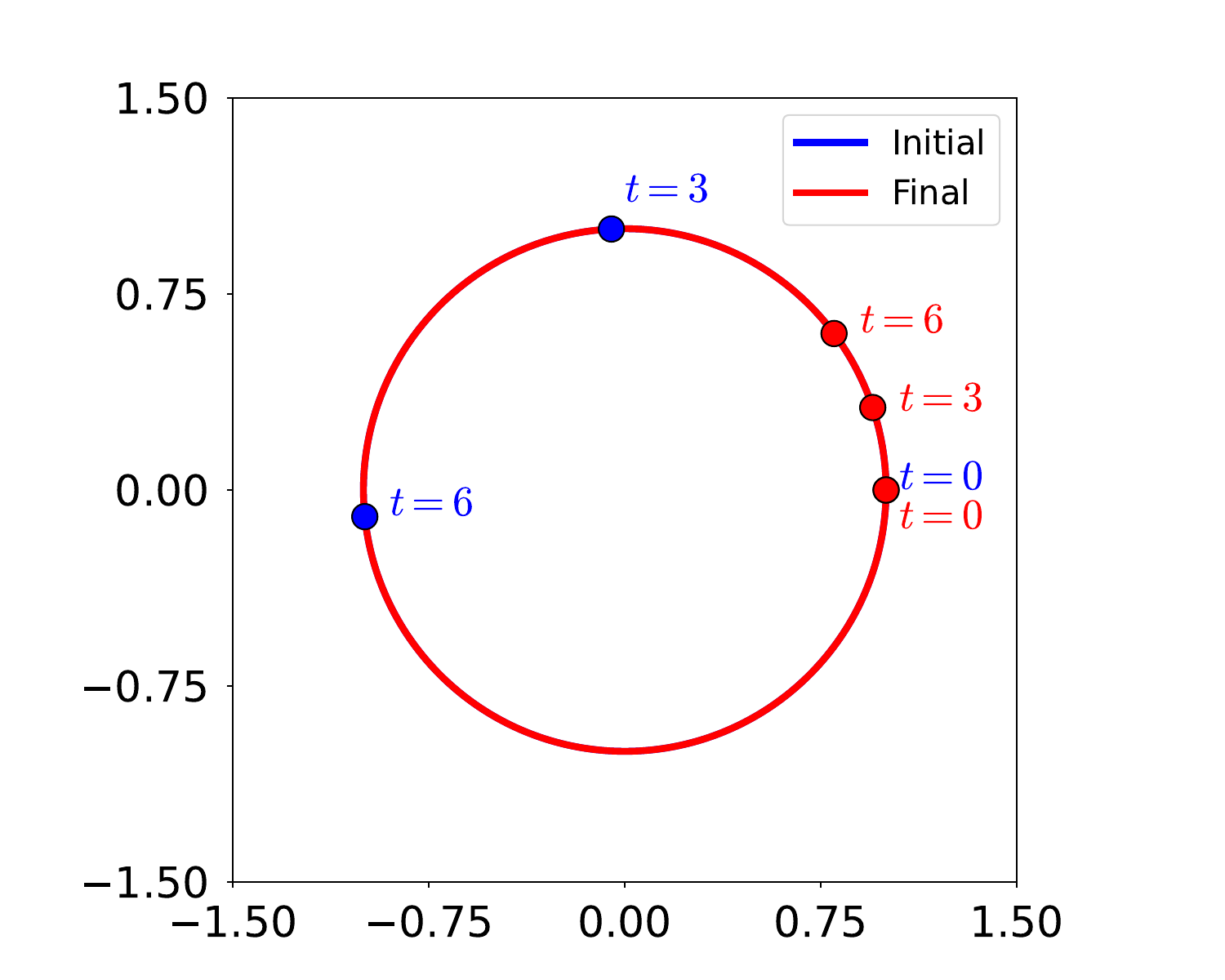}
    \caption{Tangential mode: $\hat{u}_{1}(t)$}
    \end{subfigure}
    \begin{subfigure}[t]{0.3\linewidth}
    \centering
    \includegraphics[scale=0.18]{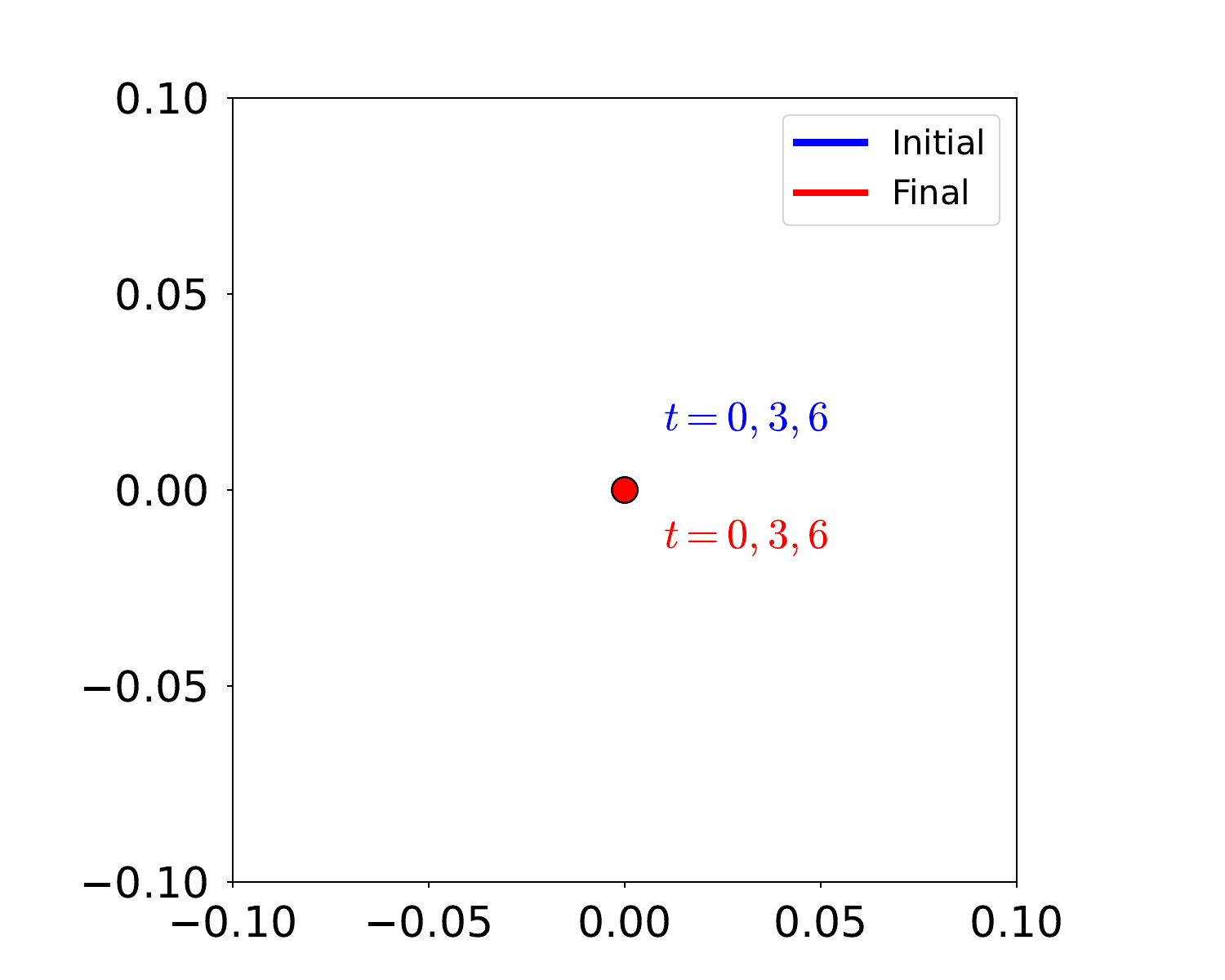}
    \caption{Normal mode: $\hat{u}_{2}(t)$}
    \end{subfigure}
    \begin{subfigure}[t]{0.3\linewidth}
    \centering
    \includegraphics[scale=0.18]{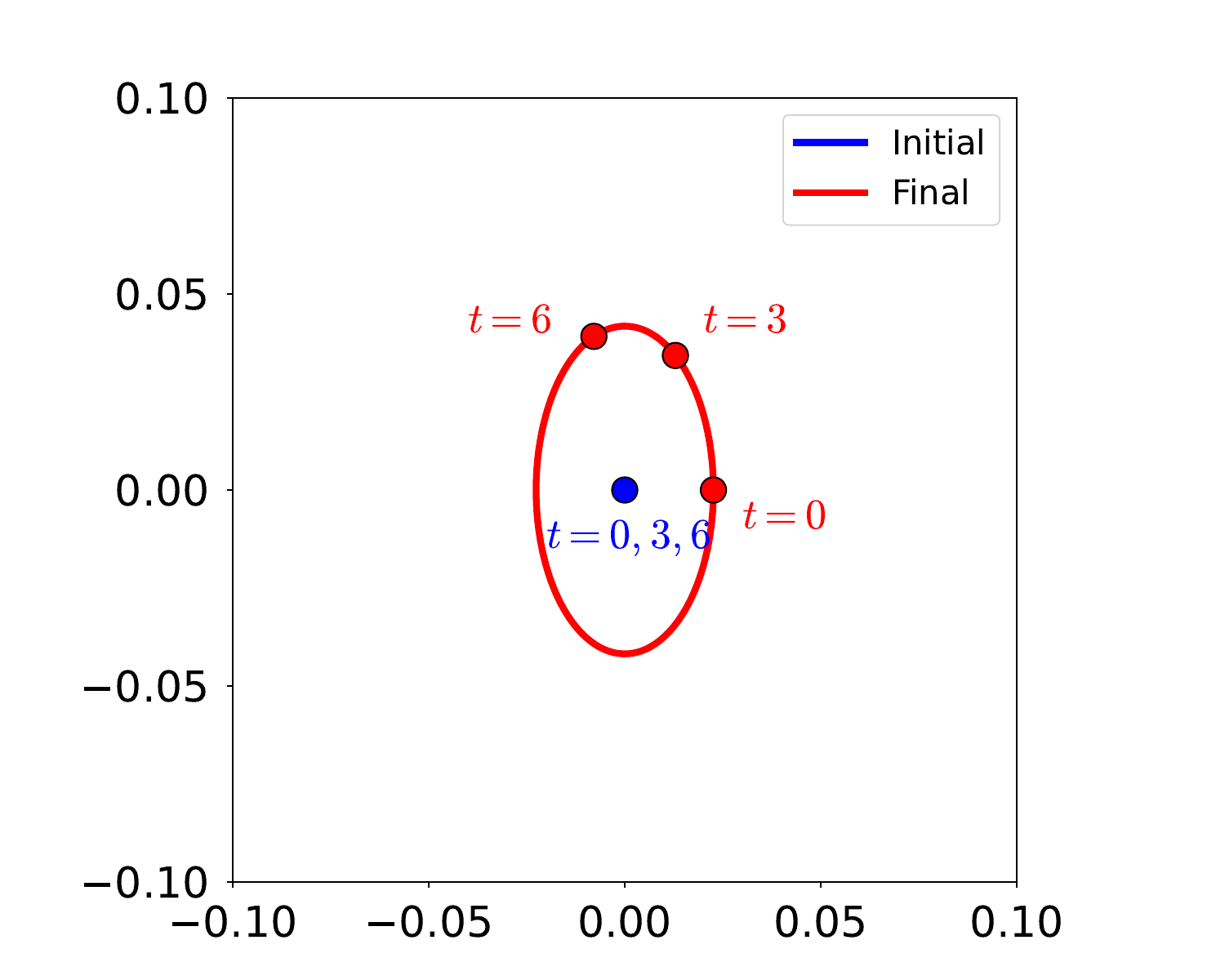}
    \caption{Normal mode: $\hat{u}_{3}(t)$}
    \end{subfigure}
    \caption{Phase space trajectories of the NLW equation~\eqref{eqn: nlw} with 
             periodic boundary conditions~\eqref{eqn: per-bdy}. 
             ``Initial'' and ``Final'' denote the initial linear solution and 
             the computed nonlinear solution, respectively. 
             Markers indicate specific time points at $t = 0$, $3$, and $6$.} 
    \label{fig: nlw-p-1torus-traj}
\end{figure}
The tangential mode $\hat{u}_{1}(t)$ retains an amplitude close to unity but 
exhibits noticeable frequency shifts relative to the linear solution. 
The normal modes, on the other hand, display two distinct types of behavior: 
the representative mode $\hat{u}_{2}(t)$ remains fixed at the origin, 
whereas $\hat{u}_{3}(t)$, which is absent from the linear solution, develops 
small-amplitude oscillations driven by the nonlinear interaction. 
Its phase-space trajectory is not circular but forms a narrow loop elongated in 
the vertical direction. 
The convergence profile is shown in~\Cref{fig: nlw-p-1torus-con}.

\begin{figure}[htb!]
    \centering
    \begin{subfigure}[t]{0.3\linewidth}
    \centering
    \includegraphics[scale=0.18]{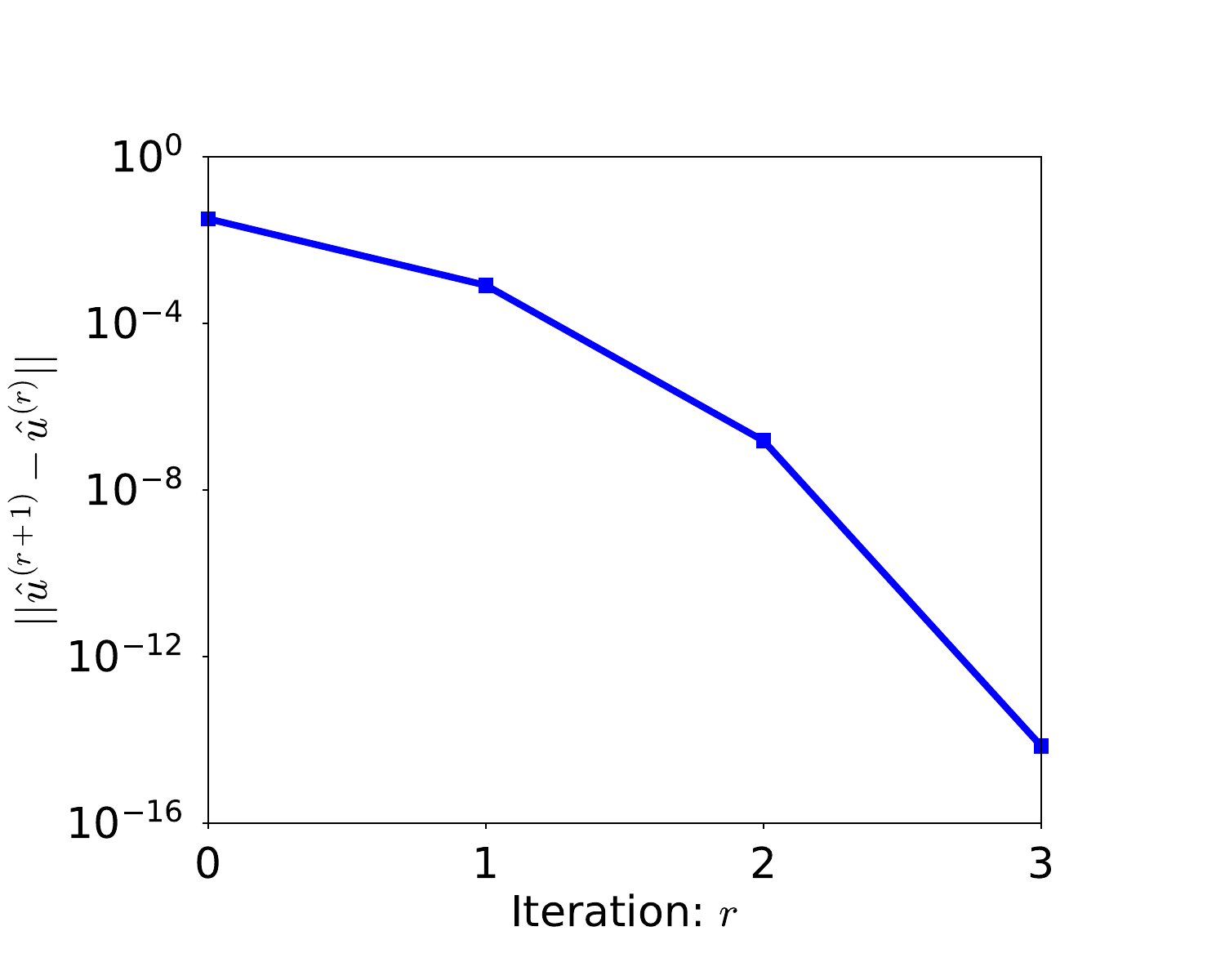}
    \caption{$\|\hat{u}_{p}^{(r+1)} - \hat{u}_{p}^{(r)}\|$}
    \end{subfigure}
    \begin{subfigure}[t]{0.3\linewidth}
    \centering
    \includegraphics[scale=0.18]{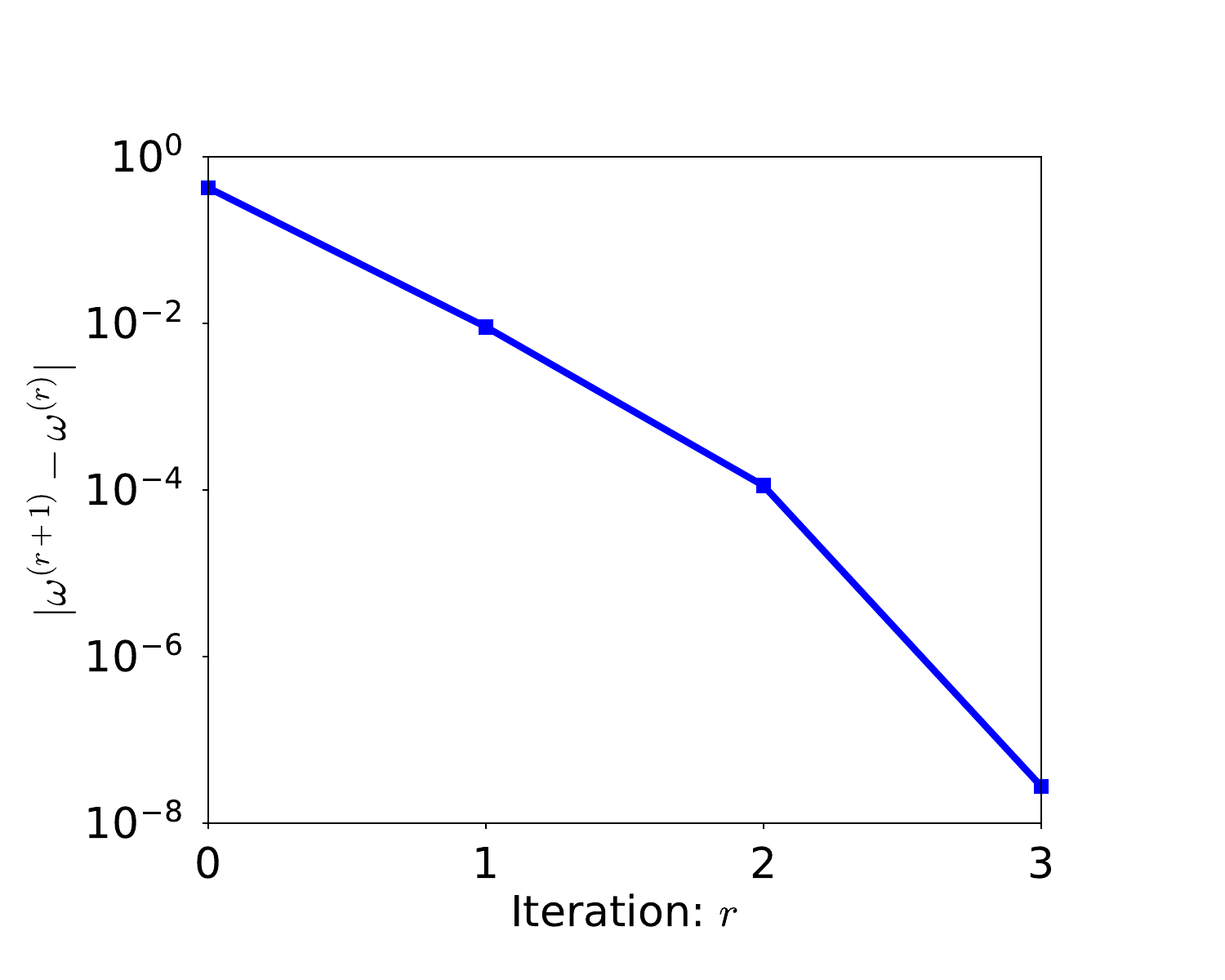}
    \caption{$|\omega^{(r+1)} - \omega^{(r)}|$}
    \end{subfigure}
    \begin{subfigure}[t]{0.3\linewidth}
    \centering
    \includegraphics[scale=0.18]{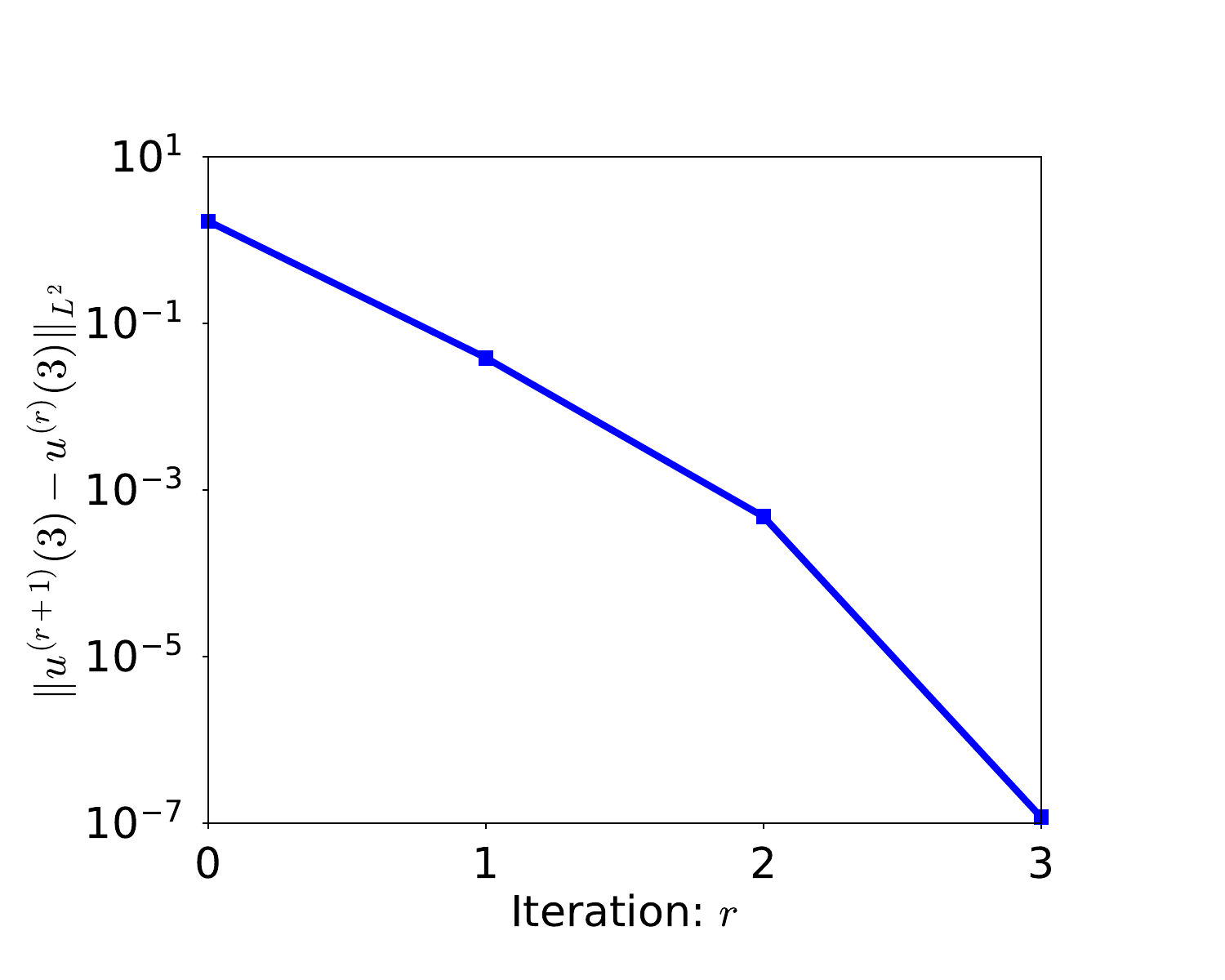}
    \caption{$\| u^{(r+1)}(x, 3) - u^{(r)}(x, 3)  \|_{L^2}$}
    \end{subfigure}
    \caption{Convergence profiles across iterations for the NLW 
             equation~\eqref{eqn: nlw} with the periodic boundary 
             condition~\eqref{eqn: per-bdy}. } 
    \label{fig: nlw-p-1torus-con}
\end{figure}

We next examine the spatiotemporal distributions of $|u(x, t)|^2$. 
Taking the squared modulus of the linear solution~\eqref{eqn: lin-1-nlw} yields
\[
    |u|^2 = 4\cos^2\left( \sqrt{1+\sqrt{2}}t + x\right).
\]
Because this expression depends on the combined phase $\sqrt{1+\sqrt{2}}t + x$, it 
forms a traveling-wave pattern moving in the negative $x$-direction, as shown 
in~\Cref{fig: nlw-p-1torus-heat-slice-lin}. 
With the nonlinear term included, the wave crests become slightly higher, and 
their trajectories indicate faster propagation in the same direction. 
At $t=3.0$ and $x=0$, the linear solution is near its second trough, whereas 
the nonlinear solution is near its third crest. 
This accumulated phase difference is consistent with an upward frequency drift 
in the nonlinear solution, as illustrated 
in~\Cref{fig: nlw-p-1torus-heat-slice-nonlin}. 
\begin{figure}[htb!]
    \centering
    \begin{subfigure}[t]{0.42\linewidth}
    \centering
    \includegraphics[scale=0.24]{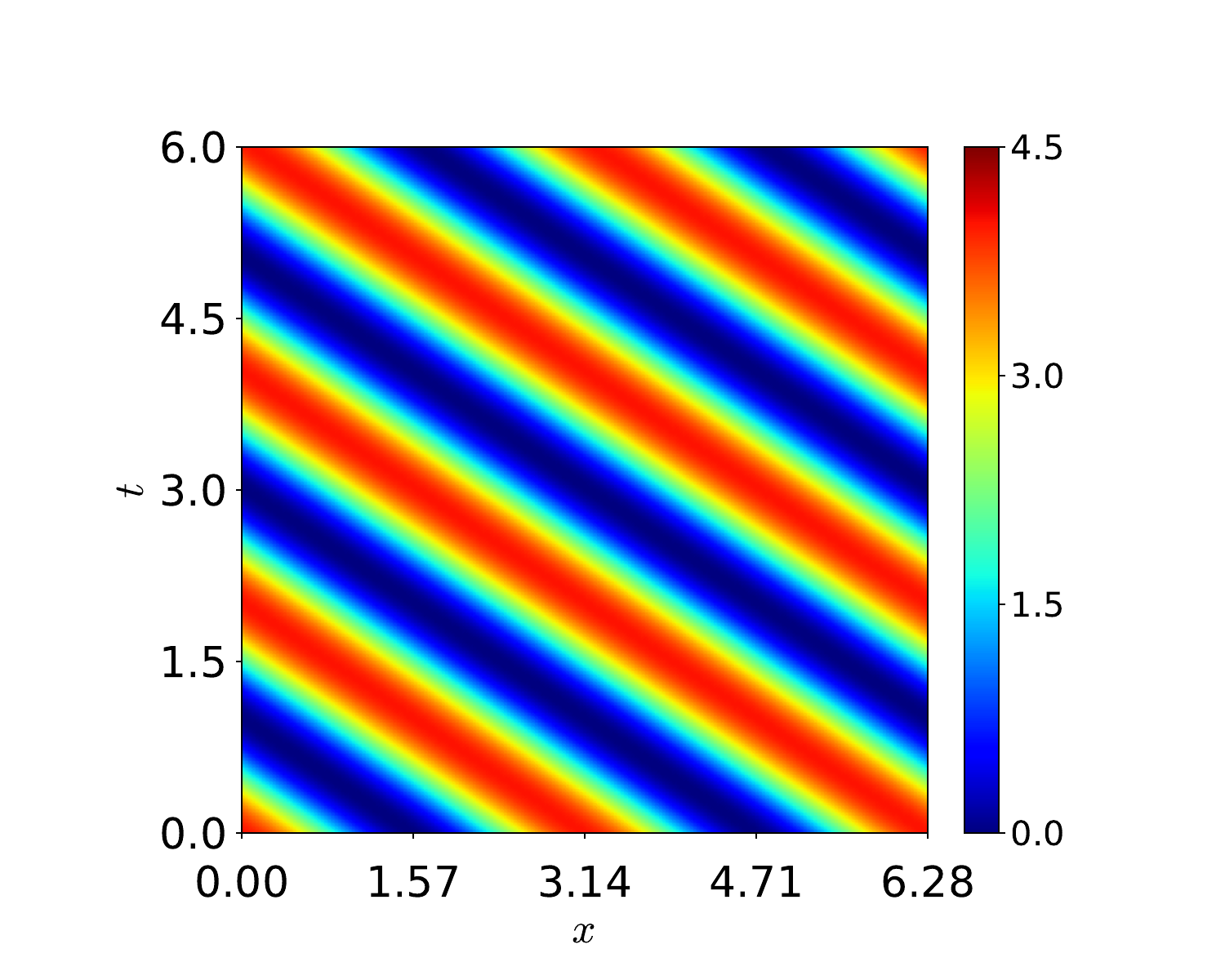}
    \caption{Linear solution}
    \label{fig: nlw-p-1torus-heat-slice-lin}
    \end{subfigure}
    \begin{subfigure}[t]{0.42\linewidth}
    \centering
    \includegraphics[scale=0.24]{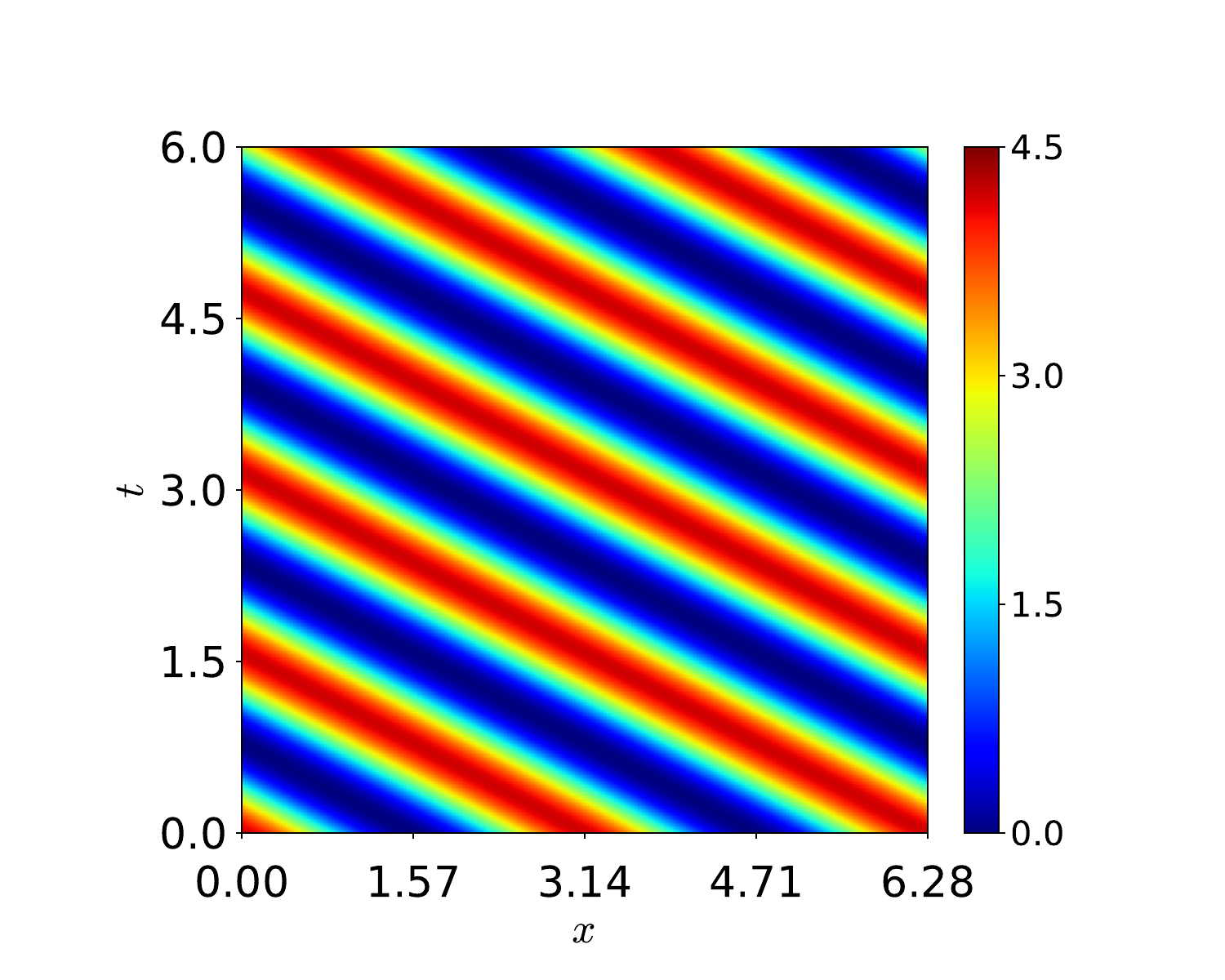}
    \caption{Nonlinear solution}
    \label{fig: nlw-p-1torus-heat-slice-nonlin}
    \end{subfigure} 
    \caption{Spatiotemporal distributions of $|u(x,t)|^2$ for the NLW 
             equation~\eqref{eqn: nlw} with the periodic boundary 
             conditions~\eqref{eqn: per-bdy}. 
             The linear solution is~\eqref{eqn: lin-1-nlw}, and the 
             nonlinear solution is obtained via the alternating numerical 
             scheme~\eqref{eqn: alternating-numerical-nlw}.}
    \label{fig: nlw-p-1torus-heat-slice}
\end{figure}

\paragraph{Quasi-periodic solutions ($m=2$)}
For $m=2$, the quasi-periodic solution for the NLW equation~\eqref{eqn: nlw} is 
a superposition of two real waves, which reads 
\begin{equation}
    \label{eqn: lin-2-nlw}
    u(x, t) = 2a_1\cos(\omega_1 t + n_1 x) + 2 a_2\cos(\omega_2 t + n_2 x).  
\end{equation}
The corresponding nonlinear solution can be written in terms of its spatial Fourier coefficients as
\begin{subequations}
    \label{eqn: nlw-periodic-2}
    \begin{empheq}[left=\empheqlbrace]{align} 
       & \hat{u}_{n_1}(t) = \overline{\hat{u}_{-n_1}(t)} = a_1 e^{i\omega_1't} + \sum_{k \in \mathbb{Z}^2\setminus \{(1, 0)\}} \hat{u}(n_1, k) e^{i\langle k, \omega' \rangle t},     \label{eqn: nlw-periodic-2-1}                \\
       & \hat{u}_{n_2}(t) = \overline{\hat{u}_{-n_2}(t)} = a_2 e^{i\omega_2't} + \sum_{k \in \mathbb{Z}^2\setminus \{(0, 1)\}} \hat{u}(n_2, k) e^{i\langle k, \omega' \rangle t},     \label{eqn: nlw-periodic-2-2}                \\
       & \hat{u}_{n}(t)   = \sum_{k \in \mathbb{Z}^2} \hat{u}(n, k) e^{i \langle k, \omega' \rangle t}, \mathrel{\phantom{===============}} n \in \mathbb{Z} \setminus \{\pm n_1, \pm n_2\},          \label{eqn: nlw-periodic-2-3}                
    \end{empheq}
\end{subequations}   
where $\omega' = (\omega_1', \omega_2')$ denotes the shifted frequency vector. 
For the numerical test, we take $n_1 = 1$, $n_2 =2$, $\mu=\sqrt{2}$, and 
$a_1=a_2 =1$, so that $\omega_1=\sqrt{1+\sqrt{2}}$ and 
$\omega_2 = \sqrt{4+\sqrt{2}}$ and the perturbation quantity $\varepsilon = 0.2$.~\Cref{fig: nlw-p-2tori-traj} compares the 
phase-space trajectories of the initial linear and computed nonlinear solutions. 
As in the periodic-solution case shown in~\Cref{fig: nlw-p-1torus-traj}, the frequency 
shifts produce visible differences between the trajectories. 
Quasi-periodic behavior is discernible, though not particularly pronounced over 
the time interval shown: the trajectories return close to previously visited 
positions without exactly retracing their paths.
\begin{figure}[htb!]
    \centering
    \begin{subfigure}[t]{0.3\linewidth}
    \centering
    \includegraphics[scale=0.18]{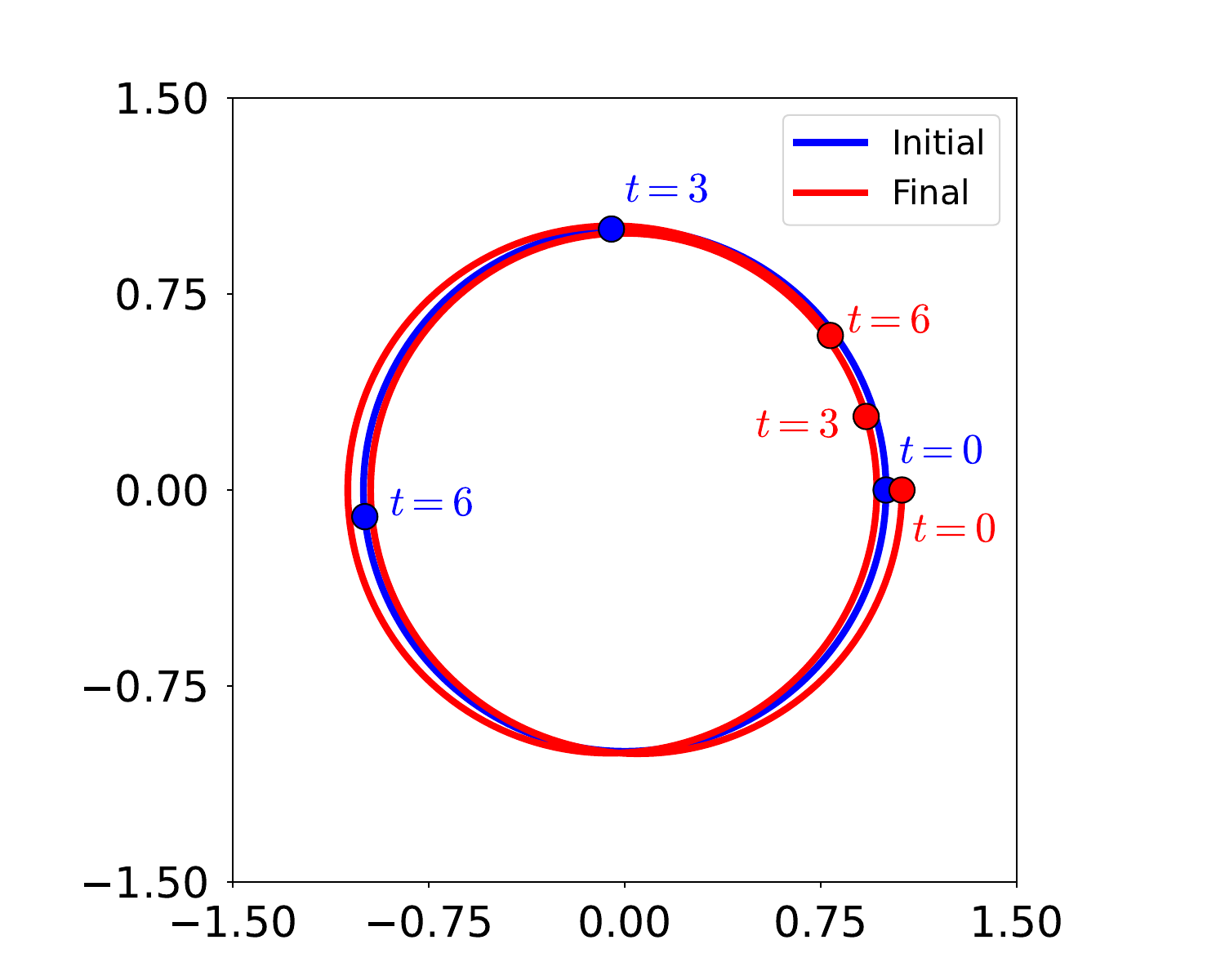}
    \caption{Tangential mode: $\hat{u}_1(t)$}
    \end{subfigure}
    \begin{subfigure}[t]{0.3\linewidth}
    \centering
    \includegraphics[scale=0.18]{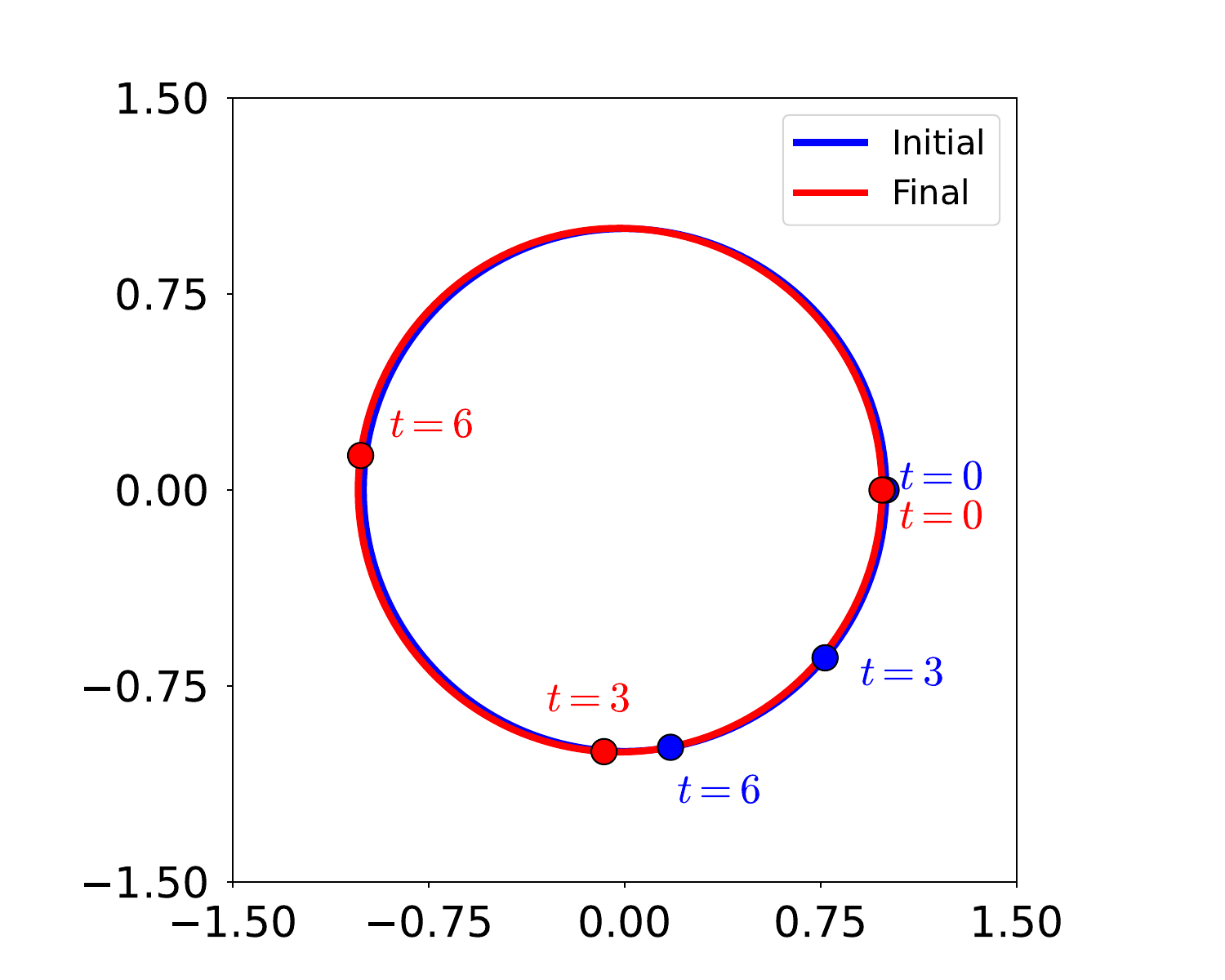}
    \caption{Tangential mode: $\hat{u}_2(t)$}
    \end{subfigure}
    \begin{subfigure}[t]{0.3\linewidth}
    \centering
    \includegraphics[scale=0.18]{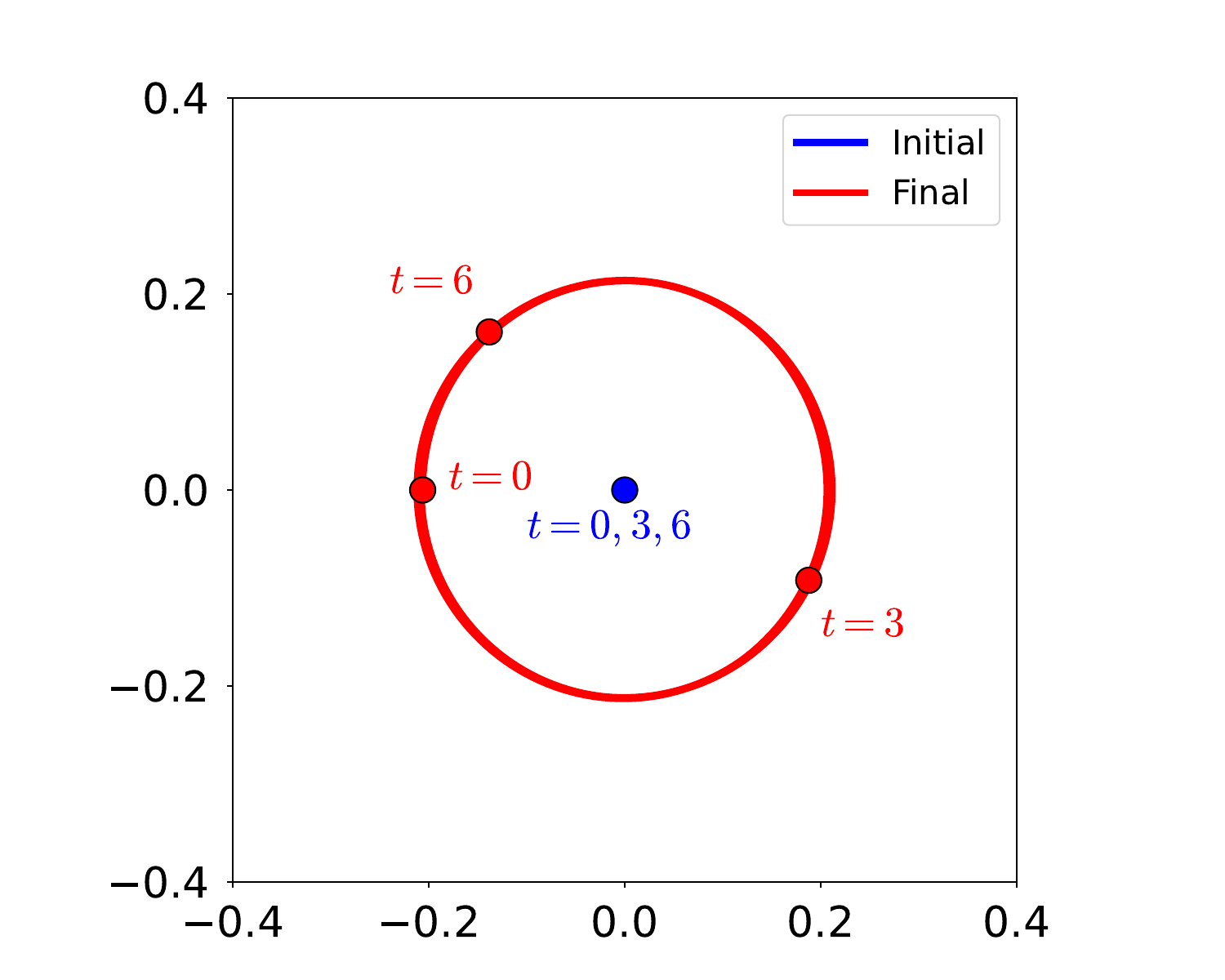}
    \caption{Normal mode: $\hat{u}_3(t)$}
    \end{subfigure}
    \caption{Phase space trajectories of the NLW equation~\eqref{eqn: nlw} with 
             periodic boundary conditions~\eqref{eqn: per-bdy}. 
             ``Initial'' and ``Final'' denote the initial linear solution and 
             the computed nonlinear solution, respectively. 
             Markers indicate specific time points at $t = 0$, $3$, and $6$.} 
    \label{fig: nlw-p-2tori-traj}
\end{figure}

Turning to the spatiotemporal distributions, the squared modulus of the linear 
solution~\eqref{eqn: lin-2-nlw}, with the parameters specified above, is
\begin{align*}
    |u|^2 = & 4\cos^2 \left( \sqrt{1+\sqrt{2}}t + x \right) + 4\cos^2 \left( \sqrt{4+\sqrt{2}}t + 2x \right) \\
            & + 8\cos \left( \sqrt{1+\sqrt{2}}t + x \right) \cos \left( \sqrt{4+\sqrt{2}}t + 2x \right).
\end{align*}
Expanding the trigonometric products shows that, apart from a constant term, 
$|u|^2$ contains four traveling-wave components with distinct phase speeds. 
Their superposition produces the diagonal bands shown 
in~\Cref{fig: nlw-p-2tori-heat-slice-lin}.  
\begin{figure}[htb!]
    \centering
    \begin{subfigure}[t]{0.42\linewidth}
    \centering
    \includegraphics[scale=0.24]{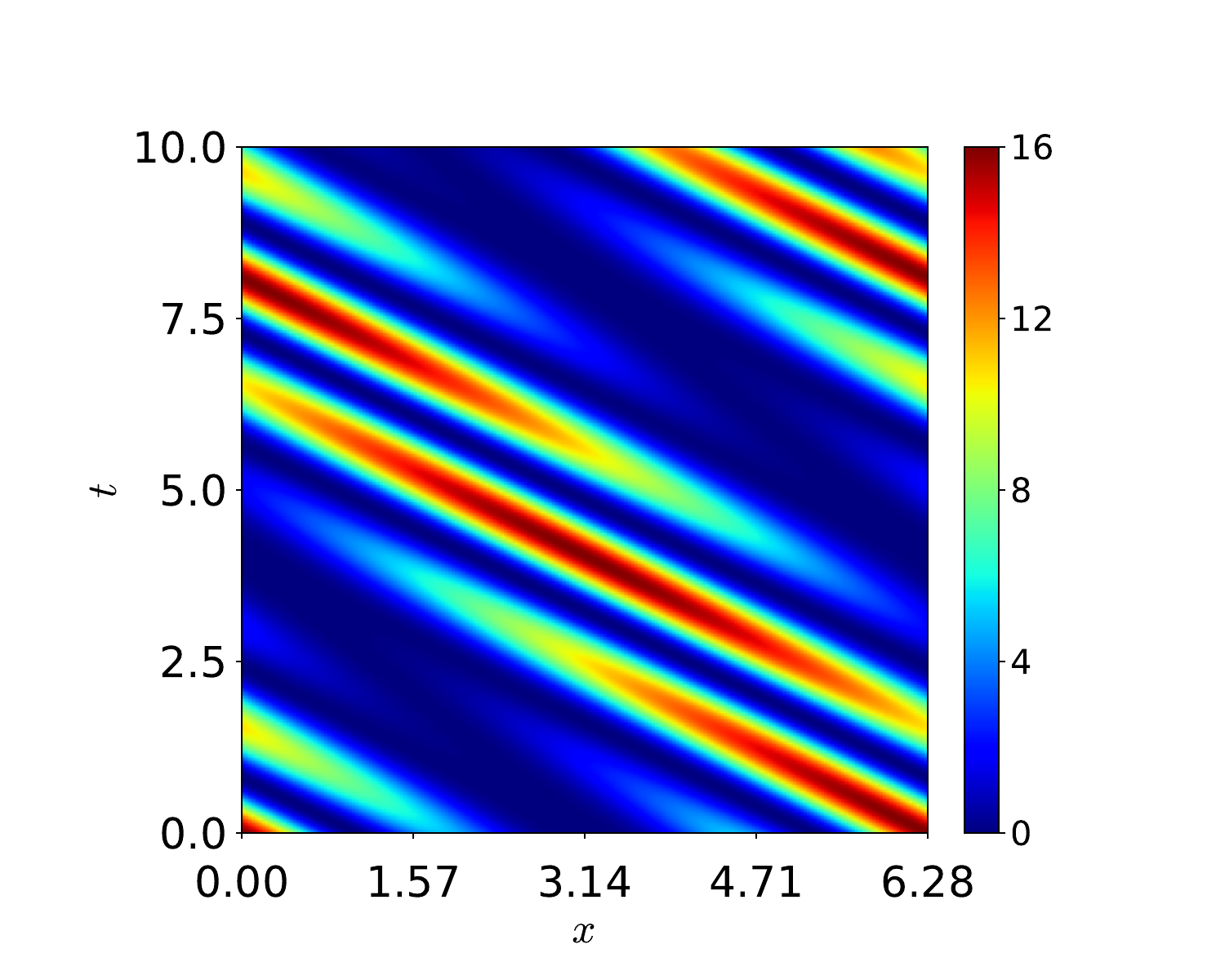}
    \caption{Linear solution}
    \label{fig: nlw-p-2tori-heat-slice-lin}    
    \end{subfigure}
    \begin{subfigure}[t]{0.42\linewidth}
    \centering
    \includegraphics[scale=0.24]{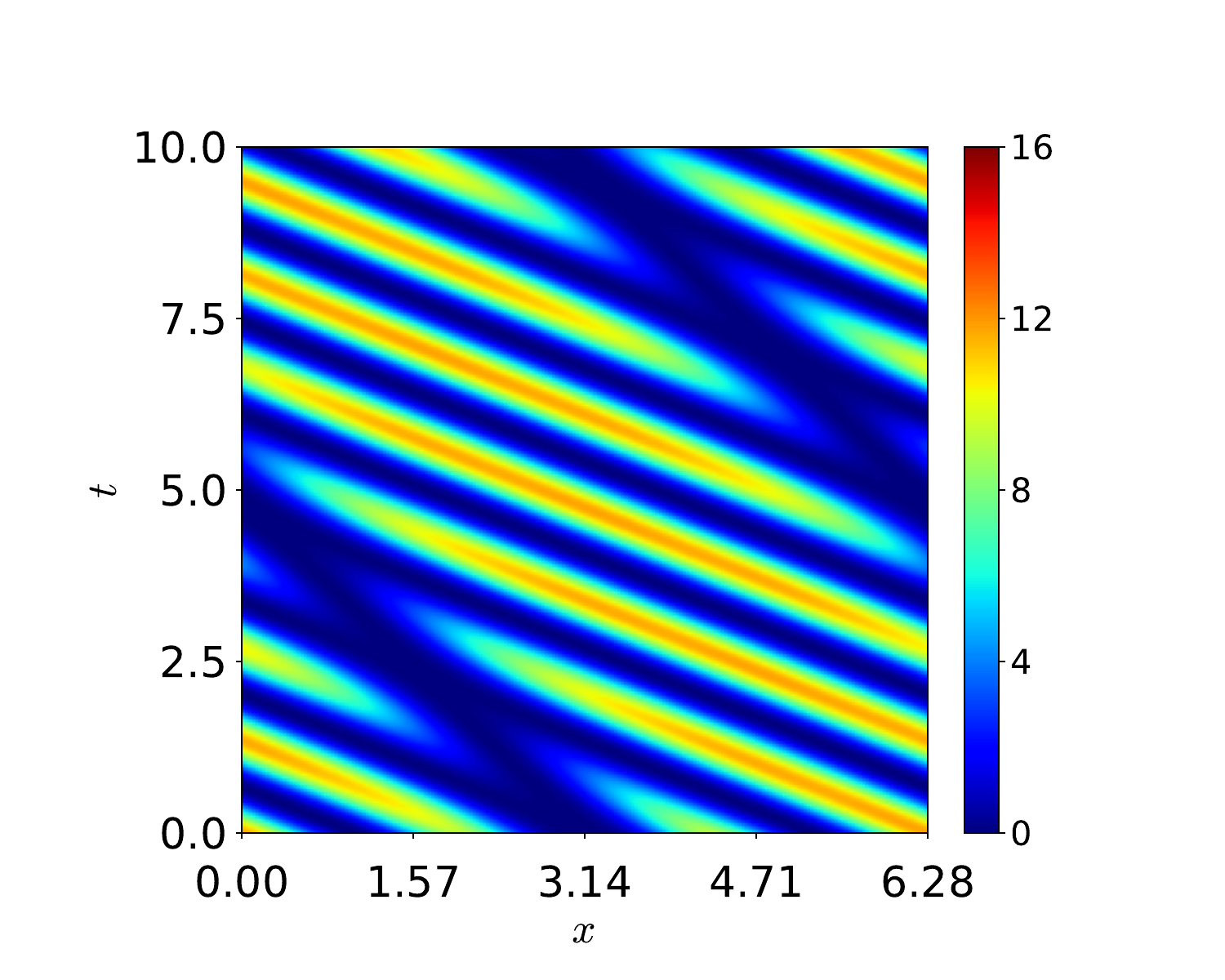}
    \caption{Nonlinear solution}
    \label{fig: nlw-p-2tori-heat-slice-nonlin}
    \end{subfigure} 
    \caption{Spatiotemporal distributions of $|u(x,t)|^2$ for the NLW 
             equation~\eqref{eqn: nlw} with the periodic boundary 
             conditions~\eqref{eqn: per-bdy}. 
             The linear solution is~\eqref{eqn: lin-2-nlw}, and the 
             nonlinear solution is obtained via the alternating numerical 
             scheme~\eqref{eqn: alternating-numerical-nlw}.}
    \label{fig: nlw-p-2tori-heat-slice}
\end{figure}
When the nonlinear perturbation is introduced, the bands remain visible, but 
their spacing and intensity patterns change. 
In~\Cref{fig: nlw-p-2tori-heat-slice-nonlin}, the bright bands are more 
closely spaced, and the high-intensity regions are rearranged. 
These changes are consistent with nonlinear frequency shifts and contributions 
from additional modes excited through nonlinear coupling.

\subsection{Generalization to the Dirichlet boundary condition}
\label{subsec: nlw-dirichlet}

We now extend the construction for the NLW equation~\eqref{eqn: nlw} from the 
periodic setting~\eqref{eqn: per-bdy} to the Dirichlet boundary 
condition~\eqref{eqn: dirichlet-bdy}.
For $x \in [-\pi, 0]$, define the odd extension by $u(x, t) = -u(-x, t)$, and 
extend the resulting function $2\pi$-periodic and odd in $x$. 
In the complex exponential expansion~\eqref{eqn: nlw-fourier-3}, spatial 
oddness implies the Fourier coefficient symmetry 
$\hat{u}(n, k) = - \hat{u}(-n, k)$. 
We select the set of tangential modes from the positive integers:
\begin{equation}
    \label{eqn: tangent-spatial-mode-nlw-dirichlet}
    \mathcal{N} = \left\{ n_1 < n_2 < \ldots < n_m \right\} \subset \mathbb{N}^+.
\end{equation}
The additional coefficient symmetry~$\hat{u}(n, k) = - \hat{u}(-n, k)$ along with
condition $\hat{u}(n,k) = \hat{u}(-n,-k)$ allows both signs of the temporal 
index to be paired. 
Consequently, each tangential mode contributes four signed combinations to the 
resonant set, rather than the two in the periodic construction: 
\begin{equation}
    \label{eqn: resonant-index-nlw-dirichlet}
    \mathcal{S} := \left\{ (\pm n_j, \pm e_j)\in \mathbb{Z}^{1+m} ~\big|~ j = 1, \ldots, m \right\}, 
\end{equation}
where the two signs are chosen independently. 
Pairing the corresponding terms in the cosine 
representation~\eqref{eqn: nlw-fourier-2}, we obtain 
\begin{equation}
    \label{eqn: nlw-fourier-dirichlet-2}
    u(x, t) = - 2 \sum_{n=1}^{\infty} \sum_{ k \in \mathbb{Z}^m} \hat{u}(n, k) \sin\left( \langle k, \omega' t + \theta_0 \rangle \right) \sin\left( nx \right), 
\end{equation}
and the corresponding linear solution is given by
\begin{equation}
    \label{eqn: nlw-fourier-dirichlet-linear}
    u(x, t) = - 4 \sum_{j=1}^{m} a_j \sin \left(\omega_j t + \theta_{0,j}\right) \sin(n_jx)
\end{equation}

We next demonstrate the performance of the alternating numerical 
scheme~\eqref{eqn: alternating-numerical-nlw}, where for simplicity we set the 
initial phase $\theta_0 = 0$.

\paragraph{Periodic solutions ($m=1$)} 
In contrast to the periodic setting, the Dirichlet boundary 
conditions~\eqref{eqn: dirichlet-bdy} impose the additional symmetry 
$\hat{u}(n, k) = - \hat{u}(-n, k)$ on the space-time Fourier coefficients, so 
that the linear solution~\eqref{eqn: lin-1-nlw} separates into a product of two 
sine functions,
\begin{equation}
    \label{eqn: lin-1-nlw-dirichlet}
    u(x, t) = -4 a_1 \sin \left(\omega_1 t \right) \sin(n_1x).
\end{equation}
Expanding the solution in a sine series in $x$, it is straightforward to show 
that the nonlinear solution takes the form
\begin{subequations}
    \label{eqn: nlw-dirichlet-1}
    \begin{empheq}[left=\empheqlbrace]{align} 
       & \hat{u}_{n_1}(t) = -4a_1 \sin(\omega_1't) - 4 \sum_{k = 2}^{\infty} \hat{u}(n_1,k) \sin(k\omega_1't),                                \label{eqn: nlw-dirichlet-1-1}                \\
       & \hat{u}_{n}(t)   = -4 \sum_{k=1}^{\infty} \hat{u}(n,k) \sin(k\omega_1't), \mathrel{\phantom{==========}} n \in \mathbb{N}^+ \setminus \{ n_1 \}.           \label{eqn: nlw-dirichlet-1-2}                
    \end{empheq}        
\end{subequations} 
Since the coefficients $\hat{u}_n(t)$ are now real, the phase-space trajectories 
are plotted in the canonical variables
\[
    \left(\hat{u}_n(t),~ \frac{\dot{\hat{u}}_{n}(t) }{\sqrt{n^2 + \mu}} \right),
\]
in which a mode oscillating at its linear frequency traces a circle. 
For the numerical test, we take $n_1 = 1$, $\mu = \sqrt{2}$, and $a_1 = 1$, so 
that $\omega_1 = \sqrt{1 + \sqrt{2}}$; we set the perterbation quantity 
$\varepsilon=0.5$; the results are shown 
in~\Cref{fig: nlw-d-1torus-traj}.
\begin{figure}[htb!]
    \centering
    \begin{subfigure}[t]{0.3\linewidth}
    \centering
    \includegraphics[scale=0.18]{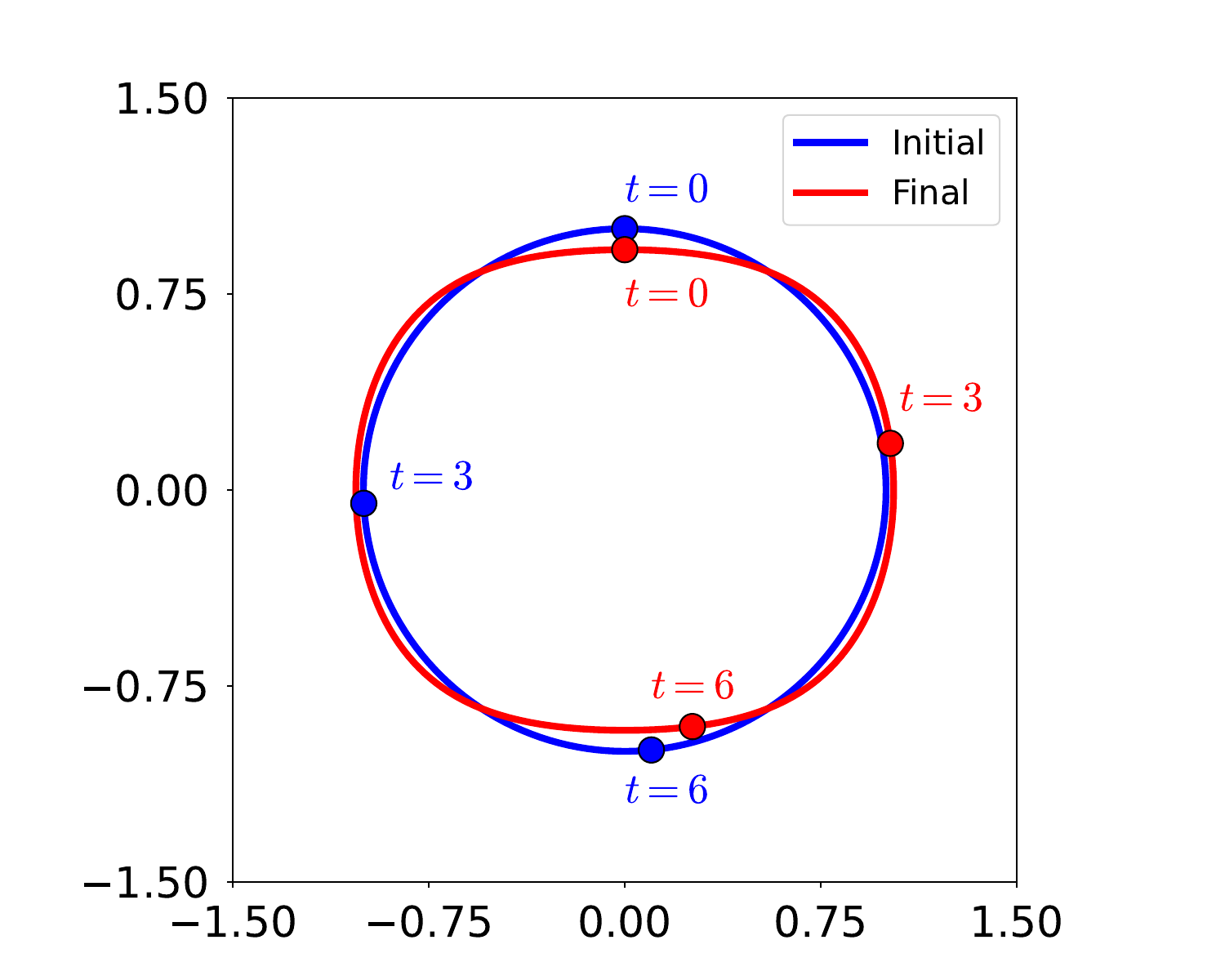}
    \caption{Tangential mode: $\hat{u}_{1}(t)$}
    \end{subfigure}
    \begin{subfigure}[t]{0.3\linewidth}
    \centering
    \includegraphics[scale=0.18]{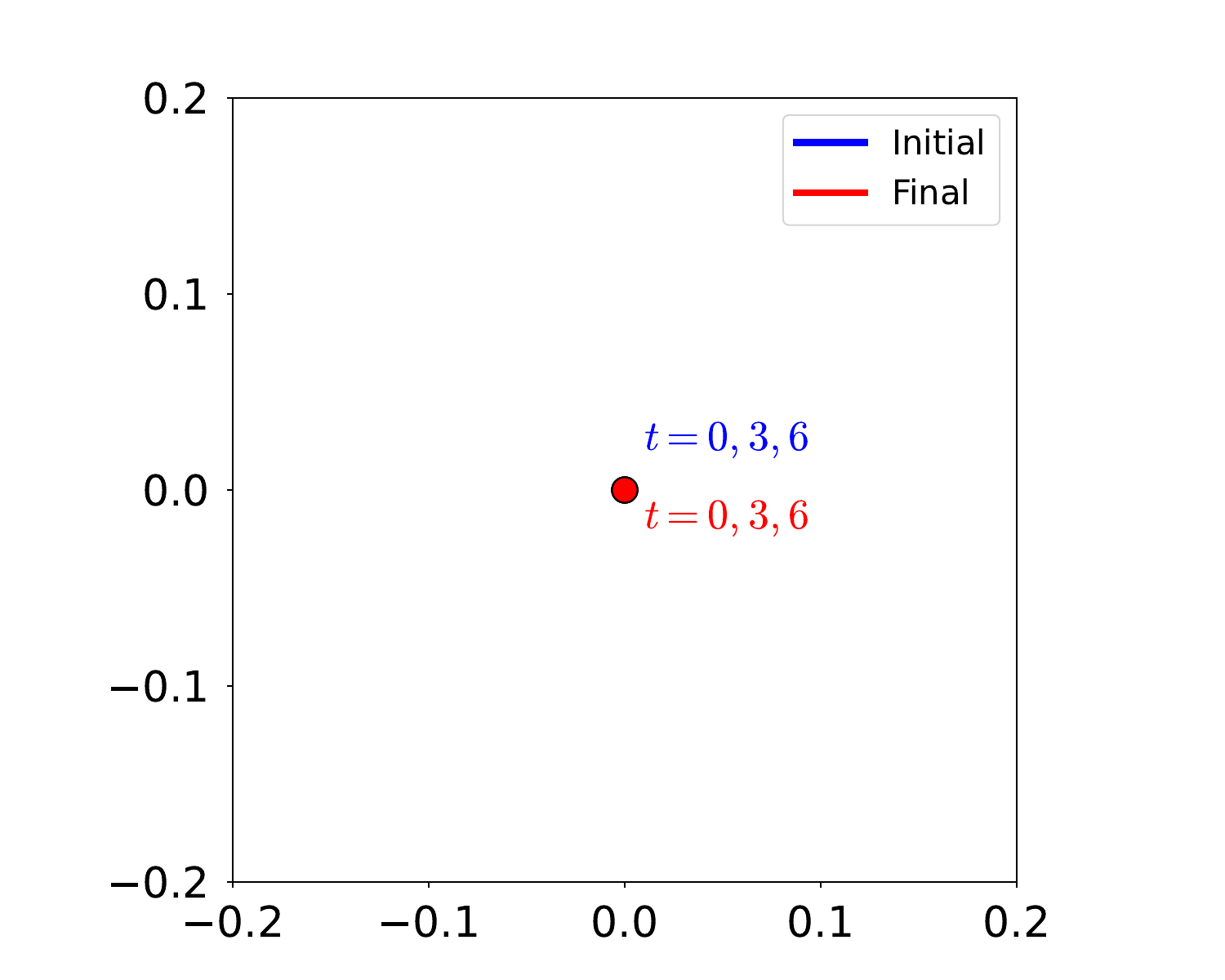}
    \caption{Normal mode: $\hat{u}_{2}(t)$}
    \end{subfigure}
    \begin{subfigure}[t]{0.3\linewidth}
    \centering
    \includegraphics[scale=0.18]{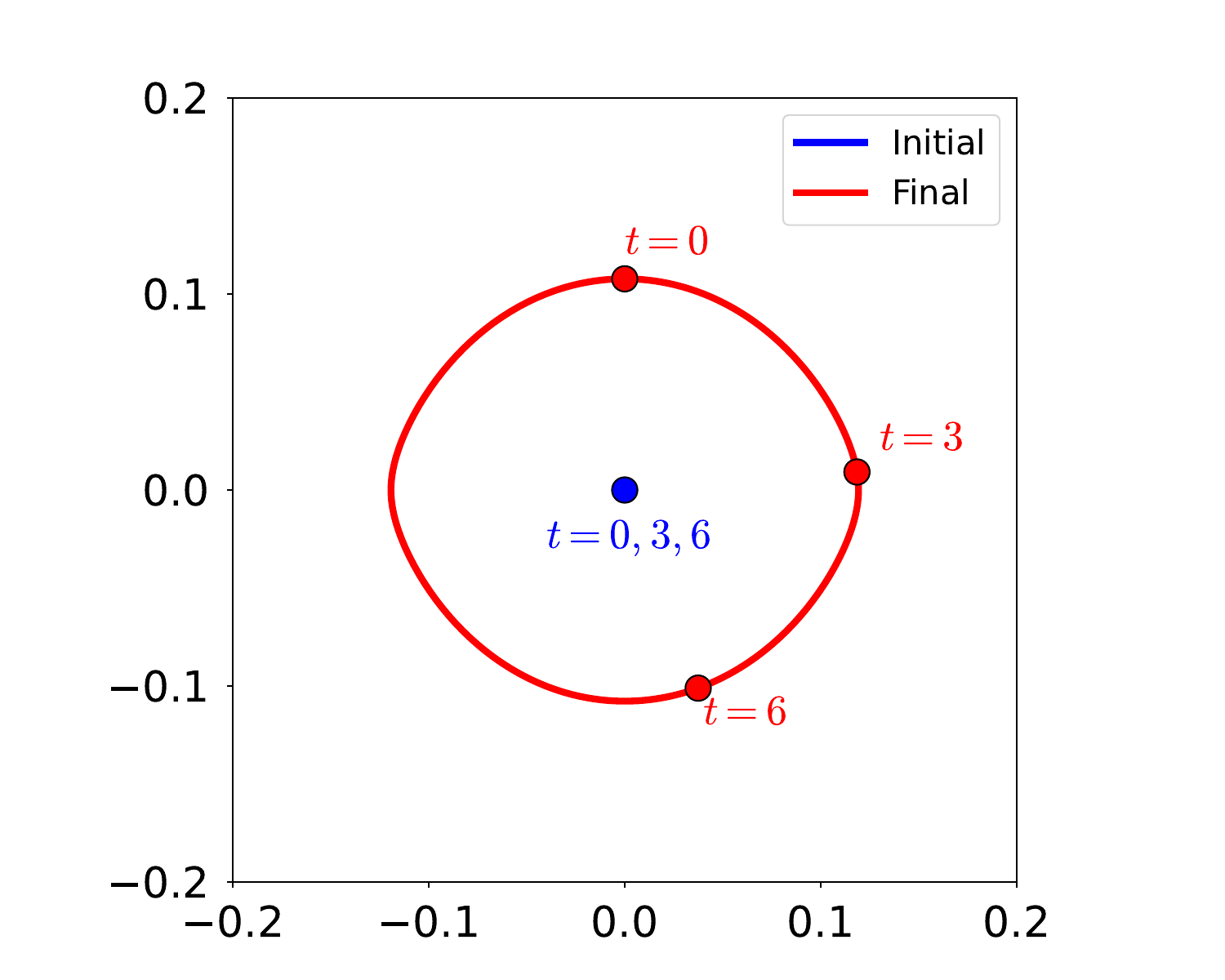}
    \caption{Normal mode: $\hat{u}_{3}(t)$}
    \end{subfigure}
    \caption{Phase space trajectories of the NLW equation~\eqref{eqn: nlw} with 
             the Dirichlet boundary conditions~\eqref{eqn: dirichlet-bdy}. 
             ``Initial'' and ``Final'' denote the initial linear solution and 
             the computed nonlinear solution, respectively. 
             Markers indicate specific time points at $t = 0$, $3$, and $6$. } 
    \label{fig: nlw-d-1torus-traj}
\end{figure}
Compared with the periodic case in~\Cref{fig: nlw-p-1torus-traj}, the 
tangential mode $\hat{u}_{1}(t)$ is slightly deformed: the trajectory of the 
nonlinear solution is a mildly flattened closed curve rather than the circle 
traced by the linear solution. 
The representative normal mode $\hat{u}_{2}(t)$ behaves exactly as in the 
periodic case and remains fixed at the origin, since the cubic nonlinearity 
couples the fundamental mode only to odd modes. 
The other representative normal mode $\hat{u}_{3}(t)$ vanishes identically 
in the linear solution, so that its ``initial" trajectory collapses to the 
origin, but it is excited by the nonlinear interaction. 
Its nonlinear trajectory is a small closed loop that is visibly flattened in 
the vertical direction, in contrast to the shape observed in the periodic case. 
The convergence profile is shown in~\Cref{fig: nlw-d-1torus-con}. 
\begin{figure}[htb!]
    \centering
    \begin{subfigure}[t]{0.3\linewidth}
    \centering
    \includegraphics[scale=0.18]{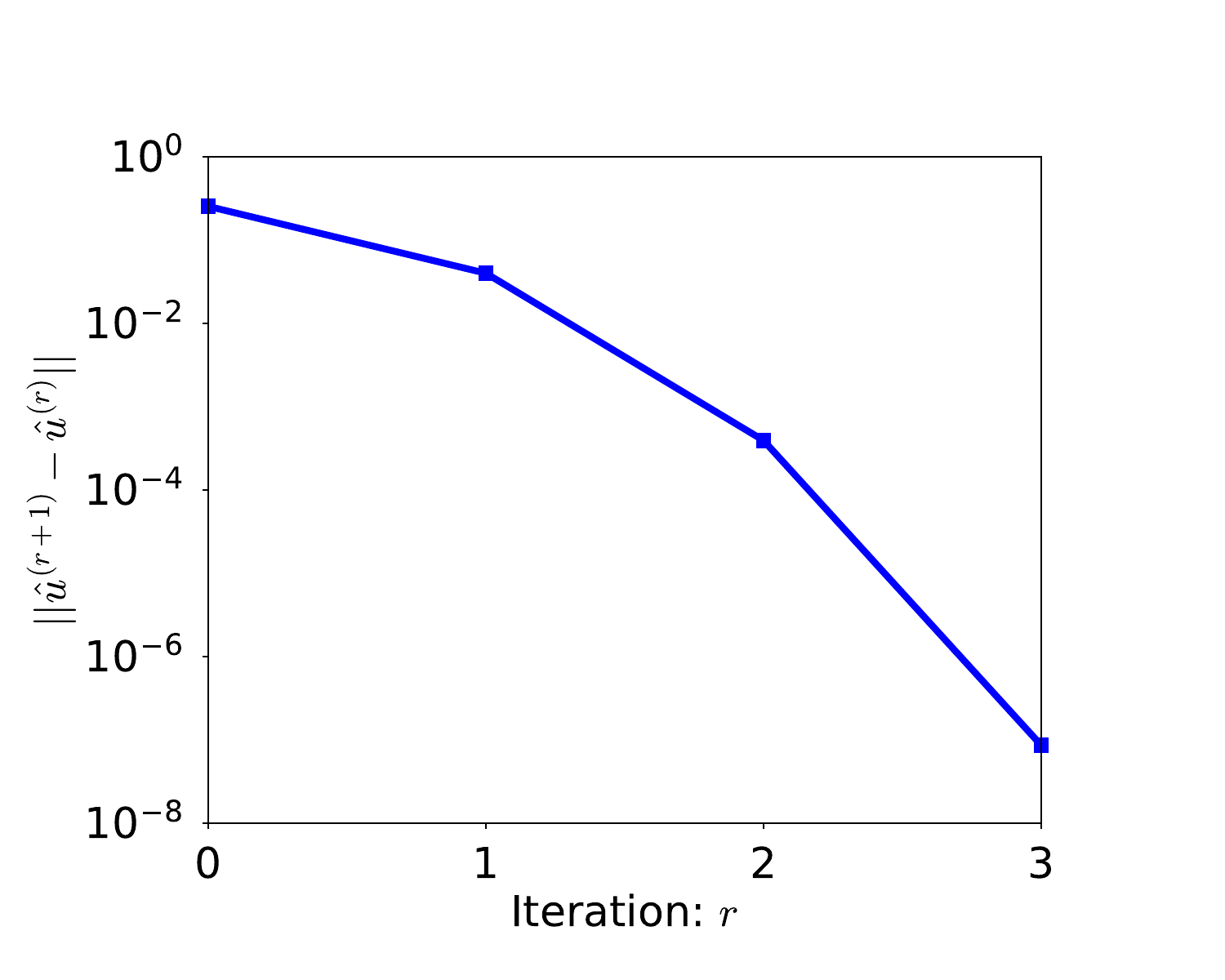}
    \caption{$\| \hat{u}_p^{(r+1)} - \hat{u}_p^{(r)} \|$}
    \end{subfigure}
    \begin{subfigure}[t]{0.3\linewidth}
    \centering
    \includegraphics[scale=0.18]{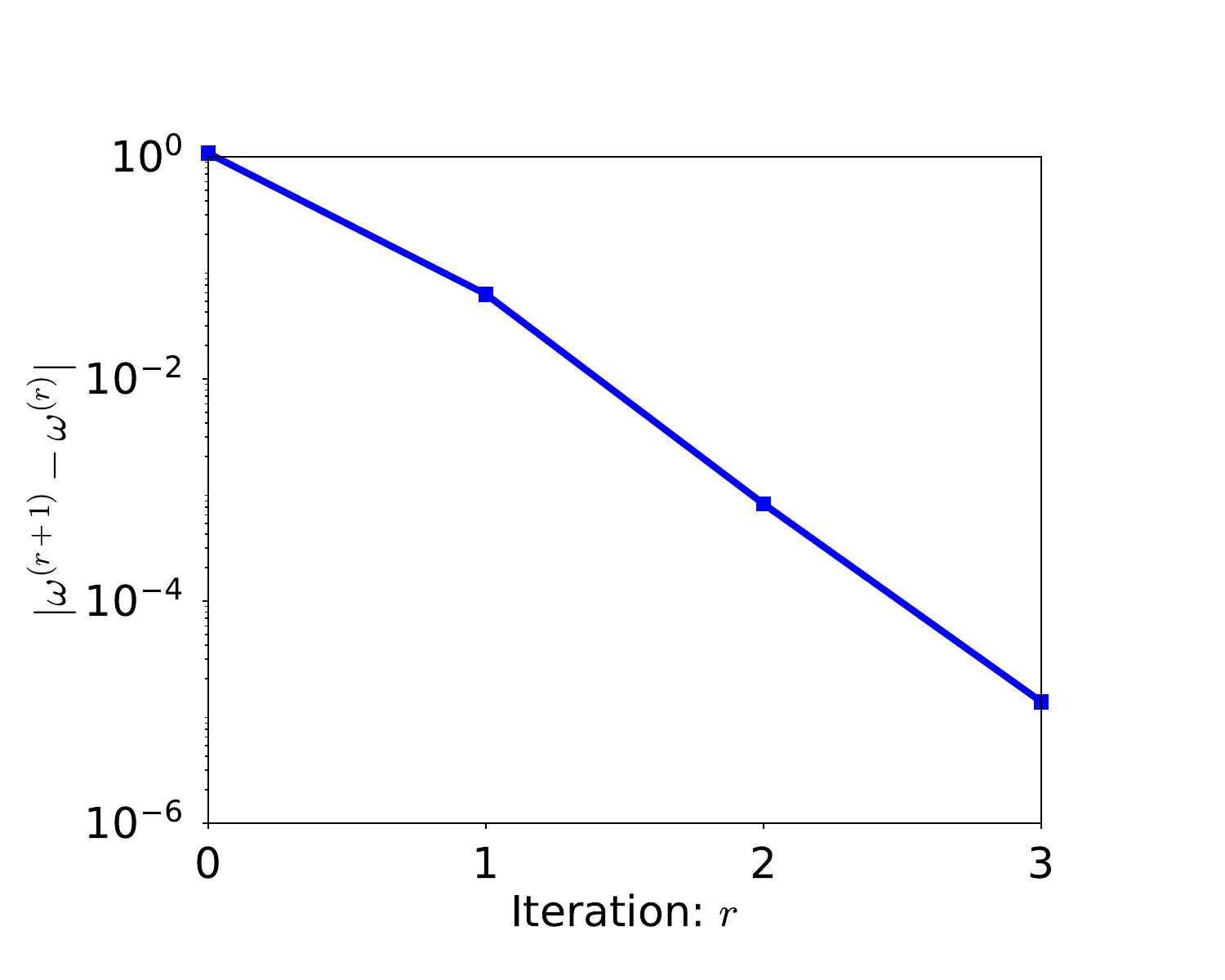}
    \caption{$| \omega^{(r+1)} - \omega^{(r)} |$}
    \end{subfigure}
    \begin{subfigure}[t]{0.3\linewidth}
    \centering
    \includegraphics[scale=0.18]{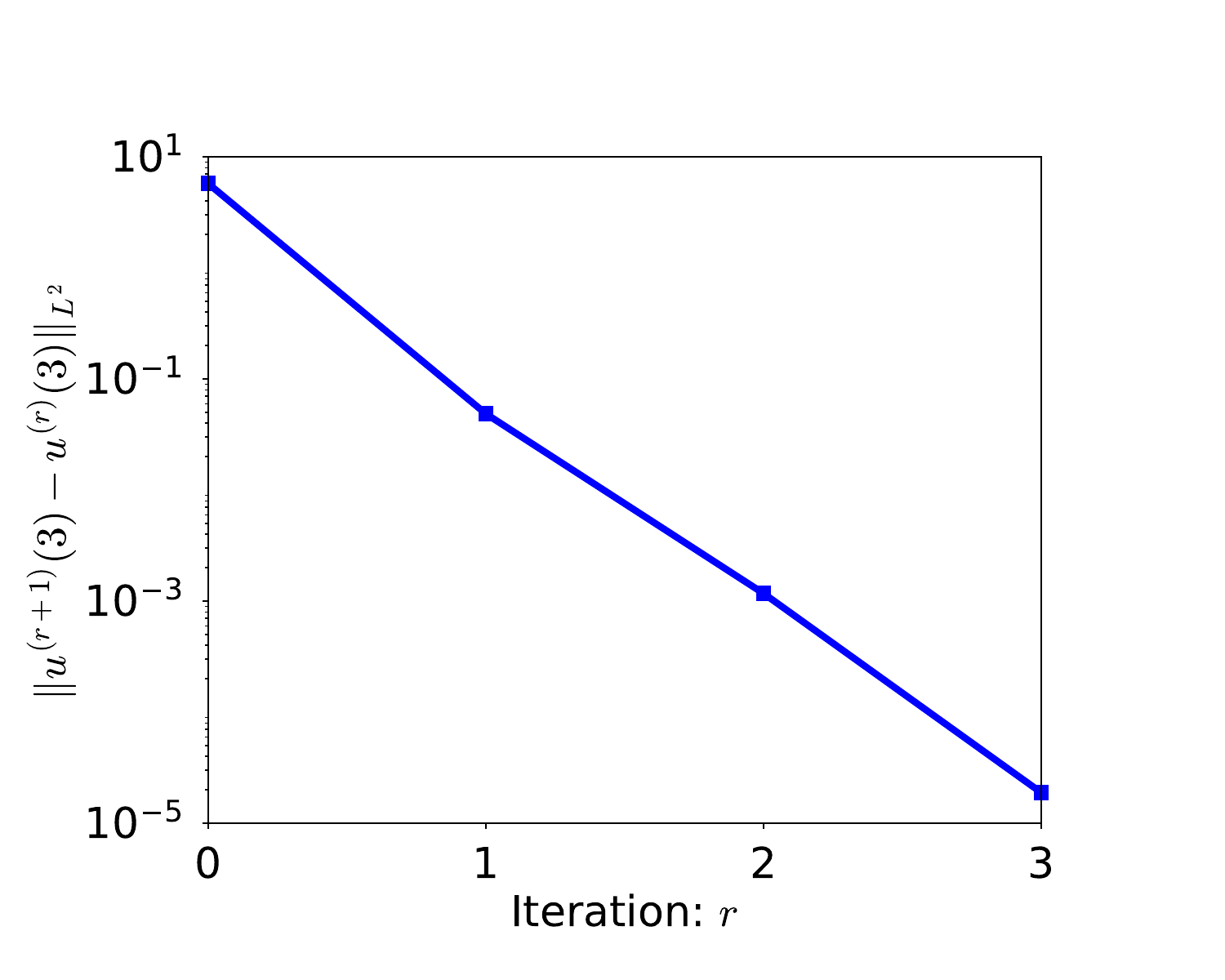}
    \caption{$\| u^{(r+1)}(x, 3) - u^{(r)}(x, 3)  \|_{L^2}$}
    \end{subfigure}
    \caption{Convergence profiles across iterations for the NLW 
             equation~\eqref{eqn: nlw} with the Dirichlet boundary 
             condition~\eqref{eqn: dirichlet-bdy}. } 
    \label{fig: nlw-d-1torus-con}
\end{figure}

We next examine the spatiotemporal distributions of $|u(x, t)|^2$. 
Taking the squared modulus of the linear 
solution~\eqref{eqn: lin-1-nlw-dirichlet} yields
\[
    |u|^2 = 16\sin^2 \left( \sqrt{1 + \sqrt{2}}t \right) \sin^2 x, 
\]
which is a standard standing wave: the nodes are fixed at $x=0$ and $x=\pi$, 
the antinode is fixed at $x=\pi/2$, as shown 
in~\Cref{fig: nlw-d-1torus-heat-slice-lin}.
\begin{figure}[htb!]
    \centering
    \begin{subfigure}[t]{0.42\linewidth}
    \centering
    \includegraphics[scale=0.24]{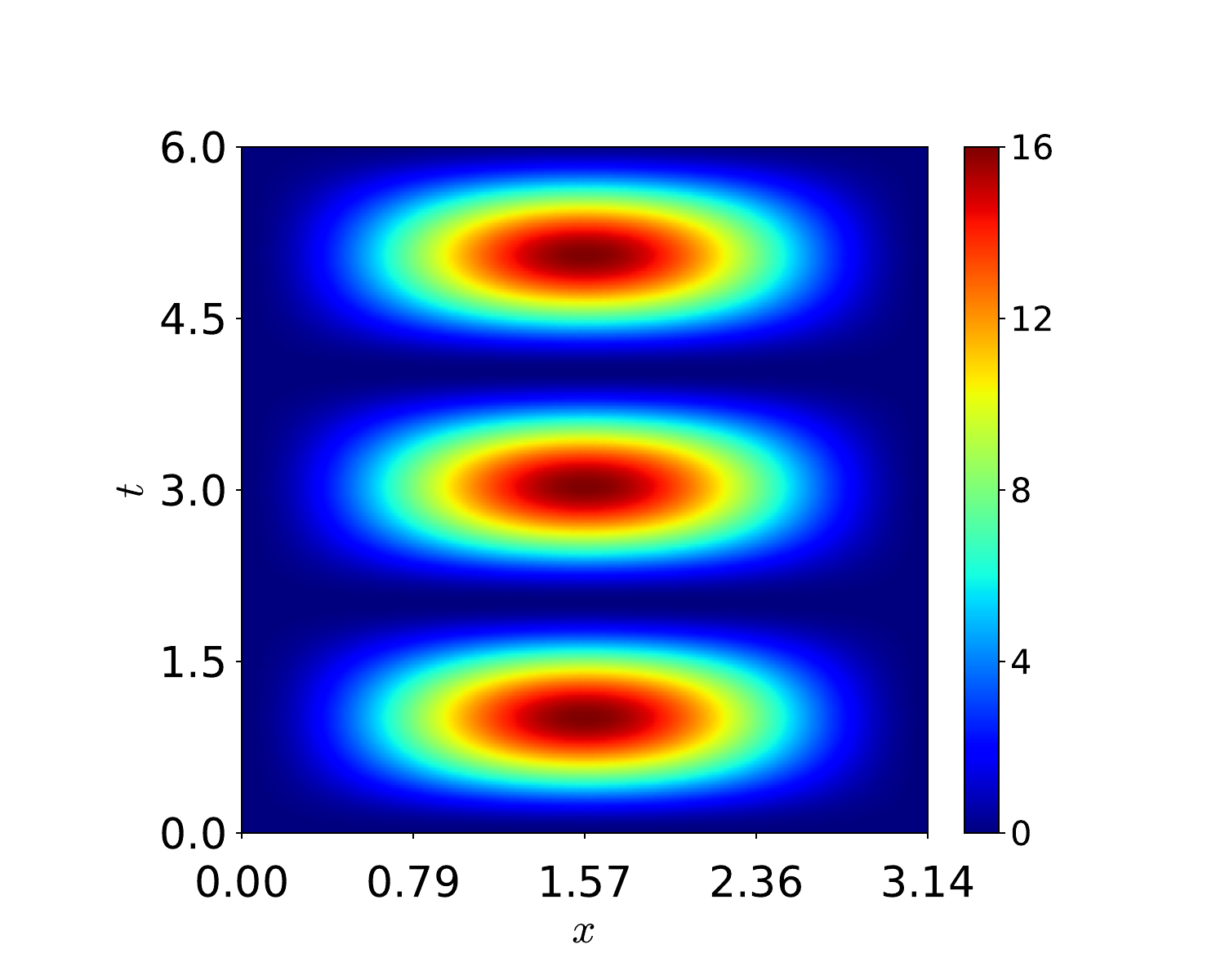}
    \caption{Linear solution}
    \label{fig: nlw-d-1torus-heat-slice-lin}
    \end{subfigure}
    \begin{subfigure}[t]{0.42\linewidth}
    \centering
    \includegraphics[scale=0.24]{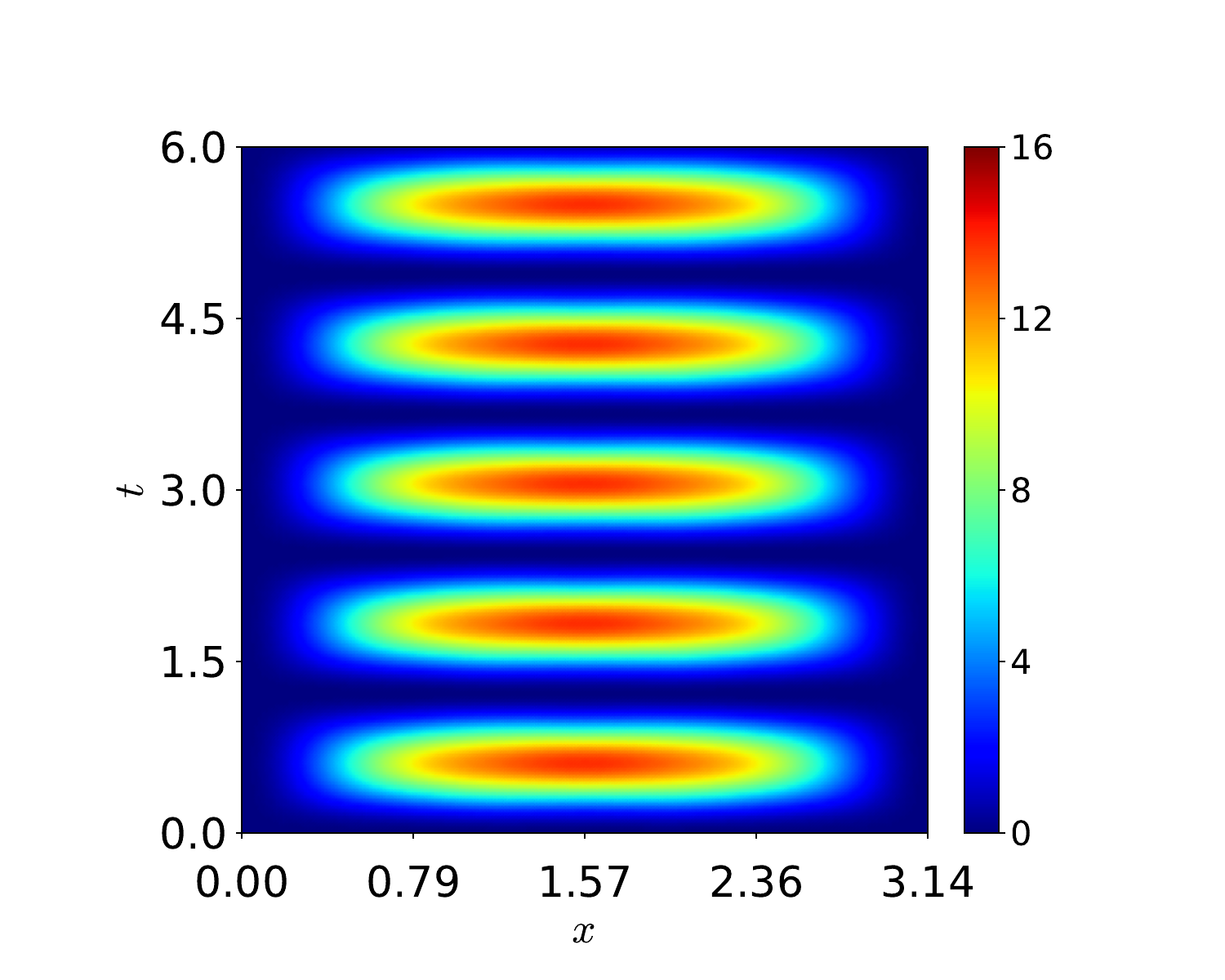}
   \caption{Nonlinear solution}
   \label{fig: nlw-d-1torus-heat-slice-nonlin}
    \end{subfigure}
    \caption{Spatiotemporal distributions of $|u(x,t)|^2$ for the NLW 
             equation~\eqref{eqn: nlw} with the Dirichlet boundary 
             conditions~\eqref{eqn: dirichlet-bdy}. 
             The linear solution is~\eqref{eqn: lin-1-nlw-dirichlet}, and the 
             nonlinear solution is obtained via the alternating numerical 
             scheme~\eqref{eqn: alternating-numerical-nlw}.}
    \label{fig: nlw-d-1torus-heat-slice}
\end{figure}
Once the nonlinear term is included, the pattern is modified in two respects, 
as shown in~\Cref{fig: nlw-d-1torus-heat-slice-nonlin}. 
In space, the profile becomes broader and its peak value is lowered, since the 
higher spatial harmonics generated by the cubic nonlinearity redistribute 
energy away from the fundamental mode $\sin x$. 
In time, the nonlinear solution completes more oscillations than the linear one 
over the same time window, reflecting the upward shift of the frequency. 
The solution remains a standing wave and no net propagation occurs; the 
nonlinearity alters the shape and the period of the pattern, but not its 
character.

\paragraph{Quasi-periodic solutions ($m=2$)} 
For $m=2$, the linear quasi-periodic solution under the Dirichlet boundary 
condition~\eqref{eqn: dirichlet-bdy} is a superposition of two standing waves, 
which reads
\begin{equation}
    \label{eqn: lin-2-nlw-dirichlet}
    u(x, t) = -4a_1 \sin(\omega_1t) \sin(n_1 x) -4a_2 \sin(\omega_2t) \sin(n_2 x). 
\end{equation}
The corresponding nonlinear solution can be written in terms of its spatial 
Fourier coefficients as
\begin{subequations}
    \label{eqn: nlw-dirichlet-2}
    \begin{empheq}[left=\empheqlbrace]{align} 
       & \hat{u}_{n_1}(t) = -4a_1 \sin(\omega_1't) - 4 \sum_{k \in \mathbb{Z}^2_+ \setminus \{ (1, 0) \}} \hat{u}(n_1,k) \sin(\langle k, \omega' \rangle t),              \label{eqn: nlw-dirichlet-2-1}       \\
       & \hat{u}_{n_2}(t) = -4a_2 \sin(\omega_2't) - 4 \sum_{k \in \mathbb{Z}^2_+ \setminus \{ (0, 1) \}} \hat{u}(n_2,k) \sin(\langle k, \omega' \rangle t),              \label{eqn: nlw-dirichlet-2-2}       \\       
       & \hat{u}_{n}(t)   = -4 \sum_{k \in \mathbb{Z}^2_+} \hat{u}(n,k) \sin(\langle k, \omega' \rangle t), \mathrel{\phantom{==========}} n \in \mathbb{N}^+ \setminus \{ n_1, n_2\}.          \label{eqn: nlw-dirichlet-2-3}                
    \end{empheq}        
\end{subequations} 
where $\omega' = (\omega_1', \omega_2')$ denotes the shifted frequency vector. 
For the numerical test, we take $n_1 = 1$, $n_2 =2$, $\mu=\sqrt{2}$, and 
$a_1=a_2 =1$, so that $\omega_1=\sqrt{1+\sqrt{2}}$ and 
$\omega_2 = \sqrt{4+\sqrt{2}}$ and the perturbation quantity $\varepsilon=0.2$.~\Cref{fig: nlw-d-2tori-traj} compares the 
phase-space trajectories of the initial linear and computed nonlinear solutions. 
\begin{figure}[htb!]
    \centering
    \begin{subfigure}[t]{0.3\linewidth}
    \centering
    \includegraphics[scale=0.18]{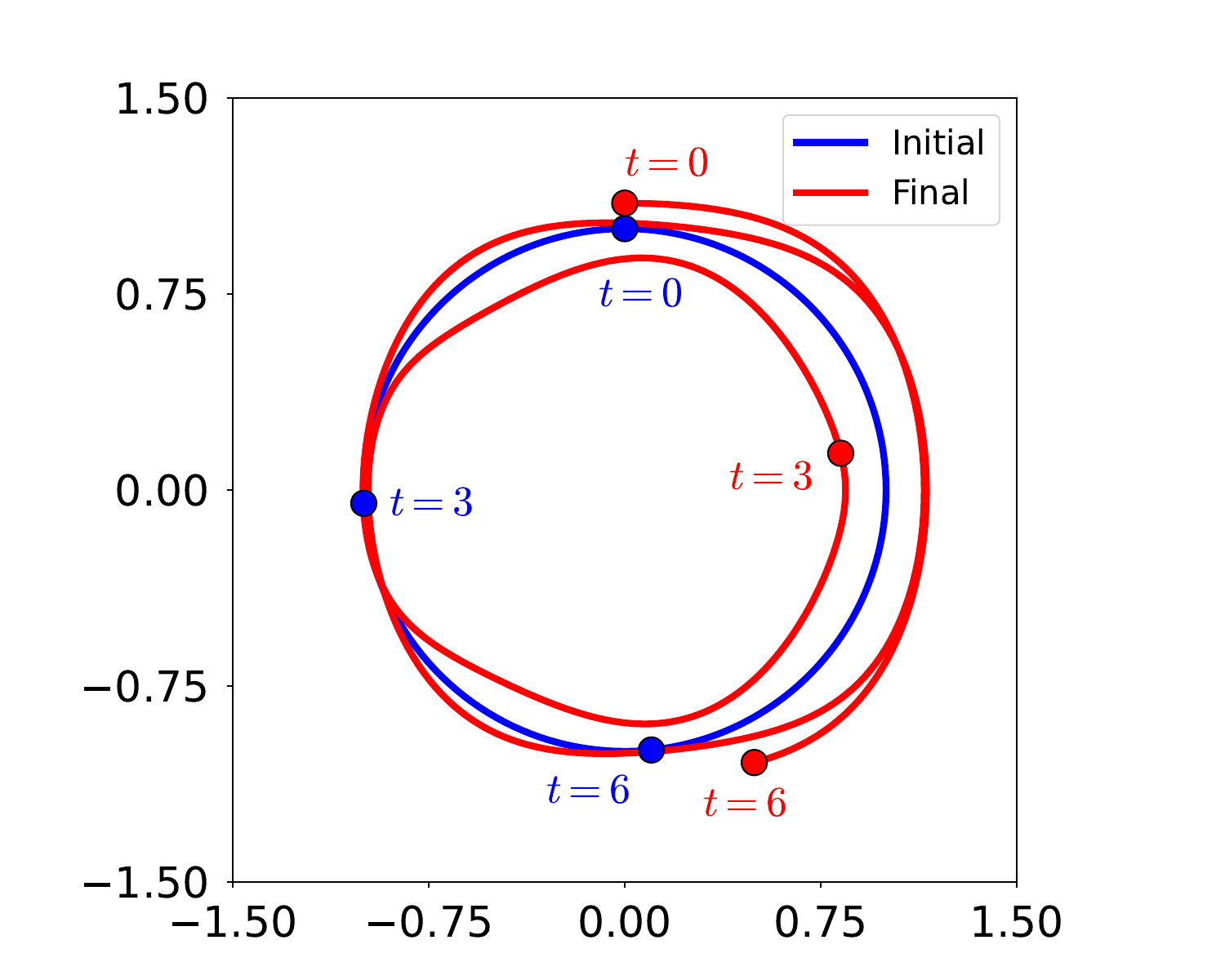}
    \caption{Tangential mode: $\hat{u}_{1}(t)$}
    \end{subfigure}
    \begin{subfigure}[t]{0.3\linewidth}
    \centering
    \includegraphics[scale=0.18]{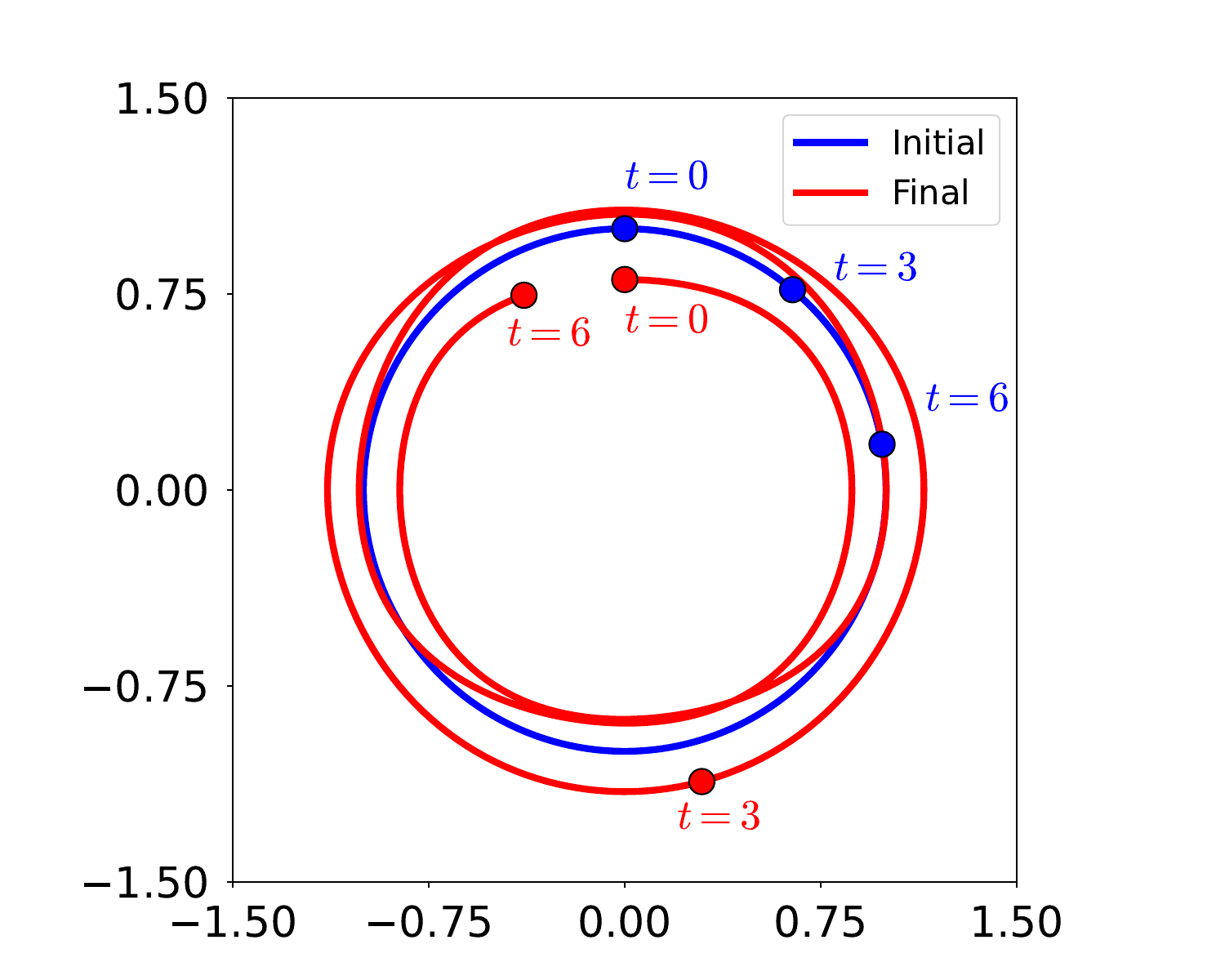}
    \caption{Tangential mode: $\hat{u}_{2}(t)$}
    \end{subfigure}
    \begin{subfigure}[t]{0.3\linewidth}
    \centering
    \includegraphics[scale=0.18]{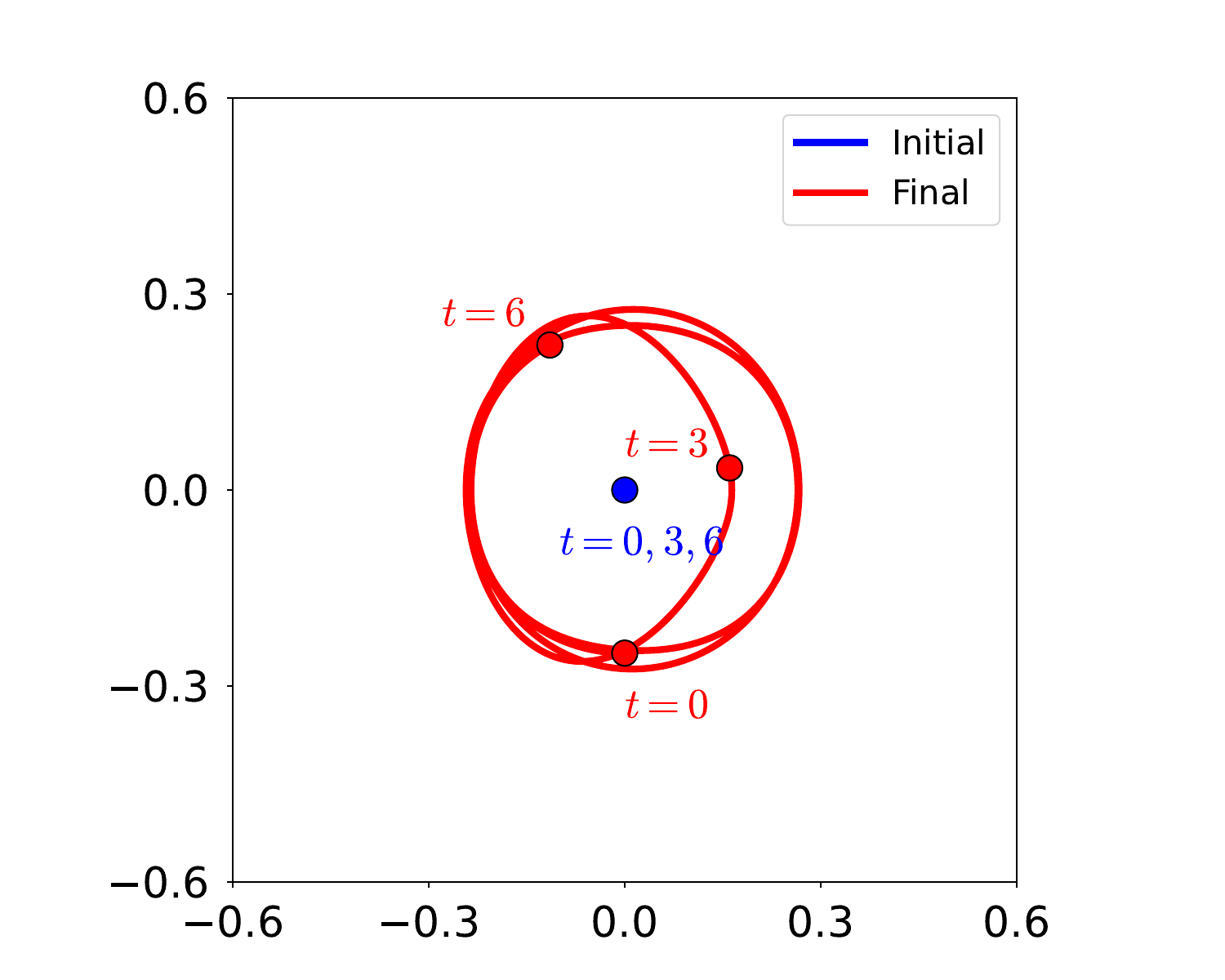}
    \caption{Normal mode: $\hat{u}_{3}(t)$}
    \end{subfigure}
    \caption{Phase space trajectories of the NLW equation~\eqref{eqn: nlw} with the Dirichlet boundary conditions~\eqref{eqn: dirichlet-bdy}. ``Initial'' and ``Final'' denote the initial linear solution and the computed nonlinear solution, respectively. Markers indicate specific time points at $t = 0$, $3$, and $6$.  } 
    \label{fig: nlw-d-2tori-traj}
\end{figure}
As in the periodic-solution case shown in~\Cref{fig: nlw-d-1torus-traj}, the frequency 
shifts produce visible differences between the two sets of trajectories. 
In addition, the quasi-periodic nature of the nonlinear solution is now 
evident: the trajectories return close to previously visited positions without 
exactly retracing their paths. 
\begin{figure}[htb!]
    \centering
    \begin{subfigure}[t]{0.42\linewidth}
    \centering
    \includegraphics[scale=0.24]{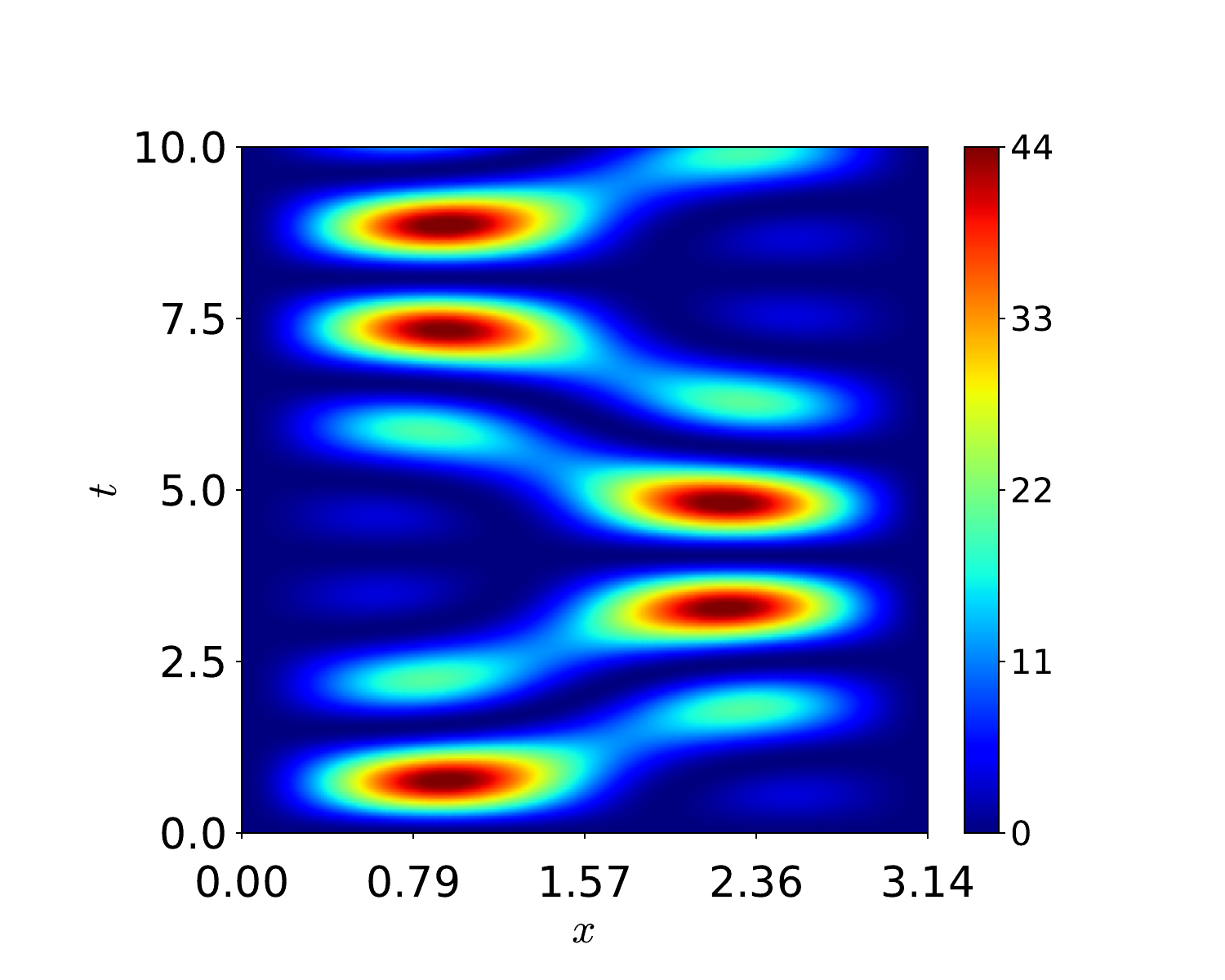}
     \caption{Linear solution}
    \label{fig: nlw-d-2tori-heat-slice-lin}     
    \end{subfigure}
    \begin{subfigure}[t]{0.42\linewidth}
    \centering
    \includegraphics[scale=0.24]{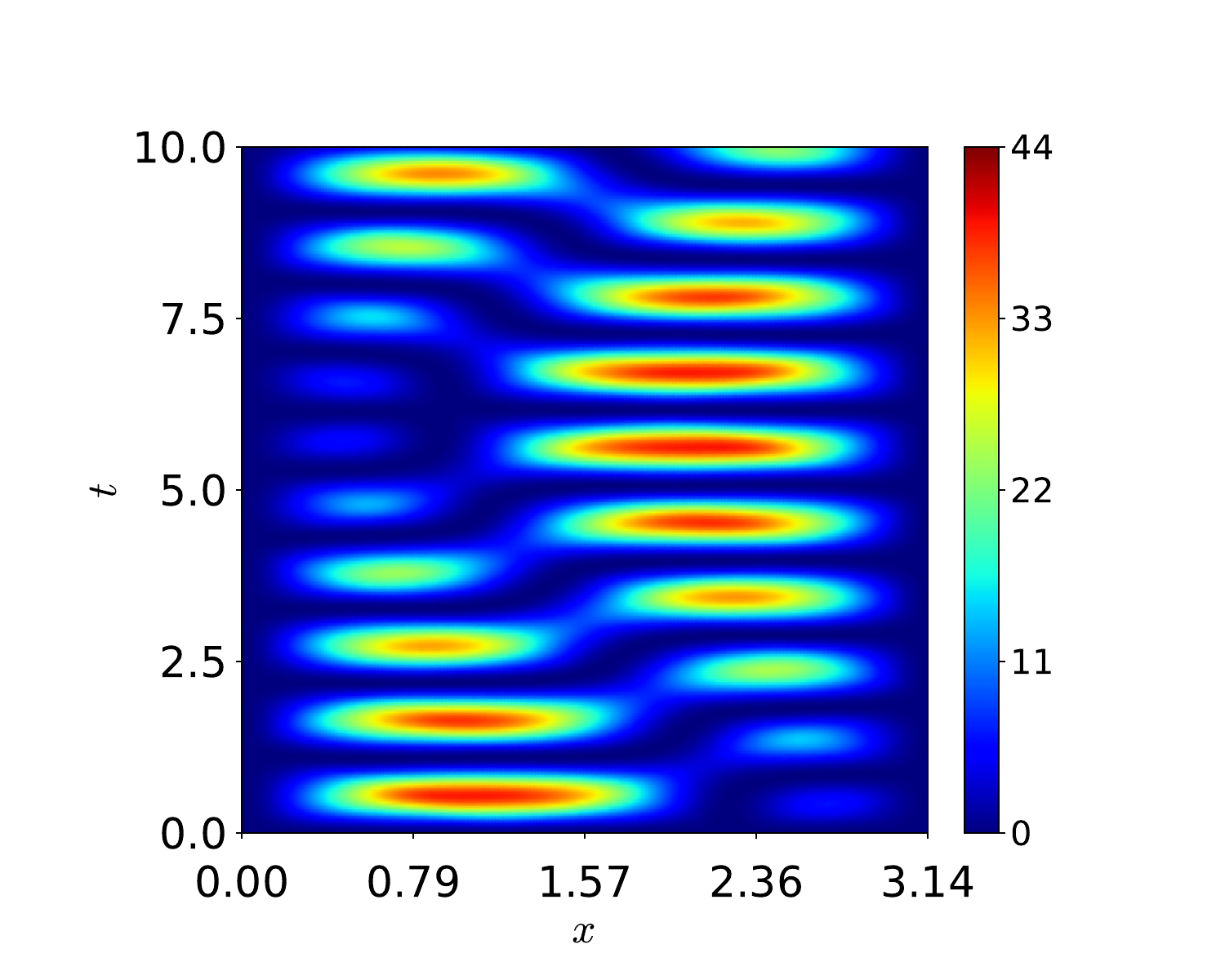}
    \caption{Nonlinear solution}
    \label{fig: nlw-d-2tori-heat-slice-nonlin}         
    \end{subfigure}
    \caption{Spatiotemporal distributions of $|u(x,t)|^2$ for the NLW 
             equation~\eqref{eqn: nlw} with the Dirichlet boundary 
             conditions~\eqref{eqn: dirichlet-bdy}. 
             The linear solution is~\eqref{eqn: lin-2-nlw-dirichlet}, and the 
             nonlinear solution is obtained via the alternating numerical 
             scheme~\eqref{eqn: alternating-numerical-nlw}.}
    \label{fig: nlw-d-2tori-heat-slice}
\end{figure}

Turning to the spatiotemporal distributions, the squared modulus of the linear 
solution~\eqref{eqn: lin-2-nlw-dirichlet} with the parameters specified above 
reads
\begin{align}
    |u|^2 = &16 \sin^2\left( \sqrt{1 + \sqrt{2}} t \right) \sin^2(x) + 16 \sin^2\left(\sqrt{4+\sqrt{2}} t \right) \sin^2(2 x)    \nonumber \\
            & + 32 \sin \left( \sqrt{1 + \sqrt{2}} t \right) \sin\left(\sqrt{4+\sqrt{2}} t \right)\sin(x) \sin(2 x),             \label{eqn: nlw-dirichlet-2-modulus}
\end{align}
where the first two terms are the intensities of the individual standing waves 
and the cross term couples them, producing a pattern shown 
in~\Cref{fig: nlw-d-2tori-heat-slice-lin}, which is considerably richer than 
in the periodic case. 
Since the cross term in~\eqref{eqn: nlw-dirichlet-2-modulus} changes sign at 
$\pi/2$, the intensity is no longer symmetric about the center: the intensity 
sways between the left and right halves of the domain, 
appearing in~\Cref{fig: nlw-d-2tori-heat-slice-lin} as bright spots (maxima) 
alternating. 
The alternation is governed by the slow difference frequency, while the fast 
modulation within each group is the sum frequency. 
Owing to the rational independence of frequencies, the supremum is approached 
but never attained. 
The nonlinear solution in~\Cref{fig: nlw-d-2tori-heat-slice-nonlin} preserves 
this qualitative structure but with clear quantitative differences: the bright 
spots are more closely spaced in time, reflecting the upward frequency shifts, 
and the left-half maxima move toward the center of the domain due to the energy 
transferred to higher spatial harmonics through the normal 
modes~\eqref{eqn: nlw-dirichlet-2-3}. 
As in the periodic-solution case, no net propagation occurs; the solution remains a 
quasi-periodic superposition of standing waves.

\section{Conclusion and further work}
\label{sec: conclu}

In the classical KAM theory for 
PDEs~\citep{kuksin1993nearly, kuksin1996invariant}, an arbitrary set of $m$ 
spatial modes can be taken as the tangential modes. 
To overcome the small-divisor problem, one excludes a set of nearly-resonant 
parameters of small total measure, and all resonances must be excluded 
simultaneously. 
In contrast,~\citet{bourgain1998quasi} takes the $m$ low-frequency spatial modes, 
that is, the large-scale modes in the physical sense, as the tangential modes, 
and uses the CWB method to construct the quasi-periodic solution by an inductive 
multiscale argument. 
The choice of low-frequency tangential modes is essential here: if a 
high-frequency spatial mode were included among them, the structure of the 
problem would no longer be preserved from one scale to the next, and the 
argument would cease to be inductive. 
Precisely for this reason,~\citet{bourgain1998quasi} is able to treat the 
nearly-resonant set locally rather than globally: within each box of scale $N$, 
only the nearly-resonant sites inside that box need to be handled. 
This local excision of the nearly-resonant set, combined with the separation of 
the spatial modes, makes it possible to handle the multiplicity of the 
eigenvalues $|n|^2$. 
For the two-dimensional torus,~\citet{bourgain1998quasi} showed that the 
eigenvalues $|n|^2$ form well-separated clusters, based on the key observation 
that the number of lattice points on a circle grows more slowly than any power 
of the radius. 
In~\citet{bourgain2005green}, the well-separation property was established in 
arbitrary dimension, thereby extending the construction of quasi-periodic 
solutions to the higher-dimensional case.

The CWB approach is also the more meaningful one from a physical point of view. 
Physical observations are concerned mainly with large-scale behavior, which 
corresponds precisely to the low-frequency spatial modes. 
Moreover, the existence of quasi-periodic solutions describes a persistent 
failure of energy to accumulate from small scales to large scales, analogous to 
the arrest of the inverse cascade in geostrophic 
turbulence~\citep{rhines1979geostrophic}. 
As regards efficient computation, in the present infinite-dimensional setting, 
as well as in the finite-dimensional one treated 
in~\citet{Fu2026numerical-a, Fu2026numerical-b}, the storage requirement for 
the Jacobian matrix in the Newton iteration becomes very large, since its 
dimension is enlarged at each step. Most of its entries, however, are very 
small, and it is natural to ask whether they can be set to zero so that the 
matrix is stored in sparse form. 
Exploiting such sparsity is a very interesting direction for further work.

\section*{Acknowledgements}
This work was partially supported by the NSFC (Grant No. 12241105) and by SIMIS (startup fund and cross-disciplinary research projects).

\bibliographystyle{abbrvnat}
\bibliography{sigproc}

\appendix
\section{Proofs of results in~\Cref{sec: basic-setting}}
\label{sec: priori-error-app}

In this appendix, we provide detailed proofs for the results 
in~\Cref{sec: basic-setting}, especially 
for~\Cref{prop: vector-field} ---~\Cref{prop: B-operator} 
(see~\Cref{subsec: vector-field} ---~\Cref{subsec: b-operator}). 
To facilitate these proofs, we first establish an elementary inequality and show 
that the Gevrey decay property is preserved under convolution. 
These preliminary results are rigorously stated in the following lemmas.

\subsection{Preliminary lemmas}
\label{subsec: preliminary}

\begin{lemma}
    \label{lem: app-1}
    Let $a, b \ge 0$ and $0 < s < 1$. Then the following inequality holds:
    \begin{equation}
        \label{eqn: ineq_power}
        a^s + b^s - (a+b)^s \ge (2-2^s) \min\{a,b\}^s. 
    \end{equation}
\end{lemma}

\begin{proof}
    Without loss of generality, we assume that $a \ge b$. 
    The case $b=0$ is immediate. 
    For $b > 0$, set $x = a/b \geq 1$ and define $f(x) = x^s + 1 - (1+x)^s$. 
    Since $0 < s <1$, its derivative satisfies $f'(x) = s(x^{s-1} - (x+1)^{s-1}) > 0$ 
    for all $x \geq 1$. 
    Thus, $f$ is strictly increasing on $[1, \infty)$, and hence 
    $f(x) \geq f(1)= 2 - 2^{s}$. 
    Multiplying this inequality by $b^s$ yields~\eqref{eqn: ineq_power}, so the 
    proof is complete. 
\end{proof}

We next use~\Cref{lem: app-1} to show that convolution preserves the Gevrey 
decay up to a multiplicative constant.

\begin{lemma}
    \label{lem: app-2}
    Let $\mathcal{K}(s)$ be the Gevrey decay set defined 
    in~\eqref{eqn: gevrey-decay-set}, with $0 < s < 1$. 
    There exists a constant $C_2(m, s) > 0$ such that, for any 
    $\hat{a}, \hat{b} \in \mathcal{K}(s)$, their convolution satisfies
    \begin{equation}    
        \label{eqn: convolution}   
        \sup_{(n, k) \in \mathbb{Z}^{1+m}} \left( |(\hat{a}*\hat{b})(n, k)| e^{ (|n| + |k|)^{s} } \right) \leq C_2(m,s).
    \end{equation}
    More generally, for any integer $\ell \geq 2$ and 
    $\hat{a}_j \in \mathcal{K}(s)$ for $j=1, \ldots, \ell$, the $\ell$-fold 
    convolution satisfies: 
    \begin{equation}    
        \label{eqn: p-convolution}   
        \sup_{(n, k) \in \mathbb{Z}^{1+m}} \left( \big| \big(\hat{a}_1 * \hat{a}_2 * \cdots * \hat{a}_{\ell} \big)(n, k) \big| e^{ (|n| + |k|)^{s} } \right) \leq C_{\ell}(m,s),
    \end{equation}   
     where $C_{\ell}(m, s) > 0$ depends only on $\ell$, $m$, and $s$.  
\end{lemma}

\begin{proof}
    Since $\hat{a}, \hat{b} \in \mathcal{K}(s)$, their convolution satisfies
    \begin{align*}
        |(\hat{a}*\hat{b})(n, k)| & \leq \sum_{(n', k') \in \mathbb{Z}^{1+m}} |\hat{a}(n-n', k-k')| |\hat{b}(n', k')|                    \\
        & \leq \sum_{(n', k') \in \mathbb{Z}^{1+m}} e^{- (|n-n'| + |k-k'|)^{s}  } e^{- (|n'| + |k'|)^{s} }                               \\
        & \leq \sum_{(n',k') \in \mathbb{Z}^{1+m}} e^{- (| n | + |k|)^{s} } e^{ - ( |n-n'| + |k-k'|)^{s}  -  (|n'| + |k'|)^{s} + ( |n-n'| + |n'| + |k-k'| + |k'| )^{s} },                                                                                                                                                                   
    \end{align*}
    where the last inequality follows from the triangle inequality. 
    Applying~\Cref{lem: app-1} to bound the sum involving the remaining 
    exponential terms, we obtain
    \[
        |(\hat{a}*\hat{b})(n, k)| \leq e^{- (|n| + |k|)^{s} } \left( \sum_{(n', k') \in \mathbb{Z}^{1+m}} e^{ - (2 - 2^{s}) \min\left\{ (|n-n'| + |k-k'|)^{s}, \; (|n'| + |k'|)^{s} \right\} } \right).     
    \]
    Since the sum of the Gevrey decay series converges in a $(1+m)$-dimensional 
    lattice $\mathbb{Z}^{1+m}$, we establish the uniform 
    bound~\eqref{eqn: convolution}. 
    By induction, this same argument extends to any finite $\ell$-fold 
    convolutional product, thereby establishing the 
    inequality~\eqref{eqn: p-convolution}.
\end{proof}

\subsection{Proof of~\Cref{prop: vector-field}}
\label{subsec: vector-field}

Since $\hat{u} \in \mathcal{K}(s)$, applying~\Cref{lem: app-2} to the 
convolution representation of the cubic perturbation $|u|^2u$ 
yields uniform Gevrey decay estimates for $\hat{X}(n, k)$. 
By the Leibniz rule, each frequency derivative 
$\partial_{\omega_{\ell}} \hat{X}(n, k)$, $\ell = 1, \ldots, m$, is a sum of 
convolutions involving the Fourier coefficients $\hat{u}(n, k)$ and their 
frequency derivatives $\partial_{\omega'_{\ell}} \hat{u}(n, k)$. 
Applying the same convolution estimate as in~\Cref{lem: app-2} yields the same Gevrey decay 
estimates for these derivatives, establishing~\eqref{eqn: vector-field-comp}. 
Finally, summing the squares of the coefficientwise bounds over 
$(n, k) \in \mathbb{Z}^{1+m}$ gives the uniform norm 
bound~\eqref{eqn: vector-field-bound}, since the corresponding Gevrey decay 
series converges. 
This completes the proof.

\subsection{Proof of~\Cref{prop: convolution}}
\label{subsec: convolution}

Let $\mathcal{H} = \partial \hat{X} / \partial \hat{u}$ denote the full tangent 
operator. 
For any $(n, k), (n', k') \in \mathbb{Z}^{1+m}$, its entries are given by
\[   
    \mathcal{H} (n - n', k- k') = \frac{4}{(2\pi)^{1+m}} \iint_{\mathbb{T}^{1+m}} |u|^2 \cos \left[ \langle k - k', \theta \rangle + (n - n') x \right] d\theta dx. 
\]
Since the kernel is even, $\mathcal{H}$ is symmetric, and so is its restriction 
$S$ to the non-resonant indices. 
As in the proof of~\Cref{prop: vector-field},~\Cref{lem: app-2} and the Leibniz 
rule yield uniform Gevrey decay estimates for the kernel and its frequency 
derivatives over the full lattice $\mathbb{Z}^{1+m}$. 
Restricting these estimates to the non-resonant indices 
gives~\eqref{eqn: tangent-comp}.

For any $\hat{u} \in \mathcal{K}(s)$, the convolution representation of 
$\mathcal{H}$ and the Cauchy–Schwarz inequality yield
\begin{align}
    \| \mathcal{H} \hat{u} \|^2 &= \sum_{(n,k) \in \mathbb{Z}^{1+m}} \bigg\| \sum_{(n',k') \in \mathbb{Z}^{1+m}}\mathcal{H}(n', k') \hat{u}(n-n', k-k') \bigg\|^2                                                                  \nonumber \\
    & \leq \sum_{(n,k) \in \mathbb{Z}^{1+m}} \left( \sum_{(n',k') \in \mathbb{Z}^{1+m}} \left\| \mathcal{H}(n',k') \right\|^{\frac12} \cdot \left\| \mathcal{H}(n', k') \right\|^{\frac12} \|\hat{u}(n-n', k-k')\| \right)^2       \nonumber \\     
    & \leq \left( \sum_{(n',k') \in \mathbb{Z}^{1+m}} \left\| \mathcal{H}(n',k') \right\| \right) \cdot \left( \sum_{(n,k) \in \mathbb{Z}^{1+m}} \sum_{(n',k') \in \mathbb{Z}^{1+m}} \left\| \mathcal{H}(n', k') \right\| \|\hat{u}(n - n', k-k')\|^2 \right)                                                                                                                     \nonumber  \\
    & = \bigg[ \sum_{ (n, k) \in \mathbb{Z}^{1+m}} \|\mathcal{H}(n, k)\| \bigg]^2 \| \hat{u} \|^2,       \label{eqn: h1-bound}
\end{align}
where all sums are taken over $\mathbb{Z}^{1+m}$, and the last equality follows 
from translation invariance of the lattice sums. 
The Gevrey decay estimate ensures that the kernel is absolutely summable, 
uniformly for $\hat{u} \in \mathcal{K}(s)$. 
Thus,~\eqref{eqn: h1-bound} yields a uniform bound for the operator norm of 
$\mathcal{H}$. 
Finally, $\partial \hat{X}_q / \partial \hat{u}_p$, $S$, and 
$\partial \hat{X} / \partial \hat{u}_p$ are obtained from $\mathcal{H}$ by 
restricting its input and/or output indices to the corresponding resonant or 
non-resonant sets. 
Since these restrictions are implemented by orthogonal projections, their 
operator norms do not exceed that of $\mathcal{H}$. 
This establishes~\eqref{eqn: tangent-bound} and completes the proof.

\subsection{Proof of~\Cref{prop: tensor}}
\label{subsec: tensor}

Following the approach used for the vector field and its tangent operator, we 
consider the second derivative $\partial^2 \hat{X} / \partial \hat{u}^{2}$ as a 
bilinear operator. 
For any $(n, k), (n', k'), (n'', k'') \in \mathbb{Z}^{1+m}$, its entries can be 
expressed as 
\begin{align*}
    \frac{\partial^2 \hat{X}(n, k)}{ \partial \hat{u}(n', k') \partial \hat{u}(n'', k'')} 
     & = \mathcal{T}(n - n' - n'', k - k' - k'') + \mathcal{T}(n - n' + n'', k - k' + k'')   \\ 
     & = \frac{4}{(2\pi)^{1+m}} \iint_{\mathbb{T}^{1+m}} \mathrm{Re} \left( u e^{-i (\langle k - k' - k'', \theta \rangle + (n - n' - n'')x) } \right)d\theta dx \\
     & \mathrel{\phantom{=}} + \frac{4}{(2\pi)^{1+m}} \iint_{\mathbb{T}^{1+m}} \mathrm{Re} \left( \overline{u} e^{-i (\langle k - k' + k'', \theta \rangle + (n - n' + n'')x) }  \right) d\theta dx.   
\end{align*}
Since $\hat{u} \in \mathcal{K}(s)$, the kernel $\mathcal{T}$ and its frequency 
derivatives satisfy a uniform Gevrey decay estimate. 
In particular, there exists a constant $\gamma_9 = \gamma_9(m, s, a, \hat{u}_p) > 0$ 
such that 
\[
    \sup_{(n, k), (n', k'), (n'', k'') \in \mathbb{Z}^{1+m}}\bigg( \| \mathcal{T}(n - n' \pm n'', k - k' \pm k'') \| e^{(|n - n' \pm n''| + | k - k' \pm k'' |)^s} \bigg) \leq \gamma_9.
\] 
Thus, each kernel contribution to the second derivative satisfies a uniform 
Gevrey decay estimate in its corresponding combined index.

To estimate the associated bilinear operators, let $\hat{u}_1$ and $\hat{u}_2$ 
be supported within $\Lambda_N^{1+m}$. 
Taking $\mathcal{T}(n - n' - n'', k - k' - k'')$ as a representative example, 
we apply the Cauchy-Schwarz inequality to derive the following bound as
\begin{align*}
    & \| \mathcal{T}( \hat{u}_1, \hat{u}_2 ) \|^2 \\
    & = \sum_{(n,k) \in \mathbb{Z}^{1+m}} \bigg\| \sum_{(n'', k'') \in \Lambda_N^{1+m}} \sum_{(n', k') \in \Lambda_N^{1+m}} \mathcal{T}(n - n' - n'', k - k' - k'') \hat{u}_1(n', k') \hat{u}_2(n'', k'' ) \bigg\|^2 \\
    & \leq \sum_{(n, k) \in \mathbb{Z}^{1+m}} \left[ \sum_{(n'', k'') \in \Lambda_N^{1+m}} \sum_{(n', k') \in \Lambda_N^{1+m}} \left\| \mathcal{T}(n - n' - n'', k - k' - k'') \right\| \right] \\
    & \mathrel{\phantom{\leq\leq=}} \cdot \left[ \sum_{(n'', k'') \in \Lambda_N^{1+m}} \sum_{(n', k') \in \Lambda_N^{1+m}}\left\| \mathcal{T}(n - n' - n'', k - k' - k'') \right\| \|\hat{u}_1(n',k')\|^2 \|\hat{u}_2(n'',k'')\|^2 \right]  \\
    & \leq \left[ \sum_{(n'', k'') \in \Lambda_N^{1+m}} \sum_{(n', k') \in \Lambda_N^{1+m}} \gamma_9 e^{-( |n - n' - n''| + |k - k' - k''|)^s} \right]  \\
    & \mathrel{\phantom{\leq\leq=}} \cdot \left[ \sum_{(n, k) \in \mathbb{Z}^{1+m}}  \sum_{(n'', k'') \in \Lambda_N^{1+m}} \sum_{(n', k') \in \Lambda_N^{1+m}} \gamma_9 e^{-( |n - n' - n''| + |k - k' - k''|)^s} \|\hat{u}_1(n', k')\|^2 \|\hat{u}_2(n'', k'')\|^2 \right] \\
    & = 2\gamma_9^2 (2N+1)^{1 + m} \left[ \sum_{(n, k) \in \mathbb{Z}^{1+m}} e^{-(|n|+|k|)^s} \right] \cdot \|\hat{u}_1\|^2 \|\hat{u}_2\|^2.                                                               
\end{align*}
The remaining kernel $\mathcal{T}(n - n' + n'', k - k' + k'')$ satisfies the 
same bound. 
Combining the bounds for the kernel contributions, we obtain
\[
    \left\| \frac{\partial^2 \hat{X}}{ \partial \hat{u}^2} \right\| \leq 2 \sqrt{2} \gamma_9 (2N+1)^{\frac{1+m}{2}} \left[ \sum_{(n, k) \in \mathbb{Z}^{1+m}} e^{-(|n|+|k|)^s} \right]^{\frac12} < \infty.  
\]
Finally, these operators, $\partial^2 \hat{X}_q / \partial \hat{u}_p^2$, 
$\partial S/\partial \hat{u}_p$, and $\partial^2 \hat{X} / \partial \hat{u}_p^2$, 
are obtained from $\partial^2 \hat{X} / \partial \hat{u}^{2}$ by restricting 
its input and/or output indices to the corresponding resonant or non-resonant 
sets. 
Since these restrictions are implemented by orthogonal projections, their 
operator norms do not exceed that of $\partial^2 \hat{X} / \partial \hat{u}^{2}$. 
This establishes~\eqref{eqn: tensor} and completes the proof.

\subsection{Proof of~\Cref{prop: B-operator}}
\label{subsec: b-operator}

Recalling the definition of $B$ in~\eqref{eqn: operator-b}, we obtain the 
following estimate for any $(n,k), (n', k') \notin \mathcal{S}$:
\[
    \left\| B \left( (n, k), (n', k') \right) \right\| \leq \frac{1}{\delta_0^2} \left\| \frac{\partial \langle k, \hat{X}_q \rangle}{\partial \hat{u}_p} \hat{u}_{p}^{\top} \right\|. 
\]
The Gevrey decay of $\hat{u} \in \mathcal{K}(s)$, together with the estimates 
for the tangent operator established in~\Cref{prop: convolution}, therefore 
yields the Gevrey bound in~\eqref{eqn: b-comp}. 
Using the rank-one structure of $B$ and applying the Cauchy-Schwarz inequality, 
we further obtain
\[
    \left\| B \left( (n, k), (n', k') \right) \right\| \leq \frac{1}{\delta_0^2} \| k \hat{u}_p(n, k) \| \cdot \left\| \frac{\partial \hat{X}_q}{\partial \hat{u}_p(n', k') } \right\|.
\] 
By utilizing the Gevrey decay of the tangent operators established 
in~\Cref{subsec: convolution} and extending the sums over the lattice 
$\mathbb{Z}^{1+m} \times \mathbb{Z}^{1+m}$, we obtain
\[
    \| B \| \leq \sum_{(n, k) \in \mathbb{Z}^{1+m}} \left( \frac{\| k \hat{u}_p(n, k) \|}{\delta_0^2} \right) \cdot \sum_{(n', k') \in \mathbb{Z}^{1+m}} \left\| \frac{\partial \hat{X}_q}{\partial \hat{u}_p(n', k') } \right\|.
\]
Finally, since $\hat{u} \in \mathcal{K}(s)$,~\Cref{prop: vector-field} 
and~\Cref{prop: convolution} ensure that these lattice sums converge and are 
uniformly bounded, thereby establishing~\eqref{eqn: b-bound}.

\section{Proofs of results in~\Cref{sec: implement}}
\label{sec: implement-app}

In this appendix, we provide the detailed proofs 
of~\Cref{lem: near-measure-initial} and~\Cref{thm: small-size} stated 
in~\Cref{sec: implement}.

\subsection{Proof of~\Cref{lem: near-measure-initial}}
\label{sec: near-measure-initial}

Recall that the frequency set $\Omega \subset \mathbb{R}^m$ is bounded, and set 
$d = \mathrm{diam}(\Omega) = \sup_{\omega'_1, \omega'_2 \in \Omega} \|\omega'_1 - \omega'_2\|_2$ 
and $R = \sup_{\omega' \in \Omega} \|\omega' \|_2$. 
For each non-zero temporal index $k \in \mathbb{Z}^m \setminus \{0\}$, the 
corresponding nearly-resonant strip lies between two parallel hyperplanes 
orthogonal to $k$. 
Estimating the measure of its intersection with $\Omega$ and summing over all 
nonzero $k$, we obtain the following bound for the tangential nearly-resonant 
set defined in~\eqref{eqn: tangential-nearly-resonant}: 
\begin{equation}
    \label{eqn: tangent-measure-estimate}
    \mathrm{mes}\left( \Omega_{0,M}^{\tau} \right) \leq \sum_{k \in \mathbb{Z}^m \setminus \{0\}} \frac{2 d^{m-1}}{\|k\|_2^{\tau+1}}. 
\end{equation}

We next estimate the measure of the single-mode nearly-resonant set 
$\Omega_{1,M}^{\tau}$ defined in~\eqref{eqn: single-nearly-resonant}. 
If the nearly-resonance condition holds for some $\omega' \in \Omega$, then 
$\left|\langle k, \omega' \rangle - n^2 - \mu \right| < 1$, which implies 
$n^2 < 1 + |\mu| + R\|k\|_2 $. 
Hence, for each fixed $k$, the number of admissible spatial indices $n$ is 
bounded by $2 \sqrt{ 1 + |\mu| + R\|k\|_2 } + 1 \leq 3 \sqrt{ 1 + |\mu| + R\|k\|_2 }$. 
Combining this counting estimate with the strip-measure estimate yields
\begin{equation}
    \label{eqn: single-measure-estimate}
    \mathrm{mes}\left( \Omega_{1,M}^{\tau} \right) 
    \leq \sum_{k \in \mathbb{Z}^m} \frac{6 d^{m-1} \sqrt{1 + |\mu| + R \|k\|_2}}{(1 + \|k\|_2)^{\tau+1}} 
    \leq \sum_{k \in \mathbb{Z}^m} \frac{6 d^{m-1} \max\big\{ \sqrt{1 + |\mu|}, \sqrt{R} \big\}}{(1 + \|k\|_2)^{\tau+\frac12}} .
\end{equation}

Finally, consider the difference nearly-resonant set $\Omega_{2,M}^{\tau}$ 
defined in~\eqref{eqn: difference-nearly-resonant}. 
The nearly-resonance condition implies $\left|\langle k, \omega' \rangle - n_1^2 + n_2^2 \right| < 1$, 
and hence $\left|n_1^2 - n_2^2 \right| < 1 + R\|k\|_2 $. 
For each fixed $k \in \mathbb{Z}^m$, the corresponding strip depends on $n_1$ 
and $n_2$ only through the integer difference $n_1^2 - n_2^2$. 
It therefore suffices to count the distinct admissible values of this 
difference rather than the pairs $(n_1, n_2)$. 
Combining this count with the strip-measure estimate gives 
\begin{equation}
    \label{eqn: diff-measure-est}
    \mathrm{mes}\left( \Omega_{2,M}^{\tau} \right) 
    \leq \sum_{k \in \mathbb{Z}^m} \frac{ 4d^{m-1} \left( 1 + R\|k\|_2 \right) }{(2 + \|k\|_2)^{\tau+1}} \leq \sum_{k \in \mathbb{Z}^m} \frac{2d^{m-1} \max\{1, 2R\}}{(2 + \|k\|_2)^{\tau}}.
\end{equation}

Since $\tau > m$, the lattice sums 
in~\eqref{eqn: tangent-measure-estimate}--\eqref{eqn: diff-measure-est} converge. 
Combining these estimates with the union 
representation~\eqref{eqn: nls-total-nearly-resonant}, we obtain the desired 
bound~\eqref{eqn: near-measure-initial}. 
This completes the proof.

\subsection{Proof of~\Cref{thm: small-size}}
\label{sec: small-size-original}

Before proving~\Cref{thm: small-size}, we establish an auxiliary lemma showing 
that Gevrey decay of matrix entries is preserved under operator multiplication. 
For a fixed exponent $0<s<1$, we say that an operator 
$A: \ell_2(\mathbb{Z}^{1+m}) \to \ell_2(\mathbb{Z}^{1+m})$ has 
\textit{proper decay} if its entries satisfy: 
\[
    \sup_{ (n, k), (n', k') \in \mathbb{Z}^{1+m}} \left( \left| A\left( (n, k), (n', k') \right) \right| e^{ (|n-n'| + |k-k'|)^s } \right) \leq 1.
\]

\begin{lemma}
    \label{lem: app-3}
    Let $A_1$ and $A_2$ be two operators with proper decay. 
    Then there exists a constant $C'_2(m, s) > 0$ such that their product 
    $\mathcal{A} = A_1 A_2$ satisfies
    \begin{equation}    
        \label{eqn: convolution-mix}   
        \sup_{(n, k), (n', k') \in \mathbb{Z}^{1+m}} \left( \left| \mathcal{A}\left( (n, k), (n', k') \right) \right| e^{ (|n-n'| + |k-k'|)^s } \right) \leq C'_2(m, s).
    \end{equation}
    More generally, for any finite collection of operators 
    $\{ A_j \}_{j=1}^{\ell}$ with proper decay, there exists a constant 
    $C'_{\ell}(m, s) > 0$ such that the $\ell$-fold product 
    $\mathcal{A} = \prod_{j=1}^{\ell} A_j$ satisfies: 
    \begin{equation}    
        \label{eqn: p-convolution-mix}   
        \sup_{(n, k), (n', k') \in \mathbb{Z}^{1+m}} \left( \left| \mathcal{A}\left( (n, k), (n', k') \right) \right| e^{ (|n-n'| + |k-k'|)^s } \right) \leq C'_{\ell}(m, s).
    \end{equation}   
\end{lemma}

\begin{proof}
    Given that both $A_1$ and $A_2$ have proper decay, the entries of their 
    product satisfy
    \begin{align*}
       \left |\mathcal{A}\left( (n, k), (n', k') \right) \right | 
        &\leq \sum_{(n'', k'') \in \mathbb{Z}^{1+m}} \left| A_1\left( (n,k), (n'',k'') \right) \right| \left| A_2 \left((n'',k''), (n',k') \right) \right| \\
        &\leq \sum_{(n'', k'') \in \mathbb{Z}^{1+m}} e^{ - (|n-n''| + |k-k''|)^s - (|n''-n'| + |k''-k'|)^s }  \\
        &\leq e^{ - (|n-n'| + |k-k'|)^s } \sum_{(n'', k'') \in \mathbb{Z}^{1+m}} e^{ - (2 - 2^s) \min\left\{ (|n-n''| + |k-k''|)^s, \; (|n''-n'| + |k''-k'|)^s \right\} },
    \end{align*}
    where the last inequality follows from~\Cref{lem: app-1}. 
    The final series converges, and its sum is bounded by a constant depending 
    only on $m$ and $s$; this establishes~\eqref{eqn: convolution-mix}. 
    Applying the same estimate successively yields~\eqref{eqn: p-convolution-mix} 
    by induction on $\ell$.
\end{proof}

We now complete the proof of~\Cref{thm: small-size}. 
The uniform lower bound~\eqref{eqn: nls-uniform-lower-original-small} ensures 
that the diagonal operator $D_{n_0, N}$ is invertible, so the restricted 
operator $L_{n_0, N}$ admits the factorization 
\[
    L_{n_0, N} = D_{n_0, N} + \varepsilon (S_{n_0, N} + B_{n_0, N}) = D_{n_0, N} \left[ I + \varepsilon D_{n_0, N}^{-1}(S_{n_0, N} + B_{n_0, N}) \right].
\]
By~\eqref{eqn: nls-uniform-lower-original-small} together 
with~\Cref{prop: convolution} and~\Cref{prop: B-operator}, we have 
\[
    \varepsilon \| D_{n_0, N}^{-1}(S_{n_0, N} + B_{n_0, N}) \|_2 \leq \frac12.
\]
Consequently, $L_{n_0, N}$ is invertible and its inverse is given by the 
Neumann series
\begin{align}
    L_{n_0, N}^{-1} & = \left[ I + \varepsilon D_{n_0,N}^{-1}(S_{n_0,N}+B_{n_0,N}) \right]^{-1} D_{n_0,N}^{-1}  \nonumber \\
                    & = \left\{ \sum_{\ell=0}^{\infty} \left[ - \varepsilon D_{n_0,N}^{-1}(S_{n_0,N}+B_{n_0,N}) \right]^{\ell} \right\} D_{n_0, N}^{-1},     \label{eqn: T-nu-expansion}
\end{align}
which converges in operator norm. 
It follows that $\sum_{\ell =0}^{\infty} \| \varepsilon D_{n_0, N}^{-1}(S_{n_0, N} + B_{n_0,  N})\|_2^{\ell} \leq 2$, 
and hence $\| L_{n_0, N}^{-1}  \|_2 \leq 2\|D_{n_0, N}^{-1} \|_2 $. 
Together with~\eqref{eqn: nls-uniform-lower-original-small}, this yields the 
norm estimate~\eqref{eqn: small-size-inverse}.

To establish the off-diagonal decay, we separate the $\ell=0$ term from the 
Neumann series~\eqref{eqn: T-nu-expansion}: 
\[
    L_{n_0, N}^{-1} = D_{n_0, N}^{-1} - \left[ \sum_{\ell=0}^{\infty} \left( - \varepsilon D_{n_0, N}^{-1}(S_{n_0, N} + B_{n_0, N}) \right)^{\ell} \right] \left(\varepsilon D_{n_0, N}^{-1}(S_{n_0, N}+B_{n_0, N}) D_{n_0, N}^{-1}\right).
\]
The Gevrey decay of $S$ and $B$ established in~\Cref{prop: convolution} 
and~\Cref{prop: B-operator}, combined 
with~\eqref{eqn: nls-uniform-lower-original-small} and~\Cref{lem: app-3}, imply 
that, upon further reducing $\varepsilon$ if necessary, each summand satisfies 
\[
    \left| \left( - \varepsilon D_{n_0, N}^{-1} (S_{n_0, N}+B_{n_0, N}) \right)^{\ell} \left( (n, k), (n', k') \right) \right| \leq \left(\frac{1}{2}\right)^{\ell+1} e^{ - (|n-n'|+|k-k'|)^s },
\]
for any $\ell \geq 0$ and any $(n, k), (n', k') \in (n_0, 0) + \Lambda_N^{1+m}$. 
Since $D_{n_0, N}^{-1}$ is diagonal with the 
bound~\eqref{eqn: nls-uniform-lower-original-small}, summing these bounds over 
$\ell \geq 0$ yields the off-diagonal decay 
estimate~\eqref{eqn: small-size-localization}. 
This completes the proof of~\Cref{thm: small-size}.

\section{Proofs of results in~\Cref{sec: iterative-scheme}}
\label{sec: completion-app}

In this appendix, we provide the detailed proof 
for~\Cref{lem: inverse-derivative}.

\subsection{Proof of~\Cref{lem: inverse-derivative}}
\label{subsec: derivative-app}

From~\Cref{prop: convolution} and~\Cref{prop: B-operator}, it follows that for 
any $(n, k) \neq (n', k')$, the off-diagonal entries of $\partial L_N$ satisfy
\[
    |\partial L_N \left( (n,k), (n', k') \right)| \leq e^{ - \left( |n-n'| + |k-k'| \right)^{s} }.
\]
In view of the operator identity in~\eqref{eqn: t-n-derivative-inverse}, each 
entry of $\partial L_{N}^{-1}(k, k')$ can be expanded as the double sum 
\begin{multline*}
    \partial L_{N}^{-1}\left( (n,k), (n', k') \right) \\ 
    = - \sum_{(n_1, k_1) \in \Lambda_N^{1+m}} \sum_{(n_2,k_2) \in \Lambda_N^{1+m}}L_{N}^{-1}\left( (n,k), (n_1, k_1) \right) \partial L_N\left( (n_1,k_1), (n_2, k_2) \right) L_{N}^{-1}\left( (n_2,k_2), (n', k') \right).
\end{multline*}
Applying the triangle inequality and inserting the bounds of~\Cref{lem: app-3}, 
we obtain: 
\begin{align*}
    & \left| \partial L_{N}^{-1}\left( (n,k), (n', k') \right) \right| \\ 
    & \leq \sum_{(n_1, k_1) \in \Lambda_N^{1+m}} \sum_{(n_2,k_2) \in \Lambda_N^{1+m}} |L_{N}^{-1}\left( (n,k), (n_1, k_1) \right)| \cdot | \partial L_N\left( (n_1,k_1), (n_2, k_2) \right)| \cdot |L_{N}^{-1}\left( (n_2,k_2), (n', k') \right)| \\
    & \leq \frac{4N^2 (2N+1)^{2m+2}}{\varepsilon_N^2} e^{ - \frac{\left( |n-n'| + |k - k'| - 2N^{\frac12} \right)^{s}}{2} } 
\end{align*}
Since the lattice sites are separated by $|n-n'| + |k-k'| \geq N^{\frac34}$, 
the prefactor is absorbed by the exponential, and the off-diagonal 
entries of $\partial L_{N}^{-1}$ satisfy the Gevrey 
bound~\eqref{eqn: inverse-restricted-entries}. 
This thus completes the proof.

\end{document}